\documentclass[10pt]{article}
\usepackage{graphicx}
\usepackage{dsfont}
\usepackage{color,graphics,graphicx}
\usepackage{mathtools}
\usepackage{bbm}
\usepackage{amssymb}
\usepackage{amsthm} 
\usepackage[square,sort,comma,numbers]{natbib}
\usepackage[font=footnotesize,labelfont=bf]{caption}
\usepackage{hyperref}

\usepackage[margin=1in]{geometry}

\newcommand\Pp{\mathbb P}
\newcommand\E{\mathbb E}

\newcommand\Z{\mathbb Z} 
\newcommand\R{\mathbb R} 
\newcommand\N{\mathbb N} 
\newcommand{\abs}[1]{\ensuremath{\left\vert#1\right\vert}}
\newcommand\Cov{\mathbb C\textnormal{ov}}
\DeclareMathOperator{\id}{id}
\DeclareMathOperator{\rv}{RV_{\hspace*{-.05cm}n}}
\newcommand*\diff{\mathop{}\!\mathrm{d}}
\newcommand\too{\rightarrow}
\newcommand\tooi{\rightarrow\infty}
\newcommand{\nor}[1]{\lVert#1\rVert}
\newcommand\Oo{\mathcal O} 
\newcommand\RV{\textnormal{RV}_{n}}
\newcommand\Var{\mathbb V\textnormal{ar}}
\newcommand{\euc}[1]{\left\lVert#1\right\rVert_2}
\newcommand\idn{i\Delta_n} 
\newcommand\iidn{(i-1)\Delta_n} 
\newcommand\ito{\text{It\^{o}}}
\newcommand\im{\mathrm{i}}
\newcommand\RVt{\text{RV}_{2n}}

\makeatletter
\newcommand*\bigcdot{\mathpalette\bigcdot@{.5}}
\newcommand*\bigcdot@[2]{\mathbin{\vcenter{\hbox{\scalebox{#2}{$\m@th#1\bullet$}}}}}
\makeatother

\newcommand\smallO{
  \mathchoice
    {{\scriptstyle\mathcal{O}}}
    {{\scriptstyle\mathcal{O}}}
    {{\scriptscriptstyle\mathcal{O}}}
    {\scalebox{.7}{$\scriptscriptstyle\mathcal{O}$}}
}

\newcommand\EatDot[1]{.}
\newcommand\KillDot[1]{}

\providecommand{\msc}[1]{\textbf{MSC Classification:} #1}

\newtheorem{theorem}{Proposition}
\newtheorem{lemma}{Lemma}

\newtheorem{assump}{Assumption}
\newtheorem{cor}{Corollary}
\newtheorem{example}{Example}

\newenvironment{keywords}{}{}
\allowdisplaybreaks

\begin{document}

\title{Parameter estimation for second-order SPDEs in multiple space dimensions}
\author{Patrick Bossert \\
	Institute of Mathematics \\
	Julius-Maximilians-Universit\"at W\"urzburg\\
	Würzburg, 97074, Germany \\
	\texttt{patrick.bossert@uni-wuerzburg.de}
}

\newpage
\maketitle
\begin{abstract}
We analyse a second-order SPDE model in multiple space dimensions and develop estimators for the parameters of this model based on discrete observations of a solution in time and space on a bounded domain. 
While parameter estimation for one and two spatial dimensions was established in recent literature, this is the first work which generalizes the theory to a general, multi-dimensional framework.
Our approach builds upon realized volatilities, enabling the construction of an oracle estimator for volatility within the underlying model. Furthermore, we show that the realized volatilities have an asymptotic illustration as response of a log-linear model with spatial explanatory variable. This yields novel and efficient estimators based on realized volatilities with optimal rates of convergence and minimal variances. For proving central limit theorems, we use a high-frequency observation scheme. To showcase our results, we conduct a Monte Carlo simulation.
\end{abstract}
\keywords{Central limit theorem under dependence, High-frequency data, Least squares estimation, Multi-dimensional SPDEs}\ \\
\msc{62F12, 62M10, 60H15}




\section{Introduction}\label{sec1}
Multidimensional stochastic partial differential equations (SPDEs) expand upon the principles of their one-dimensional counterparts to address scenarios involving multiple spatial dimensions. These equations find application across diverse scientific domains, enabling the exploration of the interplay between deterministic dynamics and stochastic fluctuations in systems spanning physics, geophysics, biology, finance, and environmental science.

Recent interest in applications of one-dimensional SPDEs and statistical methods to calibrate them is evident in the works of \cite{hambly}, \cite{fuglstad}, \cite{randolf2021}, and \cite{randolf2022} Notably, researchers have leveraged power variations, a concept well-established in financial high-frequency settings, to develop statistical inference methodologies, as evidenced by works like \cite{trabs}, \cite{cialenco}, and \cite{chong2020high}.

Multi-dimensional SPDE models, on the other hand, offer a much larger variability for modelling natural phenomena. Therefore, it is intuitive that applications of these SPDEs is also of great relevance, especially for two- and three-dimensional spaces. See, for instance, \cite{mena2019efficient} for an application in connection with the climate phenomenon El Ni\~{n}o and references therein for applications to sea temperature, \cite{pereira2020matrix} for an application in Geostatistics and dealing with seismic data and \cite{fioravanti2023interpolating} for an application in climate science. For an overview with many references to specific applications in various fields we refer to \cite{lindgren2022spde}.

While power variations have received considerable attention in the context of one-dimensional SPDEs, their utilization for SPDEs in multiple spatial dimensions remains in its nascent stages. In a pivotal contribution, \cite{tonaki2023parameter} analysed a two-dimensional SPDE model, laying the groundwork for further research in the realm of second-order, linear multi-dimensional SPDEs. 
In this endeavor, we follow a theoretical framework related to \cite{tonaki2023parameter}, adapting it to accommodate multiple spatial dimensions.

Within this multi-dimensional model, we establish the foundation for parameter estimation by utilizing quadratic increments. Building upon a high-frequency assumption over a fixed time horizon, inspired by \cite{trabs}, along with a regularity assumption, we construct a volatility estimator tailored to the multi-dimensional SPDE model. Subsequently, we link these realized volatilities to a log-linear model to enhance our understanding of the system.

To facilitate empirical investigations, we develop a simulation methodology that extends the one-dimen\-sional counterpart known as the replacement method, as introduced by \cite{hildebrandt2020}, to multiple spatial dimensions. A brief overview of the notational conventions employed in this paper can be located at the beginning of Section \ref{sec_6}.

One challenging task when working within a multi-dimensional framework is dealing with the technical difficulties that arise during the transition from one to multiple space dimensions. For instance, when opting for the spectral approach, it becomes necessary to determine a Riemann sum approximation for sums with a multi-dimensional index set. In addition to these technical challenges, the multidimensional random field exhibits significant structural changes, affecting the dependencies within the model and the behaviour of the error terms.
As a result, the generalization from a single space dimension, as researched by \cite{trabs}, to multiple spatial dimensions is not straightforward and requires careful treatment. Consequently, we anticipate that our research will provide valuable insights into statistics for SPDEs in multiple spatial dimensions and contribute efficient estimators that leverage realized volatilities.

We consider the following linear second-order SPDE in $d\in\mathbb{N}$, $d\geq 2$, spatial dimensions with additive noise:
\begin{align}
\left[
\begin{array}{ll}
\diff X_t(\textbf{y}) = A_\vartheta X_t(\textbf{y})\diff t+\sigma\diff B_t(\textbf{y}), &(t,\textbf{y})\in[0,1]\times [0,1]^d\\
X_0(\textbf{y})=\xi(\textbf{y}),& \textbf{y}\in[0,1]^d\\
X_t(\textbf{y})=0,& (t,\textbf{y})\in[0,1]\times\partial\,[0,1]^d
\end{array}
\right],\label{generalSPDEequationMulti}
\end{align}
where $\textbf{y}=(y_1,\ldots,y_d)\in[0,1]^d$. The operator $A_\vartheta$ in equation \eqref{generalSPDEequationMulti} is given by:
\begin{align}
A_\vartheta = \eta \sum_{l=1}^d \frac{\partial}{\partial y_l^2}+\sum_{l=1}^d \nu_l\frac{\partial}{\partial y_l}+\vartheta_0,\label{eqn_AvarthetaMulti}
\end{align}
with fixed parameters $\vartheta=(\vartheta_0,\nu_1,\ldots,\nu_d,\eta)$, where $\vartheta_0,\nu_1,\ldots,\nu_d\in\mathbb{R}$ and $\eta,\sigma>0$.
In the temporal domain, we consider the interval $t\in[0,1]$, which can be extended to $t\in[0,T]$ for $T>0$, while the spatial domain encompasses the $d$-dimensions unit hypercube. Additionally, we introduce $B$ as a cylindrical $Q$-Brownian motion defined over $[0,1]^d$, where the initial condition $\xi:[0,1]^d\rightarrow\mathbb{R}$ is independent to $(B_t)$.
We adopt Dirichlet boundary condition such that $X_t(\textbf{y})=0$ holds for all $(t,\textbf{y})\in[0,1]\times\partial\,[0,1]^d$.
Furthermore, we introduce the \emph{curvature} parameter, expressed as $\kappa=(\kappa_1,\ldots,\kappa_d)$, where $\kappa_l:=\nu_l/\eta\in\R$ and $l=1,\ldots,d$. Additionally, we define the \emph{normalized volatility}, denoted as $\sigma_0^2:=\sigma^2/\eta^{d/2}>0$, where the parameter $\sigma$ is called \emph{volatility}.
\ \\ 
One well-established example of a linear, second-order SPDE is the heat equation. The stochastic version of this equation in $d$-dimensions can be represented as follows:
\begin{align*}
\left[
\begin{array}{ll}
\diff X_t(\textbf{y}) =\eta \sum_{l=1}^d \frac{\partial}{\partial y_l^2}  X_t(\textbf{y})\diff t+\sigma\diff B_t(\textbf{y}), &(t,\textbf{y})\in[0,1]\times [0,1]^d\\
X_0(\textbf{y})=\xi(\textbf{y}),& \textbf{y}\in[0,1]^d\\
X_t(\textbf{y})=0,& (t,\textbf{y})\in[0,1]\times\partial\,[0,1]^d
\end{array}
\right],
\end{align*}
where $\sigma$ governs the degree of randomness in the cooling process, while $\eta$ serves as a parameter denoting thermal conductivity. In the context of the one-dimensional heat equation, this equation models the cooling process of objects like rods or thin entities. However, when extended to two or three dimensions, it characterizes the cooling process of broader surfaces or spatial volumes, with potential applications ranging from modelling the cooling of plate-like structures to the temperature dynamics of large bodies of water such as the sea surface or seawater. In all applications, the initial condition reflects the starting temperature of the object or system.

Given the prevalence of multi-dimensional models in various natural phenomena, especially when influenced by random factors, the analysis of these multi-dimensional equations becomes particularly intriguing and relevant.

\section{Probabilistic structure and statistical setup}
To investigate the multi-dimensional SPDE model introduced in \eqref{generalSPDEequationMulti}, we opt for the spectral approach. 
Different to the situation with unbounded spatial support, the differential operator $A_{\vartheta}$ in \eqref{eqn_AvarthetaMulti} admits a discrete spectrum.
Hence, the spectral approach involves representing the random field as a sum of orthogonal eigenfunctions weighted by stochastic coefficients. This methodology is widely adopted by researchers in the field, see, for instance \cite{lototsky2017stochastic}, \cite{uchida}, or \cite{hildebrandt}. 
Moreover, we extend the two-dimensional approach presented by \cite{tonaki2023parameter}.

In the context of the spectral approach, the corresponding Hilbert space is defined as follows:
\begin{align}
H_\vartheta:=\{f:[0,1]^d\too\R, \nor{f}_\vartheta<\infty \text{ and }f(\textbf{y})=0,\text{ for }\textbf{y}\in\partial\,[0,1]^d\}.\label{eqn_HilbertsapceHvarthetaMulti}
\end{align}
The norm $\nor{\cdot}_\vartheta$ is defined via the corresponding inner product $\nor{f}_\vartheta:=\langle f,f\rangle_\vartheta$ for $f\in H_\vartheta$, given by
\begin{align*}
\langle f,g\rangle_\vartheta:=\int_0^1\cdots\int_0^1f(y_1,\ldots,y_d)g(y_1,\ldots,y_d)\exp\bigg[\sum_{l=1}^d\kappa_ly_l\bigg]\diff y_1\cdots\diff y_d,
\end{align*}
where $f,g\in H_\vartheta$. The domain of the operator $A_\vartheta$ is defined as follows:
\[\mathcal{D}(A_\vartheta)=\{f\in H_\vartheta:\nor{f}_\vartheta,\nor{\partial/(\partial y_l)f}_\vartheta,\nor{\partial^2/(\partial y^2_l)f}_\vartheta<\infty, \text{ for all }l=1,\ldots,d\}.\]
The decomposition of the operator $A_\vartheta$ from \eqref{eqn_AvarthetaMulti} results in the
eigenfunctions $(e_\textbf{k})_{\textbf{k}\in\N^d}$ and eigenvalues $(-\lambda_\textbf{k})_{\textbf{k}\in\N^d}$, given by
\begin{align}
e_\textbf{k}(\textbf{y})&:=e_\textbf{k}(y_1,\ldots,y_d):=2^{d/2}\prod_{l=1}^d\sin(\pi k_ly_l)e^{-\kappa_ly_l/2},\label{eqn_ekMulti}\\
\lambda_\textbf{k}&:=-\vartheta_0+\sum_{l=1}^d\bigg(\frac{\nu_l^2}{4\eta}+\pi^2k_l^2\eta\bigg),\label{eqn_lambda_kMulti}
\end{align}
where $\textbf{k}=(k_1,\ldots,k_d)^\top\in\N^d$. 
When comparing the representation of the eigenfunctions and eigenvalues in $d$-space dimensions to those in one space dimension, as presented in \cite{trabs}, we observe that we have extended the eigenfunctions and eigenvalues in one dimension to each spatial dimension. 
The orthonormal property of the eigenfunctions in the one-dimensional case seamlessly extends to the multidimensional setting, effectively defining an orthonormal system denoted as $(e_\textbf{k})_{\textbf{k}\in\N^d}$. This observation allows us to independently decompose each spatial axis using the one-dimensional eigenfunctions, which incorporate rescaling, sine functions, and exponential terms with dependencies on the respective parameters $\kappa_l$, $l=1,\ldots,d$.
As a consequence, we derive a spectral decomposition by considering a product model over each dimension.

Moreover, the operator $A_\vartheta$ is self-adjoint within the Hilbert space $H_\vartheta$. This self-adjoint property carries significant importance, as it guarantees that the eigenfunctions collectively form a complete and orthogonal basis within $H_\vartheta$. This, in turn, empowers us to effectively represent solutions to the SPDE model outlined in equation \eqref{generalSPDEequationMulti} using a spectral decomposition.
Given that these properties can be readily deduced through standard calculations, we will forego presenting their proofs.

%

We address the $Q$-Wiener process, denoted as $W_t(y)$, within a Sobolev space defined over the bounded domain $[0,1]^d$. For a comprehensive understanding of $Q$-Wiener processes, readers are encouraged to consult references such as \cite{daPrato} or \cite{lototsky2017stochastic}.

A crucial distinction that arises when transitioning from one spatial dimension to higher dimensions is that the random filed, denoted as $X_t(\textbf{y})$, is not square integrable when considering a white noise, i.e., $\E[||X_t^{\id}||_\vartheta^2]= \infty$, where $Q=\text{id}$ denotes the identity operator. Remarkably, this phenomenon persists even in two spatial dimensions, as demonstrated by the authors in \cite{tonaki2023parameter}. To rectify this issue and ensure finite variance of the paths, it becomes necessary to employ a coloured cylindrical Wiener process instead of a white noise. This implies the introduction of an additional parameter into the model, effectively damping the Wiener process and resulting in the random field being square integrable.

In our model, we incorporate a $Q$-Wiener process to account for the stochastic noise. By implementing damping mechanisms, \cite{tonaki2023parameter} have successfully devised statistical inference techniques based on high-frequency observations, leveraging a spectral approach within the context of two spatial dimensions.

To specify the damping mechanism in our model, we adopt one natural approach by defining $(B_t)_{t\geq 0}$ as follows:
\begin{align}
\langle B_t,f\rangle_\vartheta := \sum_{\textbf{k}\in\N^d} \lambda_{\textbf{k}}^{-\alpha/2}\langle f,e_\textbf{k}\rangle_\vartheta W_t^\textbf{k},\label{eqn_QWienerProcessMulti}
\end{align}
for $f\in H_\vartheta, t\geq 0$ and independent real-valued Brownian motions $(W_t^\textbf{k})_{t\geq 0}$, $\textbf{k}\in\N^d$. 
In the preceding definition, the cylindrical Brownian motion $B$ experiences a structural transformation due to the inclusion of the term $\lambda_\textbf{k}^{-\alpha/2}$ in its spectral decomposition. This modification inherently results in a fundamental change in the probabilistic characteristics of the random field. The parameter $\alpha$ assumes special importance, as it essentially governs the Hölder regularity of the marginal processes of the random field. In the context of one spatial dimension, a related parameter was examined by \cite{chong2020high}, where it was referred to as ``spatial correlation parameter''.
Moreover, we assume that $\alpha\in (d/2-1,d/2)$, then it can be seen that
\begin{align}
\sum_{\textbf{k}\in\N^d}\nor{A_\vartheta^{-(1+\alpha)/2}e_\textbf{k}}_\vartheta^2=\sum_{\textbf{k}\in\N^d}\frac{1}{\lambda_\textbf{k}^{1+\alpha}}<\infty,\label{eqn_Bwelldefined}
\end{align}
which implies that the $Q$-Wiener process is well-defined. The operator $Q$ is then defined by
\begin{align*}
Qe_\textbf{k}=\lambda^{-\alpha}_\textbf{k}\nor{A_\vartheta^{-1/2}e_\textbf{k}}_\vartheta^2 e_\textbf{k},
\end{align*}
where the corresponding eigenvalues of $Q$ are given by $\lambda^{-\alpha}_\textbf{k}\nor{A_\vartheta^{-1/2}e_\textbf{k}}_\vartheta^2$, $\textbf{k}\in\N^d$.
Note, that the lower bound $\alpha> d/2-1$ is essential for the $Q$-Wiener process being well-defined, where the upper bound $\alpha< d/2$ serves the purpose of developing statistical inference.
For a comprehensive overview of Wiener processes on Hilbert spaces, we refer to \cite[Chapter 4]{daPrato}.
For readings on a different approach for the choice of $Q$ in two space dimensions we refer to \cite{tonaki2023parameter}.

Consider a mild solution of the SPDE model from equation \eqref{generalSPDEequationMulti}, which satisfies the integral representation:
\begin{align*}
X_t=e^{tA_\vartheta}\xi+\sigma\int_0^t e^{(t-s)A_\vartheta}\diff B_s
\end{align*}
a.s., for every $t\in[0,1]$. Then, the spectral decomposition of the random field $X_t$ is given by
\begin{align}
X_t(\textbf{y})=\sum_{\textbf{k}\in\N^d} x_\textbf{k}(t)e_\textbf{k}(\textbf{y}),~~~~\text{where}~~~~x_\textbf{k}(t):=\langle X_t,e_\textbf{k}\rangle_\vartheta.\label{eqn_fourierseriesMultiOfX}
\end{align}
The coordinate processes $(x_\textbf{k})$ follow the Ornstein-Uhlenbeck dynamics, governed by the equation:
\begin{align*}
\diff x_\textbf{k}(t)= -\lambda_\textbf{k}x_\textbf{k}(t)\diff t + \sigma\lambda_\textbf{k}^{-\alpha/2}\diff W_t^\textbf{k},
\end{align*}
for every $\textbf{k}\in\N^d$.
Utilizing that $A_\vartheta$ is self-adjoint yields the following representation for the coordinate processes: 
\begin{align}
x_\textbf{k}(t)=\langle X_t,e_\textbf{k}\rangle_\vartheta
&=e^{-\lambda_\textbf{k}t}\langle\xi,e_\textbf{k}\rangle_\vartheta+\sigma\lambda_\textbf{k}^{-\alpha/2}\int_0^t e^{-\lambda_\textbf{k}(t-s)}\diff W_s^\textbf{k},\label{eqn_coordProcessMulti}
\end{align}
where we used that $\langle e^{tA_\vartheta}f,e_\textbf{k}\rangle_\vartheta=e^{-\lambda_\textbf{k}t}\langle f,e_\textbf{k}\rangle_\vartheta$ for $f\in H_\vartheta$. In fact, by using the latter representation for $x_\textbf{k}(t)$, we can observe that the random field is square integrable, i.e.:
\begin{align*}
\E[\nor{X_t}_\vartheta^2]
&\leq C \sum_{\textbf{k}\in\N^d}\frac{1}{\lambda_\textbf{k}^{1+\alpha}}<\infty.
\end{align*}

For statistical inference, we establish a high-frequency observation scheme on a discrete grid in time and space. Similar to the one-dimensional case, it is essential to restrict the observations in order to bound correlations, which naturally arise in SPDE models. Therefore, we introduce the following mapping: 
\begin{align*}
\nor{\textbf{x}}_0:=\min_{\substack{i=1,\ldots,d \\ x_i\neq 0}} \{\abs{x_1},\ldots,\abs{x_d}\},
\end{align*}
where we set $\min\emptyset=0$ and consider the following Assumption to our model. 

\begin{assump}[Observation scheme]\label{assumption_observations_multi}
Suppose we observe a mild solution $X$ of the SPDE model from equation \eqref{generalSPDEequationMulti} on a discrete grid $(t_i,\textbf{y}_j)\in [0,1]\times[0,1]^d$, with equidistant observations in time $t_i=i\Delta_n$ for $i=1,\ldots,n$ and $\textbf{y}_1,\ldots,\textbf{y}_m\in[\delta, 1-\delta]^d$, where $n,m\in\N$ and $\delta\in(0,1/2)$. We consider the following asymptotic regime:
\begin{enumerate}
\item[(I)] $\Delta_n\too 0$ and $m=m_n\tooi$, as $n\tooi$, while $n\Delta_n=1$ and $m=\mathcal{O}(n^\rho)$, for some $\rho\in\big(0,(1-\alpha')/(d+2)\wedge 1/2\big)$,
\end{enumerate}
where $\alpha=d/2-1+\alpha'$ and $\alpha'\in(0,1)$.
Furthermore, we consider that
\[m_n\cdot\min_{\substack{j_1,j_2=1,\ldots,m_n \\ j_1\neq j_2}}\nor{\textbf{y}_{j_1}-\textbf{y}_{j_2}}_0\] 
is bounded from below, uniformly in $n$ for the asymptotic regime (I).
\end{assump}
Since our statistical inference relies on power variation based on temporal increments, it necessitates fewer spatial observations compared to temporal ones. 
Since this relation already applies in one space dimension, it is intuitive to extend this Assumption to higher space dimensions, cf. \cite{trabs}.
The damping parameter, which also influences the randomness in our model, also impacts the relationship between the resolutions of observations in temporal and spatial dimensions. As the dimensionality increases, the number of available spatial observations decreases. In particular, in one dimension, researchers such as \cite{hildebrandt} and \cite{trabs} demonstrated that in this regime, realized volatilities, expressed as:
\[\rv(\textbf{y})=\sum_{i=1}^n \big(X_{i\Delta_n}(\textbf{y})-X_{(i-1)\Delta_n}(\textbf{y})\big)^2,~j=1,\ldots,m,\]
are sufficient for estimating parameters with an optimal rate of convergence $(m_n n)^{-1/2}$. However, \cite{hildebrandt} established different optimal convergence rates when the condition $m_n/\sqrt{n}\to 0$ is violated, proposing rate-optimal estimators in this context based on double increments in both space and time. The norm $\nor{\cdot}_0$ quantifies the smallest change between two spatial observations in each dimension, effectively extending the condition used for one-dimensional SPDEs to multiple dimensions. This condition becomes especially necessary to bound covariances of the realized volatilities across different spatial coordinates.
It's worth noting that alternative choices besides $\nor{\cdot}_0$ can be considered for controlling the covariance structure in this multi-dimensional model. For instance, one might intuitively opt for the Euclidean norm instead of examining changes along each axis. Nevertheless, we use the presented mapping to directly connect the assumptions from the one-dimensional model, as outlined in \cite{trabs}, with Assumptions \ref{assumption_observations_multi}, facilitating a straightforward comparison between them.

Furthermore, we impose the following regularity condition, as introduced by \cite{tonaki2023parameter}.
\begin{assump}[Regularity]\label{assumption_regMulti}
For the SPDE model from equation \eqref{generalSPDEequationMulti} we assume that
\begin{itemize}
\item[(i)] either $\E[\langle \xi, e_\textbf{k}\rangle_\vartheta]=0$ for all $\textbf{k}\in\N^d$ and $\sup_{\textbf{k}\in\N^d} \lambda_\textbf{k}^{1+\alpha}\E[\langle\xi,e_\textbf{k}\rangle_\vartheta^2]<\infty$ or $\E[\nor{A_\vartheta^{(1+\alpha)/2}\xi}^2_\vartheta]<\infty$ holds true, for $\alpha\in(d/2-1,d/2)$,
\item[(ii)] $(\langle\xi,e_\textbf{k}\rangle_\vartheta)_{\textbf{k}\in \N^d}$ are independent.
\end{itemize}
\end{assump}

\section{Volatility estimation}\label{sec_3}
In this section, our aim is to develop an estimator for the volatility parameter $\sigma^2>0$. For this aim, we will utilize quadratic increments in time, as this statistic usually contains information about the volatility of the underlying process. For the estimation of $\sigma^2$, we assume the remaining parameters in $A_\vartheta$ to be known, as well as the damping parameter $\alpha$. Hence the orthonormal system $e_\textbf{k}$ and the eigenfunctions $\lambda_\textbf{k}$ are known, which enables us to estimate the volatility based on discrete recordings of $X_t$ in multiple time and space coordinates. More precisely, we will conduct volatility estimation using the method of moments. 

It is well established, that in one space dimension, the increments of a solution process $X_t$ behaves different than the standard setup for semi-martingales, see, for instance \cite{trabs} or for semi-martingales \cite{jacod2011discretization}. As we transfer the SPDE model to a multi-dimensional setup, it is expected, due to the coloured noise structure of $B$, that the behaviour of the increments again changes. 

Therefore, analysing the expected value of temporal squared increments and realized volatilities, gains a deeper insight into the multi-dimensional SPDE model and the capabilities of statistical inference.

\begin{theorem}\label{prop_quadIncAndrescaling}
On Assumptions \ref{assumption_observations_multi} and \ref{assumption_regMulti}, we have uniformly in $y\in[\delta,1-\delta]^d$ that
\begin{align*}
\E[(\Delta_i X)^2(\textbf{y})]=\Delta_n^{\alpha'}\sigma^2e^{-\nor{\kappa\bigcdot \textbf{y}}_1}\frac{\Gamma(1-\alpha')}{2^d(\pi\eta)^{d/2}\alpha'\Gamma(d/2)}+r_{n,i}+\Oo(\Delta_n),
\end{align*}
where $\alpha'\in(0,1)$ and a sequence $r_{n,i}$ satisfying $\sup_{1\leq i\leq n}\abs{r_{n,i}}=\Oo(\Delta_n^{\alpha'})$ and $\sum_{i=1}^nr_{n,i}=\Oo(\Delta_n^{\alpha'})$. Furthermore, rescaling yields the following:
\begin{align*}
\E\bigg[\frac{1}{n\Delta_n^{\alpha'}}\sum_{i=1}^n(\Delta_i X)^2(\textbf{y})\bigg]=\sigma^2e^{-\nor{\kappa\bigcdot \textbf{y}}	_1}\frac{\Gamma(1-\alpha')}{2^d(\pi\eta)^{d/2}\alpha'\Gamma(d/2)}+\Oo(\Delta_n^{1-\alpha'}).
\end{align*}
\end{theorem}
Comparing this result to the SPDEs in one space dimension presented in \cite{trabs} reveals some crucial differences in the structure of the random fields.
In higher space dimensions, we observe the appearance of the normalized volatility $\sigma_0^2$ and the curvature term $e^{-y\kappa}$, which are transposed from one space dimension. 
Furthermore, in higher dimensions, we obtain extra constants, among others, depending on $\alpha'$.
However, the most significant distinction when working in higher dimensions is that the parameter $\alpha'$, resulting from the coloured noise in this model, influences the leading term on one side and reduces the convergence speed of the error term on the other side.
By referring to \cite[Thm. 5.22]{daPrato}, we can see that $\alpha'$ governs the regularity in time, which is reflected in the presence of $\Delta_n^{\alpha'}$. Additionally, employing the Kolmogorov-Chentsov theorem (Kolmogorov continuity theorem) and Proposition \ref{prop_quadIncAndrescaling}, we find that the paths of $X_t$ are Hölder-continuous of almost order $\alpha'/2$. 
Note, that the space dimension $d$ of the model only affects the leading term of the expected value, while the order of the error term is solely dependent on $\alpha'$.
Additionally, the latter proposition reveals that the remainder $r_{n,i}$ becomes negligible when summing over the squared increments. As this remainder includes the initial condition, we observe that the impact of the initial condition becomes irrelevant when using the realized volatility statistic. Consequently, constructing an estimator based on the method of moments will yield better results for small $\alpha'$.
Assuming the parameters $\kappa\in\R^d$, $\eta>0$, and $\alpha'\in(0,1)$ to be known, an estimator based on the first moment method of the rescaled realized volatility for the volatility parameter $\sigma^2$ is therefore given by
\begin{align}
\hat{\sigma}^2_\textbf{y}:=\hat{\sigma}^2_n(\textbf{y}):=\frac{1}{n\Delta_n^{\alpha'}K}\sum_{i=1}^n(\Delta_iX)^2(\textbf{y})e^{\nor{\kappa\bigcdot \textbf{y}}_1},\label{sigmaY_estimator}
\end{align}
where the constant $K$ is defined by
\begin{align}
K:=\frac{\Gamma(1-\alpha')}{2^d(\pi\eta)^{d/2}\alpha'\Gamma(d/2)}.\label{eqn_KDefinition}
\end{align} 
Since the estimator $\hat{\sigma}^2_\textbf{y}$ estimates the volatility parameter $\sigma^2$ based on a single spatial point, we also introduce the following estimator:
\begin{align}
\hat{\sigma}^2:=\hat{\sigma}^2_{n,m}:=\frac{1}{nm\Delta_n^{\alpha'}K}\sum_{j=1}^m\sum_{i=1}^n(\Delta_iX)^2(\textbf{y}_j)e^{\nor{\kappa\bigcdot \textbf{y}_j}_1},\label{sigma_estimator}
\end{align}
for spatial points $\textbf{y}_1,\ldots,\textbf{y}_m\in [\delta,1-\delta]^d$.

An important distinction between coloured and white noise is that coloured noise often leads to correlated discrete increments, whereas we often find uncorrelated increments in white noise models. As demonstrated in \cite{trabs}, discrete temporal increments of a SPDE model in one spatial dimension are already negatively correlated, despite the use of white noise. This circumstance implies that we do not need to develop a fundamentally different theory, for instance, for the proofs of central limit theorems.

Nevertheless, by varying the structure of the cylindrical Brownian motion, we can expect a change in the autocovariance structure, which now depends on $\alpha'$. 
\begin{theorem}\label{prop_autocovOfIncrementsMulti}
On Assumptions \ref{assumption_observations_multi} and \ref{assumption_regMulti}, it holds for the covariance of the increments $(\Delta_iX)(y), 1\leq i\leq n$ uniformly in $\textbf{y}\in[\delta,1-\delta]$, for all $\delta\in(0,1/2)$ that
\begin{align*}
\Cov(\Delta_iX(\textbf{y}),\Delta_jX(\textbf{y}))&=-\sigma^2 e^{-\nor{\kappa\bigcdot \textbf{y}}_1}\Delta_n^{\alpha'}\frac{\Gamma(1-\alpha')}{2^{d+1}(\pi\eta)^{d/2}\alpha'\Gamma(d/2)}\\
&~~~~~\times\Big(2\abs{i-j}^{\alpha'}-\big(\abs{i-j}-1\big)^{\alpha'}-\big(\abs{i-j}+1\big)^{\alpha'}\Big)+ r_{i,j}+\Oo(\Delta_n),
\end{align*}
where $\abs{i-j}\geq 1$ and the remainders $(r_{i,j})_{i,j=1,\ldots,n}$ are negligible, i.e. $\sum_{i,j=1}^n r_{i,j}=\mathcal{O}(1)$.
\end{theorem} 
The autocovariance of the coloured noise process appears to depend solely on the spatial coordinate $\textbf{y}$ through the exponential term, which implies that the autocorrelation is independent of the spatial coordinate. Consequently, the autocorrelation structure is determined by the temporal distance or lag between increments rather than the specific temporal locations themselves. 
If we assume that $n$ is sufficiently large, the autocorrelation of temporal increments can be approximated as follows:
\begin{align*}
\rho_{(\Delta_iX),\alpha'}(\abs{i-j})&\approx -\abs{i-j}^{\alpha'}+\frac{1}{2}\big(\abs{i-j}-1\big)^{\alpha'}+\frac{1}{2}\big(\abs{i-j}+1\big)^{\alpha'}=\Oo(\abs{i-j}^{\alpha'-2}),
\end{align*}
for $i \neq j$. As $\alpha'\in(0,1)$, the autocorrelation diminishes as the lag $\abs{i-j}$ between observations increases. Furthermore, from the first derivative, we observe that the autocorrelation is monotonically decreasing. Thus, the most substantial negative correlation is found at $\abs{i-j}=1$, where the autocorrelation takes the value $\rho_{(\Delta_iX),\alpha'}(1)=2^{\alpha'-1}-1$. In the one-dimensional case with a white noise structure, corresponding to $\alpha'=1/2$, the authors \cite{trabs} demonstrated a similar behaviour. They found the most significant (negative) autocorrelation occurred at consecutive increments, with a value of $(\sqrt{2}-2)/2$. Hence, this behaviour extends to multiple spatial dimensions. 

Assuming that the initial condition is a stationary normal distribution with $\xi\sim\mathcal{N}(0,\sigma^2/(2\lambda_\textbf{k}^{1+\alpha}))$, the random field $X_t$ becomes a Gaussian random field. Proposition \ref{prop_autocovOfIncrementsMulti} provides valuable information regarding the identifiability of parameters using temporal increments statistics such as realized volatility. In a manner similar to one space dimension, it appears feasible to consistently estimate the natural parameters, given as the \emph{normalized volatility} $\sigma_0^2=\sigma^2/\eta^{d/2}$ and the \emph{curvature} parameter $\kappa$.

Although the SPDE model in multiple-space dimensions possess an alternating structural behaviour compared to its one-dimensional counterpart, we can employ the decay of the autocovariances and derive a central limit theorem (CLT) for the estimator $\hat{\sigma}^2$ based on a CLT for $\rho$-mixing triangular arrays by \cite{utev}.
\begin{theorem}\label{prop_cltVolaEstMulti} 
On Assumptions \ref{assumption_observations_multi} and \ref{assumption_regMulti}, we have 
\begin{align*}
\sqrt{nm_n}(\hat{\sigma}_{n,m_n}^2-\sigma^2)\overset{d}{\longrightarrow}\mathcal{N}(0,\Upsilon_{\alpha'}\sigma^4),
\end{align*}
for $n\tooi$, $\Upsilon_{\alpha'}$ is a numerical constant defined in equation \eqref{eqn_definingUpsilon} and $m_n=\Oo(n^\rho)$, with $\rho\in\big(0,(1-\alpha')/(d+2)\big)$.
\end{theorem}

The previous proposition establishes that a central limit theorem holds for both volatility estimators, $\hat{\sigma}^2_{n}(\textbf{y})$ from equation \eqref{sigmaY_estimator} and $\hat{\sigma}^2_{n,m}$ from equation \eqref{sigma_estimator}, with an asymptotic variance of $\Upsilon_{\alpha'}\sigma^4$. Comparing this result to a SPDE model in one space dimension, as presented in \cite{trabs}, where $\alpha'=1/2$, reveals that the same asymptotic behaviour is achieved. Hence, this asymptotic behaviour extends to multiple space dimensions.
Nevertheless, Assumption \ref{assumption_observations_multi} states a stronger restriction than in the one-dimensional case, which is necessary for the covariance $\Cov(\hat{\sigma}^2_{\textbf{y}_1},\hat{\sigma}^2_{\textbf{y}_2})$ to asymptotically vanish for two distinct space points $\textbf{y}_1,\textbf{y}_2\in[\delta,1-\delta]^d$. 

As the asymptotic variance in the latter proposition hinges on the unknown volatility parameter, we can not observe confidence intervals for the volatility parameter directly. Nevertheless, confidence intervals can be observed by utilizing the quarticity estimator:
\begin{align*}
\hat{\sigma}^4:=\hat{\sigma}_{n,m}^4:=\bigg(\frac{2^d(\pi\eta)^{d/2}\alpha'\Gamma(d/2)}{\Gamma(1-\alpha')}\bigg)^2\frac{1}{3mn\Delta_n^{2\alpha'}} \sum_{j=1}^m\sum_{i=1}^n(\Delta_iX)^4(\textbf{y}_j)e^{2\nor{\kappa\bigcdot \textbf{y}_j}_1}.
\end{align*}
Under stronger regularity assumptions, such as $\sup_{\textbf{k}\in\N^d} \lambda_\textbf{k}^{1+\alpha}\E[\langle\xi,e_\textbf{k}\rangle_\vartheta^l]<\infty$, for $l=4,8$, one can show by using the bias-variance decomposition, that the quarticity estimator consistently estimates the quarticity parameter $\sigma^4$. Applying Slutskys theorem yields asymptotic confidence intervals.

\section{Asymptotic log-linear model for realized volatilities and least squares estimation}\label{sec_4}
In one space dimension, the authors \cite{bibinger2023efficient} showed, that realized volatilities can asymptotically be linked to a log-linear model, yielding efficient parameter estimation based on ordinary least squares for the natural parameters of the respective one-dimensional SPDE model. The aim of this section is to investigate if this link can be applied in the multivariate case too and therefore, considering parameter estimation for the natural parameters $\sigma_0^2>0$ and $\kappa\in\R$. Throughout this section, we assume the damping parameter $\alpha$ to be known. 
We propose an estimator for the pure damping parameter $\alpha'$ at the end of this section.

Building upon the foundation laid by Proposition \ref{prop_cltVolaEstMulti}, it becomes apparent that rescaled realized volatilities exhibit qualitative resemblance to normal random variables when the count of temporal observations is sufficiently large. Consequently, we are enabled to assert, for adequately large values of $n$, that
\begin{align*}
\sqrt{n}\big(\hat{\sigma}^2_{\textbf{y}}-\sigma^2\big)\approx \mathcal{N}(0,\Upsilon_{\alpha'}\sigma^4),
\end{align*}
where we obtain by rearranging the latter display that
\begin{align}
\frac{\RV(\textbf{y})}{n\Delta_n^{\alpha'}}&\approx e^{-\nor{\kappa\bigcdot \textbf{y}}_1} \frac{\Gamma(1-\alpha')\sigma^2}{\eta^{d/2}\alpha'}\cdot\frac{1}{\pi^{d/2}\Gamma(d/2)2^d\sqrt{n}}\big(\sqrt{n}+\sqrt{\Upsilon}_{\alpha'}Z\big)\notag\\
&=e^{-\nor{\kappa\bigcdot \textbf{y}}_1} \frac{\Gamma(1-\alpha')\sigma^2_0}{\alpha'}\cdot\frac{1}{\pi^{d/2}\Gamma(d/2)2^d}\bigg(1+\sqrt{\frac{\Upsilon_{\alpha'}}{n}}Z\bigg),
\end{align}
with $Z\sim\mathcal{N}(0,1)$. We adopt the strategy of converting this approximation into a log-linear model, namely:
\begin{align}
\log\bigg(\frac{\RV(\textbf{y})}{n\Delta_n^{\alpha'}}\bigg)&\approx -\nor{\kappa\bigcdot \textbf{y}}_1+\log\big(\sigma_0^2K\big)+\sqrt{\frac{\Upsilon_{\alpha'}}{n}}Z\label{eqn_transitionLogLinearModelMultiApprox},
\end{align}
where $K$ is defined in \eqref{eqn_KDefinition}. To be more precise, we arrive at an approximation that resembles a multiple linear regression model. Considering the asymptotic decorrelation of the realized volatilities across different spatial locations, we can establish this linear model by examining $\log(\RV(\textbf{y}_j)/(n\Delta_n^{\alpha'}))$ for $j=1,\ldots,m$. This representation also implies a homoscedastic normal distribution for the errors within the linear model. As the log-realized volatilities are only asymptotically linkable to a log-linear model, we have to carefully analyse the error terms.

To illustrate, let us revisit the concept of the multiple linear regression model with the help of the following example.
\begin{example}\label{example_OMLR}
An ordinary multiple linear regression model is given by
\begin{align*}
Y=X\beta+\varepsilon,
\end{align*}
where
\begin{align*}
Y=\begin{pmatrix}
Y_1\\ \vdots \\ Y_m
\end{pmatrix},~~~~~ 
X=\begin{pmatrix}
1 & y_1^{(1)} & \hdots & y_d^{(1)} \\ 
\vdots &\vdots & \ddots & \vdots \\
1 & y_1^{(m)} &\hdots & y_d^{(m)}
\end{pmatrix}, ~~~~~~
\beta=\begin{pmatrix}
\beta_0 \\ \beta_1 \\ \vdots \\ \beta_d
\end{pmatrix},
\end{align*}
and homoscedastic errors $\varepsilon=(\varepsilon_1,\ldots,\varepsilon_m)^\top$, with $\E[\varepsilon_i]=0$, $\Var(\varepsilon_i)=\sigma^2>0$, for $i=1,\ldots,m$ and $\Cov(\varepsilon_i,\varepsilon_j)=0$, for all $i,j=1,\ldots,m$, with $i\neq j$. In addition, the variance-covariance matrix of $\varepsilon$ is given by $\Sigma:=\Cov(\varepsilon)=\sigma^2E_{m}$. We call the parameter $\beta_0$ intercept and the parameters $\beta_i$ as slope, where $i=1,\ldots,d$.
Suppose that $m\geq (d+1)$, and the matrix $X$ possesses a full rank of $(d+1)$. Under these assumptions, the least squares estimator for the unobservable parameter $\beta$ within this model can be expressed as follows:
\begin{align*}
\hat{\beta}=(X^\top X)^{-1}X^\top Y.
\end{align*}
Substituting the representation of $Y$ into the latter expression results in the following identity:\begin{align}
\hat{\beta}=(X^\top X)^{-1}X^\top Y=\beta+(X^\top X)^{-1}X^\top\varepsilon,\label{eqn_identityGMEForTriangularArray}
\end{align}
which shows that the estimators $\hat{\beta}$ is unbiased.
In particular, the inverse of $(X^\top X)^{-1}$ exists due to the full rank condition on the design matrix $X$.
\end{example}

The component-wise estimators highlighted in Example \ref{example_OMLR} are commonly referred to as Gauss-Markov estimators. It is well-known, that the Gauss-Markov estimators qualify as BLUE (Best Linear Unbiased Estimators), implying that they possess the minimum variance among all linear and unbiased estimators. However, it is important to acknowledge that the number of observations $(Y_j)_{1\leq j\leq m}$ is intrinsically linked to the dimensionality, specifically requiring $m\geq (d+1)$, as stated in the preceding example. Consequently, we introduce the the following full-rank assumption.
\begin{assump}\label{assumption_fullRank}
Let $\textbf{y}_1,\ldots,\textbf{y}_m\in[\delta,1-\delta]^d$, where $m\geq d$ such that the linear span
\begin{align*}
\text{span}(\textbf{Y}_m)=\R^{d+1},~~~~~\text{where}~~~~~\textbf{Y}_m:=\big\{(1,\textbf{y}_1),\ldots,(1,\textbf{y}_m)\big\},
\end{align*}
is a spanning set of $\R^{d+1}$. 
\end{assump}
In accordance with Assumption \ref{assumption_observations_multi}, it is established that the discretization of the random field is more refined in time than in space, denoted by $m=\Oo(n^\rho)$, where $\rho\in\big(0,(1-\alpha')/(d+2)\big)$. Additionally, Assumption \ref{assumption_fullRank} imposes the requirement that a minimum of $d+1$ spatial observations is necessary to construct an estimator for the natural parameters. Collectively, these assumptions enforce a minimal number of temporal points, indicated by
\begin{align*}
n>(d+1)^{\frac{d+2}{1-\alpha'}},
\end{align*}
Asymptotically, this restriction is evidently satisfied since $d$ is fixed. However, the restrictive nature becomes significant in a simulation scenario. The latter display implies that $n$ grows exponentially with the dimensions $d\geq 2$. Furthermore, if $\alpha'$ is close to one, the growth of $n$ becomes particularly pronounced. Therefore, estimating the natural parameters using this least squares approach based on realized volatilities might only be accurate for lower dimensions, such as $d=2,3$, or when a large number of temporal observations is available.

We can now establish the estimators for the natural parameters $\sigma_0^2, \kappa_1, \ldots, \kappa_d$ within the context of the SPDE model from equation \eqref{generalSPDEequationMulti}. Leveraging the approximation \eqref{eqn_transitionLogLinearModelMultiApprox} and referencing Example \ref{example_OMLR}, we proceed to define the multi-dimensional parameter and its corresponding estimator as follows:
\begin{align}
\Psi:=\begin{pmatrix}
\log(\sigma_0^2K)\\ -\kappa_1 \\ \vdots \\ -\kappa_d
\end{pmatrix}\in\R^{d+1}~~~~~\text{and}~~~~~\hat{\Psi}:=\hat{\Psi}_{n,m}:=(X^\top X)^{-1}X^\top Y\in\R^{m},\label{eqn_psiEstiamtorAndPara}
\end{align}
where 
\begin{align*}
X:=\begin{pmatrix}
1 & y_1^{(1)} & \hdots & y_d^{(1)} \\ 
\vdots &\vdots & \ddots & \vdots \\
1 & y_1^{(m)} &\hdots & y_d^{(m)}
\end{pmatrix}\in\R^{m\times (d+1)}~~~~~\text{and}~~~~~Y:=\begin{pmatrix}
\log\Big(\frac{\RV(\textbf{y}_1)}{n\Delta_n^{\alpha'}}\Big) \\ \vdots \\
\log\Big(\frac{\RV(\textbf{y}_m)}{n\Delta_n^{\alpha'}}\Big)
\end{pmatrix}\in\R^{m}.
\end{align*}
To effectively estimate the natural parameters $\sigma_0^2, \kappa_1, \ldots, \kappa_d$, we introduce the parameter $\upsilon \in (0,\infty)\times \mathbb{R}^d$ along with its associated estimator $\hat{\upsilon}$, defined as follows:
\begin{align}
\upsilon:=\begin{pmatrix}
\sigma_0^2\\ \kappa_1\\ \vdots \\ \kappa_d\end{pmatrix}
~~~~~\text{and}~~~~~
\hat{\upsilon}:=\hat{\upsilon}_{n,m}:=h^{-1}(\hat{\Psi}):=\begin{pmatrix}
h_1^{-1}(\hat{\Psi}_1)\\ \vdots \\ h_{d+1}^{-1}(\hat{\Psi}_{d+1})
\end{pmatrix},\label{eqn_beta_estimator_61}
\end{align}
where $\hat{\Psi}=(\hat{\Psi}_1,\ldots,\hat{\Psi}_{d+1})^\top$ and $h:(0,\infty)\times\R^d\rightarrow \R^{d+1}$, $h^{-1}:\R^{d+1}\rightarrow (0,\infty)\times\R^d$, with
\begin{align}
h(\textbf{x})=\begin{pmatrix}
\log(x_1K)\\ -x_2 \\ \vdots \\ -x_{d+1}
\end{pmatrix}~~~~~\text{and}~~~~~h^{-1}(\textbf{x})=\begin{pmatrix}
e^{x_1}/K\\ -x_2 \\ \vdots \\ -x_{d+1}
\end{pmatrix}.\label{eqn_functionHMultiForDeltaMethod}
\end{align}
Note, that $h(\upsilon)=\Psi$.

When considering a central limit theorem, one concern lies in determining the asymptotic variance. In the context of Example \ref{example_OMLR}, we obtain:
\begin{align*}
\Var\big(\sqrt{m}(\hat{\beta}-\beta)\big)&=\Var\big(\sqrt{m}\hat{\beta}\big)=\sigma^2c\bigg(\frac{c}{m}X^\top X\bigg)^{-1}\overset{n\tooi}{\longrightarrow}\sigma^2c\Sigma^{-1},
\end{align*}
where we make the assumption that $c/m(X^\top X)$ converges to a symmetric positive-definite variance-covariance matrix $\Sigma \in \mathbb{R}^{(d+1)\times (d+1)}$, where $c>0$ is a suitable constant. This assumption consequently entails that $\Sigma^{-1}$ is also symmetric and positive-definite.
In our model, we observe spatial coordinates $\textbf{y}_1,\ldots,\textbf{y}_m$ within the range $[\delta,1-\delta]^d$, signifying that these spatial observations are situated at least $\delta>0$ distance away from the boundaries of the unit hypercube. We can examine the structure of the matrix $X^\top X$ by utilizing the explicitly provided expression of $X$ from Example \ref{example_OMLR} and have 
\begin{align*}
\frac{1-2\delta}{m}X^\top X =\frac{1-2\delta}{m} \begin{pmatrix}
m & \sum_{j=1}^m y_1^{(j)}&\sum_{j=1}^m y_2^{(j)} & \hdots &\sum_{j=1}^m y_d^{(j)} \\
\sum_{j=1}^m y_1^{(j)} & \sum_{j=1}^m (y_1^{(j)})^2 & \sum_{j=1}^m y_1^{(j)}y_2^{(j)} & \hdots & \sum_{j=1}^m y_1^{(j)}y_d^{(j)}\\
\sum_{j=1}^m y_2^{(j)} & \sum_{j=1}^m y_2^{(j)}y_1^{(j)}& \sum_{j=1}^m(y_2^{(j)})^2 & \hdots &\sum_{j=1}^m y_2^{(j)}y_d^{(j)}\\
\vdots &\vdots & \vdots & \ddots & \vdots \\
\sum_{j=1}^m y_d^{(j)} & \sum_{j=1}^m y_d^{(j)}y_1^{(j)}&  \sum_{j=1}^m y_d^{(j)}y_2^{(j)} & \hdots & \sum_{j=1}^m(y_d^{(j)})^2
\end{pmatrix}\overset{m\tooi}{\longrightarrow}\Sigma,
\end{align*}
where $\Sigma=(\Sigma_{i,l})_{1\leq i,l\leq d+1}$, with 
\begin{align}
\Sigma_{i,l}:=\begin{cases}
1-2\delta   &,\text{ if } i=l=1,\\
\lim_{m\tooi}\frac{1-2\delta}{m}\sum_{j=1}^m y_{l-1}^{(j)}   &,\text{ if } i=1,2\leq l\leq d+1,\\
\lim_{m\tooi}\frac{1-2\delta}{m}\sum_{j=1}^m y_{i-1}^{(j)}   &,\text{ if } 2\leq i\leq d+1, l=1,\\
\lim_{m\tooi}\frac{1-2\delta}{m}\sum_{j=1}^m \big(y_{i-1}^{(j)}\big)^2   &,\text{ if } 2\leq i=l\leq d+1,\\
\lim_{m\tooi}\frac{1-2\delta}{m}\sum_{j=1}^m y_{i-1}^{(j)}y_{l-1}^{(j)}&,\text{ if }2\leq i,l\leq d+1 \text{, with }i\neq l.
\end{cases}\label{eqn_asympVarCovMatrixMLRM_general} 
\end{align}
The convergence of the Riemann sums is guaranteed by the straightforward bounds:
\begin{align*}
0\leq \frac{1-2\delta}{m}\sum_{j=1}^m a_j\leq 1,
\end{align*}
for all $m\in\mathbb{N}$, where the sequence $(a_j)$ corresponds to the relevant sequence within the Riemann sums in equation \eqref{eqn_asympVarCovMatrixMLRM_general}.

We give this elementary example, since our estimator $\hat{\Psi}$ and the asymptotic variance-covariance matrix of our estimator will be in line with the translation of the example \ref{example_OMLR} to our model, as stated in the following proposition.
\begin{theorem}\label{clt_PsiMulti}
On Assumptions \ref{assumption_observations_multi}, \ref{assumption_regMulti} and \ref{assumption_fullRank}, we have 
\begin{align*}
\sqrt{nm_n}(\hat{\Psi}_{n,m_n}-\Psi)\overset{d}{\longrightarrow}\mathcal{N}\big(\textbf{0},\Upsilon_{\alpha'}(1-2\delta)\Sigma^{-1}\big),
\end{align*}
for $n\tooi$, $\delta\in(0,1/2)$, $\textbf{0}=(0,\ldots,0)^\top\in\R^{d+1}$, $\Upsilon_{\alpha'}$ defined in equation \eqref{eqn_definingUpsilon} and $\Sigma$ defined in equation \eqref{eqn_asympVarCovMatrixMLRM_general}.
\end{theorem}

As in the one-dimensional case, our central limit theorem is readily feasible and provides asymptotic confidence intervals, as it only depends on the known parameter $\alpha'$. The latter proposition also states, that the connection between realized volatilities and a log-linear model transfers to multiple-space dimensions.

Utilizing the multivariate delta method yields a CLT for the estimator $\hat{\upsilon}$, given in the following corollary.

\begin{cor}\label{corollary_hatBetaCLT_61}
On Assumptions \ref{assumption_observations_multi}, \ref{assumption_regMulti} and \ref{assumption_fullRank}, we have 
\begin{align*}
\sqrt{nm_n}(\hat{\upsilon}_{n,m_n}-\upsilon)\overset{d}{\longrightarrow}\mathcal{N}\big(\textbf{0},\Upsilon_{\alpha'}(1-2\delta)J_{\sigma_0^2}\Sigma^{-1}J_{\sigma_0^2}\big),
\end{align*}
for $n\tooi$, $\delta\in(0,1/2)$, $\textbf{0}=(0,\ldots,0)^\top\in\R^{d+1}$, $\Upsilon_{\alpha'}$ defined in equation \eqref{eqn_definingUpsilon}, $\Sigma^{-1}$ defined in equation \eqref{eqn_asympVarCovMatrixMLRM_general} and $J_{\sigma_0^2}$ defined in equation \eqref{eqn_JSigma0SqDef}.
\end{cor}
\ \\ \\

We turn our attention to the estimation of the pure damping parameter $\alpha'\in(0,1)$, and therefore estimating the parameter $\alpha$. As this parameter necessarily arises when considering a multi-dimensional SPDE, an estimation of this parameter becomes even more pronounced then in one-space dimensions. In addition, the presented estimators in this paper were constructed under the premises, that $\alpha'$ is known. When dealing with real-world data, this assumption may not be fulfilled. 

As already mentioned, the damping parameter controls the Hölder regularity of the temporal marginal processes of a solution $X_t$, which is also effecting the correlations in our model. Therefore, we follow an approach for estimating $\alpha'$ by using a well-established concept from estimating the Hurst parameter for fractional Brownian motions. The main idea is to use two different temporal grids, one containing all the available data and the other, containing a thinned version of the original grid. Having both grids, we aim to use realized volatilities in order to gain information about the pure damping parameter. 

In detail, let us consider a mild solution $X_t(\textbf{y})$ of the SPDE model from equation \eqref{generalSPDEequationMulti}. Assume we obtain $X$ on a grid with $2n$ temporal and $m$ spatial points according to Assumption \ref{assumption_observations_multi}. 
First, we want the new grid to be equidistant in time with $\tilde{n}=\Oo(2n)$, $\tilde{n}<2n$ temporal points, such that it satisfies the observation scheme in Assumption \ref{assumption_observations_multi}. Furthermore, Proposition \ref{prop_RRVMulti}, as presented in section \ref{sec_6}, suggests to filter the original grid such that the new grid contains the maximum amount of temporal points. Intuitively, having the most possible temporal points, while respecting an equidistant order of these, should shrink the variance of the estimator. Hence, we set $\tilde{n}=n$. As we need to distinguish between both temporal resolutions we introduce the following notations. The temporal increments for both grids are denoted by
\begin{align*}
(\Delta_{2n,i_1}X)(\textbf{y}):=X_{i_1\Delta_{2n}}-X_{(i_1-1)\Delta_{2n}}~~~~~\text{and}~~~~~(\Delta_{n,i_2}X)(\textbf{y}):=X_{i_2\Delta_{n}}-X_{(i_2-1)\Delta_{n}},
\end{align*}
where $1\leq i_1\leq 2n$ and $1\leq i_2\leq n$. The increments of the filtered temporal grid can be rewritten by
\begin{align*}
(\Delta_{n,i}X)(\textbf{y})&=X_{2i\Delta_{2n}}-X_{2(i-1)\Delta_{2n}}=(\Delta_{2n,2i}X)(\textbf{y})+(\Delta_{2n,2i-1}X)(\textbf{y}),
\end{align*}
where $i=1,\ldots,n$. Furthermore, by using a index transformation, we can write:
\begin{align}
(\Delta_{n,i}X)(\textbf{y})&=\mathbbm{1}_{2\N}(i)\big((\Delta_{2n,i}X)(\textbf{y})+(\Delta_{2n,i-1}X)(\textbf{y})\big),\label{eqn_incFiltGridafterIndexShift}
\end{align}
for $i=1,\ldots,2n$, where $2\N$ denotes the set of all even and non-negative integers, i.e.: $2\N=\{0,2,4,\ldots\}$.
Thus, the realized volatilities can be defined as:
\begin{align*}
\text{RV}_{2n}(\textbf{y}):=\sum_{i=1}^{2n}(\Delta_{2n,i}X)^2(\textbf{y})~~~~~\text{and}~~~~~\text{RV}_{n}(\textbf{y}):=\sum_{i=1}^{n}(\Delta_{n,i}X)^2(\textbf{y}).
\end{align*}
By using equation \eqref{eqn_incFiltGridafterIndexShift}, we can link the filtered realized volatilities with the original grid and obtain:
\begin{align}
\text{RV}_{n}(\textbf{y})
&=\text{RV}_{2n}(\textbf{y})+2\sum_{i=2}^{2n}\mathbbm{1}_{2\N}(i)(\Delta_{2n,i}X)(\textbf{y})(\Delta_{2n,i-1}X)(\textbf{y}).\label{eqn_1212}
\end{align}
By using equation \ref{eqn_transitionLogLinearModelMultiApprox}, we can construct an estimator for the pure damping parameter using the following approach: 
\begin{align*}
\log\bigg(\frac{\RV(\textbf{y})}{n}\bigg)-\log\bigg(\frac{\RVt(\textbf{y})}{2n}\bigg)&\approx \alpha'\big(\log(\Delta_{n})-\log(\Delta_{2n})\big)+\sqrt{\frac{\Upsilon_{\alpha'}}{n}}Z_1-\sqrt{\frac{\Upsilon_{\alpha'}}{2n}}Z_2\notag\\
&=  \alpha'\log(2)+\sqrt{\frac{\Upsilon_{\alpha'}}{n}}Z_1-\sqrt{\frac{\Upsilon_{\alpha'}}{2n}}Z_2,
\end{align*} 
where $Z_1,Z_2\sim\mathcal{N}(0,1)$ and $\textbf{y}\in[\delta,1-\delta]^d$. Among others, the linear model proposes an estimator for the unknown parameter $\alpha'$ given by
\begin{align}
\hat{\alpha}':=\hat{\alpha}_{2n,m}':=\frac{1}{\log(2)m}\sum_{j=1}^m\log\bigg(\frac{2\RV(\textbf{y}_j)}{\RVt(\textbf{y}_j)}\bigg).\label{eqn_defalphaEst}
\end{align}
For analysing asymptotic properties of this estimator, it is crucial to investigate the correlation structure of quadratic temporal increments and the product of consecutive temporal increments, as evident by \eqref{eqn_1212}. Having this knowledge on the covariances, we can prove the following CLT.

\begin{theorem}\label{prop_cltAlpha}
On Assumptions \ref{assumption_observations_multi} and \ref{assumption_regMulti} we have 
\begin{align*}
\sqrt{2nm_n}(\hat{\alpha}'_{2n,m_{2n}}-\alpha')\overset{d}{\longrightarrow}\mathcal{N}\Big(0,\log(2)^{-2}\big(3\Upsilon_{\alpha'}-2^{2-\alpha'}(\Upsilon_{\alpha'}+\Lambda_{\alpha'})\big)\Big),
\end{align*}
for $n\tooi$, where $m_{2n}=\Oo\big((2n)^\rho\big)$ with $\rho\in\big(0,(1-\alpha')/(d+2)\big)$, $\Upsilon_{\alpha'}$ defined in \eqref{eqn_definingUpsilon} and $\Lambda_{\alpha'}$ defined in equation \eqref{eqn_definingLambda}.
\end{theorem}
As the proof of this central limit theorem uses analogous techniques as used for Proposition \ref{prop_cltVolaEstMulti}, we omit the proof and only provide the proof leading to the asymptotic variance in Section \ref{sec_6}. The term $-2^{2-\alpha'}(\Upsilon_{\alpha'}+\Lambda_{\alpha'})\big)$ in the asymptotic variance, as outlined in the letter CLT, represent the non-negligible covariance structures that appear when using realized volatilities on two temporal grids with different resolutions. 
Since Proposition \ref{prop_cltAlpha} also establishes the consistency of the estimator $\hat{\alpha}'$, we can conclude that the estimators $\hat{\sigma}^2_\textbf{y}$ and $\hat{\sigma}^2$ from Section \ref{sec_3}, along with $\hat{\Psi}$ and $\hat{\upsilon}$ from Section \ref{sec_4}, remain consistent when $\alpha'$ is unknown and therefore replaced via plug-in by the estimator $\hat{\alpha}'$. We can also preserve the original CLTs from the estimators $\hat{\sigma}^2,\hat{\Psi}$ and $\hat{\upsilon}$ from the Propositions \ref{prop_cltVolaEstMulti}, \ref{clt_PsiMulti} and Corollary \ref{corollary_hatBetaCLT_61}, by accepting a slightly slower rate than $n^{1/2}$.

\section{Simulation methods and Monte Carlo simulation study}
\subsection{Simulation methods}\label{sec_51}
To simulate linear one-dimensional SPDE models, two techniques have been established: the truncation method and the replacement method, as referenced in \cite{trabs} and \cite{hildebrandt2020}, respectively. In the subsequent discussion, we will explore both simulation methods in the context of multi-dimensional cases, beginning with the truncation method.

The truncation method relies on the Fourier decomposition of a mild solution $X_t(\textbf{y})$ of a SPDE model from \eqref{generalSPDEequationMulti} and allows to simulate these SPDE models with deterministic or normally distributed initial conditions $\xi$. The concept involves truncating the Fourier series from \eqref{eqn_fourierseriesMultiOfX} at a sufficiently large cut-off frequency $\textbf{K}=(K_1,\ldots,K_d)\in\N^d$, simulating only the first $K_l$ Fourier modes respectively, where $l=1,\ldots,d$. Assuming a deterministic or normally distributed initial condition, combined with Assumption \ref{assumption_observations_multi}, we find the coordinate processes normally distributed, where 
\begin{align*}
x_\textbf{k}(t)\sim \mathcal{N}\bigg(e^{-\lambda_\textbf{k} t}\mu_\xi ,~\frac{\sigma^2}{2\lambda_\textbf{k}^{1+\alpha}}\Big(1-e^{-2\lambda_\textbf{k} t}\Big) +e^{-2\lambda_\textbf{k} t}\sigma_\xi^2\bigg),
\end{align*}
for $\textbf{k}\in\N^d$. Here $\mu_\xi$ and $\sigma_{\xi}^2$ denotes the expected value and variance of the initial condition $\xi$, respectively. 
Assuming that the initial condition is deterministic, we can deduce that $x_\textbf{k}$ is normally distributed with a variance of $\sigma^2(1-e^{-2\lambda_\textbf{k} t})/(2\lambda_{\textbf{k}}^{1+\alpha})$. 
Notably, the variance of the Fourier modes $x_\textbf{k}$ is influenced by the damping parameter $\alpha$. A larger value of $\alpha$ implies a stronger damping and quicker convergence of the variance towards zero, when $\textbf{k}\rightarrow\infty$. On the other hand, a smaller value of $\alpha$ indicates a weaker damping and slower convergence of the variance. \ \\

To simulate the Fourier modes $x_\textbf{k}$, we have for $t=0,\dots,(N-1)\Delta_n$ that
\begin{align*}
x_\textbf{k}(t+\Delta_n)-x_\textbf{k}(t)e^{-\lambda_\textbf{k}\Delta_n}
&=\sigma\lambda_\textbf{k}^{-\alpha/2}\int_{t}^{t+\Delta_n}e^{-\lambda_\textbf{k}(t+\Delta_n-s)}\diff W_s^\textbf{k}.
\end{align*}
Hence, we infer the recursive representation:
\begin{align*}
x_\textbf{k}(t+\Delta_n)=x_k(t)e^{-\lambda_\textbf{k}\Delta_n}+\sigma\sqrt{\frac{1-\exp[-2\lambda_\textbf{k}\Delta_n]}{2\lambda_\textbf{k}^{1+\alpha}}}\mathcal{N}_t,
\end{align*}
with i.i.d.\ standard normals $\mathcal{N}_t$ and $x_{\textbf{k}}(0)=\langle\xi,e_{\textbf{k}}\rangle_\vartheta$, where $\xi$ is either deterministic or normal distributed. 
We therefore introduce the truncation method by approximating the Fourier series of $X_t(\textbf{y})$ using a cut-off frequency $\mathcal{K}:=\{1,\ldots,K\}^d$, where $K\in\N$. 

In one space dimension, the effectiveness of this method is strongly influenced by the chosen cut-off rate $K \in \N$. The authors \cite{uchida} observed through empirical study that insufficiently large values of $K$ lead to considerable biases in the simulations. Selecting an appropriate cut-off rate also appears to be dependent on the number of spatial and temporal observations.
Even for moderate sample sizes, a cut-off rate of $K=10^5$ is recommended, but it comes with a significant computational cost. For instance, simulating a single realization of $X$ on a grid with $M=100$ spatial points and $N=10^4$ temporal points, using a cut-off rate $K=10^5$, takes approximately 6 hours when utilizing 64 cores.
These issues becomes even more pronounced when dealing with multiple space dimensions. 
When simulating multi-dimensional SPDEs, it is reasonable to choose a cut-off frequency of at least $K=10^5$ as well, leading to $(K+1)^d$ loop iterations.
For example, in a two-dimensional case, \cite{tonaki2023parameter} performed simulations at $200\times 200$ equispaced coordinates with a temporal resolution of $N=10^3$. Using a cut-off rate of $K=10^5$, the simulation of one sample path took approximately 100 hours while using three personal computers.
This highlights the computational challenge of simulating multi-dimensional SPDEs with a large cut-off frequency, as it requires a substantial amount of computing power and time. However, the use of a sufficiently high cut-off frequency is crucial to ensure accurate and unbiased simulations of the SPDEs. 
These issues motivated the second approach, known as the replacement method.

The author \cite{hildebrandt2020} build on the work of \cite{davie2001convergence}, by replacing the higher Fourier modes instead of cutting them off. For introducing this approach, we assume $\xi\equiv 0$. The main idea of the replacement approach is to change the Hilbert space, leading to the infinite Fourier representation in \eqref{eqn_fourierseriesMultiOfX}. Therefore, we assume the spatial coordinates to be equidistant $\textbf{y}\in\{0,1/M,\ldots,(M-1)/M,1\}^d$ along each space dimension, i.e., $\textbf{y}_\textbf{j}=\textbf{j}/M=(j_1/M,\ldots,j_d/M)$ and $\textbf{j}\in\{0,\ldots,M\}^d=:\mathcal{J}$. We define the inner product by
\begin{align*}
\langle f,g,\rangle_{\vartheta,M}:=\frac{1}{M^d}\sum_{\textbf{j}\in\mathcal{J}}f(\textbf{y}_\textbf{j})g(\textbf{y}_\textbf{j})e^{\nor{\kappa\bigcdot \textbf{y}_\textbf{j}}_1},
\end{align*}
where $f,g:[0,1]^d\too \R$. It holds that $(e_k)_{1\leq k<M}$ from equation \eqref{eqn_ekMulti} form an orthonormal system with respect to the inner product $\langle\cdot,\cdot\rangle_{\vartheta,M}$. 
Hence, we can express a solution $X_t$ as:
\begin{align*}
X_t(\textbf{y}_\textbf{j})=\sum_{\textbf{m}\in\mathcal{M}}U_\textbf{m}(t)e_\textbf{m}(\textbf{y}_\textbf{j}),~~~~~\text{with}~~~~~U_\textbf{m}(t)=\langle X_t,e_\textbf{m}\rangle_{\vartheta,M},
\end{align*}
where $\mathcal{M}=\{1,\ldots,M-1\}^d$. Note, that $e_\textbf{m}(\textbf{y}_\textbf{j})=0$, if $\textbf{m}=(m_1,\ldots,m_d)$ contains at least one entry $m_l$, which is either zero or $M$, i.e. $m_l\in\{0,M\}$, for a $l\in\{1,\ldots,d\}$.
Using the Fourier representation, as given in equation \eqref{eqn_fourierseriesMultiOfX}, we have 
\begin{align*}
U_\textbf{m}(t)=\sum_{\textbf{k}\in\N^d}x_\textbf{k}(t)\langle e_\textbf{k},e_\textbf{m}\rangle_{\vartheta,M}.
\end{align*}
Let $\textbf{k}\in\N^d$, then we can decompose the inner product by
\begin{align*}
\langle e_\textbf{k},e_\textbf{m}\rangle_{\vartheta,M}=\frac{2^d}{M^d}\sum_{\textbf{j}\in\mathcal{J}}\prod_{l=1}^d\sin(\pi k_l \textbf{y}_{\textbf{j}_l})\sin(\pi m_l\textbf{y}_{\textbf{j}_l})=\prod_{l=1}^d\langle \tilde{e}_{k_l},\tilde{e}_{m_l}\rangle_{\vartheta,M,1},
\end{align*}
where $\tilde{e}_k$ and $\langle\cdot,\cdot\rangle_{\vartheta,M,1}$ denote the respective one-dimensional orthonormal basis and inner product as defined in \cite{hildebrandt2020}. Thereby, we also know that
\begin{align*}
\abs{\langle \tilde{e}_{k},\tilde{e}_{m}\rangle_{\vartheta,M,1}}=1,
\end{align*}
if $k=m+2lM$ or $k=2M-m+2lm$ for $l\in\N_0$, $k\in\N$ and $m\in\{1,\ldots,M-1\}$. Therefore, the index set $\mathcal{I}_\textbf{m}$ is given by the following $d$-fold Cartesian product:
\begin{align*}
\mathcal{I}_\textbf{m}:=\bigtimes_{l=1}^d\mathcal{I}_{m_l,1},
\end{align*}
where $\mathcal{I}_{k,1}$ denotes the one-dimensional index set introduced by \cite{hildebrandt2020}, given by
\begin{align*}
\mathcal{I}_{k,1}^+=\{k+2lM,l\in\N_0\},~~\mathcal{I}_{k,1}^-=\{2M-k+2lM,l\in\N_0\},~~\mathcal{I}_{k,1}:=\mathcal{I}_{k,1}^+\cup\mathcal{I}_{k,1}^-,
\end{align*}
where $k\in\N$.
Since $x_\textbf{l}\overset{d}{=}-x_\textbf{l}$, for all $\textbf{l}\in\N^d$, we have 
\begin{align*}
U_\textbf{m}(t)=\sum_{\textbf{k}\in\N^d}x_\textbf{k}(t)\langle e_\textbf{k},e_\textbf{m}\rangle_{\vartheta,M}=\sum_{\textbf{l}\in\mathcal{I}_\textbf{m}}x_\textbf{l}(t),
\end{align*} 
where $x_\textbf{l}$ denotes the coordinate process from equation \eqref{eqn_coordProcessMulti}. 
The covariances of the coordinate processes, given by
\begin{align*}
\Cov\big(x_\textbf{k}(t_i),x_\textbf{k}(t_j)\big)&=\frac{\sigma^2}{2\lambda_\textbf{k}^{1+\alpha}}e^{-\lambda_\textbf{k} \abs{i-j}\Delta_n}\Big(1-e^{-2\lambda_\textbf{k}\min(i,j)\Delta_n}\Big),
\end{align*}
are vanishing if $\lambda_\textbf{k}\propto \euc{\textbf{k}}^2$ is significantly larger than $1/\Delta_N$ due to the presence of the exponential term. Therefore the coordinate processes $(x_\textbf{k}(t_i))_{1\leq i\leq N}$ effectively behave like i.i.d.\ centred normal random variables, with a variance:
\begin{align*}
\Var\big(x_\textbf{k}(t_i)\big)\approx \frac{\sigma^2}{2\lambda_\textbf{k}^{1+\alpha}},
\end{align*}
for a sufficiently large $\textbf{k}\in\N^d$.
Analogously to \cite{hildebrandt2020}, we choose a bound $L\in\N$ and replace all coordinate processes $(x_\textbf{k})$ with $\textbf{k}\notin(0, LM)^d$ by a vector of independent normal random variables $(z_\textbf{l})_{\textbf{l}\in\N^d}$ with variance $\sigma^2/(2\lambda_\textbf{l}^{1+\alpha})$, i.e.:
\begin{align*}
U_\textbf{m}(t)=\sum_{\substack{\textbf{l}\in\mathcal{I}_m \\ \textbf{l}\in (0,LM)^d}}x_\textbf{l}(t)+\sum_{\substack{\textbf{l}\in\mathcal{I}_m \\ \textbf{l}\notin (0, LM)^d}}z_\textbf{l}(t).
\end{align*} 
Since the normal distribution is stable with respect to summation, we can replace the sum of the normal random variables with centred normal random variables $R_\textbf{m}\sim\mathcal{N}(0,s_\textbf{m}^2)$, where
\begin{align*}
s_\textbf{m}^2=\sum_{\substack{\textbf{l}\in\mathcal{I}_\textbf{m}\\\textbf{l}\notin (0, LM)^d}}\frac{\sigma^2}{2\lambda_\textbf{l}^{1+\alpha}}.
\end{align*}
By equation \eqref{eqn_Bwelldefined}, it is evident that the series in $s_\textbf{m}^2$ converges.

In the one-dimensional case, \cite{hildebrandt2020} developed a formula to precisely compute the one-dimensional replacement variance. One key advantage of this formula is its closed form, which enables rapid computation with minimal computational time.
\begin{figure}[t]
\centering
\includegraphics[width=0.95\textwidth]{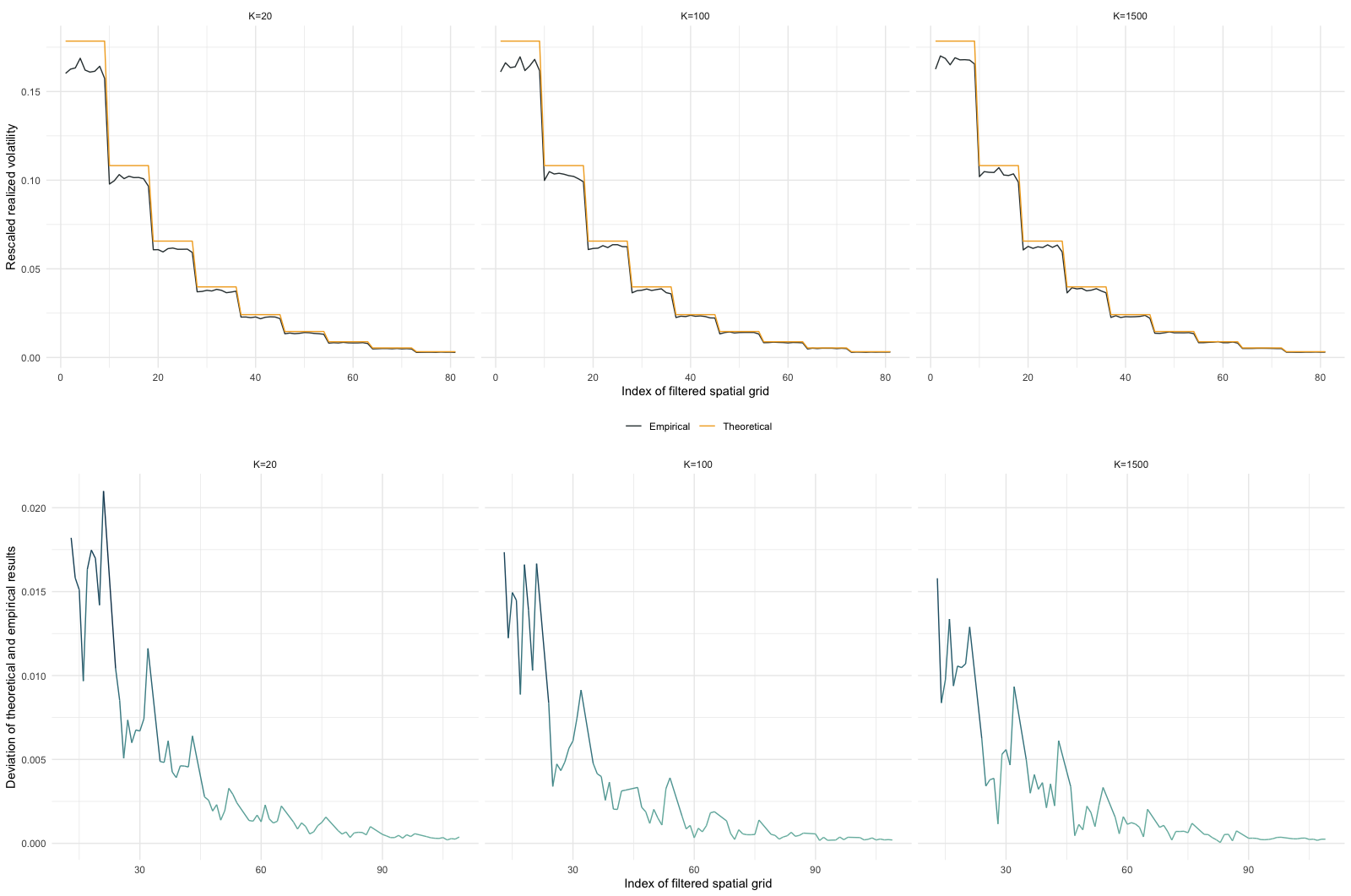}
\caption{The figure shows a comparison between the theoretical expected values of the rescaled realized volatilities, as described in Proposition \ref{prop_quadIncAndrescaling}, and their empirical counterparts (top row). The underlying two-dimensional SPDE model was simulated on an equidistant grid in both space and time, with $N=10^4$ and $M=10$. The SPDE model parameters are $\vartheta_0=0$, $\nu = (5,0)$, $\eta=1$, $\sigma=1$, and $\alpha'=4/10$, while the replacement bound was fixed at $L=10$. The results for different values of $K$ are displayed in three columns: $K=20$ (left), $K=100$ (middle), and $K=1500$ (right). The bottom row illustrates the deviation between the theoretical and empirical results.}
\label{fig_KvaluesComp}
\end{figure}
However, in the multivariate case, the series becomes more intricate due to the additional exponent $1+\alpha$ and the squaring of the summation indices. This complexity renders direct application of related series, such as the multiple zeta function or its extension, the multiple Lerch zeta function, impractical, cf. \cite{arakawa1999multiple} or \cite{gun2018multiple}.
Consequently, we currently resort to numerical approximation methods to estimate the variance $s_\textbf{m}^2$, given by
\begin{align*}
s_\textbf{m}^2\approx \sum_{\substack{\textbf{l}\in\mathcal{I}_\textbf{m} \\ \textbf{l}\in (0,KM)^d\backslash (0,LM)^d}}\frac{\sigma^2}{2\lambda_\textbf{l}^{1+\alpha}}=:\tilde{s}_\textbf{m},
\end{align*}
where $K>L$, $K\in\N$ denotes the cut-off of the approximation. The multi-dimensional replacement method is then given by
\begin{align}
X_{t_i}(\textbf{y}_\textbf{j})=\sum_{\textbf{m}\in\mathcal{M}}U_\textbf{m}(t_i)e_\textbf{m}(\textbf{y}_\textbf{j}),~~\text{where}~~U_\textbf{m}(t_i)=\sum_{\substack{\textbf{l}\in\mathcal{I}_m \\ \textbf{l}\in (0,LM)^d}}x_\textbf{l}(t_i)+\tilde{R}_\textbf{m}(i),
\end{align}
where $\tilde{R}_\textbf{m}(i)\sim\mathcal{N}(0,\tilde{s}_\textbf{m})$ denote the respective replacement random variables with the cut-off variance $\tilde{s}_\textbf{m}$ and $t_{i+1}-t_i=1/N$, where $i=1,\ldots,N$.
In this numerical approach, the quality of the simulation is highly dependent on the chosen variance cut-off $K$, as this cut-off effects the quality of the replacements $\tilde{R}_\textbf{m}$. If $K$ is selected to be too small, it will result in a negative bias in the simulations. Therefore, it is essential to carefully select an appropriate value for $K$ to ensure accurate and reliable simulations without introducing any significant bias. 

In Figure \ref{fig_KvaluesComp}, we conducted a simulation of a two-dimensional SPDE model on a grid with $N=10^4$ temporal points and $M=10$ spatial points on each axis. The top row displays a comparison between the theoretically realized values, as per Proposition \ref{prop_quadIncAndrescaling}, and the sample mean of the rescaled realized volatility for three different cut-off values: $K=20, 100, 1500$.

The bottom row illustrates the corresponding deviations between the theoretical predictions and the empirical outcomes. Notably, for the case of $K=20$, a significant negative bias is observed, while the bias diminishes as the cut-off frequency increases. 
An implementation of this method on R-programming language can be found in the R-package \texttt{SecondOrderSPDEMulti}\footnote{Link to web-page: \url{https://github.com/pabolang/SecondOrderSPDEMulti}}, available on the web-page \texttt{github.com}.
When performing a Monte Carlo study, the variance $s_\textbf{m}$, needs to be calculated only once. Since the runtime for larger $K$ values, can be enormous, we have implemented an option within the function \texttt{simulateSPDEmodelMulti} in the named R-package. This option allows for the utilization of the precomputed variance $s_\textbf{m}$ using the function \texttt{variance\_approx}, which dramatically reduces runtime when performing a Monte Carlo study.

\begin{figure}[t]
\centering
\includegraphics[width=0.95\textwidth]{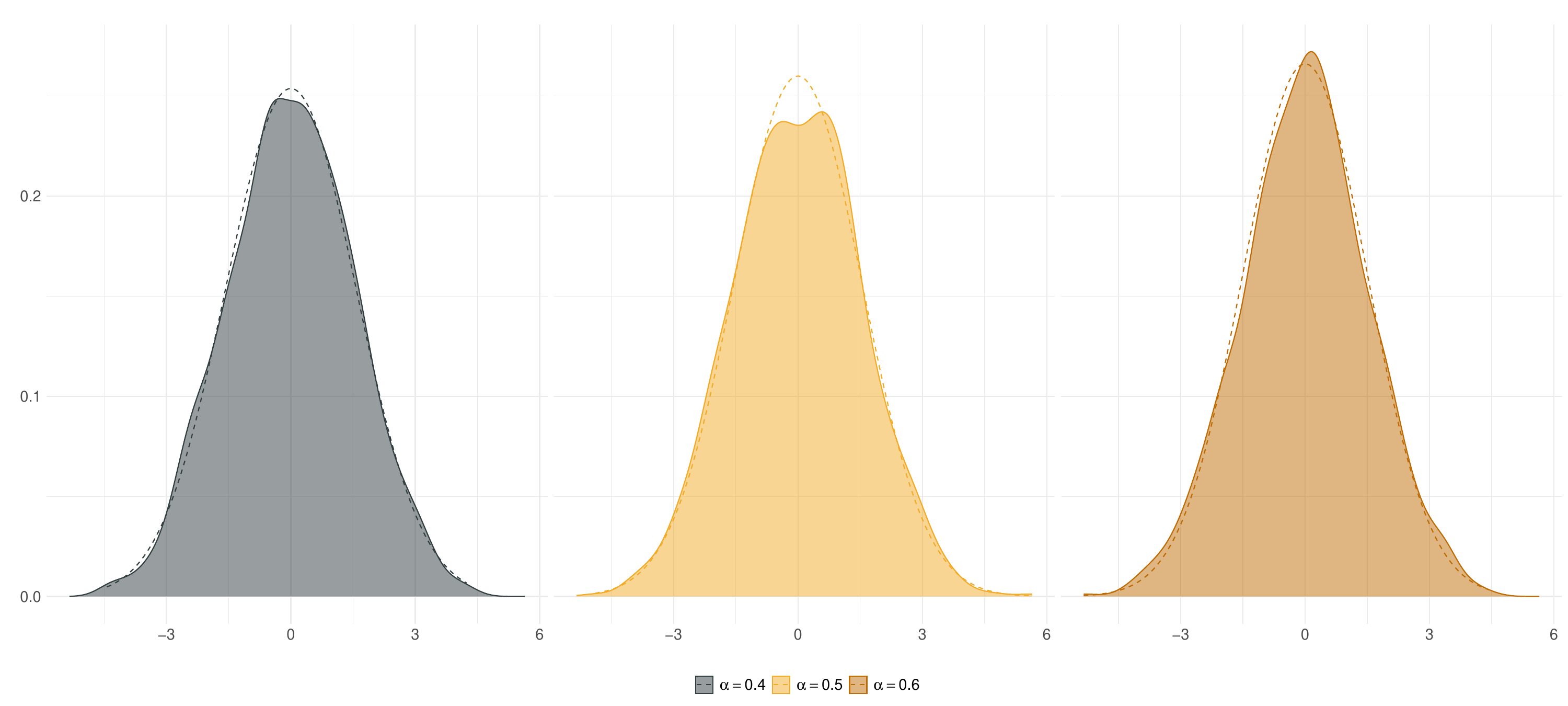}
\caption{
Comparison of the empirical distributions of normalized estimation errors for $\sigma^2$ obtained through simulation with $N=10^4$, $M=10$, and $\delta=0.05$ is presented. The kernel-density estimation utilized a Gaussian kernel with Silverman's 'rule of thumb' and was performed over 1000 Monte Carlo iterations. The specific parameter values employed were as follows: $d=2$, $\vartheta_0=0$, $\nu=(6,0)$, $\eta=1$, $\sigma=1$, $L=10$. Three scenarios were considered, with different values of $\alpha'$: $\alpha'=4/10 K=1000$ (left), $\alpha'=1/2, K=1000$ (middle), and $\alpha'=6/10, K=1300$ (right). The corresponding asymptotic distributions are represented by the dotted lines.}
\label{fig_dens_asymp_sigma_multi}
\end{figure}

\subsection{Monte Carlo simulation study}
\begin{figure}[t]
\centering
\includegraphics[width=0.9\textwidth]{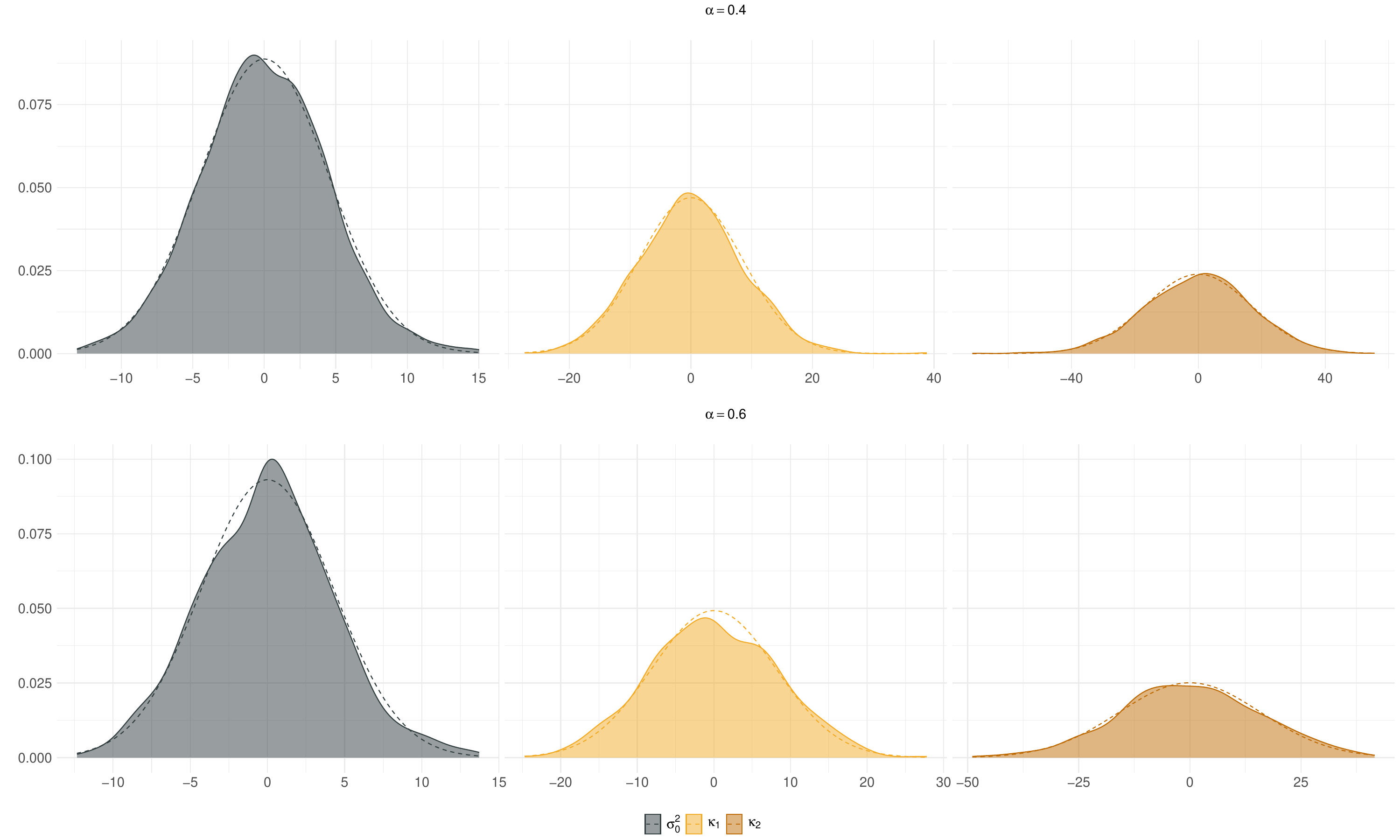}
\caption{
The figure provides a comparison of empirical distributions for centred estimation errors of the parameter vector $\upsilon=(\sigma_0^2,\kappa_1,\kappa_2)$, which are obtained through simulations on an equidistant gird in both, time and space, where $N=10^4$ and $M=10$. The kernel-density estimation employed a Gaussian kernel with Silverman's 'rule of thumb' and was conducted over 1000 Monte Carlo iterations. The specific parameter values used for the simulations are as follows: $d=2$, $\vartheta_0=0$, $\nu=(6,0)$, $\eta=1$, $\sigma=1$, and $L=10$. Two different scenarios were considered, each with a distinct value of the pure damping parameter $\alpha'$: $\alpha'=4/10, K=1000$ (top row) 
and $\alpha'=6/10, K=1300$ (bottom row). The corresponding asymptotic distributions are represented by dotted lines. The results for the normalized volatility parameter $\sigma_0^2$ are depicted in grey lines, while the estimates for $\kappa_1$ are shown in yellow lines, and those for $\kappa_2$ are represented by brown lines. For estimation of $\upsilon=(\sigma_0^2,\kappa_1,\kappa_2)$, we utilized the set of spatial points $\mathcal{S}_3$, as defined in equation \eqref{eqn_set_SP_forBetaHat}.
}
\label{fig_dens_asymp_kappaAndSigma_multi}
\end{figure}
To illustrate the central limit theorem described in Proposition \ref{prop_cltVolaEstMulti}, we conducted a Monte Carlo study. In this study, we simulated a 2-dimensional SPDE model based on equation \eqref{generalSPDEequationMulti}. Each simulation was performed on an equidistant grid in both time and space, with $N=10^4$ time steps and $M=10$ spatial steps, resulting in a total of $121$ spatial points. The simulation employed the following parameter values: $\vartheta_0=0$, $\nu=(6,0)$, $\eta=1$, $\sigma=1$, and $\alpha'$ taking on values from the set $\{4/10,1/2,6/10\}$, corresponding to three distinct damping scenarios. In each scenario, 1000 Monte Carlo iterations were executed. We utilized the replacement method detailed in Section \ref{sec_51}, with $L=10$, and for $\alpha'=4/10$ and $\alpha'=1/2$, we set a cut-off frequency of $K=10^3$, while for $\alpha'=6/10$, we used $K=1500$. 
Figure \ref{fig_dens_asymp_sigma_multi} presents a comparison between the empirical distribution of each scenario and the asymptotic normal distribution as stipulated in Proposition \ref{prop_cltVolaEstMulti}. To estimate the kernel density, we employed a Gaussian kernel with Silverman's 'rule of thumb'. As discussed in Section \ref{sec_51}, the replacement method introduced a notable negative bias due to the cut-off frequency $K$. To address this bias, we centred the data by utilizing the sample mean of the volatility estimations. This approach provided a clear basis for visually comparing the empirical and theoretical distributions.
All three scenarios exhibit a substantial fit, with the volatility estimator employing a spatial boundary of $\delta=0.05$, resulting in 81 spatial points for estimation. The sample mean of the volatility estimations were found to be $0.986$ for $\alpha'=4/10$, $0.975$ for $\alpha'=1/2$, and $0.988$ for $\alpha'=6/10$.

Figure \ref{fig_dens_asymp_kappaAndSigma_multi} depicts a comparison between the empirical distribution of each case and the asymptotic normal distribution as described in Corollary \ref{corollary_hatBetaCLT_61}. The top row shows the simulation results for $\alpha'=4/10$, 
and the bottom row presents the results for $\alpha'=6/10$. 
Each row consists of three plots, which assess the goodness of fit between the kernel density estimation and the centred normal distribution, as outlined in Corollary \ref{corollary_hatBetaCLT_61}. In these plots, grey represents the results for estimating the normalized volatility parameter $\sigma_0^2$, while the other panels in each row (yellow and brown) represent the results for the curvature parameters $\kappa_1$ and $\kappa_2$, respectively. 
To account for structural bias in the data, we centred the data by employing the sample mean of the corresponding estimates.

In this simulation study, where $N=10^4$, we must adhere to the following restriction, as outlined in Assumption \ref{assumption_observations_multi}:
\begin{align*}
M<N^{(1-\alpha')/(d+2)}\approx\begin{cases}
3.98 &,\text{ if }\alpha'=4/10 \\
3.16 &,\text{ if }\alpha'=1/2  \\
2.51 &,\text{ if }\alpha'=6/10
\end{cases}.
\end{align*}
As Assumption \ref{assumption_fullRank} necessitates a minimum of three observations for the application of the estimator $\hat{\upsilon}$, we have chosen the following observation scheme:
\begin{align}
\mathcal{S}_3:=\big\{(1/10,3/10),(4/10,2/10),(7/10,5/10)\big\}.\label{eqn_set_SP_forBetaHat}
\end{align}
Indeed, this observation scheme satisfies the Assumption \ref{assumption_fullRank}, as evident by the following calculation:
\begin{align*}
\begin{vmatrix}
1 & 1/10 & 3/10 \\
1 & 4/10 & 2/10 \\
1 & 7/10 & 5/10
\end{vmatrix}=0.12 \neq 0,
\end{align*}
where $\abs{A}$ denotes the determinant of a matrix $A\in\R^{p\times p}$, for $p\in\N$. 
For the cases $\alpha'\in\{4/10,1/2\}$, we obtain that $\abs{\mathcal{S}_3}<N^{(1-\alpha')/(d+2)}$, whereas the assumption \ref{assumption_observations_multi} is (slightly) violated for $\alpha=6/10$, since $\abs{\mathcal{S}_3}>N^{(1-\alpha')/(d+2)}$. 
Nevertheless, we present the simulation results in Figure \ref{fig_dens_asymp_kappaAndSigma_multi} for the two cases $\alpha'=4/10$ and $\alpha'=6/10$ and observe that both scenarios exhibit a substantial fit.
Since the results for the case $\alpha'=1/2$ are comparable to the two cases presented for $\alpha'$, we omit this plot.
The sample means of the respective estimations are summarized in Table \ref{table_11}.
\begin{table}[h!]
\begin{center}
\begin{tabular}{|c|c|c|c|c|c|c|c|c|c|}
\hline
$\alpha'$& mean $\hat{\sigma}_0^2$ & mean $\hat{\kappa}_1$ & mean $\hat{\kappa}_2$ \\
\hline
4/10 &0.985 &5.986 & 0.011\\
5/10 &0.972 &5.979& 0.028\\
6/10 &0.987&5.941& 0.038
\\ \hline
\end{tabular}
\caption{
The table presents the sample means of the estimations for the natural parameters $\sigma_0^2$ and $\kappa=(\kappa_1,\kappa_2)$ in a two-dimensional SPDE model. These estimations are derived from a dataset with parameters set at $\vartheta_0=0$, $\nu=(6,0)$, $\eta=1$, and $\sigma=1$, based on 1000 Monte Carlo iterations. Each row in the table corresponds to the outcomes for three selections of the pure damping parameter, where $\alpha'\in\{4/10,1/2,6/10\}$.}
\label{table_11}
\end{center}
\end{table}

We close this section by providing density plots for estimating the parameter $\alpha'$. Figure \ref{fig_dens_asymp_alpha_multi} shows a comparison between the empirical distribution of each case and the asymptotic normal distribution as described in Proposition \ref{prop_cltAlpha}. The left panel shows the simulation results for $\alpha'=4/10$, the middle panel displays the results for $\alpha'=1/2$, and the right panel presents the results for $\alpha'=6/10$. To account for structural bias in the data, we centred the data by employing the sample mean of the corresponding estimates.
\begin{figure}[t]
\centering
\includegraphics[width=0.9\textwidth]{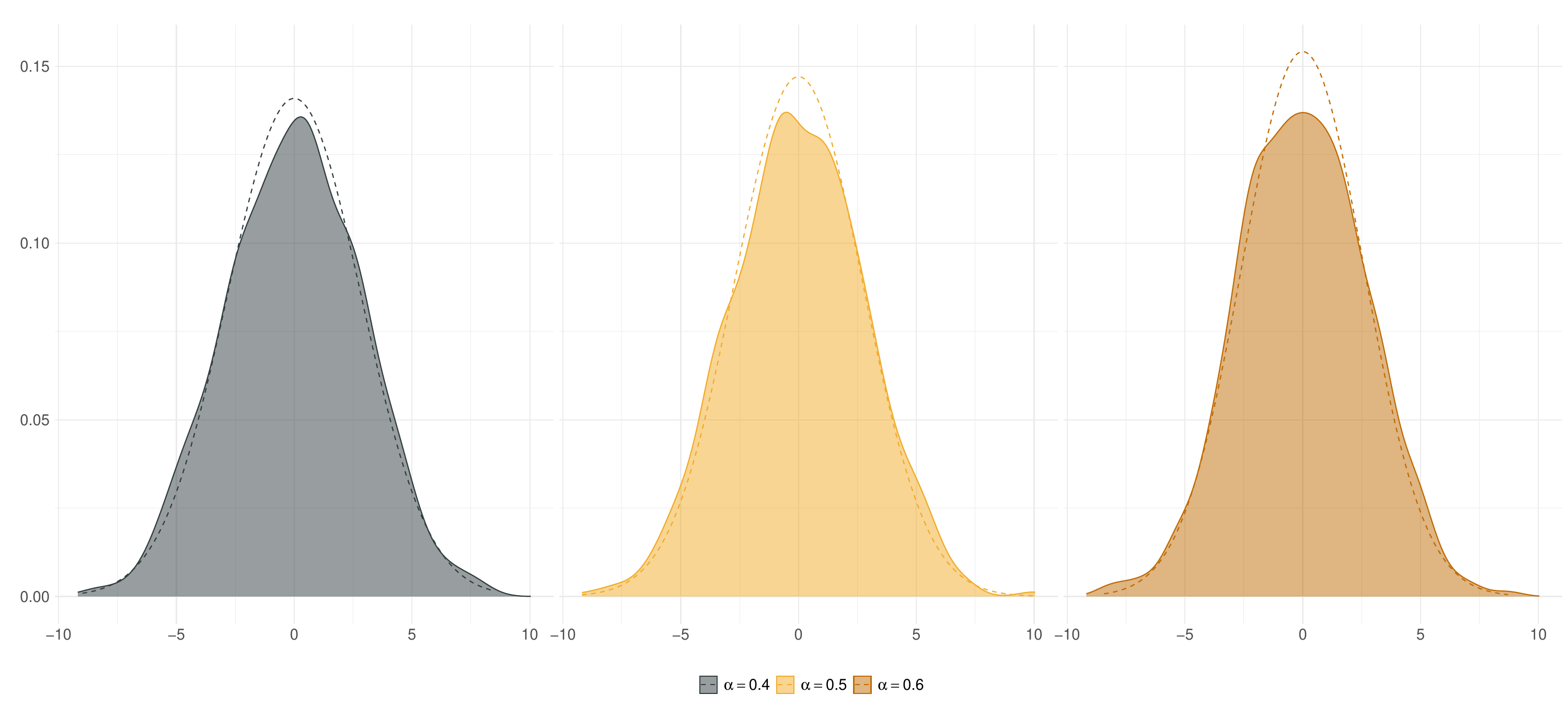}
\caption{
The figure provides a comparison of empirical distributions for centred estimation errors of $\alpha'$, which are obtained through simulations on an equidistant gird in both, time and space, where $N=10^4$ and $M=10$ and $\delta=0.05$. The kernel-density estimation employed a Gaussian kernel with Silverman's 'rule of thumb' and was conducted over 1000 Monte Carlo iterations. The specific parameter values used for the simulations were as follows: $d=2$, $\vartheta_0=0$, $\nu=(6,0)$, $\eta=1$, $\sigma=1$, and $L=10$. Three different scenarios were considered, each with a distinct value of the pure damping parameter $\alpha'$: $\alpha'=4/10 K=1000$ (left), $\alpha'=1/2, K=1000$ (middle), and $\alpha'=6/10, K=1300$ (right). The corresponding asymptotic distributions are represented by dotted lines. 
}
\label{fig_dens_asymp_alpha_multi}
\end{figure}

To estimate the damping parameter, we adopted a spatial threshold of $\delta=0.05$, which led to the utilization of 81 spatial coordinates for estimation. The parameter choices employed for the two-dimensional SPDE model are consistent with the simulation study presented earlier for the previous estimators. All three scenarios exhibit a significant fit, where we observe a qualitative difference between lower values of $\alpha'$ and higher values. This distinction can be attributed to the fact that $\alpha$ governs the Hölder regularity of the sample paths. Lower values of $\alpha$ result in rougher paths, thereby yielding a more accurate fit. The sample means of the estimates are given by $0.393$ for $\alpha'=4/10$, $0.484$ for $\alpha'=1/2$ and $0.554$ for $\alpha'=6/10$.

\section{Proofs}\label{sec_6}
We begin by clarifying some notations used in this paper:
\begin{align*}
\nor{\textbf{x}}_1:=\sum_{l=1}^d x_l,~~~~~\euc{\textbf{x}}:=\bigg(\sum_{l=1}^dx_l^2\bigg)^{1/2},~~~~~~\nor{\textbf{x}}_\infty:=\max_{l=1,\ldots,d}\abs{x_l}.
\end{align*}
Note, that the introduced notations $\nor{\cdot}_2,\nor{\cdot}_\infty$ define a norm on $\R^d$. However, the notations $\nor{\cdot}_0$ and $\nor{\cdot}_1$ do not define a norm, as they do not even map to the non-negative real numbers. Nevertheless, we use a norm notation to indicate an operation across all the spatial dimensions. 

For a measurable function $f:\R^d\too\R$ we define the $\mathcal{L}^p$-norm by
\begin{align*}
\nor{f}_{\mathcal{L}^p(D)}:=\bigg(\int_D\abs{f(\textbf{x})}^p\diff \textbf{x}\bigg)^{1/p},
\end{align*}
where $D\subseteq\R^d$. 
Finally, we define the point-wise product by 
\begin{align*}
\bigcdot: \R^d\times\R^d&\too\R^d\\ \textbf{x}\bigcdot \textbf{y} &\mapsto (x_1y_1,\ldots,x_dy_d)^\top.
\end{align*}
We say for $\textbf{k},\textbf{j}\in\N^d$ that they are not alike, i.e. $\textbf{k}\neq \textbf{j}$, if there exists at least one index $l_0\in\{1,\ldots,d\}$ with $k_{l_0}\neq j_{l_0}$. 

In the following we use the decomposition of the increments $x_\textbf{k}$ for $\textbf{k}\in\N^d$, given by
\begin{align}
\Delta_ix_\textbf{k}&=\langle\xi,e_\textbf{k}\rangle_\vartheta\big(e^{-\lambda_\textbf{k}i\Delta_n}-e^{-\lambda_\textbf{k}(i-1)\Delta_n}\big)+\sigma\lambda_\textbf{k}^{-\alpha/2}\int_0^{(i-1)\Delta_n}e^{-\lambda_\textbf{k}(i\Delta_n-s)}-e^{-\lambda_\textbf{k}((i-1)\Delta_n-s)}\diff W_s^\textbf{k}\notag\\
&~~~~~+\sigma\lambda_\textbf{k}^{-\alpha/2}\int_{(i-1)\Delta_n}^{i\Delta_n}e^{-\lambda_\textbf{k}(i\Delta_n-s)}\diff W_s^\textbf{k}\notag\\
&=A_{i,\textbf{k}}+B_{i,\textbf{k}}+C_{i,\textbf{k}},\label{equation_decompIncrement}
\end{align}
where 
\begin{align}
A_{i,\textbf{k}}&:=\langle\xi,e_\textbf{k}\rangle_\vartheta\big(e^{-\lambda_\textbf{k}i\Delta_n}-e^{-\lambda_\textbf{k}(i-1)\Delta_n}\big),\label{eqn_Aik}\\
B_{i,\textbf{k}}&:= \sigma\lambda_\textbf{k}^{-\alpha/2}\int_0^{(i-1)\Delta_n}e^{-\lambda_\textbf{k}((i-1)\Delta_n-s)}\big(e^{-\lambda_\textbf{k}\Delta_n}-1\big)\diff W_s^\textbf{k},\label{eqn_Bik}\\
C_{i,\textbf{k}}&:=\sigma\lambda_\textbf{k}^{-\alpha/2}\int_{(i-1)\Delta_n}^{i\Delta_n}e^{-\lambda_\textbf{k}(i\Delta_n-s)}\diff W_s^\textbf{k}.\label{eqn_Cik}
\end{align}

\subsection{Proofs of Section 3}
This section is structured in two parts. The first parts provides the proofs for calculating the expected value of the rescaled realized volatilities and the decay of the autocovariance, as stated in Propositions \ref{prop_quadIncAndrescaling} and \ref{prop_autocovOfIncrementsMulti}. The second part proofs the central limit theorem for the estimator $\hat{\sigma}^2$.
\subsubsection{Proofs of Propositions \ref{prop_quadIncAndrescaling} and \ref{prop_autocovOfIncrementsMulti}}
For proving Proposition \ref{prop_quadIncAndrescaling}, we need some auxiliary lemmas. 
\begin{lemma}\label{lemma_riemannApprox_multi}
Let $f:[0,\infty)\too\R$ be twice continuously differentiable with $\nor{x^{d-1}f(x^2)}_{\mathcal{L}^1([0,\infty))}$, $\nor{x^{d}f'(x^2)}_{\mathcal{L}^1([1,\infty))}$ and $\nor{ x^{d+1}f''(x^2)}_{\mathcal{L}^1([1,\infty))}\leq C$ for some $C>0$, then it holds:
\begin{itemize}
\item[(i)] 
\begin{align*}
&\Delta_n^{d/2}\sum_{\textbf{k}\in\N^d}f(\lambda_\textbf{k}\Delta_n)=\frac{1}{2^{d}(\pi\eta)^{d/2}\Gamma(d/2)}\int_0^\infty x^{d/2-1}f(x)\diff x-\sum_{\substack{\nor{\gamma}_1=1 \\ \gamma\in\{0,1\}^{d}}}^{d-1}\int_{B_\gamma}f(\pi^2\eta\euc{\textbf{z}}^2)\diff \textbf{z}\\
&~~+\Oo\bigg(\int_{0}^{\sqrt{\Delta}_n}r^{d-1}\abs{f(r^2)}\diff r\vee \Delta_n\int_{\sqrt{\Delta}_n}^1 r^{d-1}\abs{f'(r^2)}\diff r\vee\Delta_n\int_{\sqrt{\Delta}_n}^1 r^{d+1}\abs{f''(r^2)}\diff r\bigg),
\end{align*}
where $B_\gamma$ defined in equation \eqref{eqn_BGamma}.
\item[(ii)] For $\{j_1,\ldots,j_l\}\subset\{1,\ldots,d\}$, $\gamma_{j,l}\in\{0,1\}^d$, where $(\gamma_{j,l})_i=\mathbbm{1}_{i\in\{j_1,\ldots,j_l\}}$, with $i=1,\ldots,d$ and $l=1,\ldots,(d-1)$, we have 
\begin{align*}
&\Delta_n^{d/2}\sum_{k\in\N^d}f(\lambda_k\Delta_n)\cos(2\pi k_{j_1}y_{j_1})\cdot\ldots\cdot\cos(2\pi k_{j_l}y_{j_l})=(-1)^l\int_{B_{\gamma_{j,l}}}f(\pi^2\eta\euc{\textbf{z}}^2)\diff \textbf{z}\\
&~~~~~+\Oo\bigg(\max_{k=0,\ldots,l}\frac{\Delta_n^{k/2+1}}{\delta^{l+1}}\int_{\sqrt{\Delta}_n}^1r^{d-k+1}|f''(r^2)|\diff r\vee \max_{k=0,\ldots,l}\frac{\Delta_n^{k/2+1}}{\delta^{l+1}}\int_{\sqrt{\Delta}_n}^1 r^{d-k-1}|f'(r^2)|\diff r\bigg)\\
&~~~~~+ \Oo\bigg( \frac{\Delta_n^{(l+1)/2}}{\delta^{l}}\int_{\sqrt{\Delta}_n}^1 r^{d-l}|f'(r^2)|\diff r\bigg).
\end{align*}
\item[(iii)] For $\{j_1,\ldots,j_l\}=\{1,\ldots,d\}$, i.e. $l=d$, we have 
\begin{align*}
&\Delta_n^{d/2}\sum_{k\in\N^d}f(\lambda_k\Delta_n)\cos(2\pi k_{1}y_{1})\cdot\ldots\cdot\cos(2\pi k_{d}y_{d})=\Oo\big(\Delta_n^{d/2}\abs{f(\Delta_n)}\big)+\Oo\bigg(\frac{\Delta_n^{d/2}}{\delta^{d}}\int_{\sqrt{\Delta}_n}^1 r\abs{f'(r^2)}\diff r\bigg)\\
&~~+\Oo\bigg(\max_{k=0,\ldots,d-1}\frac{\Delta_n^{k/2+1}}{\delta^{d+1}}\int_{\sqrt{\Delta}_n}^1r^{d-k+1}\abs{f''(r^2)}\diff r\vee \max_{k=0,\ldots,d-1}\frac{\Delta_n^{k/2+1}}{\delta^{d+1}}\int_{\sqrt{\Delta}_n}^1 r^{d-k-1}\abs{f'(r^2)}\diff r\bigg).
\end{align*}
\end{itemize}
In particular, it holds for a $\tilde{\gamma}\in\{0,1\}^d$ with $\nor{\tilde{\gamma}}_1=l$ and $1\leq l\leq d-1$ that
\begin{align*}
\int_{B_{\tilde{\gamma}}}f(\pi^2\eta\euc{\textbf{z}}^2)\diff \textbf{z}=\Oo\bigg(\Delta_n^{l/2}\int_{\sqrt{\Delta}_n}^1r^{d-1-l}\abs{f(r^2)}\diff r\bigg),
\end{align*}
and 
\begin{align*}
&\sum_{\substack{\nor{\gamma}_1=1 \\ \gamma\in\{0,1\}^{d}}}^{d}\int_{B_\gamma}f(\pi^2\eta\euc{\textbf{z}}^2)\diff \textbf{z}=\Oo\bigg(\max_{l=1,\ldots,d-1}\Delta_n^{l/2}\int_{\sqrt{\Delta}_n}^1r^{d-1-l}\abs{f(r^2)}\diff r\vee \int_{0}^{\sqrt{\Delta}_n}r^{d-1}\abs{f(r^2)}\diff r\bigg).
\end{align*}
\end{lemma}
\begin{proof}
\label{proof_lemmaRiemannMulti}
We begin this proof by making the substitution $z_l^2=k_l^2\Delta_n$, such that
\begin{align*}
\lambda_\textbf{k}\Delta_n=\pi^2\eta\sum_{l=1}^dz_l^2+\Delta_n\bigg(\sum_{l=1}^d\Big(\frac{\nu_l^2}{4\eta}\Big)-\vartheta_0\bigg).
\end{align*}
Subsequently, employing the Taylor expansion with the Lagrange remainder, we obtain that
\begin{align*}
f(\lambda_\textbf{k}\Delta_n)=f\bigg(\pi^2\eta\sum_{l=1}^dz_l^2\bigg)+f'(\xi)\bigg(\lambda_\textbf{k}\Delta_n-\pi^2\eta\sum_{l=1}^dz_l^2\bigg)=f\bigg(\pi^2\eta\sum_{l=1}^dz_l^2\bigg)+\Oo(\Delta_n).
\end{align*}
For $\textbf{k}\in\N^d$ we define:
\begin{align*}
a_\textbf{k}:=(a_{k_1},\ldots,a_{k_d})\in\R^d_+,~~~~~\text{with}~~~~~a_{k_l}:=\sqrt{\Delta}_n(k_l+1/2),
\end{align*}
where $l=1,\ldots,d$ and 
\begin{align*}
[a_{\textbf{k-1}},a_\textbf{k}]:=[a_{k_1-1},a_{k_1}]\times\ldots\times[a_{k_d-1},a_{k_d}]\subset (0,\infty)^d.
\end{align*}
Note, that $\abs{a_{k_l}-a_{k_l-1}}=\sqrt{\Delta}_n$ for $l=1,\ldots,d$ and $a_0:=\sqrt{\Delta}_n/2$.
Moreover, by defining $\tilde{f}(x):=f(\pi^2\eta x^2)$, we observe that 
\begin{align*}
&\Delta_n^{d/2}\sum_{\textbf{k}\in\N^d}f(\lambda_\textbf{k}\Delta_n)-\int_{[\sqrt{\Delta}_n/2,\infty)^d}f(\pi^2\eta\euc{\textbf{z}}^2)\diff \textbf{z}\\
&= \Delta_n^{d/2}\sum\limits_{k_1=1}^\infty\cdots\sum\limits_{k_d=1}^\infty f\bigg(\pi^2\eta\Delta_n\sum_{l=1}^dk_l^2 \bigg)-\sum\limits_{k_1=1}^\infty\cdots\sum\limits_{k_d=1}^\infty \int_{a_{k_1-1}}^{a_{k_1}}\cdots \int_{a_{k_d-1}}^{a_{k_d}} f\bigg(\pi^2\eta\sum_{l=1}^dz_l^2\bigg)\diff z_1\cdots\diff z_d +\mathcal{O}(\Delta_n) \\
&=\sum_{\textbf{k}\in\N^d}\int_{a_{\textbf{k-1}}}^{a_\textbf{k}}f(\pi^2\eta\Delta_n\euc{\textbf{k}}^2)-f(\pi^2\eta\euc{\textbf{z}}^2)  \diff \textbf{z}+\Oo(\Delta_n)\\
&=\sum_{\textbf{k}\in\N^d}\int_{a_{\textbf{k-1}}}^{a_\textbf{k}}\tilde{f}(\sqrt{\Delta}_n\euc{\textbf{k}})-\tilde{f}(\euc{\textbf{z}})  \diff \textbf{z}+\Oo(\Delta_n)=:T_1+\Oo(\Delta_n),
\end{align*}
where $\euc{\cdot}$ denotes the euclidean norm. Define the function $g:\R_+^d\too\R_+$, with $g(\textbf{x})=\tilde{f}(\euc{\textbf{x}})$. Since $\sqrt{\Delta}_n\textbf{k}$ represents the mid-point of the interval $[a_{\textbf{k-1}},a_\textbf{k}]$ for a $\textbf{k}\in\N^d$, we can apply a Taylor expansion at the point $\sqrt{\Delta}_n\textbf{k}$, leading to the following expression:
\begin{align}
g\big(\sqrt{\Delta}_n\textbf{k}\big)-g(\textbf{z})&=g\big(\sqrt{\Delta}_n\textbf{k}\big)-\Big(g\big(\sqrt{\Delta}_n\textbf{k}\big)+\nabla g\big(\sqrt{\Delta}_n\textbf{k}\big)^\top\big(\textbf{z}-\sqrt{\Delta}_n\textbf{k}\big)\notag \\
&~~~~~+\frac{1}{2}\big(\textbf{z}-\sqrt{\Delta}_n\textbf{k}\big)^\top H_g(\xi_\textbf{k})\big(\textbf{z}-\sqrt{\Delta}_n\textbf{k}\big)\Big),\label{eqn_multidim_taylor_2ndOrder}
\end{align}
where $\nabla g$ denotes the gradient of $g$, $H_g$ the Hessian-matrix of $g$ and $\xi_\textbf{k}\in[a_{\textbf{k-1}},a_\textbf{k}]$.
Let us introduce the shorthand notation $g'_l(\textbf{z}):=\partial g(\textbf{z})/(\partial z_l)$, which represents the partial derivative of $g(\textbf{z})$ with respect to $z_l$. Then, we have:
\begin{align*}
\int_{a_{\textbf{k-1}}}^{a_{\textbf{k}}}\nabla g(\sqrt{\Delta}_n\textbf{k})^\top(\textbf{z}-\sqrt{\Delta}_n\textbf{k})\diff \textbf{z}
&=\Delta_n^{(d-1)/2}\sum_{l=1}^d g_l'(\sqrt{\Delta}_n\textbf{k})\int_{a_{k_l-1}}^{a_{k_l}} \big(z_l-\sqrt{\Delta}_nk_l\big)\diff z_l=0.
\end{align*}
Since every term in the Taylor expansion from equation \eqref{eqn_multidim_taylor_2ndOrder} disappears, we proceed by redefining the term $T_1$ as follows:
\begin{align}
T_1&:=-\sum_{\textbf{k}\in\N^d}\int_{a_{\textbf{k-1}}}^{a_\textbf{k}}\frac{1}{2}(\textbf{z}-\sqrt{\Delta}_n\textbf{k})^\top H_g(\xi_\textbf{k})(\textbf{z}-\sqrt{\Delta}_n\textbf{k})\big)\diff \textbf{z}.\label{eqn_T1_2ndDerivationBound}
\end{align}
Additionally, the order of the term $T_1$ will be analysed in display \eqref{eqn_taylor2ndErrorOrder}.
For now, our primary focus is on the main term, which can expressed by:
\begin{align*}
&\Delta_n^{d/2}\sum\limits_{k_1=1}^\infty\cdots\sum\limits_{k_d=1}^\infty f(\lambda_{(k_1,\ldots,k_d)}\Delta_n)\\
&=\int_{\frac{\sqrt{\Delta}_n}{2}}^\infty\cdots\int_{\frac{\sqrt{\Delta}_n}{2}}^\infty f\bigg(\pi^2\eta \sum_{l=1}^dz_l^2\bigg)\diff z_1\cdots\diff z_d+\Oo(T_1\vee\Delta_n)\\
&=\int_{0}^\infty\cdots\int_{0}^\infty f\bigg(\pi^2\eta \sum_{l=1}^dz_l^2\bigg)\diff z_1\cdots\diff z_d-\int_{\R^d_+\backslash[\sqrt{\Delta}_n/2,\infty)^d}f(\pi^2\eta\euc{\textbf{z}}^2)\diff \textbf{z}+\Oo(T_1\vee\Delta_n).
\end{align*}
Before delving into the analysis of the compensation integral, defined by:
\begin{align*}
\mathcal{I}:=\int_{\R^d_+\backslash[\sqrt{\Delta}_n/2,\infty)^d}f(\pi^2\eta\euc{\textbf{z}}^2)\diff \textbf{z},
\end{align*}
and the error term $T_1$, let us first examine a transformation of the main integral. 
To facilitate our analysis, we employ $d$-dimensional spherical coordinates and we have:
\begin{align*}
&\int_0^\infty\cdots\int_0^\infty f\bigg(\eta\pi^2\sum_{l=1}^dz_l^2\bigg)\diff z_1\cdots\diff z_d\\
&=\int_0^\infty r^{d-1}f(\pi^2\eta r^2)\diff r\int_0^{\pi/2}\sin^{d-2}(\varphi_1)\diff \varphi_1\cdots\int_0^{\pi/2}\sin(\varphi_{d-2})\diff\varphi_{d-2}\int_0^{\pi/2}\diff\varphi_{d-1}.
\end{align*}
For $l\in\N$, it holds that
\begin{align*}
\int_0^{\pi/2}\sin^l(x)\diff x=\frac{\sqrt{\pi}\Gamma\big(\frac{1+l}{2}\big)}{2\Gamma\big(1+\frac{l}{2}\big)},
\end{align*}
where $\Gamma(x)$ denotes the Gamma function. Furthermore, we obtain that
\begin{align*}
\prod_{l=1}^{d-2}\int_0^{\pi/2}\sin^l(x)\diff x=\frac{\pi^{d/2-1}}{2^{d-2}\Gamma(d/2)}.
\end{align*}
Thus, we have:
\begin{align*}
\int_0^\infty\cdots\int_0^\infty f\bigg(\eta\pi^2\sum_{l=1}^dz_l^2\bigg)\diff z_1\cdots\diff z_d=\frac{\pi^{d/2}}{2^{d-1}\Gamma(d/2)}\int_0^\infty r^{d-1}f(\pi^2\eta r^2)\diff r
\end{align*}
and therefore obtain:
\begin{align*}
\Delta_n^{d/2}\sum_{k\in\N^d}f(\lambda_k\Delta_n)&=\frac{1}{2^{d}(\pi\eta)^{d/2}\Gamma(d/2)}\int_0^\infty x^{d/2-1}f(x)\diff x-\mathcal{I}+\Oo(T_1\vee\Delta_n).
\end{align*}
To analyse the compensation term $\mathcal{I}$, we initiate the process by decomposing the set $\R^d_+\backslash[\sqrt{\Delta}_n/2,\infty)^d$. Let $\gamma\in\{0,1\}^d\backslash\{0\}^d$, where $\gamma=(\gamma_1,\ldots,\gamma_d)$ and let $\psi(x)=\mathbbm{1}_{[0,\sqrt{\Delta}_n/2)}(x)$. With these definitions, we can introduce the following set:
\begin{align}
B_\gamma:=\{x\in[0,\infty)^d : x_1\in\psi^{-1}(\gamma_1),\ldots,x_d\in\psi^{-1}(\gamma_d)\}\subset [0,\infty)^d.\label{eqn_BGamma}
\end{align}
Hence, we can decompose the set $\R^d_+\backslash[\sqrt{\Delta}_n/2,\infty)^d$ using the following disjoint union:
\begin{align}
\R^d_+\backslash[\sqrt{\Delta}_n/2,\infty)^d=\bigcup_{\substack{\nor{\gamma}_1=1 \\ \gamma\in\{0,1\}^d}}^dB_\gamma.\label{equation_disjointUnion}
\end{align}
which enables the decomposition of the integral $\mathcal{I}$ as follows:
\begin{align*}
\mathcal{I}&=\int_{\R^d_+\backslash[\sqrt{\Delta}_n/2,\infty)^d}f(\pi^2\eta\euc{\textbf{z}}^2)\diff \textbf{z}=\sum_{\substack{\nor{\gamma}_1=1 \\ \gamma\in\{0,1\}^d}}^d\int_{B_\gamma}f(\pi^2\eta\euc{\textbf{z}}^2)\diff \textbf{z}.
\end{align*}
Let us now focus on two cases. Firstly, the scenario where $\nor{\gamma}_1<d$, and secondly, the case where $\nor{\gamma}_1=d$.
In the first case, we assume that $\nor{\gamma}_1=l$, where $l \in \{1, \ldots, d-1\}$. This implies that there exist indices $\{i_1,\ldots,i_l\}\subset\{1,\ldots,d\}$ and $\{1,\ldots,d\}\backslash\{i_1,\ldots,i_l\}=\{j_1,\ldots,j_{d-l}\}$ with $\gamma_{i_k} = 1$ for $k = 1, \ldots, l$ and $\gamma_{j_k} = 0$ for $k = 1, \ldots, d-l$. Moreover, we assume that $i_1 < \ldots < i_l$ and $j_1 < \ldots < j_{d-l}$.

Although we are integrating over an area corresponding to an infinite hyperrectangle, transforming into $d$-dimensional spherical coordinates provides a convenient representation, facilitating the analysis of the integral's order. During the transformation into $d$-dimensional spherical coordinates, we can always ensure that the angles $\varphi_1, \ldots, \varphi_{d-1}$ are bounded by $(0, \pi/2)$, and consequently, we have:
\begin{align*}
\int_{B_\gamma}f(\pi^2\eta\euc{\textbf{z}}^2)\diff \textbf{z}=\Oo\bigg(\int_{\sqrt{\Delta}_n/2}^\infty r^{d-1}f(r^2)\diff r\bigg),
\end{align*}
where we used the fact that the radius $r$ is always greater or equal than $\sqrt{\Delta}_n/2$. However, given that $l$ dimensions vanish when integrating and as $n$ tends to infinity, we can determine the order more precisely. Therefore, we can always consider the transformation:
\begin{align}
x_{i_1}&=r\cos(\varphi_1),~~~x_{i_2}=r\sin(\varphi_1)\cos(\varphi_2),~~~\cdots~~~x_{i_l}=r\sin(\varphi_1)\cdots\sin(\varphi_{l-1})\cos(\varphi_l),~~~\ldots\notag\\ 
x_{j_1}&=r\prod_{k=1}^l\sin(\varphi_k)\cos(\varphi_{l+1}),~~~\cdots~~~x_{j_{d-l-1}}=r\prod_{k=1}^{d-2}\sin(\varphi_k)\cos(\varphi_{d-1}),~~~x_{j_{d-l}}=r\prod_{k=1}^{d-1}\sin(\varphi_k),\label{eqn_spericalTransswitch}
\end{align}
which allows without loss of generality to set $i_1=1,\ldots,i_l=l$ and $j_1=l+1,\ldots j_{d-l}=d$. 
We can bound the angles $\varphi_1, \ldots, \varphi_l$ as follows:
\begin{align*}
0\leq x_{k}=r\cos(\varphi_{k})\prod_{l=1}^{k-1}\sin(\varphi_l)\leq \frac{\sqrt{\Delta}_n}{2}&\Leftrightarrow \arccos\bigg(\frac{\sqrt{\Delta}_n}{2r\prod_{l=1}^{k-1}\sin(\varphi_l)}\bigg)\leq \varphi_{k}\leq \frac{\pi}{2},
\end{align*}
where $k=1,\ldots,l$ and $1\leq l\leq (d-1)$. By rearranging the integration order, we have 
\begin{align*}
&\int_{B_\gamma}f(\pi^2\eta\euc{\textbf{z}}^2)\diff \textbf{z}\\
&= \int_{\sqrt{\Delta}_n/2}^\infty\cdots\int_{\sqrt{\Delta}_n/2}^\infty\int_0^{\sqrt{\Delta}_n/2}\cdots\int_0^{\sqrt{\Delta}_n/2} f\bigg(\pi^2\eta\sum_{l=1}^dz_l^2\bigg)\diff z_{i_l}\cdots\diff z_{i_1}\diff z_{j_1}\cdots \diff z_{j_{d-l}}\\
&\leq\int_{\sqrt{\Delta}_n/2}^\infty f(\pi^2\eta r^2)\int_0^{\pi/2}\cdots\int_0^{\pi/2}\int_{\arccos(a_1)}^{\pi/2}\cdots\times\int_{\arccos(a_l)}^{\pi/2}\abs{J_d}\diff\varphi_l\cdots\diff\varphi_1\diff\varphi_{l+1}\cdots\diff\varphi_{d-1}\diff r,
\end{align*}
where 
\begin{align*}
b_1:=\frac{\sqrt{\Delta}_n}{2r},~~~~\cdots~~~~,b_l:=\frac{\sqrt{\Delta}_n}{2r\prod_{k=1}^{l-1}\sin(\varphi_k)}.
\end{align*}
Note, that we can use the following inequality for the determinant $\abs{J_d}$:
\begin{align*}
\abs{J_d}\leq r^{d-1}\sin(\varphi_1)^{l-1}\sin(\varphi_2)^{l-2}\cdots\sin(\varphi_{l-1}).
\end{align*}
By utilizing the identity $\pi/2 - \arccos(x) = \arcsin(x)$ and the inequality $\arcsin(x) \leq x\pi/2$, for $x \in [0,1]$, we deduce that
\begin{align*}
&\int_{\arccos(b_1)}^{\pi/2}\cdots\int_{\arccos(b_l)}^{\pi/2}\abs{J_d}\diff\varphi_l\cdots\diff\varphi_1\\
&\leq r^{d-1}\int_{\arccos(b_1)}^{\pi/2}\cdots\int_{\arccos(b_{l-1})}^{\pi/2}\int_{\arccos(b_l)}^{\pi/2}\diff\varphi_l\sin(\varphi_1)^{l-1}\cdots\sin(\varphi_{l-1})\diff\varphi_{l-1}\cdots\diff\varphi_1\\
&\leq \frac{r^{d-1}\pi}{2}\int_{\arccos(b_1)}^{\pi/2}\cdots\int_{\arccos(b_{l-1})}^{\pi/2}\frac{\sin(\varphi_1)^{l-1}\cdots\sin(\varphi_{l-1})\sqrt{\Delta}_n}{2r\sin(\varphi_1)\cdots\sin(\varphi_{l-1})}\diff\varphi_{l-1}\cdots\diff\varphi_1\\
&\leq C\Delta_n^{l/2}r^{d-1-l}.
\end{align*}
Therefore, we have 
\begin{align*}
\int_{B_\gamma}f(\pi^2\eta\euc{\textbf{z}}^2)\diff \textbf{z}=\Oo\bigg(\Delta_n^{l/2}\int_{\sqrt{\Delta}_n}^1r^{d-1-l}f(r^2)\diff r\bigg).
\end{align*}
Note, that this order applies to the derivatives as well, i.e.:
\begin{align}
\int_{B_\gamma}h(\pi^2\eta\euc{\textbf{z}}^2)\diff \textbf{z}&=\Oo\bigg(\Delta_n^{l/2}\int_{\sqrt{\Delta}_n}^1r^{d-1-l}h(r^2)\diff r\bigg),\label{eqn_BGammaOrderDerivatives}
\end{align}
where $h=f,f',f''$.
Now, let us consider the last case, where $\nor{\gamma}_1=d$. Since the radius is bounded by $\nor{\{\sqrt{\Delta}_n/2\}^d}_2=\sqrt{d\Delta}_n/2$, we can perform a transformation into $d$-dimensional spherical coordinates using the following inequality:
\begin{align*}
\int_{B_\gamma}f(\pi^2\eta\euc{\textbf{z}}^2)\diff \textbf{z}
&\leq \int_{\substack{\nor{\textbf{z}}_2\leq \sqrt{d\Delta}_n/2\\ \textbf{z}\in [0,\infty)^d}}f(\pi^2\eta\nor{\textbf{z}}_2^2)\diff z=\Oo\bigg(\int_{0}^{\sqrt{\Delta}_n}r^{d-1}f(r^2)\diff r\bigg).
\end{align*}
Consequently, we obtain the following order for the compensation integral $\mathcal{I}$:
\begin{align*}
\mathcal{I}&=\sum_{\substack{\nor{\gamma}_1=1 \\ \gamma\in\{0,1\}^{d}}}^{d-1}\int_{B_\gamma}f(\pi^2\eta\euc{\textbf{z}}^2)\diff \textbf{z}+\Oo\bigg(\int_{0}^{\sqrt{\Delta}_n}r^{d-1}f(r^2)\diff r\bigg)\\
&=\Oo\bigg(\max_{l=1,\ldots,d-1}\Delta_n^{l/2}\int_{\sqrt{\Delta}_n}^1r^{d-1-l}f(r^2)\diff r \vee \int_{0}^{\sqrt{\Delta}_n}r^{d-1}f(r^2)\diff r\bigg).
\end{align*}
Regarding the error term $T_1$ from equation \eqref{eqn_T1_2ndDerivationBound}, we obtain the following expression for $\textbf{z}\in[a_\textbf{k-1},a_\textbf{k}]$ and $\textbf{k} \in \N^d$:
\begin{align*}
(\textbf{z}-\sqrt{\Delta}_n\textbf{k})^\top &H_g(\textbf{z})(\textbf{z}-\sqrt{\Delta}_n\textbf{k})\leq \frac{\Delta_n}{4}\sum_{l_1=1}^d\sum_{l_2=1}^d \frac{\partial^2}{\partial z_{l_1}\partial z_{l_2}}f(\pi^2\eta\nor{\textbf{z}}_2^2) \\
&=C\bigg(\Delta_n\sum_{l_1=1}^d\sum_{l_2=1}^d z_{l_1}z_{l_2} f''(\pi^2\eta\nor{\textbf{z}}_2^2)+\frac{d\Delta_n}{2}f'(\pi^2\eta \nor{\textbf{z}}_2^2)\bigg)\\
&\leq C'd\Delta_n \Big(\nor{\textbf{z}}_2^2 f''(\pi^2\eta\nor{\textbf{z}}_2^2)+f'(\pi^2\eta \nor{\textbf{z}}_2^2)\Big),
\end{align*}
where $C,C'>0$ are suitable constants.
Hence, we have 
\begin{align*}
T_1=\Oo\bigg(\Delta_n\int_{[\sqrt{\Delta}_n/2,\infty)^d} \nor{\textbf{z}}_2^2 f''(\nor{\textbf{z}}_2^2)\diff\textbf{z}\vee \Delta_n\int_{[\sqrt{\Delta}_n/2,\infty)^d}f'(\nor{\textbf{z}}_2^2)\diff\textbf{z} \bigg).
\end{align*}
Once more, through the transformation into $d$-dimensional spherical coordinates, we can deduce the order of the Lagrange remainder $T_1$ as follows:
\begin{align}
\abs{T_1}
&=\Oo\bigg(\Delta_n\int_{\sqrt{\Delta}_n}^1 r^{d+1}\abs{f''(r^2)}\diff r\vee \Delta_n\int_{\sqrt{\Delta}_n}^1 r^{d-1}\abs{f'(r^2)}\diff r\bigg),
\label{eqn_taylor2ndErrorOrder}
\end{align}
which completes the proof of the first assertion. \ \\ \\
We begin the proof of (ii) by establishing the following identity:
\begin{align}
\prod_{l=1}^n\cos(x_l)=\frac{1}{2^{n-1}}\sum_{\textbf{u}\in C_n}\cos(\textbf{u}^\top \textbf{x}),\label{eqn_cosProdandSum}
\end{align} 
where $\textbf{x}=(x_1,\ldots,x_n)^\top$ and $C_n:=\{1\}\times\{-1,1\}^{n-1}$, with $\abs{C_n}=2^{n-1}$ and $n\geq 1$. 
We demonstrate that this identity can be derived using induction. For $n\in\{1,2\}$, the identity is readily observed by utilizing the elementary trigonometric identity $\cos(x \pm y) = \cos(x)\cos(y) \mp \sin(x)\sin(y)$. Now, let us assume that the advanced identity holds for an arbitrary $n\in\mathbb{N}$. For $n+1$, we consider $\textbf{x}=(\textbf{y},z)\in\R^{n+1}$, where $\textbf{y} \in \R^n$ and $z \in \R$. Then we have:
\begin{align*}
\frac{1}{2^{n}}\sum_{\textbf{u}\in C_{n+1}}\cos(\textbf{u}^\top \textbf{x})&=\frac{1}{2^{n}}\sum_{\textbf{u}\in C_{n}}\big(\cos(\textbf{u}^\top \textbf{y}+z)+(\cos(\textbf{u}^\top \textbf{y}-z)\big)\\
&=\frac{1}{2^{n-1}}\sum_{\textbf{u}\in C_{n}}\cos(\textbf{u}^\top \textbf{y})\cos(z)=\prod_{l=1}^{n+1}\cos(x_l).
\end{align*}
By utilizing equation \eqref{eqn_cosProdandSum}, we arrive at the following structure:
\begin{align*}
\Delta_n^{d/2}\sum_{\textbf{k}\in\N^d}f(\lambda_\textbf{k}\Delta_n)\cos(2\pi k_{j_1}y_{j_1})\cdot\ldots\cdot\cos(2\pi k_{j_l}y_{j_l})&=\frac{\Delta_n^{d/2}}{2^{l-1}}\sum_{\textbf{k}\in\N^d}f(\lambda_\textbf{k}\Delta_n)\sum_{\textbf{u}\in C_l}\cos(2\pi \textbf{u}^\top (\textbf{y}\bigcdot \textbf{k})_{j,l})
\\
&=\Re\bigg(\frac{\Delta_n^{d/2}}{2^{l-1}}\sum_{\textbf{k}\in\N^d}g(\sqrt{\Delta}_n\textbf{k})\sum_{\textbf{u}\in C_l}e^{\im2\pi \textbf{u}^\top (\textbf{y}\bigcdot \textbf{k})_{j,l}}\bigg)+\Oo(\Delta_n),
\end{align*}
where $(\textbf{y}\bigcdot \textbf{k})_{j,l}:=(k_{j_1}y_{j_1},\ldots,k_{j_l}y_{j_l})$ and $\{j_1,\ldots,j_l\}\subset\{1,\ldots,d\}$ and $l=1,\ldots,(d-1)$. Furthermore, it holds with $u_i\in\{-1,1\}$, $i\in\N$, that
\begin{align*}
\int_{a_{\textbf{k-1}}}^{a_\textbf{k}}e^{\im 2\pi \sum_{i=1}^lu_iy_{j_i}z_{j_i}\Delta_n^{-1/2}}\diff \textbf{z}
&=\frac{\Delta_n^{d/2}}{(\im2\pi)^l\prod_{i=1}^lu_iy_{j_i}} \prod_{i=1}^l\big(e^{\im2\pi a_{k_{j_i}}u_iy_{j_i}\Delta_n^{-1/2}}-e^{\im2\pi a_{k_{j_i}-1}u_iy_{j_i}\Delta_n^{-1/2}}\big)\\
&=\frac{\Delta_n^{d/2}\prod_{i=1}^l\sin(\pi y_{j_i})}{\pi^l\prod_{i=1}^ly_{j_i}}e^{\im 2\pi\sum_{i=1}^l u_ik_{j_i}y_{j_i}}.
\end{align*}
Defining $\textbf{y}_{j,l}:=(y_{j_1},\ldots,y_{j_l})$ and $\chi:=\chi_{j,l}:\R^l\too\R^d$, where the $i$-th component $(\chi_{j,l}(\textbf{x}))_i$ of $\chi_{j,l}(\textbf{x})$ is zero if $i\in\{1,\ldots,d\}\backslash\{j_1,\ldots,j_l\}$ or else the coordinate $x_{j_i}$, lead to:
\begin{align}
&\Re\bigg(\frac{\Delta_n^{d/2}}{2^{l-1}}\sum_{\textbf{k}\in\N^d}g(\sqrt{\Delta}_n\textbf{k})\sum_{\textbf{u}\in C_l}e^{\im2\pi \textbf{u}^\top (\textbf{y}\bigcdot \textbf{k})_{j,l}}\bigg)\notag \\
&=\sum_{\textbf{u}\in C_l}\Re\bigg(\frac{\pi^l\prod_{i=1}^ly_{j_i}}{2^{l-1}\prod_{i=1}^l\sin(\pi y_{j_i})}\mathcal{F}\bigg[\sum_{\textbf{k}\in\N^d}g(\sqrt{\Delta}_n\textbf{k})\mathbbm{1}_{(a_{\textbf{k-1}},a_{\textbf{k}}]}\bigg]\big(-2\pi\chi(\textbf{u}\bigcdot \textbf{y}_{j,l})\Delta_n^{-1/2}\big)\bigg)=:T_2+T_3,\label{eqn_decompT2T3Fourrier}
\end{align}
where $\mathcal{F}$ denotes the Fourier transformation for a $f\in\mathcal{L}^1(\R^d)$. Since we analyse functions $f:[0,\infty)^d\too\R$ the Fourier transformation is given by integrating over $[0,\infty)^d$. Hence, we define $T_2:=\sum_{\textbf{u}\in C_l}T_{2,\textbf{u}}$, $T_3:=\sum_{\textbf{u}\in C_l}T_{3,\textbf{u}}$, where the components are given by:
\begin{align}
T_{2,\textbf{u}}
&:=\Re\bigg(\frac{\pi^l\prod_{i=1}^ly_{j_i}}{2^{l-1}\prod_{i=1}^l\sin(\pi  y_{j_i})}\mathcal{F}\bigg[\sum_{\textbf{k}\in\N^d}g(\sqrt{\Delta}_n\textbf{k})\mathbbm{1}_{(a_{\textbf{k-1}},a_{\textbf{k}}]}-(-1)^lg\mathbbm{1}_{B_{\gamma_{j,l}}}\bigg]\big(-2\pi\chi(\textbf{u}\bigcdot \textbf{y}_{j,l})\Delta_n^{-1/2}\big)\bigg),\label{eqn_T2u}\\
T_{3,\textbf{u}}&:=(-1)^l\Re\bigg(\frac{\pi^l\prod_{i=1}^ly_{j_i}}{2^{l-1}\prod_{i=1}^l\sin(\pi y_{j_i})}\mathcal{F}\big[g\mathbbm{1}_{B_{\gamma_{j,l}}}\big]\big(-2\pi\chi(\textbf{u}\bigcdot \textbf{y}_{j,l})\Delta_n^{-1/2}\big)\bigg),
\end{align}
with $B_{\gamma}$ is defined in equation \eqref{eqn_BGamma} and $\gamma_{j,l}\in\{0,1\}^d$, where $(\gamma_{j,l})_i=1$ if $i\in\{j_1,\ldots,j_l\}$ or zero otherwise. 
Beginning with the analysis of the term $T_3$, we have for $1\leq l\leq (d-1)$ that
\begin{align*}
(-1)^lT_3=&\sum_{\textbf{u}\in C_l}\Re\bigg(\frac{\pi^l\prod_{i=1}^ly_{j_i}}{2^{l-1}\prod_{i=1}^l\sin(\pi y_{j_i})}\mathcal{F}\big[g\mathbbm{1}_{B_{\gamma_{j,l}}}\big]\big(-2\pi\chi(\textbf{u}\bigcdot \textbf{y}_{j,l})\Delta_n^{-1/2}\big)\bigg)\\
&=\frac{\pi^l\prod_{i=1}^ly_{j_i}}{\prod_{i=1}^l\sin(\pi y_{j_i})}\int_{B_{\gamma_{j,l}}}g(\textbf{z})\sum_{u\in C_l}\frac{1}{2^{l-1}}\cos(2\pi\chi(\textbf{u}\bigcdot \textbf{y}_{j,l})^\top \textbf{z}\Delta_n^{-1/2})\diff \textbf{z}\\
&=\frac{\pi^l\prod_{i=1}^ly_{j_i}}{\prod_{i=1}^l\sin(\pi y_{j_i})}\int_{\sqrt{\Delta}_n/2}^\infty\cdots\int_{\sqrt{\Delta}_n/2}^\infty\int_0^{\sqrt{\Delta}_n/2}\cos(2\pi y_{j_l}z_{j_l}\Delta_n^{-1/2})\cdots\\
&~~~~~ \times\int_0^{\sqrt{\Delta}_n/2} g(\textbf{z})\cos(2\pi y_{j_1}z_{j_1}\Delta_n^{-1/2})\diff z_{j_1}\cdots \diff z_{j_l}\diff z_{i_1}\cdots \diff z_{i_{d-l}}.
\end{align*}
To simplify the notation, we introduce $g(z_1, \ldots, z_d) = \tilde{g}(z_{j_1}, \ldots, z_{j_l}, z_{i_1}, \ldots, z_{i_{d-l}})$. Moreover, we can apply integration by parts to obtain:
\begin{align*}
&\frac{\pi^l\prod_{i=1}^ly_{j_i}}{\prod_{i=1}^l\sin(\pi y_{j_i})}\int_{\sqrt{\Delta}_n/2}^\infty\cdots\int_{\sqrt{\Delta}_n/2}^\infty\int_0^{\sqrt{\Delta}_n/2}\cos(2\pi y_{j_l}z_{j_l}\Delta_n^{-1/2})\cdots\\
&~~~~~ \times\int_0^{\sqrt{\Delta}_n/2} g(\textbf{z})\cos(2\pi y_{j_1}z_{j_1}\Delta_n^{-1/2})\diff z_{j_1}\cdots \diff z_{j_l}\diff z_{i_1}\cdots \diff z_{i_{d-l}}\\
&=\frac{\Delta_n^{1/2}\pi^{l-1}\prod_{i=2}^ly_{j_i}}{2\prod_{i=2}^l\sin(\pi y_{j_i})}\int_{\sqrt{\Delta}_n/2}^\infty\cdots\int_{\sqrt{\Delta}_n/2}^\infty\int_0^{\sqrt{\Delta}_n/2}\cos(2\pi y_{j_l}z_{j_l}\Delta_n^{-1/2})\cdots\\
&~~~~~ \times\int_0^{\sqrt{\Delta}_n/2} \tilde{g}(\sqrt{\Delta}_n/2,z_{j_2},\ldots,z_{j_l},z_{i_1},\ldots,z_{i_{d-l}})\cos(2\pi y_{j_2}z_{j_2}\Delta_n^{-1/2})\diff z_{j_2}\cdots \diff z_{j_l}\diff z_{i_1}\cdots \diff z_{i_{d-l}}\\
&~~~~~-\frac{\pi^l\prod_{i=1}^ly_{j_i}}{\prod_{i=1}^l\sin(\pi y_{j_i})}\int_{\sqrt{\Delta}_n/2}^\infty\cdots\int_{\sqrt{\Delta}_n/2}^\infty\int_0^{\sqrt{\Delta}_n/2}\cos(2\pi y_{j_l}z_{j_l}\Delta_n^{-1/2})\cdots\\
&~~~~~\times \int_0^{\sqrt{\Delta}_n/2} g'_{z_{j_1}}(z)\frac{\sin(2\pi y_{j_1}z_{j_1}\Delta_n^{-1/2})}{2\pi y_{j_1}\Delta_n^{-1/2}}\diff z_{j_1}\cdots \diff z_{j_l}\diff z_{i_1}\cdots \diff z_{i_{d-l}}.
\end{align*}
By induction, we have 
\begin{align}
&\sum_{\textbf{u}\in C_l}\Re\bigg(\frac{\pi^l\prod_{i=1}^ly_{j_i}}{2^{l-1}\prod_{i=1}^l\sin(\pi y_{j_i})}\mathcal{F}\big[g\mathbbm{1}_{B_{\gamma_{j,l}}}\big]\big(-2\pi\chi(\textbf{u}\bigcdot \textbf{y}_{j,l})\Delta_n^{-1/2}\big)\bigg)\notag\\
&=\bigg(\frac{\Delta_n^{1/2}}{2}\bigg)^l\int_{\sqrt{\Delta}_n/2}^\infty\cdots\int_{\sqrt{\Delta}_n/2}^\infty\tilde{g}(\sqrt{\Delta}_n/2,\ldots,\sqrt{\Delta}_n/2,z_{i_1},\ldots,z_{i_{d-l}})\diff z_{i_1}\cdots \diff z_{i_{d-l}}-\sum_{k=1}^l I_{k},\label{eqn_IntegrationByPartsInductionT3ii}
\end{align}
where we infer by a simple transformation, that
\begin{align}
I_k&:=\frac{\Delta_n^{(k-1)/2}\pi^{l-k+1}\prod_{i=k}^ly_{j_i}}{2^{k-1}\prod_{i=k}^l\sin(\pi y_{j_i})}\int_{\sqrt{\Delta}_n/2}^\infty\cdots\int_{\sqrt{\Delta}_n/2}^\infty\int_0^{\sqrt{\Delta}_n/2}\cos(2\pi y_{j_l}z_{j_l}\Delta_n^{-1/2})\cdots\notag\\
&~~~~~ \times\int_0^{\sqrt{\Delta}_n/2} \bigg(\tilde{g}'_{z_{j_k}}(\sqrt{\Delta}_n/2,\ldots,\sqrt{\Delta}_n/2,z_{j_k},\ldots,z_{j_l},z_{i_1},\ldots,z_{i_{d-l}})\notag\\
&~~~~~~~~~~\times\frac{\sin(2\pi y_{j_k}z_{j_k}\Delta_n^{-1/2})}{2\pi  y_{j_k}\Delta_n^{-1/2}}\bigg)\diff z_{j_k}\cdots \diff z_{j_l}\diff z_{i_1}\cdots \diff z_{i_{d-l}}\notag\\
&=\frac{\Delta_n^{(l+1)/2}\pi^{l-k}\prod_{i=k+1}^ly_{j_i}}{2^{k}\prod_{i=k}^l\sin(\pi y_{j_i})}\int_{\sqrt{\Delta}_n/2}^\infty\cdots\int_{\sqrt{\Delta}_n/2}^\infty\int_0^{1/2}\cos(2\pi y_{j_l}z_{j_l})\cdots\notag\\
&~~~~~ \times\int_0^{1/2} \bigg(\tilde{g}'_{z_{j_k}}(\sqrt{\Delta}_n/2,\ldots,\sqrt{\Delta}_n/2,z_{j_k}\Delta_n^{1/2},\ldots,z_{j_l}\Delta_n^{1/2},z_{i_1},\ldots,z_{i_{d-l}})\notag\\
&~~~~~~~~~~\times\sin(2\pi y_{j_k}z_{j_k})\bigg)\diff z_{j_k}\cdots \diff z_{j_l}\diff z_{i_1}\cdots \diff z_{i_{d-l}}\label{eqn_Def_I_k}.
\end{align}
In order to determine the order of the terms $I_k$ we proceed by re-transforming the integral as follows:
\begin{align*}
I_k&=\Oo\bigg(\frac{\Delta_n^{(l+1)/2}}{\delta^{l-k+1}}\int_{\sqrt{\Delta}_n/2}^\infty\cdots\int_{\sqrt{\Delta}_n/2}^\infty\int_0^{1/2}\cdots\\
&~~~~~\times\int_0^{1/2}\tilde{g}'_{z_{j_k}}(\sqrt{\Delta}_n/2,\ldots,\sqrt{\Delta}_n/2,z_{j_k}\Delta_n^{1/2},\ldots,z_{j_l}\Delta_n^{1/2},z_{i_1},\ldots,z_{i_{d-l}})\diff z_{j_k}\cdots \diff z_{j_l}\diff z_{i_1}\cdots \diff z_{i_{d-l}}\bigg)\\
&=\Oo\bigg(\frac{\Delta_n^{k/2}}{\delta^{l-k+1}}\int_{\sqrt{\Delta}_n/2}^\infty\cdots\int_{\sqrt{\Delta}_n/2}^\infty\int_0^{\sqrt{\Delta}_n/2}\cdots\int_0^{\sqrt{\Delta}_n/2}z_{j_k}f'\bigg(\sum_{i=k}^l z_{j_i}^2+\sum_{j=1}^{d-l}z_{i_j}^2\bigg)\diff z_{j_k}\cdots \diff z_{j_l}\diff z_{i_1}\cdots \diff z_{i_{d-l}}\bigg).
\end{align*}
Analogously to the determination of the error term $\mathcal{I}$, we transform into $(d-k+1)$-dimensional spherical coordinates and obtain with $1\leq k\leq l\leq (d-1)$ that
\begin{align}
I_k=\Oo\bigg(\frac{\Delta_n^{(l+1)/2}}{\delta^{l-k+1}}\int_{\sqrt{\Delta}_n}^1 r^{d-l}f'(r^2)\diff r\bigg),\label{eqn_I_k_order}
\end{align}
which implies:
\begin{align*}
\sum_{k=1}^l I_{k}=\Oo\bigg(\frac{\Delta_n^{(l+1)/2}}{\delta^{l}}\int_{\sqrt{\Delta}_n}^1 r^{d-l}f'(r^2)\diff r\bigg).
\end{align*}
Next, we have 
\begin{align*}
&\bigg(\frac{\Delta_n^{1/2}}{2}\bigg)^l\int_{\sqrt{\Delta}_n/2}^\infty\cdots\int_{\sqrt{\Delta}_n/2}^\infty\tilde{g}(\sqrt{\Delta}_n/2,\ldots,\sqrt{\Delta}_n/2,z_{i_1},\ldots,z_{i_{d-l}})\diff z_{i_1}\cdots \diff z_{i_{d-l}}\\
&=\int_{\sqrt{\Delta}_n/2}^\infty\cdots\int_{\sqrt{\Delta}_n/2}^\infty\int_0^{\sqrt{\Delta}_n/2}\cdots\int_0^{\sqrt{\Delta}_n/2}\tilde{g}(\sqrt{\Delta}_n/2,\ldots,\sqrt{\Delta}_n/2,z_{i_1},\ldots,z_{i_{d-l}})\diff z_{j_1}\ldots\diff z_{j_l}\diff z_{i_1}\cdots \diff z_{i_{d-l}}\\
&=:J_1.
\end{align*}
Utilizing Taylor expansion, we can decompose $g$ as follows:
\begin{align*}
g(z_1,\ldots,z_d)&=\tilde{g}(z_{j_1},\ldots,z_{j_l},z_{i_1},\ldots,z_{i_{d-l}})\\
&=\tilde{g}(\sqrt{\Delta}_n/2,\ldots,\sqrt{\Delta}_n/2,z_{i_1},\ldots,z_{i_{d-l}})+\sum_{k=1}^l\tilde{g}'_{z_{j_k}}(\xi_1,\ldots,\xi_l,z_{i_1},\ldots,z_{i_{d-l}})\big(z_{j_k}-\sqrt{\Delta}_n/2\big),
\end{align*}
where
\begin{align*}
\nabla_l:=\begin{pmatrix}
\frac{\partial}{\partial z_{j_1}}\\\vdots \\ \frac{\partial}{\partial z_{j_l}}\\ \text{id}\\ \vdots \\ \text{id}
\end{pmatrix},~~~~~ a:=\begin{pmatrix}
\sqrt{\Delta}_n/2\\ \vdots \\ \sqrt{\Delta}_n/2 \\ z_{i_1} \\ \vdots \\ z_{i_{d-l}}
\end{pmatrix},~~~~~\tilde{z}:=\begin{pmatrix}
z_{j_1} \\  \vdots \\ z_{j_l} \\ z_{i_1} \\ \vdots \\ z_{i_{d-l}}
\end{pmatrix},
\end{align*}
and $\xi_1,\ldots,\xi_l\in[0,\sqrt{\Delta}_n/2]$.
Thus, it holds that
\begin{align*}
\abs{J_1-\int_{B_{\gamma_{j,l}}}g(\textbf{z})\diff \textbf{z}}
&\leq \int_{B_{\gamma_{j,l}}}\abs{\tilde{g}(\sqrt{\Delta}_n/2,\ldots,\sqrt{\Delta}_n/2,z_{i_1},\ldots,z_{i_{d-l}})-g(\textbf{z})}\diff \textbf{z}\\
&=\Oo\bigg(\sqrt{\Delta}_n\sum_{k=1}^l\int_{B_{\gamma_{j,l}}}z_{j_k}\abs{f'(\nor{\textbf{z}}_2^2) }\diff \textbf{z}\bigg)\\
&=\Oo\bigg(\Delta_n^{(l+1)/2}\int_{\sqrt{\Delta}_n}^1 r^{d-l}\abs{f'(r^2)} \diff r\bigg).
\end{align*}
Hence, we have:
\begin{align*}
J_1 = \int_{B_{\gamma_{j,l}}}g(\textbf{z})\diff \textbf{z} +\Oo\bigg(\Delta_n^{(l+1)/2}\int_{\sqrt{\Delta}_n}^1 r^{d-l}\abs{f'(r^2)} \diff r\bigg),
\end{align*}
and therefore, we derive the following:
\begin{align*}
T_3&=(-1)^l\int_{B_{\gamma_{j,l}}}g(\textbf{z})\diff \textbf{z} +\Oo\bigg( \frac{\Delta_n^{(l+1)/2}}{\delta^{l}}\int_{\sqrt{\Delta}_n}^1 r^{d-l}f'(r^2)\diff r\bigg).
\end{align*}
To analyse the order of the term $T_2$, we begin by distinguishing between two cases: when $l$ is an odd natural number and when $l$ is an even natural number. Considering that the term $T_{2,\textbf{u}}$ corresponds to the Fourier transform of the function
\begin{align*}
\sum_{\textbf{k}\in\N^d}g(\sqrt{\Delta}_n\textbf{k})\mathbbm{1}_{(a_{\textbf{k-1}},a_{\textbf{k}}]}-(-1)^lg\mathbbm{1}_{B_{\gamma_{j,l}}},
\end{align*}
we can analyse the order of this term by adding the following terms:
\begin{align*}
&\sum_{\textbf{k}\in\N^d}g(\sqrt{\Delta}_n\textbf{k})\mathbbm{1}_{(a_{\textbf{k-1}},a_{\textbf{k}}]}-(-1)^lg\mathbbm{1}_{B_{\gamma_{j,l}}}\\&=\sum_{\textbf{k}\in\N^d}g(\sqrt{\Delta}_n\textbf{k})\mathbbm{1}_{(a_{\textbf{k-1}},a_{\textbf{k}}]}-(-1)^lg\cdot \big(\mathbbm{1}_{B_{\gamma_{j,l}}}+\mathbbm{1}_{(\sqrt{\Delta}_n/2,\infty)^d}-\mathbbm{1}_{(\sqrt{\Delta}_n/2,\infty)^d}\big).
\end{align*} 
If $l$ is odd, we have 
\begin{align*}
\sum_{\textbf{k}\in\N^d}g(\sqrt{\Delta}_n\textbf{k})\mathbbm{1}_{(a_{\textbf{k-1}},a_{\textbf{k}}]}-(-1)^lg\mathbbm{1}_{B_{\gamma_{j,l}}}=\sum_{\textbf{k}\in\N^d}g(\sqrt{\Delta}_n\textbf{k})\mathbbm{1}_{(a_{\textbf{k-1}},a_{\textbf{k}}]}-g\mathbbm{1}_{(\sqrt{\Delta}_n/2,\infty)^d} +  g\mathbbm{1}_{(\sqrt{\Delta}_n/2,\infty)^d\cup B_{\gamma_{j,l}}},
\end{align*}
since we have disjoint sets.
For the case where $l$ is even, we find that
\begin{align*}
\sum_{\textbf{k}\in\N^d}g(\sqrt{\Delta}_n\textbf{k})\mathbbm{1}_{(a_{\textbf{k-1}},a_{\textbf{k}}]}-(-1)^lg\mathbbm{1}_{B_{\gamma_{j,l}}}&=\sum_{\textbf{k}\in\N^d}g(\sqrt{\Delta}_n\textbf{k})\mathbbm{1}_{(a_{\textbf{k-1}},a_{\textbf{k}}]}-g\mathbbm{1}_{(\sqrt{\Delta}_n/2,\infty)^d} + g\cdot \big(\mathbbm{1}_{(\sqrt{\Delta}_n/2,\infty)^d }-\mathbbm{1}_{B_{\gamma_{j,l}}}\big)\\
&\leq \sum_{\textbf{k}\in\N^d}g(\sqrt{\Delta}_n\textbf{k})\mathbbm{1}_{(a_{\textbf{k-1}},a_{\textbf{k}}]}-g\mathbbm{1}_{(\sqrt{\Delta}_n/2,\infty)^d} +  g\mathbbm{1}_{(\sqrt{\Delta}_n/2,\infty)^d\cup B_{\gamma_{j,l}}}.
\end{align*}
Therefore, we can decompose $T_2$ for general $l=1,\ldots,d-1$ into the following parts:
\begin{align*}
T_{2,\textbf{u}}&\leq \Re\bigg(\frac{\pi^l\prod_{i=1}^ly_{j_i}}{2^{l-1}\prod_{i=1}^l\sin(\pi  y_{j_i})}\mathcal{F}\bigg[\sum_{\textbf{k}\in\N^d}g(\sqrt{\Delta}_n\textbf{k})\mathbbm{1}_{(a_{\textbf{k-1}},a_{\textbf{k}}]}-g\mathbbm{1}_{(\sqrt{\Delta}_n/2,\infty)^d}\bigg]\big(-2\pi\chi(\textbf{u}\bigcdot \textbf{y}_{j,l})\Delta_n^{-1/2}\big)\bigg)\\
&~~~~~+\Re\bigg(\frac{\pi^l\prod_{i=1}^ly_{j_i}}{2^{l-1}\prod_{i=1}^l\sin(\pi  y_{j_i})}\mathcal{F}\bigg[g\mathbbm{1}_{(\sqrt{\Delta}_n/2,\infty)^d\cup B_{\gamma_{j,l}}}\bigg]\big(-2\pi\chi(\textbf{u}\bigcdot \textbf{y}_{j,l})\Delta_n^{-1/2}\big)\bigg)=:S_{1,\textbf{u}}+S_{2,\textbf{u}}.
\end{align*}
Furthermore, we define $S_i:=\sum_{\textbf{u}\in C_l}S_{i,\textbf{u}}$, for $i=1,2$.
Starting with $S_2$, it holds for $q\in\{1,2\}$ that
\begin{align}
\abs{x_j^q\mathcal{F}[g](\textbf{x})}=\abs{\mathcal{F}\bigg[\frac{\partial^q}{\partial x_j^q}g\bigg](\textbf{x})}\leq \nor{\frac{\partial^q}{\partial x_j^q} g(\textbf{x})}_{\mathcal{L}_1}~~~~~\text{and}~~~~\abs{\mathcal{F}[g](\textbf{x})}\leq \abs{x_j^{-q}}\nor{\frac{\partial^q}{\partial x_j^{q}} g(\textbf{x})}_{\mathcal{L}_1},\label{eqn_fourrierBoundsDecay}
\end{align}
where we use $x_j\neq 0$ in the last inequality, for $j=1,\ldots,d$ and $\textbf{x}\in\R^d$. Hence, we have:
\begin{align*}
S_{2,\textbf{u}}&=\Oo\bigg( \frac{\Delta_n\pi^{l-2}\prod_{i=2}^ly_{j_i}}{2^{l+1}y_{j_1}\prod_{i=1}^l\sin(\pi y_{j_i})}\nor{\frac{\partial^2}{\partial x_{j_1}^2} g\mathbbm{1}_{[\sqrt{\Delta}_n/2,\infty)^d\cup B_{\gamma_{j,l}}}}_{\mathcal{L}_1}\bigg).
\end{align*}
To compute the $\mathcal{L}_1$ norm, we first obtain the following:
\begin{align*}
\bigg|\bigg|\frac{\partial^2}{\partial z_{j_1}^2} g\mathbbm{1}_{[\sqrt{\Delta}_n/2,\infty)^d\cup B_{\gamma_{j,l}}}\bigg|\bigg|_{\mathcal{L}_1}&=\int_{[\sqrt{\Delta}_n/2,\infty)^d\cup B_{\gamma_{j,l}}}\frac{\partial^2}{\partial z_{j_1}^2}g(\textbf{z})\diff \textbf{z}\\
&=\int_{([\sqrt{\Delta}_n/2,\infty)\dot{\cup}(0,\sqrt{\Delta}_n/2) )^l\times [\sqrt{\Delta}_n/2,\infty)^{d-l}}\frac{\partial^2}{\partial z_{j_1}^2}g(\textbf{z})\diff  \tilde{\textbf{z}},
\end{align*}
where $ \tilde{\textbf{z}}=(z_{j_1},\ldots,z_{j_l},z_{i_1},\ldots,z_{i_{d-l}})$.
At this point, it is possible that none of the integration variables $z_{j_1},\ldots,z_{j_l}$ fall within the range $(0,\sqrt{\Delta}_n/2)$, or one to all of them. Assume we have $0\leq k\leq l$ of these integration variable within the range $(0,\sqrt{\Delta}_n/2)$, then there are $\binom{l}{k}$ possible combinations to choose $k$ variables from $z_{j_1},\ldots,z_{j_l}$. As each choice results in the same order of the integral, which is evident by the argumentation followed by display \eqref{eqn_spericalTransswitch}, it is sufficient to analyse the order of the integral, where we set the first $k$ integration variables $z_{j_1},\ldots,z_{j_k}\in (0,\sqrt{\Delta}_n/2)$. Hence, we get:  
\begin{align*}
\bigg|\bigg|\frac{\partial^2}{\partial z_{j_1}^2} g\mathbbm{1}_{[\sqrt{\Delta}_n/2,\infty)^d\cup B_{\gamma_{j,l}}}\bigg|\bigg|_{\mathcal{L}_1}&=\int_{([\sqrt{\Delta}_n/2,\infty)\dot{\cup}(0,\sqrt{\Delta}_n/2) )^l\times [\sqrt{\Delta}_n/2,\infty)^{d-l}}\frac{\partial^2}{\partial z_{j_1}^2}g(\textbf{z})\diff \tilde{\textbf{z}}\\
&=\Oo\bigg(\max_{k=1,\ldots,l}\Delta_n^{k/2}\int_{\sqrt{\Delta}_n}^1r^{d-k+1}f''(r^2)\diff r\vee \int_{\sqrt{\Delta}_n}^\infty r^{d+1}f''(r^2)\diff r\\
&~~~~~\vee \max_{k=1,\ldots,l}\Delta_n^{k/2}\int_{\sqrt{\Delta}_n}^1r^{d-k-1}f'(r^2)\diff r\vee \int_{\sqrt{\Delta}_n}^\infty r^{d-1}f'(r^2)\diff r\bigg).
\end{align*}
Thus, we infer the following:
\begin{align}
S_{2}&=\sum_{\textbf{u}\in C_l}S_{2,\textbf{u}}=\Oo\bigg(\max_{k=0,\ldots,l}\frac{\Delta_n^{k/2+1}}{\delta^{l+1}}\int_{\sqrt{\Delta}_n}^1r^{d-k+1}f''(r^2)\diff r\vee \max_{k=0,\ldots,l}\frac{\Delta_n^{k/2+1}}{\delta^{l+1}}\int_{\sqrt{\Delta}_n}^1 r^{d-k-1}f'(r^2)\diff r\bigg),\label{eqn_S2_orderSave}
\end{align}
where we have used that $\textbf{y}\in[\delta,1-\delta]^d$. 
We commence the analysis of the term $S_1$. Here, we find that
\begin{align*}
\abs{S_{1,\textbf{u}}}=\frac{\pi^l\prod_{i=1}^ly_{j_i}}{2^{l-1}\prod_{i=1}^l\sin(\pi  y_{j_i})}\abs{\Re\bigg(\sum_{\textbf{k}\in\N^d}\int_{a_{\textbf{k-1}}}^{a_\textbf{k}}\big(g(\sqrt{\Delta}_n\textbf{k})-g(\textbf{z})\big)\exp\big[2\pi\im \chi(\textbf{u}\bigcdot \textbf{y}_{j,l})^\top \textbf{z}\Delta_n^{-1/2}\big]\diff \textbf{z}\bigg)}.
\end{align*}
By considering display \eqref{eqn_multidim_taylor_2ndOrder}, we can deduce:
\begin{align*}
\abs{S_{1,\textbf{u}}}&\leq \frac{\pi^l\prod_{i=1}^ly_{j_i}}{2^{l-1}\prod_{i=1}^l\sin(\pi  y_{j_i})}\abs{\Re\bigg(-\sum_{\textbf{k}\in\N^d}\int_{a_{\textbf{k-1}}}^{a_\textbf{k}}\nabla g(\sqrt{\Delta}_n\textbf{k})^\top(\textbf{z}-\sqrt{\Delta}_n\textbf{k})\exp\big[2\pi\im\chi(\textbf{u}\bigcdot \textbf{y}_{j,l})^\top \textbf{z}\Delta_n^{-1/2}\big]  \diff \textbf{z}\bigg)}\\
&~~~~~+\frac{\pi^l\prod_{i=1}^ly_{j_i}}{2^{l-1}\prod_{i=1}^l\sin(\pi  y_{j_i})}\Re\bigg(\sum_{\textbf{k}\in\N^d}\int_{a_{\textbf{k-1}}}^{a_\textbf{k}}\bigg\lvert \frac{1}{2}(\textbf{z}-\sqrt{\Delta}_n\textbf{k})^\top H_g(\xi)\\
&~~~~~~~~~~\times (\textbf{z}-\sqrt{\Delta}_n\textbf{k}) \exp\big[2\pi\im\chi(\textbf{u}\bigcdot \textbf{y}_{j,l})^\top \textbf{z}\Delta_n^{-1/2}\big]   \bigg\rvert \diff \textbf{z} \bigg)\\
&\leq \frac{\pi^l\prod_{i=1}^ly_{j_i}}{2^{l-1}\prod_{i=1}^l\sin(\pi  y_{j_i})}\abs{\Re\bigg(-\sum_{\textbf{k}\in\N^d} \int_{a_{\textbf{k-1}}}^{a_\textbf{k}}\nabla g(\sqrt{\Delta}_n\textbf{k})^\top(\textbf{z}-\sqrt{\Delta}_n\textbf{k})\exp\big[2\pi\im\chi(\textbf{u}\bigcdot \textbf{y}_{j,l})^\top \textbf{z}\Delta_n^{-1/2}\big]  \diff \textbf{z}\bigg)}\\
&~~~~~+\frac{\pi^l\prod_{i=1}^ly_{j_i}}{2^{l-1}\prod_{i=1}^l\sin(\pi  y_{j_i})}\sum_{\textbf{k}\in\N^d}\int_{a_{\textbf{k-1}}}^{a_\textbf{k}}\abs{ \frac{1}{2}(\textbf{z}-\sqrt{\Delta}_n\textbf{k})^\top H_g(\xi)(\textbf{z}-\sqrt{\Delta}_n\textbf{k})  }\diff \textbf{z}.
\end{align*}
We employ a similar approach as for the term $T_1$, given in the equations \eqref{eqn_T1_2ndDerivationBound} and \eqref{eqn_taylor2ndErrorOrder}, for the second integral, leading to the term:
\begin{align*}
\abs{S_{1,\textbf{u}}}&\leq\frac{\pi^l\prod_{i=1}^ly_{j_i}}{2^{l-1}\prod_{i=1}^l\sin(\pi  y_{j_i})}\abs{\sum_{\textbf{k}\in\N^d} \int_{a_{\textbf{k-1}}}^{a_\textbf{k}}\nabla g(\sqrt{\Delta}_n\textbf{k})^\top(\textbf{z}-\sqrt{\Delta}_n\textbf{k})\cos\big[2\pi\chi(\textbf{u}\bigcdot \textbf{y}_{j,l})^\top \textbf{z}\Delta_n^{-1/2}\big]  \diff \textbf{z}}\\
&~~~~~+\Oo\bigg(\frac{\Delta_n}{\delta^l}\int_{\sqrt{\Delta}_n}^1 r^{d+1}\abs{f''(r^2)}\diff r\vee \frac{\Delta_n}{\delta^l}\int_{\sqrt{\Delta}_n}^1 r^{d-1}\abs{f'(r^2)}diff r\bigg).
\end{align*}
Employing equation \eqref{eqn_cosProdandSum}, we obtain:
\begin{align*}
\abs{S_1}&\leq\frac{\pi^l\prod_{i=1}^ly_{j_i}}{\prod_{i=1}^l\sin(\pi  y_{j_i})}\bigg\lvert\sum_{\textbf{k}\in\N^d} \int_{a_{\textbf{k-1}}}^{a_\textbf{k}}\nabla g(\sqrt{\Delta}_n\textbf{k})^\top(\textbf{z}-\sqrt{\Delta}_n\textbf{k})\cos(2\pi y_{j_1}z_{j_1}\Delta_n^{-1/2})\cdots\cos(2\pi y_{j_l}z_{j_l}\Delta_n^{-1/2})  \diff \textbf{z}\bigg\rvert\\
&~~~~~+\Oo\bigg(\frac{\Delta_n}{\delta^l}\int_{\sqrt{\Delta}_n}^1 r^{d+1}\abs{f''(r^2)}\diff r\vee \frac{\Delta_n}{\delta^l}\int_{\sqrt{\Delta}_n}^1 r^{d-1}\abs{f'(r^2)}\diff r\bigg).
\end{align*}
Let $\textbf{k}\in\N^d$, then it holds that
\begin{align*}
&\int_{a_{\textbf{k-1}}}^{a_\textbf{k}}\nabla g(\sqrt{\Delta}_n\textbf{k})^\top(\textbf{z}-\sqrt{\Delta}_n\textbf{k})\cos(2\pi y_{j_1}z_{j_1}\Delta_n^{-1/2})\cdots\cos(2\pi y_{j_l}z_{j_l}\Delta_n^{-1/2})  \diff \textbf{z}\\
&=\sum_{l=1}^d \int_{a_{\textbf{k-1}}}^{a_\textbf{k}} g'_{z_l}(\sqrt{\Delta}_n\textbf{k})(z_l-\sqrt{\Delta}_n k_l)\cos(2\pi y_{j_1}z_{j_1}\Delta_n^{-1/2})\cdots\cos(2\pi y_{j_l}z_{j_l}\Delta_n^{-1/2})  \diff \textbf{z}.
\end{align*}
Firstly, for $\tilde{l} \notin \{j_1, \ldots, j_l\}$, we have 
\begin{align*}
\int_{a_{k_{j_1}-1}} ^{a_{k_{j_1}}}\cos(2\pi y_{j_1}z_{j_1}\Delta_n^{-1/2})\diff z_{j_1}\cdots\int_{a_{k_{j_l}-1}} ^{a_{k_{j_l}}}\cos(2\pi y_{j_l}z_{j_l}\Delta_n^{-1/2})\diff z_{j_l}\int_{a_{k_{j_{\tilde{l}}}-1}} ^{a_{k_{j_{\tilde{l}}}}}(z_{j_{\tilde{l}}}-\sqrt{\Delta}_nk_{j_{\tilde{l}}})\diff z_{j_{\tilde{l}}}=0,
\end{align*}
since it holds that
\begin{align*}
\int_{\sqrt{\Delta}_n(\tilde{k}-1/2)}^{\sqrt{\Delta}_n(\tilde{k}+1/2)}(x-\sqrt{\Delta}_n\tilde{k})\diff x&=\int_{-\sqrt{\Delta}_n/2}^{\sqrt{\Delta}_n/2}x=0,
\end{align*}
for a $\tilde{k}\in\N$.
Suppose $\tilde{l} \in \{j_1, \ldots, j_l\}$, then we obtain:
\begin{align*}
&\int_{a_{k_{j_1}-1}} ^{a_{k_{j_1}}}\cos(2\pi y_{j_1}z_{j_1}\Delta_n^{-1/2})\diff z_{j_1}\cdots\int_{a_{k_{j_{\tilde{l}}}-1}} ^{a_{k_{j_{\tilde{l}}}}}(z_{j_{\tilde{l}}}-\sqrt{\Delta}_nk_{j_{\tilde{l}}})\cos(2\pi y_{j_{\tilde{l}}}z_{j_{\tilde{l}}}\Delta_n^{-1/2})\diff z_{j_{\tilde{l}}}\\
&~~~~~\cdots\int_{a_{k_{j_l}-1}} ^{a_{k_{j_l}}}\cos(2\pi y_{j_l}z_{j_l}\Delta_n^{-1/2})\diff z_{j_l}.
\end{align*}
For $\tilde{k}\in\N$ and $\tilde{y}\in[\delta,1-\delta]$ we have 
\begin{align*}
\int_{\sqrt{\Delta}_n(\tilde{k}-1/2)}^{\sqrt{\Delta}_n(\tilde{k}+1/2)}\cos(2\pi \tilde{y}x\Delta_n^{-1/2})\diff x=\frac{\sqrt{\Delta}_n\cos(2\pi \tilde{y}\tilde{k})\sin(\pi\tilde{y})}{\pi\tilde{y}}=\Oo(\Delta_n^{1/2}/\delta),
\end{align*}
and 
\begin{align}
\int_{\sqrt{\Delta}_n(\tilde{k}-1/2)}^{\sqrt{\Delta}_n(\tilde{k}+1/2)}(x-\sqrt{\Delta}_n\tilde{k})\cos(2\pi \tilde{y}x\Delta_n^{-1/2})\diff x&=\int_{-\sqrt{\Delta}_n/2}^{\sqrt{\Delta}_n/2}x\cos(2\pi \tilde{y}(x+\sqrt{\Delta}_n\tilde{k})\Delta_n^{-1/2})\diff x\notag\\
&=\frac{\Delta_n\big(\pi \tilde{y}\cos(\pi \tilde{y})-\sin(\pi \tilde{y})\big)}{2\pi^2 \tilde{y}^2}\sin(2\pi \tilde{k}\tilde{y}).
\label{eqn_T2SolvingOneIntegralSine}
\end{align}
Hence, we get for $1\leq l \leq d-1$ that
\begin{align*}
\abs{S_1}&=\Oo\bigg(\frac{\Delta_n^{(d+1)/2}}{\delta}\sum_{i=1}^l\sum_{\textbf{k}\in\N^d}\abs{g'_{z_{j_{i}}}(\sqrt{\Delta}_n\textbf{k})}\cos(2\pi y_{j_1}k_{j_1})\cdots\cos(2\pi y_{j_{i-1}}k_{j_{i-1}})\sin(2\pi k_{j_{i}}y_{j_{i}})\\
&~~~~\times\cos(2\pi y_{j_{i+1}}k_{j_{i+1}})\cdots\cos(2\pi y_{j_l}k_{j_l})\bigg)+\Oo\bigg(\frac{\Delta_n}{\delta^l}\int_{\sqrt{\Delta}_n}^1 r^{d+1}\abs{f''(r^2)}\diff r\vee \frac{\Delta_n}{\delta^l}\int_{\sqrt{\Delta}_n}^1 r^{d-1}\abs{f'(r^2)}\diff r\bigg),
\end{align*}
where we set $y_{j_0}=y_{j_{l+1}}=0$.
It remains to determine the order of the series. 
Therefore, we use the following identity:
\begin{align*}
\sin(x_1)\cos(x_2)\cdots\cos(x_n)=\frac{1}{2^{n-1}}\sum_{\textbf{u}\in C_n}\sin(\textbf{u}^\top\textbf{x}),
\end{align*}
where $\textbf{x}=(x_1,\ldots,x_n)\in\R^n$ and $C_n=\{1\}\times\{-1,1\}^{n-1}$. 
This identity can be proven similarly to identity in display \eqref{eqn_cosProdandSum}. Without loss of generality, we set the coordinates of the sine term to be $j_1$, leading to the expression:
\begin{align*}
&\sum_{\textbf{k}\in\N^d}g'_{z_{j_1}}(\sqrt{\Delta}_n\textbf{k})\sin(2\pi k_{j_1}y_{j_1})\cos(2\pi y_{j_2}k_{j_2})\cdots\cos(2\pi y_{j_l}k_{j_l})\\&~~~~~=\frac{1}{2^{l-1}}\sum_{\textbf{u}\in C_l}\sum_{\textbf{k}\in\N^d}g'_{z_{j_1}}(\sqrt{\Delta}_n\textbf{k})\sin\big(2\pi(\textbf{u}^\top(\textbf{y}\bigcdot\textbf{k})_{j,l}\big),
\end{align*}
where $(\textbf{y}\bigcdot \textbf{k})_{j,l}:=(k_{j_1}y_{j_1},\ldots,k_{j_l}y_{j_l})$.
By following similar steps as in display \eqref{eqn_decompT2T3Fourrier}, we find that
\begin{align*}
&\Delta_n^{(d+1)/2}\sum_{\textbf{k}\in\N^d}g'_{z_{j_1}}(\sqrt{\Delta}_n\textbf{k})\sin(2\pi k_{j_1}y_{j_1})\cos(2\pi y_{j_2}k_{j_2})\cdots\cos(2\pi y_{j_l}k_{j_l})\\
&=\sum_{\textbf{u}\in C_l}\Im\bigg(\frac{\Delta_n^{1/2}\pi^l\prod_{i=1}^ly_{j_i}}{2^{l-1}\prod_{i=1}^l\sin(\pi y_{j_i})}\mathcal{F}\bigg[\sum_{\textbf{k}\in\N^d}g'_{z_{j_1}}(\sqrt{\Delta}_n\textbf{k})\mathbbm{1}_{(a_{\textbf{k-1}},a_{\textbf{k}}]}\bigg]\big(-2\pi\chi(\textbf{u}\bigcdot \textbf{y}_{j,l})\Delta_n^{-1/2}\big)\bigg)=:U_1+U_2-U_3,
\end{align*}
where $U_i:=\sum_{\textbf{u}\in C_l}U_{i,\textbf{u}}$ for $i=1,2,3$ and
\begin{align*}
U_{1,\textbf{u}}&:=  \Im\bigg(\frac{\Delta_n^{1/2}\pi^l\prod_{i=1}^ly_{j_i}}{2^{l-1}\prod_{i=1}^l\sin(\pi y_{j_i})}\mathcal{F}\bigg[\sum_{\textbf{k}\in\N^d}g'_{z_{j_1}}(\sqrt{\Delta}_n\textbf{k})\mathbbm{1}_{(a_{\textbf{k-1}},a_{\textbf{k}}]}-g'_{z_{j_1}}\mathbbm{1}_{(\sqrt{\Delta}_n/2,\infty)^d}\bigg]\big(-2\pi\chi(\textbf{u}\bigcdot \textbf{y}_{j,l})\Delta_n^{-1/2}\big)\bigg),\\
U_{2,\textbf{u}}&:=\Im\bigg(\frac{\Delta_n^{1/2}\pi^l\prod_{i=1}^ly_{j_i}}{2^{l-1}\prod_{i=1}^l\sin(\pi y_{j_i})}\mathcal{F}\bigg[g'_{z_{j_1}}\mathbbm{1}_{(\sqrt{\Delta}_n/2,\infty)^d\cup B_{\gamma_{j,l}}}\bigg]\big(-2\pi\chi(\textbf{u}\bigcdot \textbf{y}_{j,l})\Delta_n^{-1/2}\big)\bigg), \\
U_{3,\textbf{u}}&:=\Im\bigg(\frac{\Delta_n^{1/2}\pi^l\prod_{i=1}^ly_{j_i}}{2^{l-1}\prod_{i=1}^l\sin(\pi y_{j_i})}\mathcal{F}\bigg[g'_{z_{j_1}}\mathbbm{1}_{ B_{\gamma_{j,l}}}\bigg]\big(-2\pi\chi(\textbf{u}\bigcdot \textbf{y}_{j,l})\Delta_n^{-1/2}\big)\bigg).
\end{align*}
By employing the inequality $\nor{\mathcal{F}[f]}_\infty\leq \nor{f}_{\mathcal{L}_1}$, we obtain, for a $\textbf{u} \in C_l$, that
\begin{align*}
\abs{U_{1,\textbf{u}}} &\leq \frac{\Delta_n^{1/2}\pi^l\prod_{i=1}^ly_{j_i}}{2^{l-1}\prod_{i=1}^l\sin(\pi y_{j_i})}\bigg\lvert\bigg\rvert \sum_{\textbf{k}\in\N^d}g'_{z_{j_1}}(\sqrt{\Delta}_n\textbf{k})\mathbbm{1}_{(a_{\textbf{k-1}},a_{\textbf{k}}]}-g'_{z_{j_1}}\mathbbm{1}_{(\sqrt{\Delta}_n/2,\infty)^d} \bigg\rvert\bigg\rvert_{\mathcal{L}_1(\R^d)}\\
&\leq \frac{\Delta_n^{1/2}\pi^l\prod_{i=1}^ly_{j_i}}{2^{l-1}\prod_{i=1}^l\sin(\pi y_{j_i})}\sum_{\textbf{k}\in\N^d} \int_{\R^d} \big\lvert g'_{z_{j_1}}(\sqrt{\Delta}_n\textbf{k})-g'_{z_{j_1}}(\textbf{z})\big\rvert\mathbbm{1}_{(a_{\textbf{k-1}},a_{\textbf{k}}]}(\textbf{z})\diff \textbf{z}.
\end{align*}
Applying Taylor's expansion, we find that
\begin{align*}
\abs{U_{1,\textbf{u}}}&\leq  \frac{\Delta_n^{1/2}\pi^l\prod_{i=1}^ly_{j_i}}{2^{l-1}\prod_{i=1}^l\sin(\pi y_{j_i})}\sum_{\textbf{k}\in\N^d} \int_{a_\textbf{k-1}}^{a_\textbf{k}} \big\lvert\nabla g'_{z_{j_1}}(\xi_\textbf{k})^\top(\textbf{z}-\sqrt{\Delta}_n\textbf{k})\big\rvert\diff \textbf{z}.
\end{align*}
Following analogous steps as for the term $T_1$, we have for $\textbf{k}\in[a_{\textbf{k-1}},a_\textbf{k}]$ that
\begin{align*}
\nabla g'_{z_{j_1}}(\textbf{z})^\top(\textbf{z}-\sqrt{\Delta}_n\textbf{k})&=\sum_{l=1}^d \frac{\partial^2}{\partial z_{j_1}\partial z_l}g(\textbf{z})(z_l-\sqrt{\Delta}_nk_l)\\
&\leq C\sqrt{\Delta}_n\Big(f''(\pi^2\eta\nor{\textbf{z}}_2^2)(dz_{j_1}^2+\nor{\textbf{z}}_2^2)+f'(\pi^2\eta\nor{\textbf{z}}_2^2)\Big),
\end{align*}
and therefore it holds that
\begin{align*}
\abs{U_{1,\textbf{u}}}&=\Oo\bigg(\frac{\Delta_n}{\delta^l}\int_{[\sqrt{\Delta}_n/2,\infty)^d}\nor{\textbf{z}}_2^2f''(\pi^2\eta\nor{\textbf{z}}_2^2)\diff\textbf{z}+\frac{\Delta_n}{\delta^l}\int_{[\sqrt{\Delta}_n/2,\infty)^d}f'(\pi^2\eta\nor{\textbf{z}}_2^2)\diff\textbf{z}\bigg)\\
&=\Oo\bigg(\frac{\Delta_n}{\delta^l}\int_{\sqrt{\Delta}_n}^1 r^{d+1}f''(r^2)\diff r\vee \frac{\Delta_n}{\delta^l}\int_{\sqrt{\Delta}_n}^1 r^{d-1}f'(r^2)\diff r\bigg).
\end{align*}
Note, that $U_{1}$ is of the same order as $U_{1,\textbf{u}}$. Using display \eqref{eqn_fourrierBoundsDecay} with $q=1$ we have for $U_{2,\textbf{u}}$ that
\begin{align*}
\abs{U_{2,\textbf{u}}}&=\Oo\bigg(\frac{\Delta_n\pi^{l-1}\prod_{i=2}^ly_{j_i}}{2^{l}\prod_{i=1}^l\sin(\pi y_{j_i})}\bigg\lvert\bigg\lvert\frac{\partial^2}{\partial z_{j_1}^2}g\mathbbm{1}_{(\sqrt{\Delta}_n/2,\infty)^d\cup B_{\gamma_{j,l}}}  \bigg\rvert\bigg\rvert_{\mathcal{L}_1}\bigg).
\end{align*}
Utilizing the order of the term $S_2$ yields the following:
\begin{align*}
\abs{U_2}&=\Oo\bigg(\max_{k=0,\ldots,l}\frac{\Delta_n^{k/2+1}}{\delta^{l}}\int_{\sqrt{\Delta}_n}^1r^{d-k+1}\abs{f''(r^2)}\diff r\vee \max_{k=0,\ldots,l}\frac{\Delta_n^{k/2+1}}{\delta^{l}}\int_{\sqrt{\Delta}_n}^1 r^{d-k-1}\abs{f'(r^2)}\diff r\bigg).
\end{align*}
For the last term $U_3$ we have with the equations \eqref{eqn_fourrierBoundsDecay} and \eqref{eqn_BGammaOrderDerivatives} that
\begin{align*}
U_{3,\textbf{u}}&=\Oo\bigg(\frac{\Delta_n\pi^{l-1}\prod_{i=2}^ly_{j_i}}{2^{l}\prod_{i=1}^l\sin(\pi y_{j_i})}\bigg\lvert\bigg\lvert\frac{\partial^2}{\partial y_{j_1}^2}g\mathbbm{1}_{ B_{\gamma_{j,l}}}  \bigg\rvert\bigg\rvert_{\mathcal{L}_1}\bigg)\\
&=\Oo\bigg(\frac{\Delta_n^{l/2+1}}{\delta^l}\int_{\sqrt{\Delta}_n}^1 r^{d-l+1}f''(r)\diff r\vee \frac{ \Delta_n^{l/2+1}}{\delta^l}\int_{\sqrt{\Delta}_n}^1 r^{d-l-1}f'(r)\diff r\bigg).
\end{align*}
Hence, we find:
\begin{align*}
\abs{S_1}&=\Oo\bigg(\max_{k=0,\ldots,l}\frac{\Delta_n^{k/2+1}}{\delta^{l+1}}\int_{\sqrt{\Delta}_n}^1r^{d-k+1}\abs{f''(r^2)}\diff r\vee \max_{k=0,\ldots,l}\frac{\Delta_n^{k/2+1}}{\delta^{l+1}}\int_{\sqrt{\Delta}_n}^1 r^{d-k-1}\abs{f'(r^2)}\diff r\bigg)=\abs{S_2}.
\end{align*}
and 
\begin{align*}
&\Delta_n^{d/2}\sum_{\textbf{k}\in\N^d}f(\lambda_\textbf{k}\Delta_n)\cos(2\pi k_{j_1}y_{j_1})\cdot\ldots\cdot\cos(2\pi k_{j_l}y_{j_l})=T_2+T_3+\Oo(\Delta_n)\\
&=(-1)^l\int_{B_{\gamma_{j,l}}}g(\textbf{z})\diff \textbf{z} +\Oo\bigg(\Delta_n^{(l+1)/2}\int_{\sqrt{\Delta}_n}^1 r^{d+1-l}\abs{f'(r^2)} \diff r \vee \frac{\Delta_n^{(l+1)/2}}{\delta^{l}}\int_{\sqrt{\Delta}_n}^1 r^{d-l}f'(r^2)\diff r\bigg)\\
&~~~~~+\Oo\bigg(\max_{k=0,\ldots,l}\frac{\Delta_n^{k/2+1}}{\delta^{l+1}}\int_{\sqrt{\Delta}_n}^1r^{d-k+1}\abs{f''(r^2)}\diff r\vee \max_{k=0,\ldots,l}\frac{\Delta_n^{k/2+1}}{\delta^{l+1}}\int_{\sqrt{\Delta}_n}^1 r^{d-k-1}\abs{f'(r^2)}\diff r\bigg),
\end{align*}
which completes the proof of (ii).\ \\
For the proof of (iii), we proceed in a manner similar to the proof of (ii). Firstly, for a $\gamma \in \{0,1\}^d$, with $\nor{\gamma}_1 = d-1$, we find that
\begin{align*}
\Delta_n^{d/2}\sum_{\textbf{k}\in\N^d}f(\lambda_\textbf{k}\Delta_n)\cos(2\pi k_{1}y_{1})\cdot\ldots\cdot\cos(2\pi k_{d}y_{d})=:T_2+T_3-T_4+\Oo(\Delta_n),
\end{align*}
where we redefine $T_i:=\sum_{\textbf{u}\in C_d}T_{i,\textbf{u}}$, with $i=2,3,4$, by the following:
\begin{align*}
T_{2,\textbf{u}}
&:=\Re\bigg(\frac{\pi^d\prod_{i=1}^dy_{i}}{2^{d-1}\prod_{i=1}^d\sin(\pi  y_{i})}\mathcal{F}\bigg[\sum_{\textbf{k}\in\N^d}g(\sqrt{\Delta}_n\textbf{k})\mathbbm{1}_{(a_{\textbf{k-1}},a_{\textbf{k}}]}-g\mathbbm{1}_{[\sqrt{\Delta}_n/2,\infty)^d}\bigg]\big(-2\pi (\textbf{u}\bigcdot \textbf{y})\Delta_n^{-1/2}\big)\bigg)\\
T_{3,\textbf{u}}&:=\Re\bigg(\frac{\pi^d\prod_{i=1}^dy_{i}}{2^{d-1}\prod_{i=1}^d\sin(\pi y_{i})}\mathcal{F}\bigg[g\mathbbm{1}_{[\sqrt{\Delta}_n/2,\infty)^d\cup B_\gamma}\bigg]\big(-2\pi (\textbf{u}\bigcdot \textbf{y})\Delta_n^{-1/2}\big)\bigg)\\
T_{4,\textbf{u}}&:=\Re\bigg(\frac{\pi^d\prod_{i=1}^dy_{i}}{2^{d-1}\prod_{i=1}^d\sin(\pi y_{i})}\mathcal{F}\bigg[g\mathbbm{1}_{ B_\gamma}\bigg]\big(-2\pi (\textbf{u}\bigcdot \textbf{y})\Delta_n^{-1/2}\big)\bigg),
\end{align*} 
where $\textbf{y}=(y_1,\ldots,y_d)\in[\delta,1-\delta]^d$.
For $T_2$, we apply the same procedure as for $S_1$ in part (ii) to obtain:
\begin{align*}
\abs{T_2}&=\Oo\bigg(\frac{\Delta_n^{(d+1)/2}}{\delta}\sum_{i=1}^d\sum_{\textbf{k}\in\N^d}\abs{g'_{z_{i}}(\sqrt{\Delta}_n\textbf{k})}\cos(2\pi y_{1}k_{1})\cdots\cos(2\pi y_{i-1}k_{i-1})\sin(2\pi k_{i}y_{i})\\
&~~~~\times\cos(2\pi y_{i+1}k_{i+1})\cdots\cos(2\pi y_{d}k_{d})\bigg)+\Oo\bigg(\frac{\Delta_n}{\delta^d}\int_{\sqrt{\Delta}_n}^1 r^{d+1}\abs{f''(r^2)}\diff r\vee \frac{\Delta_n}{\delta^d}\int_{\sqrt{\Delta}_n}^1 r^{d-1}\abs{f'(r^2)}\diff r\bigg)\\
&=\Oo\bigg(\frac{\Delta_n^{(d+1)/2}}{\delta}\sum_{\textbf{k}\in\N^d}\abs{g'_{z_{1}}(\sqrt{\Delta}_n\textbf{k})}\sin(2\pi y_{1}k_{1})\cos(2\pi y_{2}k_{2})\cdots\cos(2\pi y_{d}k_{d})\bigg)\\
&~~~~+\Oo\bigg(\frac{\Delta_n}{\delta^d}\int_{\sqrt{\Delta}_n}^1 r^{d+1}\abs{f''(r^2)}\diff r\vee \frac{\Delta_n}{\delta^d}\int_{\sqrt{\Delta}_n}^1 r^{d-1}\abs{f'(r^2)}\diff r\bigg).
\end{align*}
Furthermore, it holds that
\begin{align*}
&\Delta_n^{(d+1)/2}\sum_{\textbf{k}\in\N^d}g'_{z_{1}}(\sqrt{\Delta}_n\textbf{k})\sin(2\pi k_{1}y_{1})\cos(2\pi y_{2}k_{2})\cdots\cos(2\pi y_{d}k_{d})\\
&=\sum_{\textbf{u}\in C_l}\Im\bigg(\frac{\Delta_n^{1/2}\pi^d\prod_{i=1}^dy_{i}}{2^{d-1}\prod_{i=1}^d\sin(\pi y_{i})}\mathcal{F}\bigg[\sum_{\textbf{k}\in\N^d}g'_{z_{1}}(\sqrt{\Delta}_n\textbf{k})\mathbbm{1}_{(a_{\textbf{k-1}},a_{\textbf{k}}]}\bigg]\big(-2\pi(\textbf{u}\bigcdot \textbf{y})\Delta_n^{-1/2}\big)\bigg)=:U_1+U_2-U_3,
\end{align*}
where we redefine $U_i:=\sum_{\textbf{u}\in C_l}U_{i,\textbf{u}}$, for $i=1,2,3$, by the following terms:
\begin{align*}
U_{1,\textbf{u}}&:=  \Im\bigg(\frac{\Delta_n^{1/2}\pi^d\prod_{i=1}^dy_{i}}{2^{d-1}\prod_{i=1}^d\sin(\pi y_{i})}\mathcal{F}\bigg[\sum_{\textbf{k}\in\N^d}g'_{z_{j_1}}(\sqrt{\Delta}_n\textbf{k})\mathbbm{1}_{(a_{\textbf{k-1}},a_{\textbf{k}}]}-g'_{z_{j_1}}\mathbbm{1}_{(\sqrt{\Delta}_n/2,\infty)^d}\bigg]\big(-2\pi(\textbf{u}\bigcdot \textbf{y})\Delta_n^{-1/2}\big)\bigg),\\
U_{2,\textbf{u}}&:=\Im\bigg(\frac{\Delta_n^{1/2}\pi^d\prod_{i=1}^dy_{i}}{2^{d-1}\prod_{i=1}^d\sin(\pi y_{i})}\mathcal{F}\bigg[g'_{z_{j_1}}\mathbbm{1}_{(\sqrt{\Delta}_n/2,\infty)^d\cup B_{\gamma}}\bigg]\big(-2\pi(\textbf{u}\bigcdot \textbf{y})\Delta_n^{-1/2}\big)\bigg), \\
U_{3,\textbf{u}}&:=-\Im\bigg(\frac{\Delta_n^{1/2}\pi^d\prod_{i=1}^dy_{i}}{2^{d-1}\prod_{i=1}^d\sin(\pi y_{i})}\mathcal{F}\bigg[g'_{z_{j_1}}\mathbbm{1}_{ B_{\gamma}}\bigg]\big(-2\pi(\textbf{u}\bigcdot \textbf{y})\Delta_n^{-1/2}\big)\bigg).
\end{align*}
For the term $U_1$, $U_2$ and $U_3$ we obtain the same order as in part (ii), resulting in:
\begin{align*}
\abs{U_1}&=\Oo\bigg(\frac{\Delta_n}{\delta^d}\int_{\sqrt{\Delta}_n}^1 r^{d+1}\abs{f''(r^2)}\diff r\vee \frac{\Delta_n}{\delta^d}\int_{\sqrt{\Delta}_n}^1 r^{d-1}\abs{f'(r^2)}\diff r\bigg),\\
\abs{U_2}&=\Oo\bigg(\max_{k=0,\ldots,d-1}\frac{\Delta_n^{k/2+1}}{\delta^{d}}\int_{\sqrt{\Delta}_n}^1r^{d-k+1}\abs{f''(r^2)}\diff r\vee \max_{k=0,\ldots,d-1}\frac{\Delta_n^{k/2+1}}{\delta^{d}}\int_{\sqrt{\Delta}_n}^1 r^{d-k-1}\abs{f'(r^2)}\diff r\bigg),\\
\abs{U_3}&=\Oo\bigg(\frac{\Delta_n^{(d+1)/2}}{\delta^d}\int_{\sqrt{\Delta}_n}^1 r^{2}\abs{f''(r)}\diff r\vee \frac{\Delta_n^{(d+1)/2}}{\delta^d}\int_{\sqrt{\Delta}_n}^1 \abs{f'(r)}\diff r\bigg).
\end{align*}
Hence, we have:
\begin{align*}
\abs{T_2}&=\Oo\bigg(\max_{k=0,\ldots,d-1}\frac{\Delta_n^{k/2+1}}{\delta^{d+1}}\int_{\sqrt{\Delta}_n}^1r^{d-k+1}\abs{f''(r^2)}\diff r\vee \max_{k=0,\ldots,d-1}\frac{\Delta_n^{k/2+1}}{\delta^{d+1}}\int_{\sqrt{\Delta}_n}^1 r^{d-k-1}\abs{f'(r^2)}\diff r\bigg).
\end{align*}
For $T_3$ we infer the same order as for $S_2$ in equation \eqref{eqn_S2_orderSave} and have $T_3=\Oo(T_2)$. For $T_4$ we set without loss of generality that $\gamma=\{1,\ldots,1,0\}\in\{0,1\}^d$ and have 
\begin{align*}
T_4&=\frac{\pi^d\prod_{i=1}^dy_{i}}{\prod_{i=1}^d\sin(\pi y_{i})}\int_{\sqrt{\Delta}_n/2}^\infty\cos(2\pi y_{d}z_{d}\Delta_n^{-1/2})\int_0^{\sqrt{\Delta}_n/2}\cos(2\pi y_{{d-1}}z_{{d-1}}\Delta_n^{-1/2})\cdots\\
&~~~~~ \times\int_0^{\sqrt{\Delta}_n/2} g(\textbf{z})\cos(2\pi y_{1}z_{1}\Delta_n^{-1/2})\diff z_{1}\cdots \diff z_{d-1}\diff z_{d}.
\end{align*}
Using analogous steps as in equations \eqref{eqn_IntegrationByPartsInductionT3ii} and \eqref{eqn_I_k_order}, we have 
\begin{align*}
T_4&=\bigg(\frac{\Delta_n^{1/2}}{2}\bigg)^{d-1}\frac{\pi y_d}{\sin(\pi y_d)}\int_{\sqrt{\Delta}_n/2}^\infty\tilde{g}(\sqrt{\Delta}_n/2,\ldots,\sqrt{\Delta}_n/2,z_d)\cos(2\pi y_{d}z_d\Delta_n^{-1/2})\diff z_d\\
&~~~~~+\Oo\bigg(\frac{\Delta_n^{d/2}}{\delta^{d}}\int_{\sqrt{\Delta}_n}^1 rf'(r^2)\diff r\bigg).
\end{align*}
Integration by parts yields:
\begin{align*}
&\frac{\pi y_d}{\sin(\pi y_d)}\int_{\sqrt{\Delta}_n/2}^\infty\tilde{g}(\sqrt{\Delta}_n/2,\ldots,\sqrt{\Delta}_n/2,z_d)\cos(2\pi y_{d}z_d\Delta_n^{-1/2})\diff z_d\\
&=\frac{\Delta_n^{1/2}}{2\sin(\pi y_d)}\bigg[\tilde{g}(\sqrt{\Delta}_n/2,\ldots,\sqrt{\Delta}_n/2,z_d)\sin(2\pi y_dz_d\Delta_n^{-1/2})\bigg]_{\sqrt{\Delta}_n/2}^\infty\\
&~~~~~-\frac{\Delta_n^{1/2}}{2\sin(\pi y_d)}\int_{\sqrt{\Delta}_n/2}^\infty\frac{\partial}{\partial z_d}\tilde{g}(\sqrt{\Delta}_n/2,\ldots,\sqrt{\Delta}_n/2,z_d)\sin(2\pi y_dz_d\Delta_n^{-1/2})\diff z_d\\
&=\Oo\Big(\Delta_n^{1/2}g(\sqrt{\Delta}_n/2,\ldots,\sqrt{\Delta}_n/2)\Big)-I_d,
\end{align*}
where 
\begin{align*}
I_d:=\frac{\Delta_n^{1/2}}{2\sin(\pi y_d)}\int_{\sqrt{\Delta}_n/2}^\infty\frac{\partial}{\partial z_d}\tilde{g}(\sqrt{\Delta}_n/2,\ldots,\sqrt{\Delta}_n/2,z_d)\sin(2\pi y_dz_d\Delta_n^{-1/2})\diff z_d.
\end{align*}
Furthermore, we have:
\begin{align*}
I_d&=\Oo\bigg(\frac{\Delta_n^{1/2}}{\delta}\int_{\sqrt{\Delta}_n/2}^\infty z_df'\big((d-1)\Delta_n/4+z_d^2\big) \diff z_d\bigg)=\Oo\bigg(\frac{\Delta_n^{1/2}}{\delta}\int_{\sqrt{\Delta}_n/2}^1 rf'\big(r^2\big) \diff r\bigg)
\end{align*}
and therefore we find that
\begin{align*}
T_4&=\Oo\big(\Delta_n^{d/2}f(\Delta_n)\big)+\Oo\bigg(\frac{\Delta_n^{d/2}}{\delta^{d}}\int_{\sqrt{\Delta}_n}^1 rf'(r^2)\diff r\bigg).
\end{align*}
Finally, we obtain that
\begin{align*}
&\Delta_n^{d/2}\sum_{\textbf{k}\in\N^d}f(\lambda_\textbf{k}\Delta_n)\cos(2\pi k_{1}y_{1})\cdot\ldots\cdot\cos(2\pi k_{d}y_{d})=\Oo\big(\Delta_n^{d/2}f(\Delta_n)\big)+\Oo\bigg(\frac{\Delta_n^{d/2}}{\delta^{d}}\int_{\sqrt{\Delta}_n}^1 rf'(r^2)\diff r\bigg)\\
&~~~~~+\Oo\bigg(\max_{k=0,\ldots,d-1}\frac{\Delta_n^{k/2+1}}{\delta^{d+1}}\int_{\sqrt{\Delta}_n}^1r^{d-k+1}\abs{f''(r^2)}\diff r\vee \max_{k=0,\ldots,d-1}\frac{\Delta_n^{k/2+1}}{\delta^{d+1}}\int_{\sqrt{\Delta}_n}^1 r^{d-k-1}\abs{f'(r^2)}\diff r\bigg),
\end{align*}
which completes the proof.
\end{proof}

Consider the class following of functions, given by
\begin{align}
\mathcal{Q}_{\beta}&:=\Big\{f:[0,\infty)\too\R\vert ~f \text{ is twice differentiable, } \nor{x^{d-1}f(x^2)}_{\mathcal{L}^1([0,\infty))},~\nor{x^{d}f^{(1)}(x^2)}_{\mathcal{L}^1([1,\infty))}\notag,\\
&~~~~~~~~\nor{x^{d+1}f^{(2)}(x^2)}_{\mathcal{L}^1([1,\infty))}\text{ and }\limsup_{x\rightarrow 0}\abs{f^{(j)}(x^2)/x^{-\beta_j}}\leq C<\infty, \text{ for }j=0,1,2\Big\},\label{eqn_QBeta_def}
\end{align}
where $f^{(j)}$ denotes the $j$-th derivative and $\beta=(\beta_0,\beta_1,\beta_2)\in(0,\infty)$. 
Then, Lemma \ref{lemma_riemannApprox_multi} can be expressed by the following corollary.
\begin{cor}\label{corollary_toLemmaRiemannApproxMulti}
Let $f\in\mathcal{Q}_\beta$ for $\beta=(\beta_0,\beta_1,\beta_2)\in(0,\infty)$, then it holds that
\begin{itemize}
\item[(i)] 
\begin{align*}
\Delta_n^{d/2}\sum_{\textbf{k}\in\N^d}f(\lambda_\textbf{k}\Delta_n)&=\frac{1}{2^{d}(\pi\eta)^{d/2}\Gamma(d/2)}\int_0^\infty x^{d/2-1}f(x)\diff x-\sum_{\substack{\nor{\gamma}_1=1 \\ \gamma\in\{0,1\}^{d}}}^{d-1}\int_{B_\gamma}f(\pi^2\eta\euc{\textbf{z}}^2)\diff \textbf{z}\\
&~~~~~+\Oo\big(\Delta_n\vee \Delta_n^{(d-\beta_0)/2}\vee\Delta_n^{(d+2-\beta_1)/2}\vee\Delta_n^{(d+4-\beta_2)/2}\big),
\end{align*}
where $B_\gamma$ is defined in equation \eqref{eqn_BGamma}.
\item[(ii)] For $\{j_1,\ldots,j_l\}\subset\{1,\ldots,d\}$, $\gamma_{j,l}\in\{0,1\}^d$, where $(\gamma_{j,l})_i=\mathbbm{1}_{i\in\{j_1,\ldots,j_l\}}$, with $i=1,\ldots,d$ and $l=1,\ldots,(d-1)$, we have 
\begin{align*}
&\Delta_n^{d/2}\sum_{k\in\N^d}f(\lambda_k\Delta_n)\cos(2\pi k_{j_1}y_{j_1})\cdot\ldots\cdot\cos(2\pi k_{j_l}y_{j_l})=(-1)^l\int_{B_{\gamma_{j,l}}}f(\pi^2\eta\euc{\textbf{z}}^2)\diff \textbf{z}\\
&~~~~~+ \Oo\big(\delta^{-(l+1)}\Delta_n\vee\delta^{-l}\Delta_n^{(l+1)/2}\vee \delta^{-(l+1)}\Delta_n^{(d+2-\beta_1)/2}\vee\delta^{-(l+1)}\Delta_n^{(d+4-\beta_2)/2}\big).
\end{align*}
\item[(iii)] For $\{j_1,\ldots,j_l\}=\{1,\ldots,d\}$, i.e. $l=d$, we have 
\begin{align*}
&\Delta_n^{d/2}\sum_{k\in\N^d}f(\lambda_k\Delta_n)\cos(2\pi k_{1}y_{1})\cdot\ldots\cdot\cos(2\pi k_{d}y_{d})=\Oo\big(\delta^{-(d+1)}\Delta_n\vee \Delta_n^{(d-\beta_0)/2}\\
&~~~~~\vee \delta^{-(d+1)}\Delta_n^{(d+2-\beta_1)/2}\vee\delta^{-(d+1)}\Delta_n^{(d+4-\beta_2)/2}\big).
\end{align*}
\end{itemize}
In particular, it holds for a $\tilde{\gamma}\in\{0,1\}^d$ with $\nor{\tilde{\gamma}}_1=l$ and $1\leq l\leq d-1$ that
\begin{align*}
\int_{B_{\tilde{\gamma}}}f(\pi^2\eta\euc{\textbf{z}}^2)\diff \textbf{z}=\Oo\Big(\Delta_n^{l/2}\vee \Delta_n^{(d-\beta_0)/2}\Big)
\end{align*}
and 
\begin{align*}
\sum_{\substack{\nor{\gamma}_1=1 \\ \gamma\in\{0,1\}^{d}}}^{d}\int_{B_\gamma}f(\pi^2\eta\euc{\textbf{z}}^2)\diff \textbf{z}=\Oo\Big(\Delta_n^{1/2}\vee \Delta_n^{(d-\beta_0)/2}\Big).
\end{align*}
\end{cor}
As the proof of the latter corollary is straightforward, we omit it. 
The following two functions:
\begin{align}
f_{\alpha}(x):=\frac{1-e^{-x}}{x^{1+\alpha}}~~~~~\text{and}~~~~~g_{\alpha,\tau}(x)=\frac{(1-e^{-x})^2}{2x^{1+\alpha}}e^{-\tau x},\label{equation_functionsAlphaandTau}
\end{align}
for $\alpha,\tau>0$ play a crucial role in the forthcoming analysis, particularly in calculating the realized volatility. In order to utilize Corollary \ref{corollary_toLemmaRiemannApproxMulti} for these functions, we need to verify their belonging to the class $\mathcal{Q}_\beta$ and determine the corresponding parameter $\beta$. The following lemma serves this purpose.
\begin{lemma}\label{lemma_checkConditionsAprroxlemma}
It holds: $f_\alpha\in\mathcal{Q}_{\beta_1}$ and $g_{\alpha,\tau}\in \mathcal{Q}_{\beta_2}$, where 
\begin{align*}
\beta_1=\big(2\alpha,2(1+\alpha),2(2+\alpha)\big)~~~~~\text{and}~~~~~\beta_2=\big(2\alpha,2(1+\alpha),2(1+\alpha)\big).
\end{align*}
\end{lemma}
\begin{proof}
We provide the proof for the function $f_\alpha$ since the proof for $g_{\alpha,\tau}$ follows in an analogous manner.
First, it holds with integration by parts that
\begin{align}
\int_0^\infty \frac{e^{- cx}}{x^{m}}\diff x&=\frac{c^{m-1}}{1-m}\Gamma(2-m)=c^{m-1}\Gamma(1-m),\label{equation_generalIntegraltoGamma}
\end{align}
for $m<1$ and $c>0$, where $\Gamma(z)=\int_0^\infty t^{z-1}e^{-t}\diff t$ denotes the Gamma function for $z\in\mathbb{C}$ and $\Re(z)\notin\{0,-1,-2,\ldots\}$. Note, that $\Gamma(1+z)=z\Gamma(z)$. 
By utilizing equation \eqref{equation_generalIntegraltoGamma}, we find:
\begin{align}
\int_0^\infty \frac{e^{-cx^2}}{x^m}\diff x =\frac{c^{(m-1)/2}}{2}\Gamma\bigg(\frac{1}{2}-\frac{m}{2}\bigg),\label{equation_generalIntegraltoGammaSquared}
\end{align}
where $m<1$ and $c>0$ and
\begin{align}
\int_0^\infty \frac{1-e^{-cx^2}}{x^m}\diff x =-\frac{c^{(m-1)/2}}{2}\Gamma\bigg(\frac{1}{2}-\frac{m}{2}\bigg)<C,\label{equation_generalIntegraltoGammaOneMinusSquared}
\end{align}
for $1<m<3$, $c>0$ and a constant $0<C<\infty$.
We begin by examining the conditions of the class $\mathcal{Q}_\beta$ for the functions $f_{\alpha}$ and $g_{\alpha,\tau}$. First and foremost, both functions $f_{\alpha}$ and $g_{\alpha,\tau}$ are evidently twice continuously differentiable. Here, we find:
\begin{align*}
f'_\alpha(x)&=\frac{e^{-x}}{x^{1+\alpha}}-(1+\alpha)\frac{1-e^{-x}}{x^{2+\alpha}},\\
f''_\alpha(x)&=-\frac{e^{-x}}{x^{1+\alpha}}-(1+\alpha)\frac{2e^{-x}}{x^{2+\alpha}}+(1+\alpha)(2+\alpha)\frac{1-e^{-x}}{x^{3+\alpha}},\\
g_{\alpha,\tau}'(x)&= e^{-x(\tau+1)}f_{\alpha}(x)-\frac{1+\alpha}{x}g_{\alpha,\tau}(x)-\tau g_{\alpha,\tau}(x), \\
g_{\alpha,\tau}''(x)&=-(\tau+1)e^{-x(\tau+1)}f_{\alpha}(x)+e^{-x(\tau+1)}f_{\alpha}'(x)+\frac{1+\alpha}{x^2}g_{\alpha,\tau}(x)-\frac{1+\alpha}{x}g'_{\alpha,\tau}(x)-\tau g'_{\alpha,\tau}(x).
\end{align*}
Furthermore, we obtain:
\begin{align}
\int_1^\infty x^me^{-x^2}\diff x=\frac{1}{2}\int_1^\infty x^{(m-1)/2}e^{-x}=\frac{1}{2}\Gamma\big((m+1)/2,1\big)\leq C,\label{eqn_intcrit1}
\end{align}
and
\begin{align}
\int_1^\infty x^m(1-e^{-x^2})\diff x&=\bigg[\frac{x^{m+1}}{m+1}(1-e^{-x^2})\bigg]_1^\infty-\frac{2}{m+1}\int_1^\infty x^{m+2}e^{-x^2}\diff x\notag\\
&=\bigg[\frac{x^{m+1}}{m+1}(1-e^{-x^2})\bigg]_1^\infty-\frac{1}{m+1}\Gamma\big((m+3)/2,1\big)\leq C,\label{eqn_intcrit2}
\end{align}
if $m<-1$. Here, $\Gamma(z,s)=\int_s^\infty t^{z-1}e^{-z}\diff z$ denotes the upper incomplete Gamma function. 
For the left limit, we obtain in general:
\begin{align}
\lim_{x\too 0}\frac{1-e^{-x^2}}{x^{m-\beta}}=\lim_{x\too 0}\frac{2e^{-x^2}}{(m-\beta)x^{m-\beta-2}}<\infty \Leftrightarrow m-2\leq \beta,\label{eqn_leftLimit1}
\end{align}
and 
\begin{align}
\lim_{x\too 0}\frac{e^{-x^2}}{x^{m-\beta}}<\infty\Leftrightarrow m\leq \beta.\label{eqn_leftLimit2}
\end{align}
Concerning the integration criteria for $f_{\alpha}$, we have by equation \eqref{equation_generalIntegraltoGammaOneMinusSquared} that $\nor{x^{d-1}f_{\alpha}(x^2)}_{\mathcal{L}^1([0,\infty))}$ since $1<1+2\alpha'<3$.
The integration criteria for the first and second derivative, $f'$ and $f''$, are established based on the equations \eqref{eqn_intcrit1} and \eqref{eqn_intcrit2}, as $d-4-2\alpha=-2-\alpha'<-1$ and $(d+1)-6-2\alpha=-3-2\alpha'<-1$.
Therefore, it remains to determine the parameters $\beta_0,\beta_1,\beta_2$ which are associated to $f,f'$ and $f''$, respectively. 
Using the displays \eqref{eqn_leftLimit1} and \eqref{eqn_leftLimit2}, we have $f\in\mathcal{Q}_\beta$, with
\begin{align*}
\beta_0=2\alpha,~~~~~\beta_1=2(1+\alpha),~~~~~\text{and}~~~~~\beta_2=2(2+\alpha).
\end{align*}
\end{proof}

\begin{lemma}\label{lemma_calcfAlphaDelta}
On Assumptions \ref{assumption_observations_multi} and \ref{assumption_regMulti} it holds that
\begin{align*}
\Delta_n^{d/2}\sum_{\textbf{k}\in\N^d}f_\alpha(\lambda_\textbf{k}\Delta_n)&=\frac{\Gamma(1-\alpha')}{2^d(\pi\eta)^{d/2}\alpha'\Gamma(d/2)}-\sum_{\substack{\nor{\gamma}_1=1 \\ \gamma\in\{0,1\}^{d}}}^{d-1}\int_{B_\gamma}f_\alpha(\pi^2\eta\euc{\textbf{z}}^2)\diff \textbf{z}+\Oo\big(\Delta_n^{1-\alpha'}\big),
\end{align*}
where $\Gamma$ denotes the Gamma function. Furthermore, it holds that
\begin{align*}
\Delta_n^{d/2}\sum_{\textbf{k}\in\N^d}g_{\alpha,\tau}(\lambda_\textbf{k}\Delta_n)&= \frac{1}{2}\left(-\tau ^{\alpha'}+2 (\tau +1)^{\alpha'}-(\tau +2)^{\alpha'}\right) \frac{\Gamma(1-\alpha')}{2^{d}(\pi\eta)^{d/2}\alpha'\Gamma(d/2)}\\
&~~~~~-\sum_{\substack{\nor{\gamma}_1=1 \\ \gamma\in\{0,1\}^{d}}}^{d-1}\int_{B_\gamma}g_{\alpha,\tau}(\pi^2\eta\euc{\textbf{z}}^2)\diff \textbf{z}+\Oo\big(\Delta_n^{1-\alpha'}\big).
\end{align*}
\end{lemma}
\begin{proof}
Given that Lemma \ref{lemma_checkConditionsAprroxlemma} establishes $f_{\alpha}\in\mathcal{Q}_{\beta_1}$ and $g_{\alpha,\tau}\in\mathcal{Q}_{\beta_2}$, with $\beta_1=\big(2\alpha,2(1+\alpha),2(2+\alpha)\big)$ and $\beta_2=\big(2\alpha,2(1+\alpha),2(1+\alpha)\big)$, we can employ Corollary \ref{corollary_toLemmaRiemannApproxMulti} on these functions. 
In addition, by utilizing analogous steps as in equation \eqref{equation_generalIntegraltoGamma}, we find:
\begin{align}
\int_0^\infty \frac{1-e^{-cx}}{x^m}=\frac{c^{m-1}}{m-1}\Gamma(2-m), \label{eqn_integralSolutionExpm}
\end{align}
for $1<m<2$ and $c>0$.
Considering $\alpha=d/2-1+\alpha'$, where $\alpha'\in(0,1)$, and equation \eqref{eqn_integralSolutionExpm}, we obtain the following:
\begin{align*}
\Delta_n^{d/2}\sum_{\textbf{k}\in\N^d}f_\alpha(\lambda_\textbf{k}\Delta_n)&=\frac{1}{2^{d}(\pi\eta)^{d/2}\Gamma(d/2)}\int_0^\infty x^{d/2-1}\frac{1-e^{-x}}{x^{1+\alpha}}\diff x+R_{n,1}\\
&=\frac{1}{2^{d}(\pi\eta)^{d/2}\Gamma(d/2)}\int_0^\infty \frac{1-e^{-x}}{x^{1+\alpha'}}\diff x+R_{n,1}\\
&=\frac{1}{2^{d}(\pi\eta)^{d/2}\Gamma(d/2)}\cdot\frac{\Gamma(1-\alpha')}{\alpha'}+R_{n,1},
\end{align*}
where
\begin{align*}
R_{n,1}&:=-\sum_{\substack{\nor{\gamma}_1=1 \\ \gamma\in\{0,1\}^{d}}}^{d-1}\int_{B_\gamma}f(\pi^2\eta\euc{\textbf{z}}^2)\diff \textbf{z}+\Oo\big(\Delta_n\vee \Delta_n^{(d-2\alpha)/2}\vee\Delta_n^{(d+2-2(1+\alpha))/2}\vee\Delta_n^{(d+4-2(2+\alpha))/2}\big)\\
&=-\sum_{\substack{\nor{\gamma}_1=1 \\ \gamma\in\{0,1\}^{d}}}^{d-1}\int_{B_\gamma}f(\pi^2\eta\euc{\textbf{z}}^2)\diff \textbf{z}+\Oo\big(\Delta_n^{1-\alpha'}\big).
\end{align*}
For statement (ii), we have 
\begin{align*}
\Delta_n^{d/2}\sum_{\textbf{k}\in\N^d}g_{\alpha,\tau}(\lambda_\textbf{k}\Delta_n)&=\frac{1}{2^{d}(\pi\eta)^{d/2}\Gamma(d/2)}\int_0^\infty \frac{(1-e^{-x})^2}{2x^{1+\alpha'}}e^{-\tau x}\diff x+R_{n,2}\\
&=\frac{1}{2^{d+1}(\pi\eta)^{d/2}\Gamma(d/2)}\bigg(\int_0^\infty \frac{e^{-\tau x}}{x^{1+\alpha'}}-2\int_0^\infty \frac{e^{-x(1+\tau)}}{x^{1+\alpha'}}+\int_0^\infty \frac{e^{-x(2+\tau)}}{x^{1+\alpha}} \bigg)+R_{n,2}.
\end{align*}
By using equation \eqref{equation_generalIntegraltoGamma}, we find that
\begin{align*}
\Delta_n^{d/2}\sum_{\textbf{k}\in\N^d}g_{\alpha,\tau}(\lambda_\textbf{k}\Delta_n)
&=\frac{\Gamma(1-\alpha')}{2^{d+1}(\pi\eta)^{d/2}\alpha'\Gamma(d/2)}\big(-\tau^{\alpha'}+2(1+\tau)^{\alpha'}-(2+\tau)^{\alpha'}\big)+R_{n,2},
\end{align*}
where 
\begin{align*}
R_{n,2}&:= -\sum_{\substack{\nor{\gamma}_1=1 \\ \gamma\in\{0,1\}^{d}}}^{d-1}\int_{B_\gamma}f(\pi^2\eta\euc{\textbf{z}}^2)\diff \textbf{z}+\Oo\big(\Delta_n\vee \Delta_n^{(d-2\alpha)/2}\vee\Delta_n^{(d+2-2(1+\alpha))/2}\vee\Delta_n^{(d+4-2(1+\alpha))/2}\big)\\
&=-\sum_{\substack{\nor{\gamma}_1=1 \\ \gamma\in\{0,1\}^{d}}}^{d-1}\int_{B_\gamma}f(\pi^2\eta\euc{\textbf{z}}^2)\diff \textbf{z}+\Oo\big(\Delta_n^{1-\alpha'}\big).
\end{align*}
The proof follows by utilizing the following identity for half-integer arguments:
\begin{align*}
\Gamma(n/2)=\frac{(n-2)!!\sqrt{\pi}}{2^{(n-1)/2}}.
\end{align*}
\end{proof}

\begin{lemma}\label{lemma_quadIncrementsFirstCalc}
On Assumptions \ref{assumption_observations_multi} and \ref{assumption_regMulti}, we have 
\begin{align*}
\E[(\Delta_iX)^2(\textbf{y})]&=\sigma^2 2^de^{-\nor{\kappa\bigcdot \textbf{y}}_1}\sum_{\textbf{k}\in\N^d}\mathcal{D}_{i,\textbf{k}}\sin^2(\pi k_1y_1)\cdot\ldots\cdot\sin^2(\pi k_dy_d)+r_{n,i},
\end{align*}
where $r_{n,i}$ is a sequence satisfying $\sum_{i=1}^nr_{n,i}=\Oo(\Delta_n^{\alpha'})$ and 
\begin{align}
\mathcal{D}_{i,\textbf{k}}
&=\Delta_n^{d/2+\alpha'}\bigg(\frac{1-e^{-\lambda_\textbf{k}\Delta_n} }{(\lambda_\textbf{k}\Delta_n)^{1+\alpha}}-\frac{(1-e^{-\lambda_\textbf{k}\Delta_n})^2}{2(\lambda_\textbf{k}\Delta_n)^{1+\alpha}}e^{-2\lambda_\textbf{k}\iidn}\bigg),\label{equation_dik}
\end{align}
with $\alpha'=1+\alpha-d/2\in(0,1)$.
\end{lemma}
\begin{proof}
First, $A_{i,\textbf{k}}, B_{i,\textbf{k}}$, and $C_{i,\textbf{k}}$ are independent of each other, where $i=1,\ldots,n$ and $\textbf{k}\in\N^d$. Exploiting the fact that $(W^\textbf{k})_{\textbf{k}\in\N^d}$ are independent Brownian motions, the \ito-integrals $B_{i,\textbf{k}}$ and $C_{i,\textbf{k}}$ are also independent and centred. Thus, we have
\begin{align*}
\E[(\Delta_iX)^2(\textbf{y})]&=\sum_{\textbf{k}_1\in\N^d}\sum_{\textbf{k}_2\in\N^d}e_{\textbf{k}_1}(\textbf{y})e_{\textbf{k}_2}(\textbf{y})\E[\Delta_ix_{\textbf{k}_1}\Delta_ix_{\textbf{k}_2}]\\
&=\sum_{\textbf{k}\in\N^d}e_{\textbf{k}}^2(\textbf{y})\big(\E[B_{i,\textbf{k}}^2]+\E[C_{i,\textbf{k}}^2]\big)+r_{n,i},
\end{align*}
where $r_{n,i}:=\sum_{\textbf{k}_1,\textbf{k}_2\in\N^d}e_{\textbf{k}_1}(\textbf{y})e_{\textbf{k}_2}(\textbf{y})\E[A_{i,\textbf{k}_1}A_{i,\textbf{k}_2}]$.
\ito-isometry yields the following:
\begin{align*}
\E[B_{i,\textbf{k}}^2]&=\E\Bigg[\bigg(\sigma\lambda_\textbf{k}^{-\alpha/2}  \int\limits_0^{\iidn} e^{-\lambda_\textbf{k}\big(\iidn-s\big)}\big(e^{-\lambda_\textbf{k}\Delta_n}-1\big)\diff W_s^\textbf{k}\bigg)^2\Bigg]=\sigma^2\big(1-e^{-\lambda_\textbf{k}\Delta_n}\big)^2  \frac{1-e^{-2\lambda_\textbf{k}\iidn}}{2\lambda_\textbf{k}^{1+\alpha}},  \\
\E[C_{i,\textbf{k}}^2]&=\sigma^2\lambda_\textbf{k}^{-\alpha} \int\limits_{(i-1)\Delta_n}^{i\Delta_n}e^{-2\lambda_\textbf{k}(i\Delta_n-s)}\diff s = \sigma^2 \frac{1-e^{-2\lambda_\textbf{k}\Delta_n}}{2\lambda_\textbf{k}^{1+\alpha}}.
\end{align*}
Additionally, we possess the following expression for the remainder $r_{n,i}$:
\begin{align*}
\E[A_{i,\textbf{k}_1}A_{i,\textbf{k}_2}]
&=\big(e^{-\lambda_{\textbf{k}_1}\iidn-\lambda_{\textbf{k}_1}\Delta_n}-e^{-\lambda_{\textbf{k}_1}\iidn}\big) \big(e^{-\lambda_{\textbf{k}_2}\iidn-\lambda_{\textbf{k}_2}\Delta_n}-e^{-\lambda_{\textbf{k}_2}\iidn}\big)\\
&~~~~~\times\E\big[\langle \xi,e_{\textbf{k}_1}\rangle_\vartheta\langle \xi,e_{\textbf{k}_2}\rangle_\vartheta\big]\\
  &=\big(1-e^{-\lambda_{\textbf{k}_1}\Delta_n}\big) \big(1-e^{-\lambda_{\textbf{k}_2}\Delta_n}\big)e^{-(\lambda_{\textbf{k}_1}+\lambda_{\textbf{k}_2})\iidn}\E\big[\langle \xi,e_{\textbf{k}_1}\rangle_\vartheta\langle \xi,e_{\textbf{k}_2}\rangle_\vartheta\big].
\end{align*}
Hence, we obtain the representation:
\begin{align*}
\E[(\Delta_iX)^2(\textbf{y})]
&=\sigma^2 2^de^{-\sum_{l=1}^d\kappa_ly_l}\sum_{\textbf{k}\in\N^d}\bigg(\big(1-e^{-\lambda_\textbf{k}\Delta_n}\big)^2  \frac{1-e^{-2\lambda_\textbf{k}\iidn}}{2\lambda_\textbf{k}^{1+\alpha}}+\frac{1-e^{-2\lambda_\textbf{k}\Delta_n}}{2\lambda_\textbf{k}^{1+\alpha}}\bigg)\\&~~~~~\times\sin^2(\pi k_1y_1)\cdot\ldots\cdot\sin^2(\pi k_dy_d)+r_{n,i}.
\end{align*}
In addition, let us define:
\begin{align*}
\mathcal{D}_{i,\textbf{k}}&:=\Delta_n^{1+\alpha}\bigg(\big(1-e^{-\lambda_\textbf{k}\Delta_n}\big)^2  \frac{1-e^{-2\lambda_\textbf{k}\iidn}}{2(\lambda_\textbf{k}\Delta_n)^{1+\alpha}}+\frac{1-e^{-2\lambda_\textbf{k}\Delta_n}}{2(\lambda_\textbf{k}\Delta_n)^{1+\alpha}}\bigg)\\
&=\Delta_n^{d/2+\alpha'}\bigg(\frac{1-e^{-\lambda_\textbf{k}\Delta_n} }{(\lambda_\textbf{k}\Delta_n)^{1+\alpha}}-\frac{(1-e^{-\lambda_\textbf{k}\Delta_n})^2}{2(\lambda_\textbf{k}\Delta_n)^{1+\alpha}}e^{-2\lambda_\textbf{k}\iidn}\bigg),
\end{align*}
where $\alpha'\in(0,1)$. Then, we have 
\begin{align*}
\E[(\Delta_iX)^2(\textbf{y})]&=\sigma^2 2^de^{-\nor{\kappa\bigcdot y}_1}\sum_{\textbf{k}\in\N^d}\mathcal{D}_{i,\textbf{k}}\sin^2(\pi k_1y_1)\cdot\ldots\cdot\sin^2(\pi k_dy_d)+r_{n,i}.
\end{align*}
The analysis of the remainder $r_{n,i}$ remains to be conducted. Here, we have
\begin{align*}
r_{n,i}=\sum_{\textbf{k}_1,\textbf{k}_2\in\N^d}e_{\textbf{k}_1}(\textbf{y})e_{\textbf{k}_2}(\textbf{y})\big(1-e^{-\lambda_{\textbf{k}_1}\Delta_n}\big) \big(1-e^{-\lambda_{\textbf{k}_2}\Delta_n}\big)e^{-(\lambda_{\textbf{k}_1}+\lambda_{\textbf{k}_2})\iidn}\E\big[\langle \xi,e_{\textbf{k}_1}\rangle_\vartheta\langle \xi,e_{\textbf{k}_2}\rangle_\vartheta\big].
\end{align*}
To demonstrate that $\sum_{i=1}^n r_{n,i} = \mathcal{O}(\Delta_n^{\alpha'})$, we use Assumption \ref{assumption_regMulti}. Under the conditions $\E[\langle \xi,e_\textbf{k}\rangle_\vartheta]=0$ and $\sup_{\textbf{k}\in\N^d}\lambda_\textbf{k}^{1+\alpha}\E[\langle \xi,e_\textbf{k}\rangle_\vartheta^2] < \infty$, we can find a constant $C > 0$ such that $\E[\langle \xi,e_\textbf{k}\rangle_\vartheta^2] \leq C/\lambda_\textbf{k}^{1+\alpha}$ for all $k\in\N^d$. Consequently, given that $\big(\langle \xi,e_\textbf{k}\rangle_\vartheta\big)_{\textbf{k}\in\N^d}$ are independent, we have 
\begin{align*}
r_{n,i}=\sum\limits_{\textbf{k}\in\N^d} \big(1-e^{-\lambda_\textbf{k}\Delta_n}\big)^2 e^{-2\lambda_\textbf{k}\iidn}e_\textbf{k}^2(\textbf{y})\E[\langle \xi,e_\textbf{k}\rangle^2_\vartheta]\leq C\sum\limits_{\textbf{k}\in\N^d} \frac{\big(1-e^{-\lambda_\textbf{k}\Delta_n}\big)^2}{\lambda_\textbf{k}^{1+\alpha}} e^{-2\lambda_\textbf{k}\iidn}e_\textbf{k}^2(\textbf{y}).
\end{align*} 
Assuming the second alternative in Assumption \ref{assumption_regMulti}, where $\E\big[\nor{A_\vartheta^{(1+\alpha)/2}\xi}_\vartheta^2\big]<\infty$, we can proceed with the following steps. Exploiting the self-adjointness of $A_\vartheta$ on $H_\vartheta$ and employing the Cauchy-Schwarz inequality, we obtain:
\begin{align*}
r_{n,i}&=\E\bigg[\Big(\sum\limits_{\textbf{k}\in\N^d} \big(1-e^{-\lambda_\textbf{k}\Delta_n}\big)e^{-\lambda_\textbf{k}\iidn}\langle\xi,e_\textbf{k}\rangle_\vartheta e_\textbf{k}(\textbf{y})\Big)^2\bigg]\\
&\leq \sum\limits_{\textbf{k}\in\N^d}\frac{\big(1-e^{-\lambda_\textbf{k}\Delta_n}\big)^2}{\lambda_\textbf{k}^{1+\alpha}}e^{-2\lambda_\textbf{k}\iidn}e_\textbf{k}^2(\textbf{y})\E\Big[\sum\limits_{\textbf{k}\in\N}^\infty \langle A_\vartheta^{(1+\alpha)/2}\xi,e_\textbf{k}\rangle_\vartheta^2 \Big].
\end{align*}
Applying Parseval's identity on the expected value gives us:
\begin{align*}
r_{n,i} \leq \E\Big[\nor{A_\vartheta^{(1+\alpha)/2}\xi}_\vartheta^2\Big]\sum\limits_{\textbf{k}\in\N^d}\frac{\big(1-e^{-\lambda_\textbf{k}\Delta_n}\big)^2}{\lambda_\textbf{k}^{1+\alpha}}e^{-2\lambda_\textbf{k}\iidn}e_\textbf{k}^2(\textbf{y}).
\end{align*}
Since we can uniformly bound the eigenfunctions $(e_{\textbf{k}})_{\textbf{k}\in\N^d}$, it is sufficient to bound the following expression
\begin{align*}
\sum\limits_{i=1}^n r_{n,i} &\leq C\sum\limits_{\textbf{k}\in\N^d} \frac{\big(1-e^{-\lambda_\textbf{k}\Delta_n}\big)^2}{\lambda_\textbf{k}^{1+\alpha}}\sum\limits_{i=1}^n e^{-2\lambda_\textbf{k}\iidn}\notag \\
&\leq  C\sum\limits_{\textbf{k}\in\N^d} \frac{\big(1-e^{-\lambda_\textbf{k}\Delta_n}\big)^2}{\lambda_\textbf{k}^{1+\alpha}\big(1-e^{-2\lambda_\textbf{k}\Delta_n}\big)}\\
&\leq  C\sum\limits_{\textbf{k}\in\N^d}  \frac{1-e^{-\lambda_\textbf{k}\Delta_n}}{\lambda_\textbf{k}^{1+\alpha}}\notag
\\&=C\Delta_n^{d/2+\alpha'}\sum\limits_{\textbf{k}\in\N^d}  \frac{1-e^{-\lambda_\textbf{k}\Delta_n}}{(\lambda_\textbf{k}\Delta_n)^{1+\alpha}},\notag
\end{align*}
for both cases in Assumption \ref{assumption_regMulti}, where we have used the closed form formula of the geometric series and a suitable constant $C>0$. Utilizing Lemma \ref{lemma_calcfAlphaDelta}, we obtain: 
\begin{align*}
\Delta_n^{d/2}\sum\limits_{\textbf{k}\in\N^d}  \frac{1-e^{-\lambda_\textbf{k}\Delta_n}}{(\lambda_\textbf{k}\Delta_n)^{1+\alpha}}=\Delta_n^{d/2}\sum\limits_{\textbf{k}\in\N^d} f_{\alpha'}(\lambda_\textbf{k}\Delta_n)=C+\smallO(1),
\end{align*}
with a suitable constant $C>0$. Hence, we have
\begin{align*}
\sum_{i=1}^nr_{n,i}=\Oo(\Delta_n^{\alpha'}),
\end{align*}
which completes the proof.
\end{proof}
\ \\
It follows the proof of Proposition \ref{prop_quadIncAndrescaling}.
\begin{proof}
We begin by recalling Lemma \ref{lemma_quadIncrementsFirstCalc}:
\begin{align*}
\E[(\Delta_iX)^2(\textbf{y})]&=\sigma^2 e^{-\nor{\kappa\bigcdot \textbf{y}}_1}\mathcal{T}_i+r_{n,i},
\end{align*}
where
\begin{align*}
\mathcal{T}_i:=&2^d\sum_{\textbf{k}\in\N^d}\sin^2(\pi k_1y_1)\cdot\ldots\cdot\sin^2(\pi k_dy_d)\mathcal{D}_{i,\textbf{k}}\\
&=2^d\sum_{\textbf{k}\in\N^d}\frac{1-\cos(2\pi k_1y_1)}{2}\cdot\ldots\cdot\frac{1-\cos(2\pi k_dy_d)}{2}\mathcal{D}_{i,\textbf{k}}\\
&=\sum_{\textbf{k}\in\N^d}\mathcal{D}_{i,\textbf{k}}+\sum_{l=1}^{d-1}\sum_{1\leq j_1<\ldots<j_l\leq d}(-1)^l \sum_{\textbf{k}\in\N^d} \mathcal{D}_{i,\textbf{k}}\cos(2\pi k_{j_1}y_{j_1})\cdot\ldots\cdot\cos(2\pi k_{j_l}y_{j_l})\\
&~~~~~+(-1)^{d}\sum_{\textbf{k}\in\N^d} \mathcal{D}_{i,\textbf{k}}\cos(2\pi k_{1}y_{1})\cdots\cos(2\pi k_{d}y_{d}),
\end{align*}
where $\mathcal{D}_{i,\textbf{k}}$ is defined as in display \eqref{equation_dik}.
Furthermore, we define:
\begin{align*}
h_{\alpha,\tau}(x):=\bigg(\frac{1-e^{-x} }{x^{1+\alpha}}-\frac{(1-e^{-x})^2}{2x^{1+\alpha}}e^{-x\tau}\bigg).
\end{align*}
Note, that $h_{\alpha,\tau}(x)=f_{\alpha}(x)-g_{\alpha,\tau}(x)$, where $f_{\alpha}$ and $g_{\alpha,\tau}$ are defined as in equation \eqref{equation_functionsAlphaandTau}.
By Lemma \ref{lemma_checkConditionsAprroxlemma} we have $h_{\alpha,\tau}\in \mathcal{Q}_\beta$, where $\beta=\big(2\alpha,2(1+\alpha),2(2+\alpha)\big)$. Then, we obtain:
\begin{align*}
\Delta_n^{-\alpha'}\sum_{\textbf{k}\in\N^d}\mathcal{D}_{i,\textbf{k}}&=\Delta_n^{d/2}\sum_{\textbf{k}\in\N^d}h_{\alpha,2(i-1)}(\lambda_\textbf{k}\Delta_n)\\
&=\frac{1}{2^d(\pi\eta)^{d/2}\Gamma(d/2)}\int_0^\infty x^{d/2-1}h_{\alpha,2(i-1)}(x)\diff x-\sum_{\substack{\nor{\gamma}_1=1 \\ \gamma\in\{0,1\}^{d}}}^{d-1}\int_{B_\gamma}h_{\alpha,2(i-1)}(\pi^2\eta\euc{\textbf{z}}^2)\diff \textbf{z}+\Oo\big(\Delta_n^{1-\alpha'}\big),
\end{align*}
and 
\begin{align*}
&\Delta_n^{d/2}\sum_{l=1}^{d-1}\sum_{1\leq j_1<\ldots<j_l\leq d}(-1)^l \sum_{\textbf{k}\in\N^d} h_{\alpha,2(i-1)}(\lambda_\textbf{k}\Delta_n)\cos(2\pi k_{j_1}y_{j_1})\cdot\ldots\cdot\cos(2\pi k_{j_l}y_{j_l})\\
&=\sum_{\substack{\nor{\gamma}_1=1 \\ \gamma\in\{0,1\}^{d}}}^{d-1}\int_{B_{\tilde{\gamma}_l}}h_{\alpha,2(i-1)}(\pi^2\eta\nor{\textbf{z}}_2^2)\diff \textbf{z}+\Oo(\Delta_n^{1-\alpha'}).
\end{align*}
Thus, by using Corollary \ref{corollary_toLemmaRiemannApproxMulti} we have
\begin{align*}
&\sum_{\textbf{k}\in\N^d}\mathcal{D}_{i,\textbf{k}}+\sum_{l=1}^{d-1}\sum_{1\leq j_1<\ldots<j_l\leq d}(-1)^l \sum_{\textbf{k}\in\N^d} \mathcal{D}_{i,\textbf{k}}\cos(2\pi k_{j_1}y_{j_1})\cdot\ldots\cdot\cos(2\pi k_{j_l}y_{j_l})\\
&~~~~~+(-1)^{d}\sum_{\textbf{k}\in\N^d} \mathcal{D}_{i,\textbf{k}}\cos(2\pi k_{1}y_{1})\cdots\cos(2\pi k_{d}y_{d})\\
&=\frac{\Delta_n^{\alpha'}}{2^d(\pi\eta)^{d/2}\Gamma(d/2)}\int_0^\infty x^{d/2-1}h_{\alpha,2(i-1)}(x)\diff x+\Oo(\Delta_n).
\end{align*}
Utilizing Lemma \ref{lemma_calcfAlphaDelta} yields
\begin{align*}
&\frac{1}{2^d(\pi\eta)^{d/2}\Gamma(d/2)}\bigg(\int_0^\infty f_{\alpha'}(x)\diff x-\int_0^\infty g_{\alpha',2(i-1)}(x)\diff x\bigg)\\
&=\frac{\Gamma(1-\alpha')}{2^{d}(\pi\eta)^{d/2}\alpha'\Gamma(d/2)}\bigg(1+\frac{1}{2}\big(2(i-1)\big)^{\alpha'}-\big(1+2(i-1)\big)^{\alpha'}+\frac{1}{2}\big(2+2(i-1)\big)^{\alpha'}\bigg).
\end{align*}
Therefore, we have with Lemma \ref{lemma_quadIncrementsFirstCalc} that
\begin{align*}
\E[(\Delta_i X)^2(\textbf{y})]&=\sigma^2e^{-\nor{\kappa \textbf{y}}_1}\frac{\Delta_n^{\alpha'}\Gamma(1-\alpha')}{2^{d}(\pi\eta)^{d/2}\alpha'\Gamma(d/2)}\bigg(1+\frac{1}{2}\big(2(i-1)\big)^{\alpha'}-\big(1+2(i-1)\big)^{\alpha'}+\frac{1}{2}\big(2+2(i-1)\big)^{\alpha'}\bigg)\\
&~~~~~+r_{n,i}+\Oo(\Delta_n)\\
&=\Delta_n^{\alpha'}\frac{\sigma^2e^{-\nor{\kappa \textbf{y}}_1}\Gamma(1-\alpha')}{2^{d}(\pi\eta)^{d/2}\alpha'\Gamma(d/2)}+\tilde{r}_{n,i}+\Oo(\Delta_n),
\end{align*}
where $\tilde{r}_{n,i}$ includes $r_{n,i}$ and the $i$ dependent term from the last display. For this $i$-dependent term we define the following function: 
\begin{align*}
t(x):=\frac{x^{\alpha'}}{2}\big(-1+2(1+1/x)^{\alpha'}-(1+2/x)^{\alpha'}\big),
\end{align*}
and have with $x=2(i-1)$ that
\begin{align}
\sum_{i=1}^nt\big(2(i-1)\big)=\Oo\bigg(\sum_{i=0}^\infty \frac{1}{i^{2-\alpha'}}\bigg)=\Oo(1).\label{eqn_orderOfiTermsGeometrricSum}
\end{align}
Hence, we have by Lemma \ref{lemma_quadIncrementsFirstCalc} that $\sum_{i=1}^n\tilde{r}_{n,i}=\Oo(\Delta_n^{\alpha'})$, which completes the proof.
\end{proof}

Next, we proof Proposition \ref{prop_autocovOfIncrementsMulti}.
\begin{proof}
We begin with the following expression:
\begin{align*}
\Cov(\Delta_iX(\textbf{y}),\Delta_jX(\textbf{y}))&=\sum\limits_{\textbf{k}_1,\textbf{k}_2\in\N^d} \Cov (\Delta_ix_{\textbf{k}_1},\Delta_jx_{\textbf{k}_2})e_{\textbf{k}_1}(\textbf{y})e_{\textbf{k}_2}(\textbf{y})\\
&=\sum\limits_{\textbf{k}\in\N^d}\Cov (A_{i,\textbf{k}}+B_{i,\textbf{k}}+C_{i,\textbf{k}},A_{j,\textbf{k}}+B_{j,\textbf{k}}+C_{j,\textbf{k}})e_\textbf{k}^2(\textbf{y}).
\end{align*}
Since $\big(\langle \xi,e_\textbf{k}\rangle_\vartheta\big){\textbf{k}\in\N^d}$ are independent by Assumption \ref{assumption_regMulti}, we can use the independence of $A_{i,\textbf{k}}$ and $B_{i,\textbf{k}}$ and analyse the covariance of the remaining terms. Here, we have, by the Itô-isometry and $i<j$:
\begin{align*}
\Sigma_{i,j}^{B,\textbf{k}}&:=\Cov(B_{i,\textbf{k}},B_{j,\textbf{k}})\\
&=\sigma^2\lambda_\textbf{k}^{-\alpha} \big(1-e^{-\lambda_\textbf{k}\Delta_n}\big)^2 e^{-\lambda_\textbf{k}(i+j-2)\Delta_n} \Cov\Bigg(\int\limits_0^{(i-1)\Delta_n}e^{\lambda_\textbf{k}s}\diff W_s^\textbf{k},\int\limits_0^{(i-1)\Delta_n}e^{\lambda_ks}\diff W_s^\textbf{k}\Bigg)\\
&=\sigma^2 \big(e^{-\lambda_\textbf{k}\Delta_n(j-i)}-e^{-\lambda_\textbf{k}(i+j-2)\Delta_n}\big)\frac{\big(1-e^{-\lambda_\textbf{k}\Delta_n}\big)^2 }{2\lambda_\textbf{k}^{1+\alpha}}.
\end{align*}
Therefore, it follows for $1\leq i,j\leq n$ that
\begin{align}
\Sigma_{i,j}^{B,\textbf{k}}=\sigma^2 \big(e^{-\lambda_\textbf{k}\Delta_n\abs{i-j}}-e^{-\lambda_\textbf{k}(i+j-2)\Delta_n}\big)\frac{\big(1-e^{-\lambda_\textbf{k}\Delta_n}\big)^2 }{2\lambda_\textbf{k}^{1+\alpha}}. \label{number:CovBB}
\end{align}
Next, we have $\Sigma_{i,j}^{C,\textbf{k}}=\Cov(C_{i,\textbf{k}},C_{j,\textbf{k}})=0$, for $i\neq j$, and we derive the following:
\begin{align}
\Sigma_{i,j}^{C,\textbf{k}} = \mathbbm{1}_{\{j=i\}}\Cov(C_{i,\textbf{k}},C_{i,\textbf{k}})
&=\mathbbm{1}_{\{j=i\}}\sigma^2 \frac{1-e^{-2\lambda_\textbf{k}\Delta_n}}{2\lambda_\textbf{k}^{1+\alpha}}.\label{eqn_covCijk}
\end{align}
It remains to analyse the covariance of $B_{i,\textbf{k}}$ and $C_{j,\textbf{k}}$. Since $\Sigma_{i,j}^{BC,\textbf{k}}:=\Cov(B_{i,\textbf{k}},C_{j,\textbf{k}})=0$ for $i\leq j$, we analyse the following:
\begin{align}
\Sigma_{i,j}^{BC,\textbf{k}} &=  \mathbbm{1}_{\{i>j\}}\Cov(B_{i,\textbf{k}},C_{j,\textbf{k}})\notag\\ 
&=\mathbbm{1}_{\{i>j\}}\sigma^2\lambda_\textbf{k}^{-\alpha}\big(e^{-\lambda_\textbf{k}\Delta_n}-1\big)e^{-\lambda_\textbf{k}(i-1)\Delta_n}e^{-\lambda_\textbf{k}j\Delta_n}\Cov\Bigg(\int_{(j-1)\Delta_n}^{j\Delta_n}e^{\lambda_\textbf{k}s}\diff W_s^\textbf{k},\int_{(j-1)\Delta_n}^{j\Delta_n}e^{\lambda_\textbf{k}s}\diff W_s^\textbf{k}\Bigg)\notag \\
&=\mathbbm{1}_{\{i>j\}}\sigma^2e^{-\lambda_\textbf{k}\Delta_n(i-j)}\big(e^{\lambda_\textbf{k}\Delta_n}-e^{-\lambda_\textbf{k}\Delta_n}\big)\frac{e^{-\lambda_\textbf{k}\Delta_n}-1}{2\lambda_\textbf{k}^{1+\alpha}}\label{number:CovBC}.
\end{align}
Similarly, we have 
\[\Sigma_{j,i}^{BC,\textbf{k}}:=\Sigma_{i,j}^{CB,\textbf{k}}:=\Cov(C_{i,\textbf{k}},B_{j,\textbf{k}})=\mathbbm{1}_{\{i<j\}}\sigma^2e^{-\lambda_\textbf{k}\Delta_n(j-i)}\big(e^{\lambda_\textbf{k}\Delta_n}-e^{-\lambda_\textbf{k}\Delta_n}\big)\frac{e^{-\lambda_\textbf{k}\Delta_n}-1}{2\lambda_\textbf{k}^{1+\alpha}}. \]
For $i<j$ we obtain:
\begin{align*}
\Cov(\Delta_iX(\textbf{y}),\Delta_jX(\textbf{y}))&=\sum\limits_{\textbf{k}\in\N^d} \Cov (A_{i,\textbf{k}}+B_{i,\textbf{k}}+C_{i,\textbf{k}},A_{j,\textbf{k}}+B_{j,\textbf{k}}+C_{j,\textbf{k}})e_\textbf{k}^2(\textbf{y}) \\
&= \sum\limits_{\textbf{k}\in\N^d} \big(\Sigma_{i,j}^{B,\textbf{k}}+\Sigma_{i,j}^{CB,\textbf{k}}\big)e_\textbf{k}^2(\textbf{y})+r_{i,j},
\end{align*}
where
\begin{align*}
r_{i,j}&:=\sum_{\textbf{k}\in\N^d}\Cov(A_{i,\textbf{k}},A_{j,\textbf{k}})e_\textbf{k}^2(\textbf{y})\\
&=\sum_{\textbf{k}\in\N^d}\big(e^{-\lambda_\textbf{k}i\Delta_n}-e^{-\lambda_\textbf{k}(i-1)\Delta_n}\big)\big(e^{-\lambda_\textbf{k}j\Delta_n}-e^{-\lambda_\textbf{k}(j-1)\Delta_n}\big)\Var\big(\langle\xi,e_\textbf{k}\rangle_\vartheta\big)e_\textbf{k}^2(\textbf{y})\\
&=\sum_{\textbf{k}\in\N^d} e^{-\lambda_\textbf{k}\Delta_n(i+j-2)}\big(e^{-\lambda_\textbf{k}\Delta_n}-1\big)^2\Var\big(\langle\xi,e_\textbf{k}\rangle_\vartheta\big)e_\textbf{k}^2(\textbf{y}).
\end{align*}
We use that the operator $A_\vartheta$ is self-adjoint on $H_\vartheta$, such that $-\lambda_\textbf{k}^{(1+\alpha)/2}\langle\xi,e_\textbf{k}\rangle=\langle A_\vartheta^{(1+\alpha)/2}\xi,e_\textbf{k}\rangle$ and derive the following inequality for the remainder:
\begin{align}
r_{i,j}&\leq \sum_{\textbf{k}\in\N^d} e^{-\lambda_\textbf{k}\Delta_n(i+j-2)}\frac{\big(e^{-\lambda_\textbf{k}\Delta_n}-1\big)^2}{\lambda_\textbf{k}^{1+\alpha}}\E\Big[\big(\lambda_\textbf{k}^{(1+\alpha)/2}\langle\xi,e_\textbf{k}\rangle_\vartheta\big)^2\Big]e_\textbf{k}^2(\textbf{y})\notag \\
&\leq C\sup\limits_{\textbf{k}\in \N}\E\Big[\langle A_\vartheta^{(1+\alpha
)/2}\xi,e_\textbf{k}\rangle_\vartheta^2\Big]\sum_{\textbf{k}\in\N^d} e^{-\lambda_\textbf{k}\Delta_n(i+j-2)}\frac{\big(1-e^{-\lambda_\textbf{k}\Delta_n}\big)^2}{\lambda_\textbf{k}^{1+\alpha}}.\label{number32:1}
\end{align}
Furthermore, for $i<j$ we have
\begin{align*}
\Cov(\Delta_iX(\textbf{y}),\Delta_jX(\textbf{y}))&=\sum_{\textbf{k}\in\N^d} \big(\Sigma_{i,j}^{B,\textbf{k}}+\Sigma_{i,j}^{CB,\textbf{k}}\big)e_\textbf{k}^2(\textbf{y})+r_{i,j} \\
&=\sigma^2\sum_{\textbf{k}\in\N^d}  e_\textbf{k}^2(\textbf{y}) e^{-\lambda_\textbf{k}\Delta_n(j-i)} 
\frac{\big(1-e^{-\lambda_\textbf{k}\Delta_n}\big)^2 +\big(e^{\lambda_\textbf{k}\Delta_n}-e^{-\lambda_\textbf{k}\Delta_n}\big)\big(e^{-\lambda_\textbf{k}\Delta_n}-1\big)}{2\lambda_\textbf{k}^{1+\alpha}} \\
&~~~~ -\sigma^2\sum_{\textbf{k}\in\N^d}e_\textbf{k}^2(\textbf{y}) e^{-\lambda_\textbf{k}(i+j-2)\Delta_n}\frac{\big(1-e^{-\lambda_\textbf{k}\Delta_n}\big)^2 }{2\lambda_\textbf{k}^{1+\alpha}}+r_{i,j}.
\end{align*}
We define the second remainder as:
\begin{align*}
s_{i,j}:=-\sigma^2\sum\limits_{\textbf{k}\in \N^d}e_\textbf{k}^2(\textbf{y}) e^{-\lambda_\textbf{k}(i+j-2)\Delta_n}\frac{\big(1-e^{-\lambda_\textbf{k}\Delta_n}\big)^2 }{2\lambda_\textbf{k}^{1+\alpha}}.
\end{align*}
Using the identity $\sin^2(x)=(1-\cos(2x))/2$, we arrive at:
\begin{align*}
\Cov(\Delta_iX(\textbf{y}),\Delta_jX(\textbf{y}))
&= \sigma^2\sum\limits_{\textbf{k}\in \N^d}  e_\textbf{k}^2(\textbf{y}) e^{-\lambda_\textbf{k}\Delta_n(j-i)} 
\frac{2-e^{-\lambda_\textbf{k}\Delta_n}-e^{\lambda_\textbf{k}\Delta_n}}{2\lambda_\textbf{k}^{1+\alpha}}+ s_{i,j}+r_{i,j} \\
&=- \sigma^2\Delta_n^{1+\alpha}\sum\limits_{\textbf{k}\in\N^d}^\infty  e_\textbf{k}^2(\textbf{y}) e^{-\lambda_\textbf{k}\Delta_n(j-i-1)} 
\frac{\big(1-e^{-\lambda_\textbf{k}\Delta_n}\big)^2}{2(\lambda_\textbf{k}\Delta_n)^{1+\alpha}}+ s_{i,j}+r_{i,j}\\
&=-\sigma^2 e^{-\nor{\kappa\bigcdot \textbf{y}}_1}\Delta_n^{1+\alpha}\sum\limits_{\textbf{k}\in \N^d} e^{-\lambda_\textbf{k}\Delta_n(j-i-1)} 
\frac{\big(1-e^{-\lambda_\textbf{k}\Delta_n}\big)^2}{2(\lambda_\textbf{k}\Delta_n)^{1+\alpha}}
\prod_{l=1}^d\big(1-\cos(2\pi k_ly_l)\big)
+ s_{i,j}+r_{i,j}.
\end{align*}
By defining the following expression:
\begin{align*}
\mathcal{S}_{i,\textbf{k}}:=e^{-\lambda_\textbf{k}\Delta_n(j-i-1)} 
\frac{\big(1-e^{-\lambda_\textbf{k}\Delta_n}\big)^2}{2(\lambda_\textbf{k}\Delta_n)^{1+\alpha}}=g_{\alpha,(j-i-1)}(\lambda_\textbf{k}\Delta_n),
\end{align*}
we obtain that
\begin{align*}
&-\sigma^2 e^{-\nor{\kappa\bigcdot \textbf{y}}_1}\Delta_n^{1+\alpha}\sum\limits_{\textbf{k}\in \N^d} e^{-\lambda_\textbf{k}\Delta_n(j-i-1)} 
\frac{\big(1-e^{-\lambda_\textbf{k}\Delta_n}\big)^2}{2(\lambda_\textbf{k}\Delta_n)^{1+\alpha}}
\prod_{l=1}^d\big(1-\cos(2\pi k_ly_l)\big)\\
&=-\sigma^2 e^{-\nor{\kappa\bigcdot \textbf{y}}_1}\Delta_n^{1+\alpha}\sum_{\textbf{k}\in \N^d}\bigg(\mathcal{S}_{i,\textbf{k}}+\mathcal{S}_{i,\textbf{k}}\sum_{l=1}^d(-1)^l\sum_{1\leq j_1<\ldots<j_l\leq n}\cos(2\pi k_{j_1}y_{j_1})\cdots\cos(2\pi k_{j_l}y_{j_l})\bigg).
\end{align*}
Since we know by Lemma \ref{lemma_checkConditionsAprroxlemma} that $g_{\alpha,\tau}\in\mathcal{Q}_\beta$, with $\beta=\big(2\alpha,2(1+\alpha),2(1+\alpha)\big)$, we have with Lemma \ref{lemma_calcfAlphaDelta} and Corollary \ref{corollary_toLemmaRiemannApproxMulti} that
\begin{align*}
\Delta_n^{d/2}\sum_{\textbf{k}\in\N^d}g_{\alpha,\tau}(\lambda_\textbf{k}\Delta_n)&=\frac{\Gamma(1-\alpha')}{2^{d}(\pi\eta)^{d/2}\alpha'\Gamma(d/2)}\bigg(-\frac{1}{2}\tau^{\alpha'}+\big(1+\tau\big)^{\alpha'}-\frac{1}{2}\big(2+\tau\big)^{\alpha'}\bigg)\\
&~~~~~-\sum_{\substack{\nor{\gamma}_1=1\\ \gamma\in\{0,1\}^d }}^{d-1}\int_{B_\gamma}g_{\alpha,\tau}(\pi^2\eta\nor{\textbf{z}}_2^2)\diff \textbf{z}+\Oo(\Delta_n^{1-\alpha'}),
\end{align*}
and 
\begin{align*}
&\Delta_n^{d/2}\sum_{l=1}^d(-1)^l\sum_{1\leq j_1<\ldots<j_l\leq n}\sum_{\textbf{k}\in\N^d}g_{\alpha,\tau}(\lambda_\textbf{k}\Delta_n)\cos(2\pi k_{j_1}y_{j_1})\cdots\cos(2\pi k_{j_l}y_{j_l})\\
&=\sum_{\substack{\nor{\gamma}_1=1\\ \gamma\in\{0,1\}^d }}^{d-1}\int_{B_\gamma}g_{\alpha,\tau}(\pi^2\eta\nor{\textbf{z}}_2^2)\diff \textbf{z}+\Oo(\Delta_n^{1-\alpha'}).
\end{align*}
In line with Proposition \ref{prop_quadIncAndrescaling} and with $\tau=(j-i-1)$, we have 
\begin{align*}
&\Cov(\Delta_iX(\textbf{y}),\Delta_jX(\textbf{y}))\\&=-\sigma^2 e^{-\nor{\kappa\bigcdot \textbf{y}}_1}\Delta_n^{\alpha'}\frac{\Gamma(1-\alpha')}{2^{d}(\pi\eta)^{d/2}\alpha'\Gamma(d/2)}\bigg(-\frac{1}{2}\big(j-i-1\big)^{\alpha'}+\big(j-i\big)^{\alpha'}-\frac{1}{2}\big(j-i+1\big)^{\alpha'}\bigg)\\
&~~~~~+ s_{i,j}+r_{i,j}+\Oo(\Delta_n).
\end{align*}
It remains to show that $\sum_{i,j=1}^n(s_{i,j}+r_{i,j})=\Oo(1)$. Therefore, we use display \eqref{number32:1} and obtain:
\begin{align*}
\sum\limits_{i,j=1}^n(s_{i,j}+r_{i,j})
&\leq C\bigg(\sigma^2+\sup\limits_{\textbf{k}\in \N}\E\Big[\langle A_\vartheta^{(1+\alpha
)/2}\xi,e_\textbf{k}\rangle_\vartheta^2\Big]\bigg)\sum\limits_{i,j=1}^n\sum_{\textbf{k}\in\N^d} e^{-\lambda_\textbf{k}\Delta_n(i+j-2)}\frac{\big(1-e^{-\lambda_\textbf{k}\Delta_n}\big)^2}{\lambda_\textbf{k}^{1+\alpha}} \\
&= C\bigg(\sigma^2+\sup\limits_{\textbf{k}\in \N}\E\Big[\langle A_\vartheta^{(1+\alpha
)/2}\xi,e_\textbf{k}\rangle_\vartheta^2\Big]\bigg)\sum_{\textbf{k}\in\N^d} \frac{\big(1-e^{-\lambda_\textbf{k}\Delta_n}\big)^2}{\lambda_\textbf{k}^{1+\alpha}}\bigg(\frac{1-e^{-\lambda_\textbf{k}n\Delta_n}}{1-e^{-\lambda_\textbf{k}\Delta_n}}\bigg)^2\\
&\leq C\bigg(\sigma^2+\sup\limits_{\textbf{k}\in \N}\E\Big[\langle A_\vartheta^{(1+\alpha
)/2}\xi,e_\textbf{k}\rangle_\vartheta^2\Big]\bigg)\sum_{\textbf{k}\in\N^d} \frac{1}{\lambda_\textbf{k}^{1+\alpha}}=\Oo(1).
\end{align*}
Analogous computations for $i>j$ complete the proof.
\end{proof}
\subsubsection{Proof of the central limit theorem from Proposition \ref{prop_cltVolaEstMulti}}
We begin this section by decomposing a temporal increment of a mild solution $\tilde{X}_t$ with a stationary initial condition $\langle\xi,e_\textbf{k}\rangle_\vartheta\sim\mathcal{N}(0,\sigma^2/(2\lambda_\textbf{k}^{1+\alpha}))$. Analogously to \cite{trabs}, we decompose the coordinate processes as follows: 
\begin{align}
\Delta_i\tilde{x}_\textbf{k} &= \tilde{A}_{i,\textbf{k}}+B_{i,\textbf{k}}+C_{i,\textbf{k}}\label{eqn_procedureForstatInitCond}\\
&=\sigma\lambda_\textbf{k}^{-\alpha/2}\int\limits_{-\infty}^0  e^{-\lambda_\textbf{k}\big(\iidn-s\big)}\big(e^{-\lambda_\textbf{k}\Delta_n}-1\big) \diff W_s^\textbf{k}\notag \\
&~~~~~ +\sigma\lambda_\textbf{k}^{-\alpha/2}\int\limits_0^{\iidn} e^{-\lambda_\textbf{k}\big(\iidn-s\big)}\big(e^{-\lambda_\textbf{k}\Delta_n}-1\big) \diff W_s^\textbf{k} +C_{i,\textbf{k}}\notag\\
&= \tilde{B}_{i,\textbf{k}}+C_{i,\textbf{k}},\notag
\end{align}
where 
\begin{align}
\tilde{B}_{i,\textbf{k}}&:=\sigma\lambda_\textbf{k}^{-\alpha/2}\int\limits_{-\infty}^{\iidn}  e^{-\lambda_\textbf{k}\big(\iidn-s\big)}\big(e^{-\lambda_\textbf{k}\Delta_n}-1\big) \diff W_s^\textbf{k},\label{eqn_tildeB_multi}\\
C_{i,\textbf{k}}&=\sigma\lambda_\textbf{k}^{-\alpha/2}\int_{(i-1)\Delta_n}^{i\Delta_n}e^{-\lambda_\textbf{k}(i\Delta_n-s)}\diff W_s^\textbf{k}.
\end{align}
Thus, $(\Delta_i\tilde{X})(\textbf{y})$ is centred, Gaussian and stationary.

To prove a central limit theorem for the volatility estimators, utilize the following theorem.
\begin{theorem}\label{prop_clt_utev}
Let $(Z_{k_n,i})_{1\leq i\leq k_n}$ a centred triangular array, with a sequence $(k_n)_{n\in\N}$. Then it holds:
\begin{align*}
\sum_{i=1}^{k_n} Z_{k_n,i}\overset{d}{\longrightarrow} \mathcal{N}(0,\upsilon^2), 
\end{align*}
with $\upsilon^2=\lim_{n\tooi}\Var(\sum_{i=1}^{k_n}Z_{k_n,i})<\infty$ if the following conditions hold:
\begin{itemize}
\item[(I)] $\Var\Big(\sum\limits_{i=a}^bZ_{k_n,i}\Big)\leq C\sum\limits_{i=a}^b\Var(Z_{k_n,i})$, for all $1\leq a\leq b\leq k_n$,
\item[(II)] $\limsup\limits_{n\tooi}\sum\limits_{i=1}^{k_n}\E[Z_{k_n,i}^2]<\infty$,
\item[(III)] $\sum\limits_{i=1}^{k_n}\E\Big[Z_{k_n,i}^2\mathbbm{1}_{\{|Z_{k_n,i}|>\varepsilon\}}\Big]\overset{n\tooi}{\longrightarrow}0$, for all $\varepsilon>0$,
\item[(IV)] $\Cov\Big(e^{\im t\sum_{i=a}^bZ_{k_n,i}},e^{\im t\sum_{i=b+u}^c Z_{k_n,i}}\Big)\leq \rho_t(u)\sum\limits_{i=a}^c\Var(Z_{k_n,i})$, for all $1\leq a\leq b< b+u\leq c\leq k_n$ and $t\in\R$,
\end{itemize}
where $C>0$ is a universal constant and $\rho_t(u)\geq 0$ is a function with $\sum_{j=1}^\infty\rho_t(2^j)<\infty$. 
\end{theorem}

The associated preliminary triangular arrays for the volatility estimator from \eqref{sigma_estimator} is defined as follows:
\begin{align*}
\xi_{n,i}:=\frac{2^d(\pi\eta)^{d/2}\alpha'\Gamma(d/2)}{\sqrt{nm}\Delta_n^{\alpha'}\Gamma(1-\alpha')}\sum_{j=1}^m(\Delta_iX)^2(\textbf{y}_j)e^{\nor{\kappa\bigcdot \textbf{y}_j}_1}.
\end{align*}
In the following lemma, we proof that working with triangular arrays based on a SPDEs with a stationary initial condition, i.e.:
\begin{align}
\tilde{\xi}_{n,i}:=\frac{2^d(\pi\eta)^{d/2}\alpha'\Gamma(d/2)}{\sqrt{nm}\Delta_n^{\alpha'}\Gamma(1-\alpha')}\sum_{j=1}^m(\Delta_i\tilde{X})^2(\textbf{y}_j)e^{\nor{\kappa\bigcdot \textbf{y}_j}_1},\label{eqn_triangularArrayZetaMulti}
\end{align}
is sufficient.

\begin{lemma}\label{lemma74}
On Assumptions \ref{assumption_observations_multi} und \ref{assumption_regMulti}, it holds that
\begin{align*}
\sqrt{m}_n\sum\limits_{i=1}^n\Big((\Delta_i\tilde{X})^2(\textbf{y})-(\Delta_iX)^2(\textbf{y})\Big)\overset{\Pp}{\too}0,
\end{align*}
for $n\tooi$.
\end{lemma}
\begin{proof}
We initiate the proof with the following:
\begin{align*}
(\Delta_i\tilde{X})^2(\textbf{y})-(\Delta_iX)^2(\textbf{y})
&=\sum\limits_{\textbf{k}_1,\textbf{k}_2\in\N} \Big(\Delta_i\tilde{x}_{\textbf{k}_1}\Delta_i\tilde{x}_{\textbf{k}_2}-\Delta_ix_{\textbf{k}_1}\Delta_ix_{\textbf{k}_2} \Big)e_{\textbf{k}_1}(\textbf{y})e_{\textbf{k}_2}(\textbf{y})= \tilde{T}_i-T_i,
\end{align*}
where we define:
\begin{align*}
\tilde{T}_i &:=\sum\limits_{\textbf{k}_1,\textbf{k}_2\in\N^d} \Big(\tilde{A}_{i,\textbf{k}_1}\tilde{A}_{i,\textbf{k}_2}+\tilde{A}_{i,\textbf{k}_1}\big(B_{i,\textbf{k}_2}+C_{i,\textbf{k}_2}\big)+\tilde{A}_{i,\textbf{k}_2}\big(B_{i,\textbf{k}_1}+C_{i,\textbf{k}_1}\big)\Big)e_{\textbf{k}_1}(\textbf{y})e_{\textbf{k}_2}(\textbf{y}), \\
T_i&:=\sum\limits_{\textbf{k}_1,\textbf{k}_2\in\N^d} \Big(A_{i,\textbf{k}_1}A_{i,\textbf{k}_2}+A_{i,\textbf{k}_1}\big(B_{i,\textbf{k}_2}+C_{i,\textbf{k}_2}\big)+A_{i,\textbf{k}_2}\big(B_{i,\textbf{k}_1}+C_{i,\textbf{k}_1}\big)\Big)e_{\textbf{k}_1}(\textbf{y})e_{\textbf{k}_2}(\textbf{y}).
\end{align*}
It remains to show that $\sqrt{m}_n\sum_{i=1}^nT_i\overset{\Pp}{\too}0$, since this implies $\sqrt{m}_n\sum_{i=1}^n\tilde{T}_i\overset{\Pp}{\too}0$. Here, we have the following:
\begin{align}
\sum\limits_{i=1}^nT_i=\sum\limits_{i=1}^n\bigg(\sum\limits_{\textbf{k}\in\N^d} A_{i,\textbf{k}}e_\textbf{k}(\textbf{y})\bigg)^2+2\sum\limits_{i=1}^n\bigg(\sum\limits_{\textbf{k}\in\N^d} A_{i,\textbf{k}}e_\textbf{k}(\textbf{y})\bigg)\bigg(\sum\limits_{\textbf{k}\in\N^d}\Big(B_{i,\textbf{k}}+C_{i,\textbf{k}}\Big)e_\textbf{k}(\textbf{y})\bigg).\label{eqn_sumTiConvergenceinProp}
\end{align}
Using Hölder's inequality we obtain:
\begin{align*}
\E\Bigg[\sum\limits_{i=1}^n\bigg(\sum\limits_{\textbf{k}\in\N^d} A_{i,\textbf{k}}e_\textbf{k}(\textbf{y})\bigg)^2\Bigg]&=  \E\Bigg[\sum\limits_{i=1}^n\sum\limits_{\textbf{k}\in\N^d} A_{i,\textbf{k}}^2e_\textbf{k}^2(\textbf{y})\Bigg]+\E\Bigg[\sum\limits_{i=1}^n\sum\limits_{\substack{ \textbf{k}_1,\textbf{k}_2\in\N^d\\ \textbf{k}_1\neq \textbf{k}_2}} A_{i,\textbf{k}_1}A_{i,\textbf{k}_2}e_{\textbf{k}_1}(\textbf{y})e_{\textbf{k}_2}(\textbf{y})\Bigg] \\
&\leq C \E\Bigg[\sum\limits_{i=1}^n\sum\limits_{\textbf{k}\in\N^d} A_{i,\textbf{k}}^2\Bigg]+\E\Bigg[\bigg|\sum\limits_{i=1}^n\sum\limits_{\substack{ \textbf{k}_1,\textbf{k}_2\in\N^d\\ \textbf{k}_1\neq \textbf{k}_2}} A_{i,\textbf{k}_1}A_{i,\textbf{k}_2}e_{\textbf{k}_1}(\textbf{y})e_{\textbf{k}_2}(\textbf{y})\bigg|^2\Bigg]^{1/2},
\end{align*}
where $C>0$ is a suitable constant. Let $C_{\xi}:=\sup_{\textbf{k}\in\N^d}\lambda_\textbf{k}^{1+\alpha}\E[\langle\xi,e_\textbf{k}\rangle_\vartheta^2]$. With analogous steps as in Lemma \ref{lemma_quadIncrementsFirstCalc}, we find:
\begin{align}
 \sum\limits_{i=1}^n\sum\limits_{\textbf{k}\in\N^d}\E\big[ A_{i,\textbf{k}}^2\big]&= \sum\limits_{i=1}^n\sum\limits_{\textbf{k}\in\N^d} \big(e^{-\lambda_\textbf{k}\idn}-e^{-\lambda_\textbf{k}\iidn}\big)^2\E\Big[\langle \xi,e_\textbf{k}\rangle^2_\vartheta\Big]\notag \\
 &\leq C_\xi\sum\limits_{\textbf{k}\in\N^d} \frac{\big(1-e^{-\lambda_\textbf{k}\Delta_n}\big)^2}	{\lambda_\textbf{k}^{1+\alpha}}\sum\limits_{i=1}^n e^{-2\lambda_\textbf{k}\iidn}\notag \\
 &\leq C_\xi\sum\limits_{\textbf{k}\in\N^d} \frac{1-e^{-\lambda_\textbf{k}\Delta_n}}{\lambda_\textbf{k}^{1+\alpha}}=\mathcal{O}\big(\Delta_n^{\alpha'}\big).\label{eqn_boundAA}
\end{align}
Furthermore, we have
\begin{align*}
&\E\Bigg[\bigg|\sum\limits_{i=1}^n\sum\limits_{\substack{ \textbf{k}_1,\textbf{k}_2\in\N^d\\ \textbf{k}_1\neq \textbf{k}_2}} A_{i,\textbf{k}_1}A_{i,\textbf{k}_2}e_{\textbf{k}_1}(\textbf{y})e_{\textbf{k}_2}(\textbf{y})\bigg|^2\Bigg]\\
&= \sum\limits_{i,j=1}^n\sum\limits_{\substack{ \textbf{k}_1,\textbf{k}_2\in\N^d\\ \textbf{k}_1\neq \textbf{k}_2}}\sum\limits_{\substack{ \textbf{k}_3,\textbf{k}_4\in\N^d\\ \textbf{k}_3\neq \textbf{k}_4}} e_{\textbf{k}_1}(\textbf{y})e_{\textbf{k}_2}(\textbf{y})e_{\textbf{k}_3}(\textbf{y})e_{\textbf{k}_4}(\textbf{y})\E\Big[A_{i,\textbf{k}_1}A_{i,\textbf{k}_2}A_{j,\textbf{k}_3}A_{j,\textbf{k}_4}\Big].
\end{align*}
Let us assume that $\E[\langle \xi,e_\textbf{k}\rangle_\vartheta]=0$ from Assumption \ref{assumption_regMulti}. Then, for $\textbf{k}_1=\textbf{k}_3$ and $\textbf{k}_2=\textbf{k}_4$, we have 
\begin{align*}
\sum\limits_{i,j=1}^n\sum\limits_{\substack{ \textbf{k}_1,\textbf{k}_2\in\N^d\\ \textbf{k}_1\neq \textbf{k}_2}} \E\big[A_{i,\textbf{k}_1}A_{i,\textbf{k}_2}A_{j,\textbf{k}_1}&A_{j,\textbf{k}_2}\big]  
=\sum\limits_{i,j=1}^n\sum\limits_{\substack{ \textbf{k}_1,\textbf{k}_2\in\N^d\\ \textbf{k}_1\neq \textbf{k}_2}} \big(1-e^{-\lambda_{\textbf{k}_1}\Delta_n}\big)^2\big(1-e^{-\lambda_{\textbf{k}_2}\Delta_n}\big)^2\\
&~~~~~\times e^{-\lambda_{\textbf{k}_1}(i+j-2)\Delta_n-\lambda_{\textbf{k}_2}(i+j-2)\Delta_n}\E\Big[\langle \xi,e_{\textbf{k}_1}\rangle^2_\vartheta\Big] \E\Big[\langle \xi,e_{\textbf{k}_2}\rangle^2_\vartheta\Big] \\
&\leq C_\xi^2\sum\limits_{\substack{ \textbf{k}_1,\textbf{k}_2\in\N^d\\ \textbf{k}_1\neq \textbf{k}_2}} \frac{\big(1-e^{-\lambda_{\textbf{k}_1}\Delta_n}\big)^2\big(1-e^{-\lambda_{\textbf{k}_2}\Delta_n}\big)^2}{\lambda_{\textbf{k}_1}^{1+\alpha}\lambda_{\textbf{k}_2}^{1+\alpha}}\sum\limits_{i,j=1}^ne^{-(\lambda_{\textbf{k}_1}+\lambda_{\textbf{k}_2})(i+j-2)\Delta_n},
\end{align*}
where the case $\textbf{k}_1=\textbf{k}_4$ and $\textbf{k}_2=\textbf{k}_3$ works analogously. By using the geometric series, we obtain:
\begin{align*}
\sum\limits_{i,j=1}^ne^{-(\lambda_{\textbf{k}_1}+\lambda_{\textbf{k}_2})(i+j-2)\Delta_n}
&\leq \frac{1}{(1-e^{-(\lambda_{\textbf{k}_1}+\lambda
_{\textbf{k}_2})\Delta_n})^2},
\end{align*}
and therefore, we have 
\begin{align*}
\sum\limits_{i,j=1}^n\sum\limits_{\substack{ \textbf{k}_1,\textbf{k}_2\in\N^d\\ \textbf{k}_1\neq \textbf{k}_2}} \E\big[A_{i,\textbf{k}_1}A_{i,\textbf{k}_2}A_{j,\textbf{k}_1}A_{j,\textbf{k}_2}\big] &\leq C_\xi^2\sum\limits_{\textbf{k}_1,\textbf{k}_2\in\N^d} \frac{\big(1-e^{-\lambda_{\textbf{k}_1}\Delta_n}\big)^2\big(1-e^{-\lambda_{\textbf{k}_2}\Delta_n}\big)^2}{\lambda_{\textbf{k}_1}^{1+\alpha}\lambda_{\textbf{k}_2}^{1+\alpha} \big(1-e^{-(\lambda_{\textbf{k}_1}+\lambda_{\textbf{k}_2})\Delta_n}\big)^2} \\
&\leq C_\xi^2\sum\limits_{\textbf{k}_1,\textbf{k}_2\in\N^d} \frac{\big(1-e^{-\lambda_{\textbf{k}_1}\Delta_n}\big)\big(1-e^{-\lambda_{\textbf{k}_2}\Delta_n}\big)}{\lambda_{\textbf{k}_1}^{1+\alpha}\lambda_{\textbf{k}_2}^{1+\alpha} } 
=\mathcal{O}(\Delta_n^{2\alpha'}),
\end{align*}
where we have used $(1-p)(1-q)/(1-pq)\leq 1-p$, for $0\leq p,q<1$.
For the second option in Assumption \ref{assumption_regMulti}, we use an analogous procedure as in Lemma \ref{lemma_quadIncrementsFirstCalc}. Here, we have with $C_\xi':=\sum_{\textbf{k}\in\N}\lambda_\textbf{k}^{1+\alpha}\E[\langle\xi,e_\textbf{k}\rangle_\vartheta^2]<\infty$ and Parseval's identity that
\begin{align*}
\sum\limits_{i,j=1}^n\sum\limits_{\substack{ \textbf{k}_1,\textbf{k}_2\in\N^d\\ \textbf{k}_1\neq \textbf{k}_2}}\sum\limits_{\substack{ \textbf{k}_3,\textbf{k}_4\in\N^d\\ \textbf{k}_3\neq \textbf{k}_4}} &\E\big[A_{i,\textbf{k}_1}A_{i,\textbf{k}_2}A_{j,\textbf{k}_3}A_{j,\textbf{k}_4}\big]  \leq\Bigg(\sum\limits_{i=1}^n\bigg(\sum\limits_{\textbf{k}\in\N^d} \E[A_{i,\textbf{k}}] \bigg)^2\Bigg)^2 \\
&\leq\Bigg(\sum\limits_{i=1}^n\bigg(\sum\limits_{\textbf{k}\in\N^d} \frac{\big(e^{-\lambda_\textbf{k}\Delta_n}-1\big)e^{-\lambda_\textbf{k}\iidn}}{\lambda_\textbf{k}^{(1+\alpha)/2}}\lambda_\textbf{k}^{(1+\alpha)/2}\E\Big[\langle \xi,e_\textbf{k}\rangle_\vartheta^2\Big] ^{1/2} \bigg)^2\Bigg)^2\\
&\leq \Bigg(\sum\limits_{i=1}^n\bigg(\sum\limits_{\textbf{k}\in\N^d} \frac{\big(1-e^{-\lambda_\textbf{k}\Delta_n}\big)^2e^{-2\lambda_\textbf{k}\iidn}}{\lambda_\textbf{k}^{1+\alpha}}\bigg)\bigg(\sum\limits_{\textbf{k}\in\N^d}\lambda_\textbf{k}^{1+\alpha}\E\Big[\langle \xi,e_\textbf{k}\rangle_\vartheta^2\Big]  \bigg)\Bigg)^2 \\
&\leq C_\xi'^2\Bigg(\sum\limits_{\textbf{k}\in\N^d}\frac{(1-e^{-\lambda_\textbf{k}\Delta_n})^2}{\lambda_\textbf{k}^{1+\alpha}(1-e^{-2\lambda_\textbf{k}\Delta_n})}\Bigg)^2=\mathcal{O}(\Delta_n^{2\alpha'}).
\end{align*}
By using Markov's inequality, we conclude with
\begin{align*}
\sum\limits_{i=1}^n\bigg(\sum\limits_{\textbf{k}\in\N^d} A_{i,\textbf{k}}e_\textbf{k}(\textbf{y})\bigg)^2 = \mathcal{O}_\Pp\big(\Delta_n^{\alpha'}\big).
\end{align*}
Continuing, we proceed to bound the following term:
\begin{align*}
2\sum\limits_{i=1}^n\bigg(\sum\limits_{\textbf{k}\in\N^d} A_{i,\textbf{k}}e_\textbf{k}(\textbf{y})\bigg)\bigg(\sum\limits_{\textbf{k}\in\N^d}\Big(B_{i,\textbf{k}}+C_{i,\textbf{k}}\Big)e_\textbf{k}(\textbf{y})\bigg).
\end{align*}
We can make use of the independence of $A_{i,\textbf{k}},B_{i,\textbf{k}}$ and $C_{i,\textbf{k}}$ to show:
\begin{align*}
&\E\Bigg[\bigg|\sum\limits_{i=1}^n\bigg(\sum\limits_{\textbf{k}\in\N^d} A_{i,\textbf{k}}e_\textbf{k}(\textbf{y})\bigg)\bigg(\sum\limits_{\textbf{k}\in\N^d}\Big(B_{i,\textbf{k}}+C_{i,\textbf{k}}\Big)e_\textbf{k}(\textbf{y})\bigg)\bigg|^2\Bigg]\\
&=\sum\limits_{i,j=1}^n\Bigg(\sum\limits_{\textbf{k}_1,\textbf{k}_2\in\N^d} \E\big[A_{i,\textbf{k}_1} A_{j,\textbf{k}_2}\big]e_{\textbf{k}_1}(\textbf{y})e_{\textbf{k}_2}(\textbf{y})  \Bigg)
\Bigg(\sum\limits_{\textbf{k}\in\N^d}\E\Big[\big(B_{i,\textbf{k}}+C_{i,\textbf{k}}\big)\big(B_{j,\textbf{k}}+C_{j,\textbf{k}}\big)\Big]e_\textbf{k}^2(\textbf{y})\Bigg) =:\sum\limits_{i,j=1}^nR_{i,j}S_{i,j},
\end{align*}
where 
\begin{align*}
R_{i,j}&:=\sum\limits_{\textbf{k}_1,\textbf{k}_2\in\N^d}^\infty \E\big[A_{i,\textbf{k}_1} A_{j,\textbf{k}_2}\big]e_{\textbf{k}_1}(\textbf{y})e_{\textbf{k}_2}(\textbf{y}),  \\
S_{i,j}&:=\sum\limits_{\textbf{k}\in\N^d}\E\Big[\big(B_{i,\textbf{k}}+C_{i,\textbf{k}}\big)\big(B_{j,\textbf{k}}+C_{j,\textbf{k}}\big)\Big]e_\textbf{k}^2(\textbf{y}).
\end{align*}
Assuming the first option in Assumption \ref{assumption_regMulti} holds, we can analogously obtain, as in equation \eqref{eqn_boundAA}, that
\begin{align*}
R_{i,j}
&= \sum\limits_{\textbf{k}\in\N^d} \E\big[A_{i,\textbf{k}} A_{j,\textbf{k}}\big]e_\textbf{k}^2(\textbf{y})\leq CC_\xi\sum\limits_{\textbf{k}\in\N^d} \frac{(1-e^{-\lambda_\textbf{k}\Delta_n})^2}{\lambda_\textbf{k}^{1+\alpha}}e^{-\lambda_\textbf{k}(i+j-2)\Delta_n}  =\mathcal{O}\big(\Delta_n^{\alpha'}\big),
\end{align*}
and therefore it holds that $\sum_{i,j=1}^nR_{i,j}=\mathcal{O}(\Delta_n^{\alpha'})$ and $\sup_{i,j=1,\ldots,n} \abs{R_{i,j}}=\mathcal{O}(\Delta_n^{\alpha'})$ as well as\linebreak $\sup_{j=1,\ldots,n}\sum_{i=1}^n \abs{R_{i,j}}=\mathcal{O}(\Delta_n^{\alpha'})$. For the second option in Assumption \ref{assumption_regMulti}, we find:
\begin{align*}
R_{i,j}&\leq C\sum\limits_{\textbf{k}_1,\textbf{k}_2\in\N^d} \frac{\big(1-e^{-\lambda_{\textbf{k}_1}\Delta_n}\big)\big(1-e^{-\lambda_{\textbf{k}_2}\Delta_n}\big)}{\lambda_{\textbf{k}_1}^{(1+\alpha)/2}\lambda_{\textbf{k}_2}^{(1+\alpha)/2}}e^{-\big(\lambda_{\textbf{k}_1}(i-1)+\lambda_{\textbf{k}_2}(j-1)\big)\Delta_n}\\
&~~~~~\times\lambda_{\textbf{k}_1}^{(1+\alpha)/2}\E\Big[\abs{\langle \xi,e_{\textbf{k}_1}\rangle_\vartheta} \Big] \lambda_{\textbf{k}_2}^{(1+\alpha)/2}\E\Big[\abs{\langle \xi,e_{\textbf{k}_2}\rangle_\vartheta} \Big] \\
&\leq C\sum\limits_{\textbf{k}_1,\textbf{k}_2\in\N^d} \frac{\big(1-e^{-\lambda_{\textbf{k}_1}\Delta_n}\big)\big(1-e^{-\lambda_{\textbf{k}_2}\Delta_n}\big)}{\lambda_{\textbf{k}_1}^{(1+\alpha)/2}\lambda_{\textbf{k}_2}^{(1+\alpha)/2}}e^{-\big(\lambda_{\textbf{k}_1}(i-1)+\lambda_{\textbf{k}_2}(j-1)\big)\Delta_n}\\
&~~~~~\times\lambda_{\textbf{k}_1}^{(1+\alpha)/2}\E\Big[\abs{\langle \xi,e_{\textbf{k}_1}\rangle_\vartheta}^2 \Big]^{1/2} \lambda_{\textbf{k}_2}^{(1+\alpha)/2}\E\Big[\abs{\langle \xi,e_{\textbf{k}_2}\rangle_\vartheta}^2 \Big]^{1/2}\\
&=C\sum\limits_{\textbf{k}_1\in\N^d} \frac{\big(1-e^{-\lambda_{\textbf{k}_1}\Delta_n}\big)}{\lambda_{\textbf{k}_1}^{(1+\alpha)/2}}e^{-\lambda_{\textbf{k}_1}(i-1)\Delta_n}\lambda_{\textbf{k}_1}^{(1+\alpha)/2}\E\Big[\abs{\langle \xi,e_{\textbf{k}_1}\rangle_\vartheta} ^2\Big]^{1/2}\sum\limits_{\textbf{k}_2\in\N^d}\frac{\big(1-e^{-\lambda_{\textbf{k}_2}\Delta_n}\big)}{\lambda_{\textbf{k}_2}^{(1+\alpha)/2}}\\
&~~~~~~~\times e^{-\lambda_{\textbf{k}_2}(j-1)\Delta_n}\lambda_{\textbf{k}_2}^{(1+\alpha)/2}\E\Big[\abs{\langle \xi,e_{\textbf{k}_2}\rangle_\vartheta}^2 \Big]^{1/2}\\
&\leq CC_\xi'\Bigg(\sum\limits_{\textbf{k}_1\in\N^d} \frac{\big(1-e^{-\lambda_{\textbf{k}_1}\Delta_n}\big)^2}{\lambda_{\textbf{k}_1}^{1+\alpha}}e^{-2\lambda_{\textbf{k}_1}(i-1)\Delta_n}\sum\limits_{\textbf{k}_2\in\N^d}\frac{\big(1-e^{-\lambda_{\textbf{k}_2}\Delta_n}\big)^2}{\lambda_{\textbf{k}_2}^{1+\alpha}}e^{-2\lambda_{\textbf{k}_2}(j-1)\Delta_n}\bigg)^{1/2},
\end{align*}
and therefore, we have
\begin{align*}
\sum_{i,j=1}^nR_{i,j}&\leq CC_\xi'\sum_{\textbf{k}\in\N^d}\frac{1-e^{-\lambda_{\textbf{k}}\Delta_n}}{\lambda_{\textbf{k}}^{1+\alpha}}=\Oo\big(\Delta_n^{\alpha'}\big).
\end{align*}
Thus, we infer for both options in Assumption \ref{assumption_regMulti}, that $\sup_{i,j} \abs{R_{i,j}}=\mathcal{O}(\Delta_n^{\alpha'})$ and $\sup_j\sum_{i=1}^n \abs{R_{i,j}}=\mathcal{O}(\Delta_n^{\alpha'})$. For the term $S_{i,j}$, we obtain:
\begin{align*}
S_{i,j}&=\sum\limits_{\textbf{k}\in\N^d}\E\Big[\big(B_{i,\textbf{k}}+C_{i,\textbf{k}}\big)\big(B_{j,\textbf{k}}+C_{j,\textbf{k}}\big)\Big]e_\textbf{k}^2(\textbf{y})\\
&=\sum\limits_{\textbf{k}\in\N^d}\Big(\Sigma_{i,j}^{B,\textbf{k}}+\Sigma_{i,j}^{BC,\textbf{k}}+\Sigma_{j,i}^{BC,\textbf{k}}+\Sigma_{i,j}^{C,\textbf{k}}\Big)e_\textbf{k}^2(\textbf{y}),
\end{align*}
where we used the notation of the proof of Proposition \ref{prop_autocovOfIncrementsMulti}, where 
\begin{align*}
\Sigma_{i,j}^{B,\textbf{k}}:=\Cov(B_{i,\textbf{k}},B_{j,\textbf{k}}),~~~~~~~~\Sigma_{i,j}^{BC,\textbf{k}}:=\Cov(B_{i,\textbf{k}},C_{j,\textbf{k}}),~~~~~~~~\Sigma_{i,j}^{C,\textbf{k}}:=\Cov(C_{i,\textbf{k}},C_{j,\textbf{k}}).
\end{align*} 
Upon inserting the calculations of Proposition \ref{prop_autocovOfIncrementsMulti}, we infer for $i<j$ that
\begin{align*}
S_{i,j}&=\sum\limits_{\textbf{k}\in\N^d}\Big(\Sigma_{i,j}^{B,\textbf{k}}+\Sigma_{j,i}^{BC,\textbf{k}}\Big)e_\textbf{k}^2(\textbf{y}) \\
&\leq -\sigma^2 e^{-\nor{\kappa\bigcdot \textbf{y}}_1}\Delta_n^{\alpha'}\frac{\Gamma(1-\alpha')}{2^{d}(\pi\eta)^{d/2}\alpha'\Gamma(d/2)}\bigg(-\frac{1}{2}\big(j-i-1\big)^{\alpha'}+\big(j-i\big)^{\alpha'}-\frac{1}{2}\big(j-i+1\big)^{\alpha'}\bigg) \\
&~~~~~+C\sigma^2\sum_{\textbf{k}\in\N^d} e^{-\lambda_\textbf{k}(i+j-2)\Delta_n}\frac{\big(1-e^{-\lambda_\textbf{k}\Delta_n}\big)^2 }{\lambda_\textbf{k}^{1+\alpha}}+\mathcal{O}(\Delta_n).
\end{align*}
For $i=j$ we obtain:
\begin{align*}
S_{i,i}&=\sum\limits_{\textbf{k}\in\N^d}\Big(\Sigma_{ii}^{B,\textbf{k}}+\Sigma_{ii}^{C,\textbf{k}}\Big)e_\textbf{k}^2(\textbf{y}) 
\leq C\sigma^2\sum\limits_{\textbf{k}\in\N^d} \bigg( \frac{\big(1-e^{-\lambda_\textbf{k}\Delta_n}\big)^2 }{2\lambda_\textbf{k}^{1+\alpha}}+\frac{1-e^{-2\lambda_\textbf{k}\Delta_n}}{2\lambda_\textbf{k}^{1+\alpha}} \bigg)=C\sigma^2\sum\limits_{\textbf{k}\in\N^d} \frac{1-e^{-\lambda_\textbf{k}\Delta_n} }{\lambda_\textbf{k}^{1+\alpha}} .
\end{align*}
Utilizing equation \eqref{eqn_orderOfiTermsGeometrricSum} we find that
\begin{align*}
\sum\limits_{i,j=1}^nR_{i,j}S_{i,j}&\leq C\sum\limits_{i,j=1}^n\bigg(\sum_{\textbf{k}\in\N^d}\frac{(1-e^{-\lambda_\textbf{k}\Delta_n})^2}{\lambda_\textbf{k}^{1+\alpha}}e^{-\lambda_\textbf{k}(i+j-2)\Delta_n}\bigg)\bigg(\mathbbm{1}_{\{i\neq j\}}\Delta_n^{\alpha'}\abs{i-j}^{\alpha'-2} \\
&~~~~~~~~~~+\sum\limits_{\textbf{k}\in\N^d} \Big( \frac{\big(1-e^{-\lambda_\textbf{k}\Delta_n}\big)^2 }{\lambda_\textbf{k}^{1+\alpha}}e^{-\lambda_\textbf{k}(i+j-2)\Delta_n}+\mathbbm{1}_{\{i=j\}}\frac{1-e^{-\lambda_\textbf{k}\Delta_n}}{\lambda_\textbf{k}^{1+\alpha}} \Big)+\Oo(\Delta_n)\bigg)\\
&=\mathcal{O}\bigg(\Delta_n^{2\alpha'}\sum\limits_{j=1}^\infty j^{\alpha'-2}+\Delta_n^{2\alpha'}\bigg)=\mathcal{O}(\Delta_n^{2\alpha'}),
\end{align*}
where $C>0$ is a suitable constant.
From the analysis above, we find that both terms in display \eqref{eqn_sumTiConvergenceinProp} are of order $\Oo_\Pp(\Delta_n^{\alpha'})$. Therefore, we conclude that $\sqrt{m}_n\sum_{i=1}^nT_i\overset{\Pp}{\too}0$, which completes the proof.
\end{proof}
Thanks to the previous lemma, we have 
\begin{align*}
\sum\limits_{i=1}^n\big(\tilde{\xi}_{n,i}-\xi_{n,i}\big)\overset{\Pp}{\too}0,
\end{align*}
as $n\tooi$, which allows us to investigate a mild solution under a stationary condition from now on.

We follow up by investigating the variance-covariance structure of the following term:
\begin{align}
V_{p,\Delta_n}(\textbf{y}):=\frac{1}{p\Delta_n^{\alpha'}}\sum_{i=1}^{p}(\Delta_i\tilde{X})^2(\textbf{y})e^{\nor{\kappa\bigcdot \textbf{y}}_1},\label{eqn_rescaledRealizedVolaMulti}
\end{align}
for $\textbf{y}\in[\delta,1-\delta]^d$. We refer to this expression as rescaled realized volatility.

\begin{theorem}\label{prop_RRVMulti}
On the Assumptions \ref{assumption_observations_multi} and \ref{assumption_regMulti}, we have for the rescaled realized volatility in two spacial coordinates $\textbf{y}_1,\textbf{y}_2\in[\delta,1-\delta]^d$ that
\begin{align*}
\Cov\big(V_{p,\Delta_n}(\textbf{y}_1),V_{p,\Delta_n}(\textbf{y}_2)\big)&=\mathbbm{1}_{\{\textbf{y}_1=\textbf{y}_2\}}\frac{\Upsilon_{\alpha'}}{p}\bigg(\frac{\Gamma(1-\alpha')\sigma^2}{2^d(\pi \eta)^{d/2}\alpha'\Gamma(d/2)}\bigg)^2\Bigg(1+\mathcal{O}\bigg(\Delta_n^{1/2}\vee \frac{\Delta_n^{1-\alpha'}}{\delta^{d+1}}\vee\frac{1}{p}\bigg)\Bigg)\\
&~~~~~ +\Oo\bigg(\frac{\Delta_n^{1-\alpha'}}{p}\Big(\mathbbm{1}_{\{\textbf{y}_1\neq \textbf{y}_2\}}\nor{\textbf{y}_1-\textbf{y}_2}_0^{-(d+1)}+\delta^{-(d+1)}\Big)\bigg),
\end{align*}
where $\Upsilon_{\alpha'}$ is a numerical constant depending on $\alpha'\in(0,1)$, defined in equation \eqref{eqn_definingUpsilon}. 
In particular we have 
\begin{align*}
\Var\big(V_{n,\Delta_n}(\textbf{y})\big)=\frac{\Upsilon_{\alpha'}}{n}\bigg(\frac{\Gamma(1-\alpha')\sigma^2}{2^d(\pi\eta)^{d/2}\alpha'\Gamma(d/2)}\bigg)^2\Big(1+\Oo\big(\Delta_n^{1/2}\vee \Delta_n^{1-\alpha'}\big)\Big).
\end{align*}
\end{theorem}
\begin{proof}
It holds that
\begin{align*}
&\Cov\big(V_{p,\Delta_n}(\textbf{y}_1),V_{p,\Delta_n}(\textbf{y}_2)\big)\\
&=\frac{2e^{\nor{\kappa\bigcdot(\textbf{y}_1+\textbf{y}_2)}_1}}{p^2\Delta_n^{2\alpha'}}\sum_{i,j=1}^p\bigg(\sum_{\textbf{k}_1,\textbf{k}_2\in\N^d}e_{\textbf{k}_1}(\textbf{y}_1)e_{\textbf{k}_1}(\textbf{y}_2)e_{\textbf{k}_2}(\textbf{y}_1)e_{\textbf{k}_2}(\textbf{y}_2)\Cov\Big(\Delta_i\tilde{x}_{\textbf{k}_1}\Delta_i\tilde{x}_{\textbf{k}_2},~\Delta_j\tilde{x}_{\textbf{k}_1}\Delta_j\tilde{x}_{\textbf{k}_2}\Big)\bigg)\\
&=\frac{2e^{\nor{\kappa\bigcdot(\textbf{y}_1+\textbf{y}_2)}_1}}{p\Delta_n^{2\alpha'}}\sum_{\textbf{k}_1,\textbf{k}_2\in\N^d}e_{\textbf{k}_1}(\textbf{y}_1)e_{\textbf{k}_1}(\textbf{y}_2)e_{\textbf{k}_2}(\textbf{y}_1)e_{\textbf{k}_2}(\textbf{y}_2)D_{\textbf{k}_1,\textbf{k}_2},
\end{align*}
where 
\begin{align*}
D_{\textbf{k}_1,\textbf{k}_2}:=\frac{1}{p}\sum_{i,j=1}^p\Cov\Big(\big(\tilde{B}_{i,\textbf{k}_1}+C_{i,\textbf{k}_1}\big)\big(\tilde{B}_{i,\textbf{k}_2}+C_{i,\textbf{k}_2}\big),~\big(\tilde{B}_{j,\textbf{k}_1}+C_{j,\textbf{k}_1}\big)\big(\tilde{B}_{j,\textbf{k}_2}+C_{j,\textbf{k}_2}\big)\Big).
\end{align*}
Consider $(Z_\textbf{k})_{\textbf{k}\in\N^d}$ as independent standard normal distributed random variables, which are independent to $B_{i,\textbf{k}}$. We can express $\tilde{B}_{i,\textbf{k}}$ as:
\begin{align*}
\tilde{B}_{i,\textbf{k}}=B_{i,\textbf{k}}+\frac{1}{(2\lambda_\textbf{k}^{1+\alpha})^{1/2}}\sigma\big( e^{-\lambda_\textbf{k}\Delta_n}-1\big)e^{-\lambda_\textbf{k}(i-1)\Delta_n}Z_\textbf{k}.
\end{align*}
Hence, we derive the following covariance structures:
\begin{align}
\Cov\Big(\tilde{B}_{i,\textbf{k}},C_{j,\textbf{k}}\Big)&=\Cov(B_{i,\textbf{k}},C_{j,\textbf{k}})=\Sigma_{i,j}^{BC,\textbf{k}},\notag \\
\Cov\Big(\tilde{B}_{i,\textbf{k}},\tilde{B}_{j,\textbf{k}}\Big)&=\Cov(B_{i,\textbf{k}},B_{j,\textbf{k}})+\frac{\sigma^2}{2\lambda_\textbf{k}^{1+\alpha}}\big(e^{-\lambda_\textbf{k}\Delta_n}-1\big)^2e^{-\lambda_\textbf{k}(i+j-2)\Delta_n }\Var(Z_\textbf{k})\notag \\
&=\frac{\sigma^2}{2\lambda_\textbf{k}^{1+\alpha}}\big(e^{-\lambda_\textbf{k}\Delta_n}-1\big)^2e^{-\lambda_\textbf{k}\Delta_n\abs{i-j}}=:\tilde{\Sigma}_{i,j}^{B,\textbf{k}},\label{eqn_CovBBtilde}
\end{align}
where we have applied equation \eqref{number:CovBB}. As $\tilde{B}_{i,\textbf{k}}+C_{i,\textbf{k}}$ is centred normally distributed, we can use Isserlis' theorem to deduce that
\begin{align*}
D_{\textbf{k}_1,\textbf{k}_2}&=\frac{1}{p} \sum\limits_{i,j=1}^p
\bigg(\E\Big[\big(\tilde{B}_{i,\textbf{k}_1}+C_{i,\textbf{k}_1}\big)\big(\tilde{B}_{j,\textbf{k}_1}+C_{j,\textbf{k}_1}\big)\Big]
\E\Big[\big(\tilde{B}_{i,\textbf{k}_2}+C_{i,\textbf{k}_2}\big)\big(\tilde{B}_{j,\textbf{k}_2}+C_{j,\textbf{k}_2}\big)\Big]\\
&~~~~~~~+\E\Big[\big(\tilde{B}_{i,\textbf{k}_1}+C_{i,\textbf{k}_1}\big)
\big(\tilde{B}_{j,\textbf{k}_2}+C_{j,\textbf{k}_2}\big)\Big]\E\Big[\big(\tilde{B}_{i,\textbf{k}_2}+C_{i,\textbf{k}_2}\big)\big(\tilde{B}_{j,\textbf{k}_1}+C_{j,\textbf{k}_1}\big)\Big]\bigg).
\end{align*}
For further reading on the Isserlis theorem, we recommend referring to \cite{isserlis1918formula}.
Assume $\textbf{k}_1\neq \textbf{k}_2$, then we have
\begin{align}
D_{\textbf{k}_1,\textbf{k}_2}&=\frac{1}{p} \sum\limits_{i,j=1}^p
\E\bigg[\big(\tilde{B}_{i,\textbf{k}_1}+C_{i,\textbf{k}_1}\big)\big(\tilde{B}_{j,\textbf{k}_1}+C_{j,\textbf{k}_1}\big)\bigg]
\E\bigg[\big(\tilde{B}_{i,\textbf{k}_2}+C_{i,\textbf{k}_2}\big)\big(\tilde{B}_{j,\textbf{k}_2}+C_{j,\textbf{k}_2}\big)\bigg)\bigg]\notag\\
&=\frac{1}{p} \sum\limits_{i,j=1}^p
\Big(\tilde{\Sigma}_{i,j}^{B,\textbf{k}_1}+\Sigma_{i,j}^{BC,\textbf{k}_1}+\Sigma_{j,i}^{BC,\textbf{k}_1}+\Sigma_{i,j}^{C,\textbf{k}_1}\Big)\Big(\tilde{\Sigma}_{i,j}^{B,\textbf{k}_2}+\Sigma_{i,j}^{BC,\textbf{k}_2}+\Sigma_{j,i}^{BC,\textbf{k}_2}+\Sigma_{i,j}^{C,\textbf{k}_2}\Big).\label{eqn_MainCalcDkl}
\end{align}
We calculate each combination separately. To do this, we use the following identity:
\begin{align}
\sum_{i,j=1}^pq^{|i-j|}=2\frac{q^{p+1}-q}{(1-q)^2}+p\frac{1+q}{1-q},\label{eqn_geometricSeriesAbsVal}
\end{align}
for $q\neq 1$. Then, we have
\begin{align*}
\frac{1}{p} \sum\limits_{i,j=1}^p \tilde{\Sigma}_{i,j}^{B,\textbf{k}_1}\tilde{\Sigma}_{i,j}^{B,\textbf{k}_2}&=\sigma^4\frac{\big(e^{-\lambda_{\textbf{k}_1}\Delta_n}-1\big)^2\big(e^{-\lambda_{\textbf{k}_2}
\Delta_n}-1\big)^2}{4p\lambda_{\textbf{k}_1}^{1+\alpha}\lambda_{\textbf{k}_2}^{1+\alpha}}\sum\limits_{i,j=1}^p e^{-(\lambda_{\textbf{k}_1}+\lambda_{\textbf{k}_2})\Delta_n\abs{i-j}}\\
&=\sigma^4\frac{\big(e^{-\lambda_{\textbf{k}_1}\Delta_n}-1\big)^2\big(e^{-\lambda_{\textbf{k}_2}
\Delta_n}-1\big)^2}{4\lambda_{\textbf{k}_1}^{1+\alpha}\lambda_{\textbf{k}_2}^{1+\alpha}}\cdot\frac{1+e^{-(\lambda_{\textbf{k}_1}+\lambda_{\textbf{k}_2})\Delta_n}}{1-e^{-(\lambda_{\textbf{k}_1}+\lambda_{\textbf{k}_2})\Delta_n}}\\
&~~~~~\times\bigg(1+\mathcal{O}\Big(1\wedge \frac{p^{-1}}{1-e^{-(\lambda_{\textbf{k}_1}+\lambda_{\textbf{k}_2})\Delta_n}}\Big)\bigg).
\end{align*}
By utilizing equation \eqref{eqn_covCijk}, we obtain:
\begin{align*}
\frac{1}{p} \sum\limits_{i,j=1}^p \Sigma_{i,j}^{C,\textbf{k}_1}\Sigma_{i,j}^{C,\textbf{k}_2}&=\sigma^4 \frac{\big(1-e^{-2\lambda_{\textbf{k}_1}\Delta_n}\big)\big(1-e^{-2\lambda_{\textbf{k}_2}\Delta_n}\big)}{4\lambda_{\textbf{k}_1}^{1+\alpha}\lambda_{\textbf{k}_2}^{1+\alpha}}.
\end{align*}
Using equation \eqref{number:CovBC} and the identity
\begin{align*}
\sum_{i,j=1}^p\mathbbm{1}_{\{i>j\}}q^{i-j}=\frac{pq}{1-q}+\frac{q-q^{p+1}}{(1-q)^2},
\end{align*}
yields that
\begin{align*}
\frac{1}{p} \sum\limits_{i,j=1}^p \Sigma_{i,j}^{BC,\textbf{k}_1}\Sigma_{i,j}^{BC,\textbf{k}_2}&=\frac{1}{p} \sum\limits_{i,j=1}^p\mathbbm{1}_{\{i>j\}}\sigma^4\frac{\big(e^{-\lambda_{\textbf{k}_1}\Delta_n}-1\big)\big(e^{-\lambda_{\textbf{k}_2}\Delta_n}-1\big)}{4\lambda_{\textbf{k}_1}^{1+\alpha}\lambda_{\textbf{k}_2}^{1+\alpha}}\\
&~~~~~\times e^{-(\lambda_{\textbf{k}_1}+\lambda_{\textbf{k}_2})\Delta_n(i-j)}\big(e^{\lambda_{\textbf{k}_1}\Delta_n}-e^{-\lambda_{\textbf{k}_1}\Delta_n}\big)
\big(e^{\lambda_{\textbf{k}_2}\Delta_n}-e^{-\lambda_{\textbf{k}_2}\Delta_n}\big)\\
&=\sigma^4\frac{\big(e^{-\lambda_{\textbf{k}_1}\Delta_n}-1\big)\big(e^{-\lambda_{\textbf{k}_2}\Delta_n}-1\big)}{4\lambda_{\textbf{k}_1}^{1+\alpha}\lambda_{\textbf{k}_2}^{1+\alpha}}\cdot
\frac{(1-e^{-2\lambda_{\textbf{k}_1}\Delta_n})(1-e^{-2\lambda_{\textbf{k}_2}\Delta_n})}{1-e^{-(\lambda_{\textbf{k}_1}+\lambda_{\textbf{k}_2})\Delta_n}}\\
&~~~~~\times\bigg(1+\mathcal{O}\Big(1\wedge \frac{p^{-1}}{1-e^{-(\lambda_{\textbf{k}_1}+\lambda_{\textbf{k}_2})\Delta_n}}\Big)\bigg).
\end{align*}
The same calculations apply to $\Sigma_{j,i}^{BC,\textbf{k}_1}\Sigma_{j,i}^{BC,\textbf{k}_2}$. As for the cross-terms, we obtain:
\begin{align*}
&\frac{1}{p} \sum\limits_{i,j=1}^p \tilde{\Sigma}_{i,j}^{B,\textbf{k}_1}\big(\Sigma_{i,j}^{BC,\textbf{k}_2}+\Sigma_{j,i}^{BC,\textbf{k}_2}\big) \\
&=\sigma^4\frac{\big(e^{-\lambda_{\textbf{k}_1}\Delta_n}-1\big)^2\big(e^{-\lambda_{\textbf{k}_2}\Delta_n}-1\big)}{4\lambda_{\textbf{k}_1}^{1+\alpha}\lambda_{\textbf{k}_2}^{1+\alpha}}
\big(e^{\lambda_{\textbf{k}_2}\Delta_n}-e^{-\lambda_{\textbf{k}_2}\Delta_n}\big) \frac{1}{p} \sum\limits_{i,j=1}^p\mathbbm{1}_{\{i>j\}} e^{-(\lambda_{\textbf{k}_1}+\lambda_{\textbf{k}_2})\Delta_n(i-j)}\\
&~~~~~+ \sigma^4\frac{\big(e^{-\lambda_{\textbf{k}_1}\Delta_n}-1\big)^2\big(e^{-\lambda_{\textbf{k}_2}\Delta_n}-1\big)}{4\lambda_{\textbf{k}_1}^{1+\alpha}\lambda_{\textbf{k}_2}^{1+\alpha}}
\big(e^{\lambda_{\textbf{k}_2}\Delta_n}-e^{-\lambda_{\textbf{k}_2}\Delta_n}\big) \frac{1}{p} \sum\limits_{i,j=1}^p\mathbbm{1}_{\{j>i\}} e^{-(\lambda_{\textbf{k}_1}+\lambda_{\textbf{k}_2})\Delta_n(j-i)}\\
&=\sigma^4\frac{\big(e^{-\lambda_{\textbf{k}_1}\Delta_n}-1\big)^2\big(e^{-\lambda_{\textbf{k}_2}\Delta_n}-1\big)}{2\lambda_{\textbf{k}_1}^{1+\alpha}\lambda_{\textbf{k}_2}^{1+\alpha}}
e^{-\lambda_{\textbf{k}_1}\Delta_n}\frac{1-e^{-2\lambda_{\textbf{k}_2}\Delta_n}}{1-e^{-(\lambda_{\textbf{k}_1}+\lambda_{\textbf{k}_2})\Delta_n}}\bigg(1+\mathcal{O}\Big(1\wedge \frac{p^{-1}}{1-e^{-(\lambda_{\textbf{k}_1}+\lambda_{\textbf{k}_2})\Delta_n}}\Big)\bigg),
\end{align*}
and
\begin{align*}
\frac{1}{p} \sum\limits_{i,j=1}^p \tilde{\Sigma}_{i,j}^{B,\textbf{k}_1}\Sigma_{i,j}^{C,\textbf{k}_2}&= \frac{1}{p} \sum\limits_{i,j=1}^p\frac{\sigma^2}{2\lambda_{\textbf{k}_1}^{1+\alpha}}\big(e^{-\lambda_{\textbf{k}_1}\Delta_n}-1\big)^2e^{-\lambda_{\textbf{k}_1}\Delta_n\abs{i-j}}\mathbbm{1}_{\{i=j\}}\sigma^2 \frac{1-e^{-2\lambda_{\textbf{k}_2}\Delta_n}}{2\lambda_{\textbf{k}_2}^{1+\alpha}}\\
&=\sigma^4\frac{\big(e^{-\lambda_{\textbf{k}_1}\Delta_n}-1\big)^2\big(1-e^{-2\lambda_{\textbf{k}_2}\Delta_n}\big)}{4\lambda_{\textbf{k}_1}^{1+\alpha}\lambda_{\textbf{k}_2}^{1+\alpha}}.
\end{align*}
Furthermore, the following cross-terms vanish:
\begin{align*}
\frac{1}{p} \sum\limits_{i,j=1}^p \Sigma_{i,j}^{BC,\textbf{k}_1}\Sigma_{i,j}^{C,\textbf{k}_2}=\frac{1}{p} \sum\limits_{i,j=1}^p \Sigma_{j,i}^{BC,\textbf{k}_1}\Sigma_{i,j}^{C,\textbf{k}_2}=\frac{1}{p} \sum\limits_{i,j=1}^p \Sigma_{i,j}^{BC,\textbf{k}_1}\Sigma_{j,i}^{BC,\textbf{k}_2}=0.
\end{align*}
Inserting the auxiliary calculations into equation \eqref{eqn_MainCalcDkl} results in:
\begin{align*}
D_{\textbf{k}_1,\textbf{k}_2}&=\frac{1}{p} \sum\limits_{i,j=1}^p
\Big(\tilde{\Sigma}_{i,j}^{B,\textbf{k}_1}+\Sigma_{i,j}^{BC,\textbf{k}_1}+\Sigma_{j,i}^{BC,\textbf{k}_1}+\Sigma_{i,j}^{C,\textbf{k}_1}\Big)
\Big(\tilde{\Sigma}_{i,j}^{B,\textbf{k}_2}+\Sigma_{i,j}^{BC,\textbf{k}_2}+\Sigma_{j,i}^{BC,\textbf{k}_2}+\Sigma_{i,j}^{C,\textbf{k}_2}\Big) \\
&=\frac{1}{p} \sum\limits_{i,j=1}^p\Big(\tilde{\Sigma}_{i,j}^{B,\textbf{k}_1}\tilde{\Sigma}_{i,j}^{B,\textbf{k}_2}+\tilde{\Sigma}_{i,j}^{B,\textbf{k}_1}\big(\Sigma_{i,j}^{BC,\textbf{k}_2}+\Sigma_{j,i}^{BC,\textbf{k}_2}\big)+\tilde{\Sigma}_{i,j}^{B,\textbf{k}_1}\Sigma_{i,j}^{C,\textbf{k}_2}+\big(\Sigma_{i,j}^{BC,\textbf{k}_1}+\Sigma_{j,i}^{BC,\textbf{k}_1}\big)\tilde{\Sigma}_{i,j}^{B,\textbf{k}_2}\\
&~~~~~+\Sigma_{i,j}^{BC,\textbf{k}_1}\Sigma_{i,j}^{BC,\textbf{k}_2}+\Sigma_{j,i}^{BC,\textbf{k}_1}\Sigma_{j,i}^{BC,\textbf{k}_2}+\Sigma_{i,j}^{C,\textbf{k}_1}\tilde{\Sigma}_{i,j}^{B,\textbf{k}_2}+\Sigma_{i,j}^{C,\textbf{k}_1}\Sigma_{i,j}^{C,\textbf{k}_2}\Big)\\
&=\sigma^4\bigg(
\frac{\big(e^{-\lambda_{\textbf{k}_1}\Delta_n}-1\big)^2\big(e^{-\lambda_{\textbf{k}_2}
\Delta_n}-1\big)^2}{4\lambda_{\textbf{k}_1}^{1+\alpha}\lambda_{\textbf{k}_2}^{1+\alpha}}\Big(\frac{1+e^{-(\lambda_{\textbf{k}_1}+\lambda_{\textbf{k}_2})\Delta_n}}{1-e^{-(\lambda_{\textbf{k}_1}+\lambda_{\textbf{k}_2})\Delta_n}}\\
&~~~~~+\frac{e^{-\lambda_{\textbf{k}_1}\Delta_n}(1-e^{-2\lambda_{\textbf{k}_2}\Delta_n})}{1-e^{-(\lambda_{\textbf{k}_1}+\lambda_{\textbf{k}_2})\Delta_n}}\cdot\frac{2}{e^{-\lambda_{\textbf{k}_2}\Delta_n}-1}+\frac{e^{-\lambda_{\textbf{k}_2}\Delta_n}(1-e^{-2\lambda_{\textbf{k}_1}\Delta_n})}{1-e^{-(\lambda_{\textbf{k}_1}+\lambda_{\textbf{k}_2})\Delta_n}}\cdot\frac{2}{e^{-\lambda_{\textbf{k}_1}\Delta_n}-1}\\
&~~~~~+\frac{(1-e^{-2\lambda_{\textbf{k}_1}\Delta_n})(1-e^{-2\lambda_{\textbf{k}_2}\Delta_n})}{1-e^{-(\lambda_{\textbf{k}_1}+\lambda_{\textbf{k}_2})\Delta_n}}\cdot \frac{2}{\big(e^{-\lambda_{\textbf{k}_1}\Delta_n}-1\big)\big(e^{-\lambda_{\textbf{k}_2}\Delta_n}-1\big)}\Big)\\
&~~~~~+\frac{\big(e^{-\lambda_{\textbf{k}_1}\Delta_n}-1\big)^2\big(1-e^{-2\lambda_{\textbf{k}_2}\Delta_n}\big)}{4\lambda_{\textbf{k}_1}^{1+\alpha}\lambda_{\textbf{k}_2}^{1+\alpha}}+\frac{\big(e^{-\lambda_{\textbf{k}_2}\Delta_n}-1\big)^2\big(1-e^{-2\lambda_{\textbf{k}_1}\Delta_n}\big)}{4\lambda_{\textbf{k}_1}^{1+\alpha}\lambda_{\textbf{k}_2}^{1+\alpha}}\\
&~~~~~+\frac{\big(1-e^{-2\lambda_{\textbf{k}_1}\Delta_n}\big)\big(1-e^{-2\lambda_{\textbf{k}_2}\Delta_n}\big)}{4\lambda_{\textbf{k}_1}^{1+\alpha}\lambda_{\textbf{k}_2}^{1+\alpha}}
\bigg)\bigg(1+\mathcal{O}\Big(1\wedge \frac{p^{-1}}{1-e^{-(\lambda_{\textbf{k}_1}+\lambda_{\textbf{k}_2})\Delta_n}}\Big)\bigg).
\end{align*}
Using the identity $(e^{2x}-1)/(e^x-1)=e^x+1$, we have
\begin{align*}
D_{\textbf{k}_1,\textbf{k}_2}
&=\sigma^4\bigg(\frac{\big(e^{-\lambda_{\textbf{k}_1}\Delta_n}-1\big)^2\big(e^{-\lambda_{\textbf{k}_2}
\Delta_n}-1\big)^2}{4\lambda_{\textbf{k}_1}^{1+\alpha}\lambda_{\textbf{k}_2}^{1+\alpha}}\cdot \frac{3-e^{-(\lambda_{\textbf{k}_1}+\lambda_{\textbf{k}_2})\Delta_n}}{1-e^{-(\lambda_{\textbf{k}_1}+\lambda_{\textbf{k}_2})\Delta_n}}\\
&~~~~~+\frac{\big(e^{-\lambda_{\textbf{k}_1}\Delta_n}-1\big)^2\big(1-e^{-2\lambda_{\textbf{k}_2}\Delta_n}\big)}{4\lambda_{\textbf{k}_1}^{1+\alpha}\lambda_{\textbf{k}_2}^{1+\alpha}}+\frac{\big(e^{-\lambda_{\textbf{k}_2}\Delta_n}-1\big)^2\big(1-e^{-2\lambda_{\textbf{k}_1}\Delta_n}\big)}{4\lambda_{\textbf{k}_1}^{1+\alpha}\lambda_{\textbf{k}_2}^{1+\alpha}}\\
&~~~~~+\frac{\big(1-e^{-2\lambda_{\textbf{k}_1}\Delta_n}\big)\big(1-e^{-2\lambda_{\textbf{k}_2}\Delta_n}\big)}{4\lambda_{\textbf{k}_1}^{1+\alpha}\lambda_{\textbf{k}_2}^{1+\alpha}}
\bigg)\times\bigg(1+\mathcal{O}\Big(1\wedge \frac{p^{-1}}{1-e^{-(\lambda_{\textbf{k}_1}+\lambda_{\textbf{k}_2})\Delta_n}}\Big)\bigg)\\
&=\sigma^4\bigg(\frac{\big(e^{-\lambda_{\textbf{k}_1}\Delta_n}-1\big)^2\big(e^{-\lambda_{\textbf{k}_2}
\Delta_n}-1\big)^2}{4\lambda_{\textbf{k}_1}^{1+\alpha}\lambda_{\textbf{k}_2}^{1+\alpha}}\cdot \frac{4-2e^{-(\lambda_{\textbf{k}_1}+\lambda_{\textbf{k}_2})\Delta_n}}{1-e^{-(\lambda_{\textbf{k}_1}+\lambda_{\textbf{k}_2})\Delta_n}}+\frac{\big(1-e^{-\lambda_{\textbf{k}_1}\Delta_n}\big)\big(1-e^{-\lambda_{\textbf{k}_2}\Delta_n}\big)}{4\lambda_{\textbf{k}_1}^{1+\alpha}\lambda_{\textbf{k}_2}^{1+\alpha}}\\
&~~~~~\times2\Big(2-\big(1-e^{-\lambda_{\textbf{k}_1}\Delta_n}\big)\big(1-e^{-\lambda_{\textbf{k}_2}\Delta_n}\big)\Big)\bigg)\bigg(1+\mathcal{O}\Big(1\wedge \frac{p^{-1}}{1-e^{-(\lambda_{\textbf{k}_1}+\lambda_{\textbf{k}_2})\Delta_n}}\Big)\bigg)\\
&=\sigma^4\bigg(\frac{\big(1-e^{-\lambda_{\textbf{k}_1}\Delta_n}\big)^2\big(1-e^{-\lambda_{\textbf{k}_2}
\Delta_n}\big)^2}{2\lambda_{\textbf{k}_1}^{1+\alpha}\lambda_{\textbf{k}_2}^{1+\alpha}}\cdot \frac{1}{1-e^{-(\lambda_{\textbf{k}_1}+\lambda_{\textbf{k}_2})\Delta_n}}+\frac{\big(1-e^{-\lambda_{\textbf{k}_1}\Delta_n}\big)\big(1-e^{-\lambda_{\textbf{k}_2}\Delta_n}\big)}{\lambda_{\textbf{k}_1}^{1+\alpha}\lambda_{\textbf{k}_2}^{1+\alpha}}\bigg)\\
&~~~~~\times\bigg(1+\mathcal{O}\Big(1\wedge \frac{p^{-1}}{1-e^{-(\lambda_{\textbf{k}_1}+\lambda_{\textbf{k}_2})\Delta_n}}\Big)\bigg).
\end{align*}
Recalling the calculations of the covariance yields:
\begin{align*}
&\Cov\big(V_{p,\Delta_n}(\textbf{y}_1),V_{p,\Delta_n}(\textbf{y}_2)\big)\\
&=\frac{2e^{\nor{\kappa\bigcdot(\textbf{y}_1+\textbf{y}_2)}_1}\sigma^4}{p\Delta_n^{2\alpha'}}\sum_{\substack{\textbf{k}_1,\textbf{k}_2\in\N^d \\ \textbf{k}_1 \neq \textbf{k}_2}}e_{\textbf{k}_1}(\textbf{y}_1)e_{\textbf{k}_1}(\textbf{y}_2)e_{\textbf{k}_2}(\textbf{y}_1)e_{\textbf{k}_2}(\textbf{y}_2)\bar{D}_{\textbf{k}_1,\textbf{k}_2}\bigg(1+\mathcal{O}\Big(1\wedge \frac{p^{-1}}{1-e^{-(\lambda_{\textbf{k}_1}+\lambda_{\textbf{k}_2})\Delta_n}}\Big)\bigg)\\
&~~~~~+\frac{2e^{\nor{\kappa\bigcdot(\textbf{y}_1+\textbf{y}_2)}_1}}{p\Delta_n^{2\alpha'}}\sum_{\textbf{k}\in\N^d}e^2_{\textbf{k}}(\textbf{y}_1)e^2_{\textbf{k}}(\textbf{y}_2)D_{\textbf{k},\textbf{k}} ,
\end{align*}
where we define:
\begin{align}
\bar{D}_{\textbf{k}_1,\textbf{k}_2}:=\frac{\big(1-e^{-\lambda_{\textbf{k}_1}\Delta_n}\big)^2\big(1-e^{-\lambda_{\textbf{k}_2}
\Delta_n}\big)^2}{2\lambda_{\textbf{k}_1}^{1+\alpha}\lambda_{\textbf{k}_2}^{1+\alpha}}\cdot \frac{1}{1-e^{-(\lambda_{\textbf{k}_1}+\lambda_{\textbf{k}_2})\Delta_n}}+\frac{\big(1-e^{-\lambda_{\textbf{k}_1}\Delta_n}\big)\big(1-e^{-\lambda_{\textbf{k}_2}\Delta_n}\big)}{\lambda_{\textbf{k}_1}^{1+\alpha}\lambda_{\textbf{k}_2}^{1+\alpha}}.\label{eqn_barDk1k2ref}
\end{align}
Regarding the remainder, we utilize the inequality $(1-e^{-(x+y)})^{-1}\leq (1-e^{-x})^{-1/2}(1-e^{-y})^{-1/2}$. For a sufficiently large $p$, we deduce that
\begin{align*}
\frac{1}{p^2\Delta_n^{2\alpha'}}\sum\limits_{\substack{ \textbf{k}_1,\textbf{k}_2\in\N^d\\ \textbf{k}_1\neq \textbf{k}_2}} \frac{\overline{D}_{\textbf{k}_1,\textbf{k}_2}}{1-e^{-(\lambda_{\textbf{k}_1}+\lambda_{\textbf{k}_2})\Delta_n}}
&\leq \frac{3}{p^2\Delta_n^{2\alpha'}}\bigg(\Delta_n^{1+\alpha}\sum\limits_{\textbf{k}\in\N^d}\frac{\big(1-e^{-\lambda_{\textbf{k}}\Delta_n}\big)^{1/2}}{2(\lambda_{\textbf{k}}\Delta_n)^{1+\alpha}}\bigg)^2\\
&=\frac{3}{p^2}\bigg(\Delta_n^{d/2}\sum\limits_{\textbf{k}\in\N^d}\frac{\big(1-e^{-\lambda_{\textbf{k}}\Delta_n}\big)^{1/2}}{2(\lambda_{\textbf{k}}\Delta_n)^{1+\alpha}}\bigg)^2.
\end{align*}
Thanks to Lemma \ref{lemma_riemannApprox_multi}, we obtain the convergence of the series, such that
\begin{align*}
\frac{1}{p^2\Delta_n^{2\alpha'}}\sum\limits_{\substack{ \textbf{k}_1,\textbf{k}_2\in\N^d\\ \textbf{k}_1\neq \textbf{k}_2}} \frac{\overline{D}_{\textbf{k}_1,\textbf{k}_2}}{1-e^{-(\lambda_{\textbf{k}_1}+\lambda_{\textbf{k}_2})\Delta_n}}=\Oo\bigg(\frac{1}{p^2}\Big(\int_0^\infty \frac{\sqrt{1-e^{-x}}}{x^{1+\alpha'}}\diff x\Big)^2\bigg)=\Oo(p^{-2}).
\end{align*} 
For small $p$ we always obtain a bound of order $\Oo(p^{-1})$, and obtain:
\begin{align*}
&\frac{2e^{\nor{\kappa\bigcdot(\textbf{y}_1+\textbf{y}_2)}_1}\sigma^4}{p\Delta_n^{2\alpha'}}\sum_{\substack{\textbf{k}_1,\textbf{k}_2\in\N^d \\ \textbf{k}_1\neq \textbf{k}_2}}e_{\textbf{k}_1}(\textbf{y}_1)e_{\textbf{k}_1}(\textbf{y}_2)e_{\textbf{k}_2}(\textbf{y}_1)e_{\textbf{k}_2}(\textbf{y}_2)\bar{D}_{\textbf{k}_1,\textbf{k}_2}\cdot\mathcal{O}\bigg(1\wedge \frac{p^{-1}}{1-e^{-(\lambda_{\textbf{k}_1}+\lambda_{\textbf{k}_2})\Delta_n}}\bigg)\\
&=\Oo\bigg(\frac{1}{p}\Big(1\wedge \frac{1}{p}\Big)\bigg).
\end{align*}
Thus, we find:
\begin{align*}
\Cov\big(V_{p,\Delta_n}(\textbf{y}_1),V_{p,\Delta_n}(\textbf{y}_2)\big)&=\frac{2e^{\nor{\kappa\bigcdot(\textbf{y}_1+\textbf{y}_2)}_1}\sigma^4}{p\Delta_n^{2\alpha'}}\sum_{\substack{\textbf{k}_1,\textbf{k}_2\in\N^d \\ \textbf{k}_1 \neq \textbf{k}_2}}e_{\textbf{k}_1}(\textbf{y}_1)e_{\textbf{k}_1}(\textbf{y}_2)e_{\textbf{k}_2}(\textbf{y}_1)e_{\textbf{k}_2}(\textbf{y}_2)\bar{D}_{\textbf{k}_1,\textbf{k}_2}\\
&~~~~~+\frac{2e^{\nor{\kappa\bigcdot(\textbf{y}_1+\textbf{y}_2)}_1}}{p\Delta_n^{2\alpha'}}\sum_{\textbf{k}\in\N^d}e^2_{\textbf{k}}(\textbf{y}_1)e^2_{\textbf{k}}(\textbf{y}_2)D_{\textbf{k},\textbf{k}}+\Oo\bigg(\frac{1}{p}\Big(1\wedge \frac{1}{p}\Big)\bigg).
\end{align*}
For $\textbf{k}_1=\textbf{k}_2=\textbf{k}$ we have
\begin{align*}
D_{\textbf{k},\textbf{k}}&=\frac{1}{p} \Bigg(\sum\limits_{i,j=1}^p\E\bigg[\big(\tilde{B}_{i,\textbf{k}}+C_{i,\textbf{k}}\big)
\big(\tilde{B}_{i,\textbf{k}}+C_{i,\textbf{k}}\big)\big(\tilde{B}_{j,\textbf{k}}+C_{j,\textbf{k}}\big)\big(\tilde{B}_{j,\textbf{k}}+C_{j,\textbf{k}}\big)\bigg)\bigg]\\
&~~~~~~~-\E\bigg[\big(\tilde{B}_{i,\textbf{k}}+C_{i,\textbf{k}}\big)
\big(\tilde{B}_{i,\textbf{k}}+C_{i,\textbf{k}}\big)\bigg]\E\bigg[\big(\tilde{B}_{j,\textbf{k}}+C_{j,\textbf{k}}\big)\big(\tilde{B}_{j,\textbf{k}}+C_{j,\textbf{k}}\big)\bigg)\bigg]\Bigg)\\
&\leq \frac{4}{p} \sum\limits_{i,j=1}^p \Big(\tilde{\Sigma}_{i,j}^{B,\textbf{k}}\Big)^2+2\Big(\Sigma_{i,j}^{BC,\textbf{k}}\Big)^2+\Big(\Sigma_{i,j}^{C,\textbf{k}}\Big)^2.
\end{align*}
Calculating the covariance terms results in:
\begin{align*}
\frac{1}{p} \sum\limits_{i,j=1}^p \tilde{\Sigma}_{i,j}^{B,\textbf{k}}\tilde{\Sigma}_{i,j}^{B,\textbf{k}}&=\sigma^4\frac{\big(1-e^{-\lambda_\textbf{k}\Delta_n}\big)^4}{4\lambda_\textbf{k}^{2(1+\alpha)}}\frac{1+e^{-2\lambda_\textbf{k}\Delta_n}}{1-e^{-2\lambda_\textbf{k}\Delta_n}}\bigg(1+\mathcal{O}\Big(1\wedge \frac{p^{-1}}{1-e^{-2\lambda_\textbf{k}\Delta_n}}\Big)\bigg), \\
\frac{1}{p} \sum\limits_{i,j=1}^p \Sigma_{i,j}^{BC,\textbf{k}}\Sigma_{i,j}^{BC,\textbf{k}}&=\sigma^4\frac{\big(1-e^{-\lambda_\textbf{k}\Delta_n}\big)^2}{4\lambda_\textbf{k}^{2(1+\alpha)}}\cdot\frac{(1-e^{-2\lambda_\textbf{k}\Delta_n})^2}{1-e^{-2\lambda_\textbf{k}\Delta_n}}\bigg(1+\mathcal{O}\Big(1\wedge \frac{p^{-1}}{1-e^{-2\lambda_\textbf{k}\Delta_n}}\Big)\bigg),\\
\frac{1}{p} \sum\limits_{i,j=1}^p \Sigma_{i,j}^{C,\textbf{k}}\Sigma_{i,j}^{C,\textbf{k}}
&=\sigma^4 \frac{\big(1-e^{-2\lambda_\textbf{k}\Delta_n}\big)^2}{4\lambda_\textbf{k}^{2(1+\alpha)}},
\end{align*}
where we used analogous steps as for $\textbf{k}_1\neq \textbf{k}_2$. For $\textbf{k}_1=\textbf{k}_2=\textbf{k}$ we derive that
\begin{align*}
D_{\textbf{k},\textbf{k}}&\leq \sigma^4\Bigg(\frac{\big(1-e^{-\lambda_\textbf{k}\Delta_n}\big)^4}{\lambda_\textbf{k}^{2(1+\alpha)}}\frac{1+e^{-2\lambda_\textbf{k}\Delta_n}}{1-e^{-2\lambda_\textbf{k}\Delta_n}}\\
&~~~~~ +2\frac{\big(1-e^{-\lambda_\textbf{k}\Delta_n}\big)^2}{\lambda_\textbf{k}^{2(1+\alpha)}}\big(1-e^{-2\lambda_\textbf{k}\Delta_n}\big)+ \frac{\big(1-e^{-2\lambda_\textbf{k}\Delta_n}\big)^2}{\lambda_\textbf{k}^{2(1+\alpha)}} \Bigg)\bigg(1+\mathcal{O}\Big(1\wedge \frac{p^{-1}}{1-e^{-2\lambda_\textbf{k}\Delta_n}}\Big)\bigg),
\end{align*}
where we define:
\begin{align*}
\overline{D}_{\textbf{k},\textbf{k}}:=\frac{\big(1-e^{-\lambda_\textbf{k}\Delta_n}\big)^4}{\lambda_\textbf{k}^{2(1+\alpha)}}\frac{1+e^{-2\lambda_\textbf{k}\Delta_n}}{1-e^{-2\lambda_\textbf{k}\Delta_n}} +2\frac{\big(1-e^{-\lambda_\textbf{k}\Delta_n}\big)^2}{\lambda_\textbf{k}^{2(1+\alpha)}}\big(1-e^{-2\lambda_\textbf{k}\Delta_n}\big)+\frac{\big(1-e^{-2\lambda_\textbf{k}\Delta_n}\big)^2}{\lambda_\textbf{k}^{2(1+\alpha)}}.
\end{align*}
We demonstrate that $D_{\textbf{k},\textbf{k}}$ is negligible, as can be seen by the following:
\begin{align}
\frac{1}{p\Delta_n^{2\alpha'}}\sum\limits_{\textbf{k}\in\N^d}\overline{D}_{\textbf{k},\textbf{k}}&=\frac{1}{p\Delta_n^{2\alpha'}}\sum\limits_{\textbf{k}\in\N^d} \frac{\big(1-e^{-2\lambda_\textbf{k}\Delta_n}\big)^2}{\lambda_\textbf{k}^{2(1+\alpha)}}\bigg(\frac{\big(1-e^{-\lambda_\textbf{k}\Delta_n}\big)^4\big(1+e^{-2\lambda_\textbf{k}\Delta_n}\big)}{\big(1-e^{-2\lambda_\textbf{k}\Delta_n}\big)^3}+2\frac{\big(1-e^{-\lambda_\textbf{k}\Delta_n}\big)^2}{1-e^{-2\lambda_\textbf{k}\Delta_n}}+1\bigg)\notag\\
&\leq \frac{4}{p\Delta_n^{2\alpha'}}\sum\limits_{\textbf{k}\in\N^d} \frac{\big(1-e^{-2\lambda_\textbf{k}\Delta_n}\big)^2}{\lambda_\textbf{k}^{2(1+\alpha)}}\notag\\
&=\frac{4\Delta_n^{d/2}}{p}\Delta_n^{d/2}\sum_{\textbf{k}\in\N^d}\bigg(\frac{1-e^{-2\lambda_\textbf{k}\Delta_n}}{(\lambda_\textbf{k}\Delta_n)^{1+\alpha}}\bigg)^2=
\Oo(p^{-1}\Delta_n^{2(1-\alpha')}),\label{eqn_negelct_kkTerminCov}
\end{align}
where we can use analogous steps as in Lemma \ref{lemma_checkConditionsAprroxlemma} to show that
\begin{align*}
\bigg(\frac{1-e^{-2x}}{x^{1+\alpha}}\bigg)^2=f_\alpha^2(x)\in \mathcal{Q}_\beta,~~~~~\text{with }\beta=\big(4\alpha,1+4\alpha,2+4\alpha\big).
\end{align*}
Hence, we have
\begin{align*}
\Cov\big(V_{p,\Delta_n}(\textbf{y}_1),V_{p,\Delta_n}(\textbf{y}_2)\big)
&=\frac{2\sigma^4e^{\nor{\kappa\bigcdot(\textbf{y}_1+\textbf{y}_2)}_1}}{p\Delta_n^{2\alpha'}}\sum_{\substack{\textbf{k}_1,\textbf{k}_2\in\N^d\\\textbf{k}_1\neq \textbf{k}_2}}e_{\textbf{k}_1}(\textbf{y}_1)e_{\textbf{k}_1}(\textbf{y}_2)e_{\textbf{k}_2}(\textbf{y}_1)e_{\textbf{k}_2}(\textbf{y}_2)\bar{D}_{\textbf{k}_1,\textbf{k}_2} \\
&~~~~~+\mathcal{O}\bigg(\frac{1}{p}\Big(\Delta_n^{2(1-\alpha')}+\frac{1}{p}\wedge1\Big)\bigg).
\end{align*}
We can represent the term $\bar{D}_{\textbf{k}_1,\textbf{k}_2}$ from equation \eqref{eqn_barDk1k2ref} as:
\begin{align*}
\bar{D}_{\textbf{k}_1,\textbf{k}_2}=\frac{\big(1-e^{-\lambda_{\textbf{k}_1}\Delta_n}\big)^2\big(1-e^{-\lambda_{\textbf{k}_2}
\Delta_n}\big)^2}{2\lambda_{\textbf{k}_1}^{1+\alpha}\lambda_{\textbf{k}_2}^{1+\alpha}}\sum_{r=0}^\infty e^{-r(\lambda_{\textbf{k}_1}+\lambda_{\textbf{k}_2})\Delta_n}+\frac{\big(1-e^{-\lambda_{\textbf{k}_1}\Delta_n}\big)\big(1-e^{-\lambda_{\textbf{k}_2}\Delta_n}\big)}{\lambda_{\textbf{k}_1}^{1+\alpha}\lambda_{\textbf{k}_2}^{1+\alpha}},
\end{align*}
and decompose as follows:
\begin{align*}
\bar{D}_{\textbf{k}_1,\textbf{k}_2}^1&:=\sum_{r=0}^\infty  \frac{\big(1-e^{-\lambda_{\textbf{k}_1}\Delta_n}\big)^2\big(1-e^{-\lambda_{\textbf{k}_2}
\Delta_n}\big)^2}{2\lambda_{\textbf{k}_1}^{1+\alpha}\lambda_{\textbf{k}_2}^{1+\alpha}}e^{-r(\lambda_{\textbf{k}_1}+\lambda_{\textbf{k}_2})\Delta_n},\\
\bar{D}_{\textbf{k}_1,\textbf{k}_2}^2&:=\frac{\big(1-e^{-\lambda_{\textbf{k}_1}\Delta_n}\big)\big(1-e^{-\lambda_{\textbf{k}_2}\Delta_n}\big)}{\lambda_{\textbf{k}_1}^{1+\alpha}\lambda_{\textbf{k}_2}^{1+\alpha}}.
\end{align*}
Assume $\textbf{y}_1\neq \textbf{y}_2$, then we have 
\begin{align}
e_{\textbf{k}}(\textbf{y}_1)e_{\textbf{k}}(\textbf{y}_2)
&=e^{-\nor{\kappa\bigcdot (\textbf{y}_1+\textbf{y}_2)}_1}\prod_{l=1}^d\Big(\cos\big(\pi k_l(y^{(1)}_l-y^{(2)}_l)\big)-\cos\big(\pi k_l(y^{(1)}_l+y^{(2)}_l)\big)\Big).\label{eqn_ekDecompInCosine}
\end{align}
Let $x_l^{(1)},x_l^{(2)}\in\{(y^{(1)}_l-y^{(2)}_l)/2,(y^{(1)}_l+y^{(2)}_l)/2\}$, then we find: 
\begin{align*}
&\frac{1}{p\Delta_n^{2\alpha'}}\sum_{\textbf{k}_1,\textbf{k}_2\in\N^d}\bar{D}_{\textbf{k}_1,\textbf{k}_2}^1\prod_{l=1}^d  \cos(2\pi k^{(1)}_lx_l^{(1)})\cos(2\pi k^{(2)}_lx_l^{(2)})\\
&=\frac{2}{p}\sum_{r=0}^\infty\bigg(\Delta_n^{d/2}\sum_{\textbf{k}_1\in\N^d}g_{\alpha,r}(\lambda_{\textbf{k}_1}\Delta_n)\prod_{l=1}^d\cos(2\pi k^{(1)}_lx_l^{(1)})\bigg)\bigg(\Delta_n^{d/2}\sum_{\textbf{k}_2\in\N^d}g_{\alpha,r}(\lambda_{\textbf{k}_2}\Delta_n)\prod_{l=1}^d\cos(2\pi k^{(2)}_lx_l^{(2)})\bigg).
\end{align*}
Note, that $\textbf{y}_1\neq \textbf{y}_2$ only implies that one coordinate $y_l^{(1)}\neq y_l^{(2)}$ differs. 
To analyse the order of one of the series in the last display, we can utilize Corollary \ref{corollary_toLemmaRiemannApproxMulti} (ii) and (iii) on the function $g_{\alpha,\tau}\in\mathcal{Q}_{(2\alpha,2(1+\alpha),2(1+\alpha))}$ from display \eqref{equation_functionsAlphaandTau}, which gives the following:
\begin{align*}
\Delta_n^{d/2}\sum_{\textbf{k}_2\in\N^d}g_{\alpha,r}(\lambda_{\textbf{k}_2}\Delta_n)\prod_{l=1}^d\cos(2\pi k^{(2)}_lx_l^{(2)})&=\Oo\bigg(\frac{\Delta_n^{1-\alpha'}}{\nor{\textbf{y}_1-\textbf{y}_2}_0^{d+1}}+\frac{\Delta_n^{1-\alpha'}}{\delta^{d+1}}\bigg).
\end{align*}
Here, we considered the case when $\textbf{y}_1\neq \textbf{y}_2$ differing in every component, i.e., we used the order from Lemma \ref{corollary_toLemmaRiemannApproxMulti} (iii) and took into account that $x_l$ can exceed and fall below the limit of $1-\delta$ and $\delta$, respectively, by inserting the bounds $\nor{\textbf{y}_1-\textbf{y}_2}_0$ and $\delta$. Hence, we have 
\begin{align}
&\frac{2}{p}\sum_{r=0}^\infty\bigg(\Delta_n^{d/2}\sum_{\textbf{k}_1\in\N^d}g_{\alpha,r}(\lambda_{\textbf{k}_1}\Delta_n)\prod_{l=1}^d\cos(2\pi k^{(1)}_lx_l^{(1)})\bigg)\bigg(\Delta_n^{d/2}\sum_{\textbf{k}_2\in\N^d}g_{\alpha,r}(\lambda_{\textbf{k}_2}\Delta_n)\prod_{l=1}^d\cos(2\pi k^{(2)}_lx_l^{(2)})\bigg)\notag\\
&=\Oo\bigg(\frac{\Delta_n^{1-\alpha'}}{p}\big(\nor{\textbf{y}_1-\textbf{y}_2}_0^{-(d+1)}+\delta^{-(d+1)}\big)\sum_{r=0}^\infty\Delta_n^{d/2}\sum_{\textbf{k}\in\N^d}|g_{\alpha,r}(\lambda_\textbf{k}\Delta_n)|\bigg)\notag\\
&=\Oo\bigg(\frac{\Delta_n^{1-\alpha'}}{p}\big(\nor{\textbf{y}_1-\textbf{y}_2}_0^{-(d+1)}+\delta^{-(d+1)}\big)\Big(\Delta_n^{d/2}\sum_{\textbf{k}\in\N^d}\frac{1-e^{-\lambda_\textbf{k}\Delta_n}}{(\lambda_\textbf{k}\Delta_n)^{1+\alpha}}\Big)\bigg)\notag\\
&=\Oo\bigg(\frac{\Delta_n^{1-\alpha'}}{p}\big(\nor{\textbf{y}_1-\textbf{y}_2}_0^{-(d+1)}+\delta^{-(d+1)}\big)\bigg).\label{eqn_covEksRemainder}
\end{align}
Analogously, we consider the second term $\bar{D}^2$ with the function $f_\alpha$ from equation \eqref{equation_functionsAlphaandTau}, which gives us the following:
\begin{align*}
\Cov\big(V_{p,\Delta_n}(\textbf{y}_1),V_{p,\Delta_n}(\textbf{y}_2)\big)&=\Oo\bigg(\frac{\Delta_n^{1-\alpha'}}{p}\big(\nor{\textbf{y}_1-\textbf{y}_2}_0^{-(d+1)}+\delta^{-(d+1)}\big)\bigg)+\mathcal{O}\bigg(\frac{1}{p}\Big(\Delta_n^{2(1-\alpha')}+\frac{1}{p}\wedge1\Big)\bigg)\\
&=\Oo\bigg(\frac{\Delta_n^{1-\alpha'}}{p}\big(\nor{\textbf{y}_1-\textbf{y}_2}_0^{-(d+1)}+\delta^{-(d+1)}\big)\bigg),
\end{align*}
for $\textbf{y}_1\neq \textbf{y}_2$.
Thus, it remains to compute the variance, where $\textbf{y}_1=\textbf{y}_2=\textbf{y}\in[\delta,1-\delta]^d$. Again, utilizing
\begin{align*}
e_{\textbf{k}}(\textbf{y})e_{\textbf{k}}(\textbf{y})&=e^{-2\nor{\kappa\bigcdot \textbf{y}}_1}\prod_{l=1}^d\Big(\cos(0)-\cos(2\pi k_ly_l)\Big),
\end{align*}
and having $x_l^{(1)},x_l^{(2)}\in\{0,y_l\}$, we infer analogously to display \eqref{eqn_covEksRemainder} that
\begin{align*}
&\frac{1}{p\Delta_n^{2\alpha'}}\sum_{\textbf{k}_1,\textbf{k}_2\in\N^d}\bar{D}_{\textbf{k}_1,\textbf{k}_2}^1\prod_{l=1}^d  \cos(2\pi k^{(1)}_lx_l^{(1)})\cos(2\pi k^{(2)}_lx_l^{(2)})\\
&=\frac{2}{p}\sum_{r=0}^\infty\bigg(\Delta_n^{d/2}\sum_{\textbf{k}_1\in\N^d}g_{\alpha,r}(\lambda_{\textbf{k}_1}\Delta_n)\prod_{l=1}^d\cos(2\pi k^{(1)}_lx_l^{(1)})\bigg)\bigg(\Delta_n^{d/2}\sum_{\textbf{k}_2\in\N^d}g_{\alpha,r}(\lambda_{\textbf{k}_2}\Delta_n)\prod_{l=1}^d\cos(2\pi k^{(2)}_lx_l^{(2)})\bigg).
\end{align*} 
Now assume, without loss of generality, that $\sum_{j=1}^d\mathbbm{1}_{\{x_j^{(1)}\neq 0\}}=l$, for $1\leq l\leq d$. Then, by Corollary \ref{corollary_toLemmaRiemannApproxMulti} (ii) and (iii), we have
\begin{align*}
\Delta_n^{d/2}\sum_{\textbf{k}_2\in\N^d}g_{\alpha,r}(\lambda_{\textbf{k}_2}\Delta_n)\prod_{l=1}^d\cos(2\pi k^{(2)}_lx_l^{(2)})&=\Oo\bigg(\Delta_n^{l/2}\vee\frac{\Delta_n^{1-\alpha'}}{\delta^{d+1}}\bigg).
\end{align*}
Hence, within this setting, we conclude that
\begin{align*}
&\frac{1}{p\Delta_n^{2\alpha'}}\sum_{\textbf{k}_1,\textbf{k}_2\in\N^d}\bar{D}_{\textbf{k}_1,\textbf{k}_2}^1\prod_{l=1}^d  \cos(2\pi k^{(1)}_lx_l^{(1)})\cos(2\pi k^{(2)}_lx_l^{(2)})=\Oo\bigg(\frac{1}{p}\Big(\Delta_n^{1/2}\vee \frac{\Delta_n^{1-\alpha'}}{\delta^{d+1}}\Big)\bigg),
\end{align*}
and it follows that
\begin{align}
\Var\big(V_{p,\Delta_n}(\textbf{y})\big)&=\frac{2\sigma^4e^{2\nor{\kappa\bigcdot \textbf{y}}_1}}{p\Delta_n^{2\alpha'}}\sum_{\substack{\textbf{k}_1,\textbf{k}_2\in\N^d\\\textbf{k}_1\neq \textbf{k}_2}}e^2_{\textbf{k}_1}(\textbf{y})e^2_{\textbf{k}_2}(\textbf{y})\bar{D}_{\textbf{k}_1,\textbf{k}_2}+\mathcal{O}\bigg(\frac{1}{p}\Big(\Delta_n^{2(1-\alpha')}+\frac{1}{p}\wedge1\Big)\bigg)\notag\\
&=\frac{2\sigma^4}{p\Delta_n^{2\alpha'}}\sum_{\substack{\textbf{k}_1,\textbf{k}_2\in\N^d\\\textbf{k}_1\neq \textbf{k}_2}}\bar{D}_{\textbf{k}_1,\textbf{k}_2}+\mathcal{O}\bigg(\frac{1}{p}\Big(\Delta_n^{1/2}\vee \frac{\Delta_n^{1-\alpha'}}{\delta^{d+1}}+\frac{1}{p}\wedge1\Big)\bigg).\label{eqn_rrvPropEqnToLastTerm}
\end{align}
For the leading term we obtain:
\begin{align*}
\frac{2\sigma^4}{p\Delta_n^{2\alpha'}}\sum_{\substack{\textbf{k}_1,\textbf{k}_2\in\N^d\\\textbf{k}_1\neq \textbf{k}_2}}\bar{D}_{\textbf{k}_1,\textbf{k}_2}&=\frac{\sigma^4}{p}\bigg(\sum_{r=0}^\infty\Big(2\Delta_n^{d/2}\sum_{\textbf{k}\in\N^d}\frac{(1-e^{-\lambda_\textbf{k}\Delta_n})^2}{2(\lambda_\textbf{k}\Delta_n)^{1+\alpha}}e^{-r\lambda_\textbf{k}\Delta_n}\Big)^2+2\Big(\Delta_n^{d/2}\sum_{\textbf{k}\in\N^d}\frac{1-e^{-\lambda_\textbf{k}\Delta_n}}{(\lambda_\textbf{k}\Delta_n)^{1+\alpha}}\Big)^2\bigg)\\
&=\frac{\sigma^4}{p}\bigg(\sum_{r=0}^\infty\Big(2\Delta_n^{d/2}\sum_{\textbf{k}\in\N^d}g_{\alpha,r}(\lambda_{\textbf{k}}\Delta_n)\Big)^2+2\Big(\Delta_n^{d/2}\sum_{\textbf{k}\in\N^d}f_\alpha(\lambda_\textbf{k}\Delta_n)\Big)^2\bigg),
\end{align*}
and by Lemma \ref{lemma_calcfAlphaDelta} we have 
\begin{align*}
\Var\big(V_{p,\Delta_n}(\textbf{y})\big)&=\frac{1}{p}\bigg(\frac{\Gamma(1-\alpha')\sigma^2}{2^d(\pi\eta)^{d/2}\alpha'\Gamma(d/2)}\bigg)^2\bigg(\sum_{r=0}^\infty\big(-r ^{\alpha'}+2 (r +1)^{\alpha'}-(r +2)^{\alpha'}\big)^2 +2\bigg)\\
&~~~~~+\mathcal{O}\bigg(\frac{1}{p}\Big(\Delta_n^{1/2}\vee \frac{\Delta_n^{1-\alpha'}}{\delta^{d+1}}+\frac{1}{p}\wedge1\Big)\bigg).
\end{align*}
Defining the constant
\begin{align}
\Upsilon_{\alpha'}:=\bigg(\sum_{r=0}^\infty\big(-r ^{\alpha'}+2 (r +1)^{\alpha'}-(r +2)^{\alpha'}\big)^2 +2\bigg)\label{eqn_definingUpsilon}
\end{align}
completes the proof.
\end{proof}
The following proposition and corollary prove the general mixing-type Condition (IV) from Proposition \ref{prop_clt_utev}.
\begin{theorem}\label{porp_timeDepenMulit}
Grant the Assumptions \ref{assumption_observations_multi} and \ref{assumption_regMulti}.
Let $y\in [\delta,1-\delta]^d$ for a $\delta>0$, $1\leq r<r+u\leq v\leq n$ natural numbers and 
\begin{align*}
Q_1^r=\sum\limits_{i=1}^r(\Delta_i\tilde{X})^2(\textbf{y}),~~~~~~~~Q_{r+u}^v=\sum\limits_{i=r+u}^v(\Delta_i\tilde{X})^2(\textbf{y}),
\end{align*}
then there exists a constant $C$, where $0<C<\infty$, such that it holds for all $t\in \R$ that
\begin{align*}
\abs{\Cov\bigg(e^{\im t(Q_1^r-\E[Q_1^r])},~e^{\im t(Q_{r+u}^v-\E[Q_{r+u}^v])}\bigg)}\leq \frac{Ct^2}{u^{1-\alpha'/2}}\sqrt{\Var(Q_1^r)\Var(Q_{r+u}^v)}.
\end{align*}
\end{theorem}
\begin{proof}
Assume $Q^v_{r+u}=A_1+A_2$ with some $A_2$ which is independent to $Q_1^r$. Then we know by \cite[Prop.\ 6.6.]{trabs} that
\begin{align*}
\Cov\Big(e^{\im t\bar{Q}_1^r},e^{\im t\bar{Q}_{r+u}^v}\Big)\leq 2t^2\E\Big[\big(\bar{Q}_1^r\big)^2\Big]^{1/2}\E\Big[\big(\bar{A}_1\big)^2\Big] ^{1/2},
\end{align*}
where $\bar{X}=X-\E[X]$. For $r\leq i-1$ we obtain:
\begin{align*}
\Delta_i\tilde{X}(\textbf{y})&=\sum\limits_{\textbf{k}\in\N^d}\bigg(\sigma\lambda_\textbf{k}^{-\alpha/2}\int\limits_{-\infty}^{r\Delta_n}  e^{-\lambda_\textbf{k}\big(\iidn-s\big)}\big(e^{-\lambda_\textbf{k}\Delta_n}-1\big) \diff W_s^\textbf{k}\bigg)e_\textbf{k}(\textbf{y}) \\
&~~~~~+\sum\limits_{\textbf{k}\in\N^d}\bigg(\sigma\lambda_\textbf{k}^{-\alpha/2} \int\limits_{r\Delta_n}^{\iidn}  e^{-\lambda_\textbf{k}\big(\iidn-s\big)}\big(e^{-\lambda_\textbf{k}\Delta_n}-1\big) \diff W_s^\textbf{k}\\
&~~~~~+ \sigma\lambda_\textbf{k}^{-\alpha/2}\int\limits_{\iidn}^{\idn}e^{-\lambda_\textbf{k}(\idn-s)} \diff W_s ^\textbf{k}\Bigg)e_\textbf{k}(\textbf{y}) \\
&=\sum\limits_{\textbf{k}\in\N^d} D_1^{\textbf{k},i}e_\textbf{k}(\textbf{y})+\sum\limits_{\textbf{k}\in\N^d} D_2^{\textbf{k},i}e_\textbf{k}(\textbf{y}),
\end{align*}
where
\begin{align}
D_1^{\textbf{k},i}&:= \sigma\lambda_\textbf{k}^{-\alpha/2}\int\limits_{-\infty}^{r\Delta_n}  e^{-\lambda_\textbf{k}\big(\iidn-s\big)}\big(e^{-\lambda_\textbf{k}\Delta_n}-1\big) \diff W_s^\textbf{k},\label{number76:D1}\\
D_2^{\textbf{k},i}&:=\sigma\lambda_\textbf{k}^{-\alpha/2}\int\limits_{r\Delta_n}^{\iidn}  e^{-\lambda_\textbf{k}\big(\iidn-s\big)}\big(e^{-\lambda_\textbf{k}\Delta_n}-1\big) \diff W_s^\textbf{k} + \sigma\lambda_\textbf{k}^{-\alpha/2}\int\limits_{\iidn}^{\idn}e^{-\lambda_\textbf{k}(\idn-s)} \diff W_s ^\textbf{k}.\label{number76:D2}
\end{align}
We can establish that $D_1^{\textbf{k},i}$ and $D_2^{\textbf{k},i}$ are independent, thus yielding the following result:
\begin{align*}
Q_{r+u}^v
&=\sum\limits_{i=r+u}^v \bigg(\sum\limits_{\textbf{k}\in\N^d} D_1^{\textbf{k},i}e_\textbf{k}(\textbf{y})\bigg)^2+ 2\sum\limits_{i=r+u}^v \bigg(\sum\limits_{\textbf{k}\in\N^d} D_1^{\textbf{k},i}e_\textbf{k}(\textbf{y})\bigg)\bigg(\sum\limits_{\textbf{k}\in\N^d} D_2^{\textbf{k},i}e_\textbf{k}(\textbf{y})\bigg) +\sum\limits_{i=r+u}^v \bigg(\sum\limits_{\textbf{k}\in\N^d} D_2^{\textbf{k},i}e_\textbf{k}(\textbf{y})\bigg)^2,
\end{align*}
which implies the following decomposition:
\begin{align*}
A_1&:= \sum\limits_{i=r+u}^v \bigg(\sum\limits_{\textbf{k}\in\N^d} D_1^{\textbf{k},i}e_\textbf{k}(\textbf{y})\bigg)^2+ 2\sum\limits_{i=r+u}^v \bigg(\sum\limits_{\textbf{k}\in\N^d} D_1^{\textbf{k},i}e_\textbf{k}(\textbf{y})\bigg)\bigg(\sum\limits_{\textbf{k}\in\N^d}^\infty D_2^{\textbf{k},i}e_\textbf{k}(\textbf{y})\bigg), \\
A_2&:=\sum\limits_{i=r+u}^v \bigg(\sum\limits_{\textbf{k}\in\N^d} D_2^{\textbf{k},i}e_\textbf{k}(\textbf{y})\bigg)^2,
\end{align*}
where $A_2$ is independent to $Q_1^r$. 
Hence, our focus shifts to bounding the term $\E[\bar{A}_1^2]$, which is equivalent to computing $\Var(A_1)$. We begin with the following considerations:
\begin{align*}
\E[\bar{A}_1^2]&\leq \E[A_1^2]\\
&= \E\left[\Bigg(\sum\limits_{i=r+u}^v \bigg(\sum\limits_{\textbf{k}\in\N^d} D_1^{\textbf{k},i}e_\textbf{k}(\textbf{y})\bigg)^2+ 2\sum\limits_{i=r+u}^v \bigg(\sum\limits_{\textbf{k}\in\N^d} D_1^{\textbf{k},i}e_\textbf{k}(\textbf{y})\bigg)\bigg(\sum\limits_{\textbf{k}\in\N^d} D_2^{\textbf{k},i}e_\textbf{k}(\textbf{y})\bigg)\Bigg)^2\right] \\
&= \sum\limits_{i,j=r+u}^v\E\Bigg[\bigg(\sum\limits_{\textbf{k}\in\N^d} D_1^{\textbf{k},i}e_\textbf{k}(\textbf{y})\bigg)^2\bigg(\sum\limits_{\textbf{k}\in\N^d} D_1^{\textbf{k},j}e_\textbf{k}(\textbf{y})\bigg)^2\Bigg] \\
&~~~~~+ 4\sum\limits_{i,j=r+u}^v\E\Bigg[\bigg(\sum\limits_{\textbf{k}\in\N^d} D_1^{\textbf{k},i}e_\textbf{k}(\textbf{y})\bigg)^2\bigg(\sum\limits_{\textbf{k}\in\N^d} D_1^{\textbf{k},j}e_\textbf{k}(\textbf{y})\bigg)\bigg(\sum\limits_{\textbf{k}\in\N^d} D_2^{\textbf{k},j}e_\textbf{k}(\textbf{y})\bigg)\Bigg] \\
&~~~~~+4\sum\limits_{i,j=r+u}^v\E\Bigg[\bigg(\sum\limits_{\textbf{k}\in\N^d} D_1^{\textbf{k},i}e_\textbf{k}(\textbf{y})\bigg)\bigg(\sum\limits_{\textbf{k}\in\N^d} D_2^{\textbf{k},i}e_\textbf{k}(\textbf{y})\bigg)
\bigg(\sum\limits_{\textbf{k}\in\N^d} D_1^{\textbf{k},j}e_\textbf{k}(\textbf{y})\bigg)\bigg(\sum\limits_{\textbf{k}\in\N^d} D_2^{\textbf{k},j}e_\textbf{k}(\textbf{y})\bigg)\Bigg],
\end{align*}
where the cross-term between $D_1^{\textbf{k},i},D_2^{\textbf{k},i}$ vanishes as both terms are centred normally distributed. Therefore, we use $\E[\bar{A}_1^2]\leq T_1+4T_2$, where we define:
\begin{align}
T_1&:=\sum\limits_{i,j=r+u}^v\E\bigg[\Big(\sum\limits_{\textbf{k}\in\N^d} D_1^{\textbf{k},i}e_\textbf{k}(\textbf{y})\Big)^2\Big(\sum\limits_{\textbf{k}\in\N^d} D_1^{\textbf{k},j}e_\textbf{k}(\textbf{y})\Big)^2\bigg],\label{number76:T1} \\
T_2&:=\sum\limits_{i,j=r+u}^v\E\bigg[\Big(\sum\limits_{\textbf{k}\in\N^d} D_1^{\textbf{k},i}e_\textbf{k}(\textbf{y})\Big)\Big(\sum\limits_{\textbf{k}\in\N^d} D_2^{\textbf{k},i}e_\textbf{k}(\textbf{y})\Big)
\Big(\sum\limits_{\textbf{k}\in\N^d} D_1^{\textbf{k},j}e_\textbf{k}(\textbf{y})\Big)\Big(\sum\limits_{\textbf{k}\in\N^d} D_2^{\textbf{k},j}e_\textbf{k}(\textbf{y})\Big)\bigg].\label{number76:T2}
\end{align}
To bound the term $T_1$, we can utilize the expression $D_1^{\textbf{k},i}=e^{-\lambda_\textbf{k}(i-r-1)\Delta_n}\tilde{B}_{r+1,\textbf{k}}$, where $\tilde{B}_{i,\textbf{k}}$ is defined in equation \eqref{eqn_tildeB_multi}, leading to the following calculation:
\begin{align*}
T_1
&=\sum\limits_{i,j=r+u}^v~~\sum\limits_{\textbf{k}_1,\textbf{k}_2,\textbf{k}_3,\textbf{k}_4\in\N^d}\E\Big[e^{-\lambda_{\textbf{k}_1}(i-r-1)\Delta_n}\tilde{B}_{r+1,\textbf{k}_1}e_{\textbf{k}_1}(\textbf{y})e^{-\lambda_{\textbf{k}_2}(i-r-1)\Delta_n}\tilde{B}_{r+1,\textbf{k}_2}e_{\textbf{k}_2}(\textbf{y})\\
&~~~~~\times e^{-\lambda_{\textbf{k}_3}(j-r-1)\Delta_n}\tilde{B}_{r+1,\textbf{k}_3}e_{\textbf{k}_3}(\textbf{y}) e^{-\lambda_{\textbf{k}_4}(j-r-1)\Delta_n}\tilde{B}_{r+1,\textbf{k}_4}e_{\textbf{k}_4}(\textbf{y})\Big].
\end{align*}
Note, that any combination of indices results in a value of zero, unless, exactly two indices are the same, or all four indices are equal.
Thus, we obtain for $\textbf{k}_1=\ldots=\textbf{k}_4=\textbf{k}$ that
\begin{align*}
\sum\limits_{i,j=r+u}^v~~\sum\limits_{\textbf{k}\in\N^d} e^{-2\lambda_\textbf{k}(i+j-2r-2)\Delta_n}\E\Big[\tilde{B}_{r+1,\textbf{k}}^4\Big]e_\textbf{k}^4(\textbf{y}).
\end{align*}
For $\textbf{k}_1=\textbf{k}_2$ and $\textbf{k}_3=\textbf{k}_4$, with $\textbf{k}_1\neq \textbf{k}_3$ we find that
\begin{align*}
&\sum\limits_{i,j=r+u}^v\sum\limits_{\substack{\textbf{k}_1,\textbf{k}_2\in\N^d\\ \textbf{k}_1\neq \textbf{k}_2}} e^{-2\lambda_{\textbf{k}_1}(i-r-1)\Delta_n}e^{-2\lambda_{\textbf{k}_2}(j-r-1)\Delta_n}\E\Big[\tilde{B}_{r+1,\textbf{k}_1}^2\tilde{B}_{r+1,\textbf{k}_2}^2\Big]e_{\textbf{k}_1}^2(\textbf{y})e_{\textbf{k}_2}^2(\textbf{y}) \\
&=\sum\limits_{i,j=r+u}^v\sum\limits_{\substack{\textbf{k}_1,\textbf{k}_2\in\N^d\\ \textbf{k}_1\neq \textbf{k}_2}} e^{-2\lambda_{\textbf{k}_1}(i-r-1)\Delta_n-2\lambda_{\textbf{k}_2}(j-r-1)\Delta_n}\E\Big[\tilde{B}_{r+1,\textbf{k}_1}^2\Big]\E\Big[\tilde{B}_{r+1,\textbf{k}_2}^2\Big]e_{\textbf{k}_1}^2(\textbf{y})e_{\textbf{k}_2}^2(\textbf{y}).
\end{align*}
The remaining combinations yield the following:
\begin{align*}
&\sum\limits_{i,j=r+u}^v\sum\limits_{\substack{\textbf{k}_1,\textbf{k}_2\in\N^d\\ \textbf{k}_1\neq \textbf{k}_2}} e^{-\lambda_{\textbf{k}_1}(i+j-2r-2)\Delta_n}e^{-\lambda_{\textbf{k}_2}(i+j-2r-2)\Delta_n}\E\Big[\tilde{B}_{r+1,\textbf{k}_1}^2\tilde{B}_{r+1,\textbf{k}_2}^2\Big]e_{\textbf{k}_1}^2(\textbf{y})e_{\textbf{k}_2}^2(\textbf{y}) \\
&=\sum\limits_{i,j=r+u}^v\sum\limits_{\substack{\textbf{k}_1,\textbf{k}_2\in\N^d\\ \textbf{k}_1\neq \textbf{k}_2}}^\infty e^{-(\lambda_{\textbf{k}_1}+\lambda_{\textbf{k}_2})(i+j-2r-2)\Delta_n}\E\Big[\tilde{B}_{r+1,\textbf{k}_1}^2\Big]\E\Big[\tilde{B}_{r+1,\textbf{k}_2}^2\Big]e_{\textbf{k}_1}^2(\textbf{y})e_{\textbf{k}_2}^2(\textbf{y}),
\end{align*}
and we observe:
\begin{align*}
T_1&= \sum\limits_{i,j=r+u}^v\sum\limits_{\substack{\textbf{k}_1,\textbf{k}_2\in\N^d\\ \textbf{k}_1\neq \textbf{k}_2}} e^{-2\lambda_{\textbf{k}_1}(i-r-1)\Delta_n-2\lambda_{\textbf{k}_2}(j-r-1)\Delta_n}\E\Big[\tilde{B}_{r+1,\textbf{k}_1}^2\Big]\E\Big[\tilde{B}_{r+1,\textbf{k}_2}^2\Big]e_{\textbf{k}_1}^2(\textbf{y})e_{\textbf{k}_2}^2(\textbf{y})\\
&~~~~~+2\sum\limits_{i,j=r+u}^v\sum\limits_{\substack{\textbf{k}_1,\textbf{k}_2\in\N^d\\ \textbf{k}_1\neq \textbf{k}_2}} e^{-(\lambda_{\textbf{k}_1}+\lambda_{\textbf{k}_2})(i+j-2r-2)\Delta_n}\E\Big[\tilde{B}_{r+1,\textbf{k}_1}^2\Big]\E\Big[\tilde{B}_{r+1,\textbf{k}_2}^2\Big]e_{\textbf{k}_1}^2(\textbf{y})e_{\textbf{k}_2}^2(\textbf{y})\\
&~~~~~+\sum\limits_{i,j=r+u}^v~~\sum\limits_{\textbf{k}\in\N^d} e^{-2\lambda_\textbf{k}(i+j-2r-2)\Delta_n}\E\Big[\tilde{B}_{r+1,\textbf{k}}^4\Big]e_\textbf{k}^4(\textbf{y}) \\
&=\sigma^4\sum\limits_{\substack{\textbf{k}_1,\textbf{k}_2\in\N^d\\ \textbf{k}_1\neq \textbf{k}_2}} \frac{\big(1-e^{-\lambda_{\textbf{k}_1}\Delta_n}\big)^2\big(1-e^{-\lambda_{\textbf{k}_2}\Delta_n}\big)^2}{4\lambda_{\textbf{k}_1}^{1+\alpha}\lambda_{\textbf{k}_2}^{1+\alpha}}\Bigg(\sum\limits_{i=r+u}^v e^{-2\lambda_{\textbf{k}_1}(i-r-1)\Delta_n}\Bigg)\Bigg(\sum\limits_{j=r+u}^v e^{-2\lambda_{\textbf{k}_2}(j-r-1)\Delta_n}\Bigg)\\
&~~~~~~~~~~\times e_{\textbf{k}_1}^2(\textbf{y})e_{\textbf{k}_2}^2(\textbf{y})\\
&~~~~~+\sigma^4\sum\limits_{\substack{\textbf{k}_1,\textbf{k}_2\in\N^d\\ \textbf{k}_1\neq \textbf{k}_2}} \frac{\big(1-e^{-\lambda_{\textbf{k}_1}\Delta_n}\big)^2\big(1-e^{-\lambda_{\textbf{k}_2}\Delta_n}\big)^2}{4\lambda_{\textbf{k}_1}^{1+\alpha}\lambda_{\textbf{k}_2}^{1+\alpha}}2\Bigg(\sum\limits_{i=r+u}^v e^{-(\lambda_{\textbf{k}_1}+\lambda_{\textbf{k}_2})(i-r-1)\Delta_n}\Bigg)^2 e_{\textbf{k}_1}^2(\textbf{y})e_{\textbf{k}_2}^2(\textbf{y})\\
&~~~~~+3\sigma^4\sum\limits_{\textbf{k}\in\N^d}^\infty \frac{\big(1-e^{-\lambda_\textbf{k}\Delta_n}\big)^4}{4\lambda_\textbf{k}^{2(1+\alpha)}} \Bigg(\sum\limits_{i=r+u}^v e^{-2\lambda_\textbf{k}(i-r-1)\Delta_n}\Bigg)^2e_\textbf{k}^4(\textbf{y}),
\end{align*}
where we used equation \eqref{eqn_CovBBtilde}, which implies:
\begin{align*}
\E\Big[\tilde{B}_{r+1,\textbf{k}}^2\Big]=\frac{\sigma^2}{2\lambda_{\textbf{k}}^{1+\alpha}}\big(1-e^{-\lambda_{\textbf{k}}\Delta_n}\big)^2.
\end{align*}
Let $\bar{p}=v-r-u+1$ and $u\geq 2$. 
We begin by bounding the eigenfunctions $(e_k)$ with a suitable constant $C>0$. Additionally, we have the following:
\begin{align*}
\sum\limits_{i=r+u}^v e^{-2\lambda_{\textbf{k}}(i-r-1)\Delta_n}
&= e^{-2\lambda_{\textbf{k}}(u-1)\Delta_n}\frac{1-e^{-2\lambda_{\textbf{k}}\Delta_n\bar{p}}}{1-e^{-2\lambda_{\textbf{k}}\Delta_n}}, \\
\sum\limits_{i=r+u}^v e^{-(\lambda_{\textbf{k}_1}+\lambda_{\textbf{k}_2})(i-r-1)\Delta_n}&=e^{-(\lambda_{\textbf{k}_1}+\lambda_{\textbf{k}_2})(u-1)\Delta_n}\frac{1-e^{-(\lambda_{\textbf{k}_1}+\lambda_{\textbf{k}_2})\Delta_n \bar{p}}}{1-e^{-(\lambda_{\textbf{k}_1}+\lambda_{\textbf{k}_2})\Delta_n}}.
\end{align*}
Thus, we obtain:
\begin{align*}
\Bigg(\sum\limits_{i=r+u}^v e^{-2\lambda_{\textbf{k}_1}(i-r-1)\Delta_n}\Bigg)\Bigg(\sum\limits_{j=r+u}^v e^{-2\lambda_{\textbf{k}_2}(j-r-1)\Delta_n}\Bigg)
&\leq e^{-2(\lambda_{\textbf{k}_1}+\lambda_{\textbf{k}_2})(u-1)\Delta_n}\frac{1-e^{-2\lambda_{\textbf{k}_2}\Delta_n\bar{p}}}{\big(1-e^{-2\lambda_{\textbf{k}_1}\Delta_n}\big)\big(1-e^{-2\lambda_{\textbf{k}_2}\Delta_n}\big)} \\
&\leq e^{-2(\lambda_{\textbf{k}_1}+\lambda_{\textbf{k}_2})(u-1)\Delta_n}\frac{\bar{p}}{\big(1-e^{-2\lambda_{\textbf{k}_1}\Delta_n}\big)},
\end{align*}
as well as 
\begin{align*}
\Bigg(\sum\limits_{i=r+u}^v e^{-(\lambda_{\textbf{k}_1}+\lambda_{\textbf{k}_2})(i-r-1)\Delta_n}\Bigg)^2
&\leq e^{-2(\lambda_{\textbf{k}_1}+\lambda_{\textbf{k}_2})(u-1)\Delta_n}\frac{\bar{p}}{1-e^{-(\lambda_{\textbf{k}_1}+\lambda_{\textbf{k}_2})\Delta_n}}, \\
\Bigg(\sum\limits_{i=r+u}^v e^{-2\lambda_{\textbf{k}}(i-r-1)\Delta_n}\Bigg)^2
&\leq e^{-4\lambda_{\textbf{k}}(u-1)\Delta_n}\frac{\bar{p}}{1-e^{-2\lambda_{\textbf{k}}\Delta_n}}.
\end{align*}
Finally, we conclude with the following calculations:
\begin{align*}
T_1&\leq C^4\sigma^4\bigg(\sum\limits_{\substack{\textbf{k}_1,\textbf{k}_2\in\N^d\\ \textbf{k}_1\neq \textbf{k}_2}} \frac{\big(1-e^{-\lambda_{\textbf{k}_1}\Delta_n}\big)^2\big(1-e^{-\lambda_{\textbf{k}_2}\Delta_n}\big)^2}{4\lambda_{\textbf{k}_1}^{1+\alpha}\lambda_{\textbf{k}_2}^{1+\alpha}}e^{-2(\lambda_{\textbf{k}_1}+\lambda_{\textbf{k}_2})(u-1)\Delta_n}\frac{\bar{p}}{\big(1-e^{-2\lambda_{\textbf{k}_1}\Delta_n}\big)}\\
&~~~~~+\sum\limits_{\substack{\textbf{k}_1,\textbf{k}_2\in\N^d\\ \textbf{k}_1\neq \textbf{k}_2}} \frac{\big(1-e^{-\lambda_{\textbf{k}_1}\Delta_n}\big)^2\big(1-e^{-\lambda_{\textbf{k}_2}\Delta_n}\big)^2}{4\lambda_{\textbf{k}_1}^{1+\alpha}\lambda_{\textbf{k}_2}^{1+\alpha}}e^{-2(\lambda_{\textbf{k}_1}+\lambda_{\textbf{k}_2})(u-1)\Delta_n}\frac{2\bar{p}}{1-e^{-(\lambda_{\textbf{k}_1}+\lambda_{\textbf{k}_2})\Delta_n}}\\
&~~~~~+3\sum\limits_{\textbf{k}\in\N^d}^\infty \frac{\big(1-e^{-\lambda_\textbf{k}\Delta_n}\big)^4}{4\lambda_\textbf{k}^{2(1+\alpha)}}e^{-4\lambda_{\textbf{k}}(u-1)\Delta_n}\frac{\bar{p}}{1-e^{-2\lambda_{\textbf{k}}\Delta_n}}
\bigg)\\
&\leq C^4\sigma^43\bar{p}\bigg(\sum_{\textbf{k}_1\in\N^d}\frac{\big(1-e^{-\lambda_{\textbf{k}_1}\Delta_n}\big)}{2\lambda_{\textbf{k}_1}^{1+\alpha}}e^{-2\lambda_{\textbf{k}_1}(u-1)\Delta_n}\bigg)\bigg(\sum_{\textbf{k}_2\in\N^d}\frac{\big(1-e^{-\lambda_{\textbf{k}_2}\Delta_n}\big)^2}{2\lambda_{\textbf{k}_2}^{1+\alpha}}e^{-2\lambda_{\textbf{k}_2}(u-1)\Delta_n}\bigg)\\
&\leq C'\sigma^4\bar{p}\Delta_n^{2\alpha'}\bigg(\int_0^\infty x^{d/2-1}\frac{(1-e^{-x})}{2x^{1+\alpha}}e^{-2x(u-1)}\diff x\bigg)\bigg(\int_0^\infty x^{d/2-1}\frac{(1-e^{-x})^2}{2x^{1+\alpha}}e^{-2x(u-1)}\diff x\bigg).
\end{align*}
Utilizing analogous steps as for Lemma \ref{lemma_calcfAlphaDelta}, we obtain for both integrals that
\begin{align*}
\int_0^\infty x^{d/2-1}\frac{(1-e^{-x})^l}{2x^{1+\alpha}}e^{-2x\tau}\diff x=\Oo\bigg(\frac{1}{\tau^{l-\alpha'}}\bigg),
\end{align*}
where $l=1,2$.
Therefore, we conclude
\begin{align*}
T_1\leq C\sigma^4\frac{\bar{p}\Delta_n^{2\alpha'}}{(u-1)^{3-2\alpha'}},
\end{align*}
for a suitable $C>0$.
For the term $T_2$, according to equation \eqref{number76:T2}, we have the following expression:
\begin{align*}
T_2&=\sum\limits_{i,j=r+u}^v\E\bigg[\Big(\sum\limits_{\textbf{k}\in\N^d} D_1^{\textbf{k},i}e_\textbf{k}(\textbf{y})\Big)\Big(\sum\limits_{\textbf{k}\in\N^d} D_2^{\textbf{k},i}e_\textbf{k}(\textbf{y})\Big)
\Big(\sum\limits_{\textbf{k}\in\N^d} D_1^{\textbf{k},j}e_\textbf{k}(\textbf{y})\Big)\Big(\sum\limits_{\textbf{k}\in\N^d} D_2^{\textbf{k},j}e_\textbf{k}(\textbf{y})\Big)\bigg]\\
&=\sum\limits_{i,j=r+u}^v\bigg(\sum\limits_{\textbf{k}\in\N^d}\E\big[D_1^{\textbf{k},i}D_1^{\textbf{k},j}\big]e_\textbf{k}^2(\textbf{y})\bigg)\bigg(\sum\limits_{\textbf{k}\in\N^d}\E\big[D_2^{\textbf{k},i}D_2^{\textbf{k},j}\big]e_\textbf{k}^2(\textbf{y})\bigg).
\end{align*}
For the first expected value, we find that
\begin{align*}
\E\big[D_1^{\textbf{k},i}D_1^{\textbf{k},j}\big]&=\sigma^2\lambda_\textbf{k}^{-\alpha}(1-e^{-\lambda_\textbf{k}\Delta_n})^2e^{-\lambda_\textbf{k}(i+j-2)\Delta_n}\int_{-\infty}^{r\Delta_n}e^{2\lambda_\textbf{k}s}\diff s\\
&=\sigma^2\frac{(1-e^{-\lambda_\textbf{k}\Delta_n})^2}{2\lambda_\textbf{k}^{1+\alpha}}e^{-\lambda_\textbf{k}(i+j-2r-2)\Delta_n}.
\end{align*}
The second expected value calculates for $i\leq j$ as follows:
\begin{align*}
\E\big[D_2^{\textbf{k},i}D_2^{\textbf{k},j}\big]&=\E\bigg[\Big(\sigma\lambda_\textbf{k}^{-\alpha/2}\int\limits_{r\Delta_n}^{\iidn}  e^{-\lambda_\textbf{k}\big(\iidn-s\big)}\big(e^{-\lambda_\textbf{k}\Delta_n}-1\big) \diff W_s^\textbf{k} + C_{i,\textbf{k}}\Big)\\
&~~~~~\times\Big(\sigma\lambda_\textbf{k}^{-\alpha/2}\int\limits_{r\Delta_n}^{(j-1)\Delta_n}  e^{-\lambda_\textbf{k}\big((j-1)\Delta_n-s\big)}\big(e^{-\lambda_\textbf{k}\Delta_n}-1\big) \diff W_s^\textbf{k} +C_{j,\textbf{k}}\Big)\bigg]\\
&=\sigma^2\frac{(1-e^{-\lambda_\textbf{k}\Delta_n})^2}{2\lambda_\textbf{k}^{1+\alpha}}\big(e^{-\lambda_\textbf{k}(j-i)\Delta_n}-e^{-\lambda_\textbf{k}(i+j-2r-2)\Delta_n}\big)+\Sigma_{j,i}^{BC,\textbf{k}}+\Sigma_{i,j}^{C,\textbf{k}}.
\end{align*}
As discussed in Proposition \ref{prop_autocovOfIncrementsMulti}, we find that $\Sigma_{j,i}^{BC,\textbf{k}}=0$, when $i=j$, and $\Sigma_{i,j}^{C,\textbf{k}}=0$, when $i\neq j$. In particular, for the case when $i<j$, we have the following expression:
\begin{align*}
\E\big[D_2^{\textbf{k},i}D_2^{\textbf{k},j}\big]&=\sigma^2\frac{(1-e^{-\lambda_\textbf{k}\Delta_n})^2}{2\lambda_\textbf{k}^{1+\alpha}}\big(e^{-\lambda_\textbf{k}(j-i)\Delta_n}-e^{-\lambda_\textbf{k}(i+j-2r-2)\Delta_n}\big)\\
&~~~~~+\sigma^2e^{-\lambda_\textbf{k}\Delta_n(j-i)}\Big(e^{\lambda_\textbf{k}\Delta_n}-e^{-\lambda_\textbf{k}\Delta_n}\Big)\frac{e^{-\lambda_\textbf{k}\Delta_n}-1}{2\lambda_\textbf{k}^{1+\alpha}}\\
&\leq \sigma^2e^{-\lambda_\textbf{k}(j-i)\Delta_n}\frac{1-e^{-\lambda_\textbf{k}\Delta_n}}{2\lambda_\textbf{k}^{1+\alpha}}\big(1-e^{\lambda_\textbf{k}\Delta_n}\big)\leq 0.
\end{align*}
Using this calculations along with equation \eqref{eqn_covCijk}, we can derive the following:
\begin{align*}
T_2&\leq C^4\sigma^4\sum_{i=r+u}^{v}\bigg(\sum_{\textbf{k}\in\N^d}\frac{(1-e^{-\lambda_\textbf{k}\Delta_n})^2}{2\lambda_\textbf{k}^{1+\alpha}}e^{-2\lambda_\textbf{k}(i-r-1)\Delta_n}\bigg)\\
&~~~~~\times\bigg(\sum_{\textbf{k}\in\N^d}\frac{(1-e^{-\lambda_\textbf{k}\Delta_n})^2}{2\lambda_\textbf{k}^{1+\alpha}}\big(1-e^{-2\lambda_\textbf{k}(i-r-1)\Delta_n}\big)+\sigma^{-2}\Sigma_{i,i}^{C,\textbf{k}}\bigg)\\
&~~~~~+2\sum_{\substack{i,j=r+u\\i<j}}^v\bigg(\sum_{\textbf{k}\in\N^d}\frac{(1-e^{-\lambda_\textbf{k}\Delta_n})^2}{2\lambda_\textbf{k}^{1+\alpha}}e^{-\lambda_\textbf{k}(i+j-2r-2)\Delta_n}e_\textbf{k}^2(\textbf{y})\bigg)\bigg(\sum\limits_{\textbf{k}\in\N^d}\E\big[D_2^{\textbf{k},i}D_2^{\textbf{k},j}\big]e_\textbf{k}^2(\textbf{y})\bigg)\\
&\leq  C^4\sigma^4\sum_{i=r+u}^{v}\bigg(\sum_{\textbf{k}\in\N^d}\frac{(1-e^{-\lambda_\textbf{k}\Delta_n})^2}{2\lambda_\textbf{k}^{1+\alpha}}e^{-2\lambda_\textbf{k}(i-r-1)\Delta_n}\bigg)\bigg(\sum_{\textbf{k}\in\N^d}\frac{(1-e^{-\lambda_\textbf{k}\Delta_n})^2+1-e^{-2\lambda_\textbf{k}\Delta_n}}{2\lambda_\textbf{k}^{1+\alpha}}\bigg)\\
&\leq C^4\sigma^4 \Delta_n^{2\alpha'}\bar{p} \bigg(\Delta_n^{d/2}\sum_{\textbf{k}\in\N^d}\frac{(1-e^{-\lambda_\textbf{k}\Delta_n})^2}{2(\lambda_\textbf{k}\Delta_n)^{1+\alpha}}e^{-2\lambda_\textbf{k}(u-1)\Delta_n}\bigg)\bigg(\Delta_n^{d/2}\sum_{\textbf{k}\in\N^d}\frac{1-e^{-\lambda_\textbf{k}\Delta_n}}{(\lambda_\textbf{k}\Delta_n)^{1+\alpha}}\bigg).
\end{align*}
By utilizing analogous steps as for the term $T_1$, we obtain the following expression for $T_2$:
\begin{align*}
T_2\leq C\sigma^4\frac{\bar{p}\Delta_n^{2\alpha'}}{(u-1)^{2-\alpha'}},
\end{align*}
with a suitable constant $C>0$. Thereby, we conclude for $u\geq 2$ that
\begin{align}
\E[\bar{A}_1^2]\leq C\sigma^4\frac{\bar{p}\Delta_n^{2\alpha'}}{(u-1)^{2-\alpha'}}.\label{eqn_A1BarBound}
\end{align}
Finally, using Proposition \ref{prop_quadIncAndrescaling}, we find that
\begin{align*}
\Var\Big(\big(Q_{r+u}^v\big)^2\Big)\geq C\E\Big[\big(Q_{r+u}^v\big)^2\Big]=C\E\Bigg[\bigg(\sum\limits_{i=r+u}^v(\Delta_i\tilde{X})^2(y)\bigg)^2\Bigg]\geq C\sum\limits_{i=r+u}^v\E\Big[(\Delta_i\tilde{X})^4(y)\Big]\geq C''\sigma^4\bar{p}\Delta_n^{2\alpha'}.
\end{align*}
This, and a simple bound for $u=1$ complete the proof.
\end{proof}\ \\

\begin{cor}\label{corollary_4thConditionMultiDimAndMultiSpace}
On the assumptions of Proposition \ref{porp_timeDepenMulit}, it holds for $1\leq r<r+u\leq v\leq n$ and
\begin{align*}
\tilde{Q}_1^r=\sum\limits_{i=1}^r \tilde{\xi}_{n,i},~~~~~~~~\tilde{Q}_{r+u}^v=\sum\limits_{i=r+u}^v \tilde{\xi}_{n,i},
\end{align*}
that there is a constant $C$, with $0<C<\infty$ and $\tilde{\xi}_{n,i}$ from equation \eqref{eqn_triangularArrayZetaMulti}, such that for all $t\in\R$ it holds:
\begin{align*}
\abs{\Cov\bigg(e^{\im t(\tilde{Q}_1^r-\E[\tilde{Q}_1^r])},~e^{\im t(\tilde{Q}_{r+u}^v-\E[\tilde{Q}_{r+u}^v])}\bigg)}\leq \frac{Ct^2}{u^{1-\alpha'/2}}\sqrt{\Var(\tilde{Q}_1^r)\Var(\tilde{Q}_{r+u}^v)}.
\end{align*}

\end{cor}
\begin{proof}
We present the proof analogously to Proposition \ref{porp_timeDepenMulit} and begin by decomposing the term $\tilde{Q}_{r+u}^v$ as follows:
\begin{align*}
\tilde{Q}_{r+u}^v=\sum_{i=r+u}^v \tilde{\xi}_{n,i}=\frac{2^d(\pi\eta)^{d/2}\alpha'\Gamma(d/2)}{\sqrt{nm}\Delta_n^{\alpha'}\Gamma(1-\alpha')}\sum_{j=1}^m\big(A_1(\textbf{y}_j)+A_2(\textbf{y}_j)\big)e^{\nor{\kappa\bigcdot \textbf{y}_j}_1},
\end{align*}
where
\begin{align*}
A_1(\textbf{y})&:= \sum\limits_{i=r+u}^v \bigg(\sum\limits_{\textbf{k}\in\N^d} D_1^{\textbf{k},i}e_\textbf{k}(\textbf{y})\bigg)^2+ 2\sum\limits_{i=r+u}^v \bigg(\sum\limits_{\textbf{k}\in\N^d} D_1^{\textbf{k},i}e_\textbf{k}(\textbf{y})\bigg)\bigg(\sum\limits_{\textbf{k}\in\N^d}^\infty D_2^{\textbf{k},i}e_\textbf{k}(\textbf{y})\bigg), \\
A_2(\textbf{y})&:=\sum\limits_{i=r+u}^v \bigg(\sum\limits_{\textbf{k}\in\N^d} D_2^{\textbf{k},i}e_\textbf{k}(\textbf{y})\bigg)^2,
\end{align*}
and an analogous definition of $D_1^{\textbf{k},i}$ and $D_2^{\textbf{k},i}$ as in the equations \eqref{number76:D1} and \eqref{number76:D2}. Thereby, we need to bound the following expression:
\begin{align*}
\Var\bigg(\frac{K}{\sqrt{nm}\Delta_n^{\alpha'}}\sum_{j=1}^m A_1(\textbf{y}_j)e^{\nor{\kappa\bigcdot \textbf{y}_j}_1}\bigg)&=\frac{K^2}{nm\Delta_n^{2\alpha'}}\sum_{j=1}^m\Var\big(A_1(\textbf{y}_j)\big) e^{2\nor{\kappa\bigcdot \textbf{y}_j}_1 }\\
&~~~~~+\frac{K^2}{nm\Delta_n^{2\alpha'}}\sum_{\substack{j_1,j_2=1 \\ j_1\neq j_2}}^m\Cov\big(A_1(\textbf{y}_{j_1}),A_1(\textbf{y}_{j_2})\big)e^{\nor{\kappa\bigcdot (\textbf{y}_{j_1}+\textbf{y}_{j_2})}_1 },
\end{align*}
where
\begin{align*}
K:=\frac{2^d(\pi\eta)^{d/2}\alpha'\Gamma(d/2)}{\Gamma(1-\alpha')}.
\end{align*}
Let $\bar{p}=v-r-u+1$ and $u\geq 2$. Thanks to Proposition \ref{porp_timeDepenMulit}, we obtain the following:
\begin{align*}
\frac{K^2}{nm\Delta_n^{2\alpha'}}\sum_{j=1}^m\Var\big(A_1(\textbf{y}_j)\big) e^{2\nor{\kappa\bigcdot \textbf{y}_j}_1 }\leq C\sigma^4\frac{\bar{p}K^2\Delta_n}{(u-1)^{2-\alpha'}},
\end{align*}
where we used the bound for $\E[\bar{A}_1^2]$ from display \eqref{eqn_A1BarBound}. 
For the covariance, we exploit the independence of $D_1^{\textbf{k},i}$ and $D_2^{\textbf{k},i}$, along with both terms being centred normals. This allows us to derive the following:
\begin{align*}
\Cov\big(A_1(\textbf{y}_{1}),A_1(\textbf{y}_{2})\big)&=
\sum\limits_{i,j=r+u}^v\E\bigg[ \bigg(\sum\limits_{\textbf{k}\in\N^d} D_1^{\textbf{k},i}e_\textbf{k}(\textbf{y}_1)\bigg)^2 \bigg(\sum\limits_{\textbf{k}\in\N^d} D_1^{\textbf{k},j}e_\textbf{k}(\textbf{y}_2)\bigg)^2\bigg]\\
&~~~~~+4\sum_{i,j=r+u}^v\E\bigg[\bigg(\sum\limits_{\textbf{k}\in\N^d} D_1^{\textbf{k},i}e_\textbf{k}(\textbf{y}_1)\bigg)\bigg(\sum\limits_{\textbf{k}\in\N^d}^\infty D_2^{\textbf{k},i}e_\textbf{k}(\textbf{y}_1)\bigg)\\
&~~~~~\times\bigg(\sum\limits_{\textbf{k}\in\N^d} D_1^{\textbf{k},j}e_\textbf{k}(\textbf{y}_2)\bigg)\bigg(\sum\limits_{\textbf{k}\in\N^d}^\infty D_2^{\textbf{k},j}e_\textbf{k}(\textbf{y}_2)\bigg)\bigg]\\
&~~~~~-\sum_{i,j=r+u}^v\E\bigg[\bigg(\sum_{\textbf{k}\in\N^d} D_1^{\textbf{k},i}e_\textbf{k}(\textbf{y}_1)\bigg)^2\bigg]\E\bigg[\bigg(\sum_{\textbf{k}\in\N^d} D_1^{\textbf{k},j}e_\textbf{k}(\textbf{y}_2)\bigg)^2\bigg].
\end{align*}
Since we can bound the eigenfunctions $(e_\textbf{k})$ by a suitable constant $C>0$ for all $\textbf{k}\in\N^d$, we observe that the covariance includes the terms $T_1$ and $T_2$ from the displays \eqref{number76:T1} and \eqref{number76:T2}, respectively. Therefore, we can repeat the calculations from Proposition \ref{porp_timeDepenMulit} concerning the eigenfunctions, leading to the following:
\begin{align*}
\tilde{T}_1
&\leq \sigma^4\bar{p}\bigg(\sum\limits_{\textbf{k}\in\N^d} \frac{\big(1-e^{-\lambda_{\textbf{k}}\Delta_n}\big)}{2\lambda_{\textbf{k}}^{1+\alpha}}e^{-2\lambda_{\textbf{k}}(u-1)\Delta_n}e_{\textbf{k}}^2(\textbf{y}_1)\bigg)\bigg(\sum\limits_{\textbf{k}\in\N^d} \frac{\big(1-e^{-\lambda_{\textbf{k}}\Delta_n}\big)^2}{2\lambda_{\textbf{k}}^{1+\alpha}}e^{-2\lambda_{\textbf{k}}(u-1)\Delta_n}e_{\textbf{k}}^2(\textbf{y}_2)\bigg)\\
&~~~~~+\sigma^42\bar{p}\bigg(\sum\limits_{\textbf{k}\in\N^d} \frac{\big(1-e^{-\lambda_{\textbf{k}}\Delta_n}\big)}{2\lambda_{\textbf{k}}^{1+\alpha}}e^{-2\lambda_{\textbf{k}}(u-1)\Delta_n}e_{\textbf{k}}(\textbf{y}_1)e_{\textbf{k}}(\textbf{y}_2)\bigg)\\
&~~~~~~~~~~\times\bigg(\sum\limits_{\textbf{k}\in\N^d} \frac{\big(1-e^{-\lambda_{\textbf{k}}\Delta_n}\big)^2}{2\lambda_{\textbf{k}}^{1+\alpha}}e^{-2\lambda_{\textbf{k}}(u-1)\Delta_n}e_{\textbf{k}}(\textbf{y}_1)e_{\textbf{k}}(\textbf{y}_2)\bigg).
\end{align*}
Furthermore, we obtain that 
\begin{align*}
&\tilde{T}_1-\sum_{i,j=r+u}^v\E\bigg[\bigg(\sum_{\textbf{k}\in\N^d} D_1^{\textbf{k},i}e_\textbf{k}(\textbf{y}_1)\bigg)^2\bigg]\E\bigg[\bigg(\sum_{\textbf{k}\in\N^d} D_1^{\textbf{k},j}e_\textbf{k}(\textbf{y}_2)\bigg)^2\bigg]\\
&\leq \sigma^42\bar{p}\bigg(\sum\limits_{\textbf{k}\in\N^d} \frac{\big(1-e^{-\lambda_{\textbf{k}}\Delta_n}\big)}{2\lambda_{\textbf{k}}^{1+\alpha}}e^{-2\lambda_{\textbf{k}}(u-1)\Delta_n}e_{\textbf{k}}(\textbf{y}_1)e_{\textbf{k}}(\textbf{y}_2)\bigg)\bigg(\sum\limits_{\textbf{k}\in\N^d} \frac{\big(1-e^{-\lambda_{\textbf{k}}\Delta_n}\big)^2}{2\lambda_{\textbf{k}}^{1+\alpha}}e^{-2\lambda_{\textbf{k}}(u-1)\Delta_n}e_{\textbf{k}}(\textbf{y}_1)e_{\textbf{k}}(\textbf{y}_2)\bigg).
\end{align*}
Thus, we can bound the latter term by using display \eqref{eqn_ekDecompInCosine} and Lemma \ref{lemma_riemannApprox_multi}. Similar to Proposition \ref{prop_RRVMulti}, we find that 
\begin{align}
&\tilde{T}_1-\sum_{i,j=r+u}^v\E\bigg[\bigg(\sum_{\textbf{k}\in\N^d} D_1^{\textbf{k},i}e_\textbf{k}(\textbf{y}_1)\bigg)^2\bigg]\E\bigg[\bigg(\sum_{\textbf{k}\in\N^d} D_1^{\textbf{k},j}e_\textbf{k}(\textbf{y}_2)\bigg)^2\bigg]\notag\\
&=\Oo\bigg(\sigma^4\frac{\bar{p}\Delta_n^{2\alpha'}}{(u-1)^{2-\alpha'}}\Delta_n		^{1-\alpha'}\nor{\textbf{y}_1-\textbf{y}_2}_0^{-(d+1)}\bigg),\label{eqn_tildet1Multi_1}
\end{align}
where we used analogous steps as in display \eqref{eqn_covEksRemainder}.
For the last term in the covariance, we redefine:
\begin{align*}
\tilde{T}_2&:=\sum_{i,j=r+u}^v\E\bigg[\bigg(\sum\limits_{\textbf{k}\in\N^d} D_1^{\textbf{k},i}e_\textbf{k}(\textbf{y}_1)\bigg)\bigg(\sum\limits_{\textbf{k}\in\N^d}^\infty D_2^{\textbf{k},i}e_\textbf{k}(\textbf{y}_1)\bigg)\\
&~~~~~\times\bigg(\sum\limits_{\textbf{k}\in\N^d} D_1^{\textbf{k},j}e_\textbf{k}(\textbf{y}_2)\bigg)\bigg(\sum\limits_{\textbf{k}\in\N^d}^\infty D_2^{\textbf{k},j}e_\textbf{k}(\textbf{y}_2)\bigg)\bigg]\\
&=\sum\limits_{i,j=r+u}^v\bigg(\sum\limits_{\textbf{k}\in\N^d}\E\big[D_1^{\textbf{k},i}D_1^{\textbf{k},j}\big]e_\textbf{k}(\textbf{y}_1)e_\textbf{k}(\textbf{y}_2)\bigg)\bigg(\sum\limits_{\textbf{k}\in\N^d}\E\big[D_2^{\textbf{k},i}D_2^{\textbf{k},j}\big]e_\textbf{k}(\textbf{y}_1)e_\textbf{k}(\textbf{y}_2)\bigg).
\end{align*}
With similar steps as in Proposition \ref{prop_RRVMulti}, we obtain:
\begin{align*}
\tilde{T}_2
&\leq \sigma^4 \Delta_n^{2\alpha'}\bar{p} \bigg(\Delta_n^{d/2}\sum_{\textbf{k}\in\N^d}\frac{(1-e^{-\lambda_\textbf{k}\Delta_n})^2}{2(\lambda_\textbf{k}\Delta_n)^{1+\alpha}}e^{-2\lambda_\textbf{k}(u-1)\Delta_n}e_\textbf{k}(\textbf{y}_1)e_\textbf{k}(\textbf{y}_2)\bigg)\bigg(\Delta_n^{d/2}\sum_{\textbf{k}\in\N^d}\frac{1-e^{-\lambda_\textbf{k}\Delta_n}}{(\lambda_\textbf{k}\Delta_n)^{1+\alpha}}e_\textbf{k}(\textbf{y}_1)e_\textbf{k}(\textbf{y}_2)\bigg)\\
&=\Oo\bigg(\sigma^4\bar{p}\Delta_n^{2\alpha'}\Delta_n^{1-\alpha'}\nor{\textbf{y}_1-\textbf{y}_2}_0^{-(d+1)}\Big(\Delta_n^{d/2}\sum_{\textbf{k}\in\N^d}\frac{(1-e^{-\lambda_\textbf{k}\Delta_n})^2}{(\lambda_\textbf{k}\Delta_n)^{1+\alpha}}e^{-2\lambda_\textbf{k}(u-1)\Delta_n}\Big)\bigg)=\Oo(\tilde{T}_1).
\end{align*}
Hence, we have
\begin{align*}
\frac{K^2}{nm\Delta_n^{2\alpha'}}\sum_{\substack{j_1,j_2=1\\j_1\neq j_2}}^m\Cov\big(A_1(\textbf{y}_{j_1}),A_1(\textbf{y}_{j_2})\big)&= \Oo\bigg(\sigma^4\frac{\bar{p}\Delta_n}{(u-1)^{2-\alpha'}}\cdot \frac{\Delta_n^{1-\alpha'}}{m}\sum_{\substack{j_1,j_2=1\\j_1\neq j_2}}^m\frac{1}{\nor{\textbf{y}_{j_1}-\textbf{y}_{j_2}}_0^{d+1}}\bigg).
\end{align*}
According to Assumption \ref{assumption_observations_multi}, the distance between any two arbitrary spatial coordinates is bounded from below, leading to the following order:
\begin{align}
\sum_{\substack{j_1,j_2=1\\j_1\neq j_2}}^m\bigg(\frac{1}{\nor{\textbf{y}_{j_1}-\textbf{y}_{j_2}}_0}\bigg)^{d+1}=\Oo\bigg(m^{d+1}\sum_{\substack{j_1,j_2=1\\j_1\neq j_2}}^m\Big(\frac{1}{m\nor{\textbf{y}_{j_1}-\textbf{y}_{j_2}}_0}\Big)^{d+1}\bigg)=\Oo\big(m^{d+3}\big).\label{eqn_OrderDistinctSpatialCoordiantes}
\end{align}
Thus, we conclude that
\begin{align*}
\frac{K^2}{nm\Delta_n^{2\alpha'}}\sum_{\substack{j_1,j_2=1\\j_1\neq j_2}}^m\Cov\big(A_1(\textbf{y}_{j_1}),A_1(\textbf{y}_{j_2})\big)&= \Oo\bigg(\sigma^4\frac{\bar{p}\Delta_n}{(u-1)^{2-\alpha'}}\Delta_n^{1-\alpha'}m^{d+2}\bigg).
\end{align*}
The proof follows with:
\begin{align*}
\E[(Q_{r+u}^v)^2]\geq \sum_{i=r+u}^v\E[\xi_{n,i}^2]\geq C\frac{K^2\bar{p}}{nm\Delta_n^{2\alpha'}}\sum_{j=1}^m\E\big[(\Delta_i\tilde{X})^4(\textbf{y}_j)\big]\geq C'\sigma^4\Delta_n\bar{p}.
\end{align*}
\end{proof}

Now we are able to prove the central limit theorem from Proposition \ref{prop_cltVolaEstMulti}.
\begin{proof}
To prove this central limit theorem, we will employ Proposition \ref{prop_clt_utev}. Hence, we define:
\begin{align*}
\Xi_{n,i}:=\tilde{\xi}_{n,i}-\E[\tilde{\xi}_{n,i}],
\end{align*}
where $\tilde{\xi}_{n,i}$ is defined in equation \eqref{eqn_triangularArrayZetaMulti}.
The asymptotic variance is given by:
\begin{align*}
\Var\bigg(\sum_{i=1}^n\Xi_{n,i}\bigg)&=\Var\bigg(\sum_{i=1}^n\tilde{\xi}_{n,i}\bigg)=\frac{K^2}{nm_n\Delta_n^{2\alpha'}}\Var\bigg(\sum_{j=1}^{m_n}\sum_{i=1}^n(\Delta_i\tilde{X})^2(\textbf{y}_j)e^{\nor{\kappa\bigcdot \textbf{y}_j}_1}\bigg)\\
&=\frac{K^2n^2\Delta_n^{2\alpha'}}{nm_n\Delta_n^{2\alpha'}}\bigg(\sum_{j=1}^{m_n}\Var\big(V_{n,\Delta_n}(\textbf{y}_j)\big)+\sum_{\substack{j_1,j_2=1\\j_1\neq j_2}}^{m_n}\Cov\big(V_{n,\Delta_n}(\textbf{y}_{j_1}),V_{n,\Delta_n}(\textbf{y}_{j_2}\big)\bigg)\\
&=\frac{K^2n}{m_n}\cdot \frac{m_n\Upsilon_{\alpha'}\sigma^4}{n}\bigg(\frac{\Gamma(1-\alpha')}{2^d(\pi\eta)^{d/2}\alpha'\Gamma(d/2)}\bigg)^2\big(1+\Oo(\Delta_n^{1/2}\vee \Delta_n^{1-\alpha'})\big)\\
&~~~~~+\Oo\bigg(\frac{K^2n}{m_n}\cdot\frac{\Delta_n^{1-\alpha'}}{n}\sum_{\substack{j_1,j_2=1\\j_1\neq j_2}}^{m_n}\Big( \nor{\textbf{y}_1-\textbf{y}_2}_0^{-(d+1)}+\delta^{-(d+1)}\Big) \bigg)\\
&=\Upsilon_{\alpha'}\sigma^4\big(1+\Oo(\Delta_n^{1/2}\vee \Delta_n^{1-\alpha'})\big)+\Oo(m_n^{d+2}\Delta_n^{1-\alpha'})\overset{n\tooi}{\longrightarrow}\Upsilon_{\alpha'}\sigma^4,
\end{align*}
where we used Proposition \ref{prop_RRVMulti} and equation \eqref{eqn_OrderDistinctSpatialCoordiantes} and $K$ defined in \eqref{eqn_KDefinition}.
It remains to prove the Conditions (I)-(III) from Proposition \ref{prop_clt_utev}, since the last condition is proved by Corollary \ref{corollary_4thConditionMultiDimAndMultiSpace}.
\begin{itemize}
\item[(I)] By Proposition \ref{prop_RRVMulti} we have
\begin{align*}
\sum_{i=a}^b\Var(\Xi_{n,i})&=\sum_{i=a}^b\Var(\tilde{\xi}_{n,i})=\frac{K^2\Delta_n^{2\alpha'}}{nm_n\Delta_n^{2\alpha'}}\sum_{i=a}^b \Var\bigg(\sum_{j=1}^{m_n}\frac{1}{\Delta_n^{\alpha'}}(\Delta_i\tilde{X})(\textbf{y}_j)e^{\nor{\kappa\bigcdot \textbf{y}_j}_1}\bigg)\\
&=\frac{K^2}{nm_n}\sum_{i=a}^b\bigg(\sum_{j=1}^{m_n}\Var\big(V_{1,\Delta_n}(\textbf{y}_j)\big)+\sum_{\substack{j_1,j_2=1\\ j_1\neq j_2}}^{m_n}\Cov\big(V_{1,\Delta_n}(\textbf{y}_{j_1}),V_{1,\Delta_n}(\textbf{y}_{j_2})\big)\bigg)\\
&=\Oo\bigg(\Delta_n(b-a+1)+\Delta_n(b-a+1)\Delta_n^{1-\alpha'}m_n^{d+2}\bigg)=\Oo\big(\Delta_n(b-a+1)\big).
\end{align*}
We utilize the calculations for the asymptotic variance as shown in this proof and thus conclude:
\begin{align*}
\Var\bigg(\sum_{i=a}^b\Xi_{n,i}\bigg)&=\Var\bigg(\sum_{i=a}^b\tilde{\xi}_{n,i}\bigg)\\
&=\Oo\bigg(\frac{K^2(b-a+1)^2}{nm_n}\cdot \frac{m_n}{(b-a+1)K^2}+\frac{(b-a+1)^2}{nm_n}\cdot \frac{\Delta_n^{1-\alpha'}m_n^{d+3}}{(b-a+1)}\bigg)\\
&=\Oo\big(\Delta_n(b-a+1)\big),
\end{align*}
which shows the first condition as well as the second condition.
\item[(III)] 
We prove that a Lyapunov condition is satisfied. By using the Cauchy-Schwarz inequality, we have
\begin{align*}
\E[\tilde{\xi}_{n,i}^4]&=\frac{K^4}{n^2m_n^2\Delta_n^{4\alpha'}}\sum_{j_1,\ldots,j_4=1}^{m_n} e^{\nor{\kappa \bigcdot (\textbf{y}_{j_1}+\ldots+\textbf{y}_{j_4})}_1}\E[(\Delta_i\tilde{X})^2(\textbf{y}_{j_1})\cdots(\Delta_i\tilde{X})^2(\textbf{y}_{j_4})]\\
&\leq \frac{K^4}{n^2m_n^2\Delta_n^{4\alpha'}}\sum_{j_1,\ldots,j_4=1}^{m_n} e^{\nor{\kappa \bigcdot (\textbf{y}_{j_1}+\ldots+\textbf{y}_{j_4})}_1}\E[(\Delta_i\tilde{X})^8(\textbf{y}_{j_1})]^{1/4}\cdots\E[(\Delta_i\tilde{X})^8(\textbf{y}_{j_4})]^{1/4}\\
&\leq \frac{K^4}{n^2\Delta_n^{4\alpha'}}m_n^2 e^{4\nor{\kappa}_1}\max_{\textbf{y}\in\{\textbf{y}_1,\ldots\textbf{y}_{m_n}\}}\E[(\Delta_i\tilde{X})^8(\textbf{y})].
\end{align*}
Since $(\Delta_n\tilde{X})(\textbf{y})$ is a centred Gaussian random variable, we can infer that $\E[(\Delta_i\tilde{X})^8(\textbf{y})]=\Oo(\Delta_n^{4\alpha})$ by using Proposition \ref{prop_quadIncAndrescaling}. Thus, we have
\begin{align*}
\sum_{i=1}^n\E[\tilde{\xi}_{n,i}^4]=\Oo(\Delta_nm^2)=\smallO(1),
\end{align*}
which shows the third condition.
\end{itemize}

\end{proof}

\subsection{Proofs of Section 4}
We proceed to tackle the methodology section for the estimator $\hat{\Psi}$ by deriving the corresponding multidimensional triangular array for $\hat{\Psi}$. Notably, demonstrating a central limit theorem for the estimator $\hat{\Psi}$ alone suffices, as the estimator $\hat{\upsilon}$ is a transformation of $\hat{\Psi}$. This enables us to deduce a central limit theorem for $\hat{\upsilon}$ using the multidimensional delta method.
To construct the multidimensional triangular array, we leverage the Taylor expansion for $\log(a+x)$ and obtain that
\begin{align*}
\log\bigg(\frac{\RV(\textbf{y})}{n\Delta_n^{\alpha'}}\bigg)&=\log(\sigma_0^2K)-\nor{\kappa\bigcdot \textbf{y}}_1+\frac{\sum_{i=1}^n\overline{(\Delta_i\tilde{X})^2(\textbf{y})}}{n\Delta_n^{\alpha'}\sigma_0^2K}e^{\nor{\kappa\bigcdot\textbf{y}}_1}+\Oo\big(\Delta_n^{1-\alpha'}\big)+\Oo_\Pp\Bigg(\bigg(\frac{\overline{\RV(\textbf{y})}}{n\Delta_n^{\alpha'}}\bigg)^2\Bigg),
\end{align*}
where the constant $K$ is defined in equation \eqref{eqn_KDefinition}.
Utilizing Proposition \ref{prop_RRVMulti} we conclude that 
\begin{align}
\log\bigg(\frac{\RV(\textbf{y})}{n\Delta_n^{\alpha'}}\bigg)&=\log(\sigma_0^2K)-\nor{\kappa\bigcdot \textbf{y}}_1+\frac{\sum_{i=1}^n\overline{(\Delta_i\tilde{X})^2(\textbf{y})}}{n\Delta_n^{\alpha'}\sigma_0^2K}e^{\nor{\kappa\bigcdot\textbf{y}}_1}+\Oo\big(\Delta_n^{1-\alpha'}\big)+\Oo_\Pp(\Delta_n).\label{eqn_DecompRVMulti}
\end{align}
The previous expression simplifies the analysis by allowing us to focus on the term:
\begin{align*}
\log(\sigma_0^2K)-\nor{\kappa\bigcdot \textbf{y}}_1+\frac{\sum_{i=1}^n\overline{(\Delta_i\tilde{X})^2(\textbf{y})}}{n\Delta_n^{\alpha'}\sigma_0^2K}e^{\nor{\kappa\bigcdot\textbf{y}}_1},
\end{align*}
where the last component represents the model error.
Our goal is to establish a central limit theorem in the form of $\sqrt{nm}(\hat{\Psi}-\Psi)$. To achieve this, we can develop the triangular array associated with the estimator $\hat{\Psi}$ by employing equation \eqref{eqn_identityGMEForTriangularArray}. This triangular array is defined as $\Xi_{n,i}:=\xi_{n,i}-\E[\xi_{n,i}]$, where:
\begin{align}
\xi_{n,i}&:=\sqrt{nm}\cdot \frac{1-2\delta}{m}\bigg(\frac{1-2\delta}{m}X^\top X\bigg)^{-1}X^\top\begin{pmatrix}
\frac{(\Delta_i\tilde{X})^2(\textbf{y}_1)}{n\Delta_n^{\alpha'}\sigma_0^2K}e^{\nor{\kappa\bigcdot\textbf{y}_1}_1}\\ \vdots \\ \frac{(\Delta_i\tilde{X})^2(\textbf{y}_m)}{n\Delta_n^{\alpha'}\sigma_0^2K}e^{\nor{\kappa\bigcdot\textbf{y}_m}_1}
\end{pmatrix}\notag\\[3ex]
&=  \frac{\sqrt{n}(1-2\delta)}{\sqrt{m}K\sigma_0^2}\bigg(\frac{1-2\delta}{m}X^\top X\bigg)^{-1}X^\top\begin{pmatrix}
\frac{(\Delta_i\tilde{X})^2(\textbf{y}_1)}{n\Delta_n^{\alpha'}}e^{\nor{\kappa\bigcdot\textbf{y}_1}_1}\\ \vdots \\ \frac{(\Delta_i\tilde{X})^2(\textbf{y}_m)}{n\Delta_n^{\alpha'}}e^{\nor{\kappa\bigcdot\textbf{y}_m}_1}
\end{pmatrix}.\label{eqn_ximultiMLRM}
\end{align}
With the triangular array $\Xi_{n,i}$ in place, we can now proceed to the preparations for a CLT. For proving Proposition \ref{clt_PsiMulti}, we utilize a generalization of Proposition \ref{prop_clt_utev}, which directly follows by the Cram\'{e}r-Wold theorem.
\begin{cor}\label{prop_CLTUtev_multivariate}
Let $(Z_{k_n,i})_{1\leq i\leq k_n}$ a centred triangular array, with a sequence $k_n$, where $Z_{n,k_n}\in\R^d$ are random vectors. Then it holds that 
\begin{align*}
\sum_{i=1}^{k_n} Z_{n,i}\overset{d}{\longrightarrow} \mathcal{N}(0,\Sigma), 
\end{align*}
for $n\tooi$ and $\Sigma$ denotes a variance-covariance matrix, which satisfies the equation:
\begin{align*}
\lim_{n\tooi}\Var\bigg(\sum_{i=1}^{k_n}\beta^\top Z_{n,i}\bigg)=\beta^\top\Sigma \beta<\infty,
\end{align*}
for any $\beta\in\R^d$, if the following conditions hold for any $\beta\in\R^d$:
\begin{itemize}
\item[(I)] $\Var\Big(\sum\limits_{i=a}^b\beta^\top Z_{n,i}\Big)\leq C\sum\limits_{i=a}^b\Var(\beta^\top Z_{n,i})$, for all $1\leq a\leq b\leq k_n$ ,
\item[(II)] $\limsup\limits_{n\tooi}\sum\limits_{i=1}^{k_n}\E[\beta^\top Z_{n,i}^2]<\infty$,
\item[(III)] $\sum\limits_{i=1}^{k_n}\E\Big[\beta^\top Z_{k_n,i}^2\mathbbm{1}_{\{|\beta^\top Z_{k_n,i}|>\varepsilon\}}\Big]\overset{n\tooi}{\longrightarrow}0$, for all $\varepsilon>0$,
\item[(IV)] $\Cov\Big(e^{\im t\sum_{i=a}^b\beta^\top Z_{n,i}},e^{\im t\sum_{i=b+u}^c \beta^\top Z_{n,i}}\Big)\leq \rho_t(u)\sum\limits_{i=a}^c\Var(\beta^\top Z_{n,i})$, for all $1\leq a\leq b\leq b+u\leq c\leq k_n$ and $t\in\R$,
\end{itemize}
where $C>0$ is a universal constant and $\rho_t(u)\geq 0$ is a function with $\sum_{j=1}^\infty\rho_t(2^j)<\infty$. 
\end{cor}\ \\

We start by calculating the asymptotic variance. 
\begin{lemma}\label{lemma_assympVarMultiMLRM}
On the Assumptions \ref{assumption_observations_multi}, \ref{assumption_regMulti} and \ref{assumption_fullRank}, we have 
\begin{align*}
\lim_{n\tooi}\Var\bigg(\sum_{i=1}^n\gamma^\top \Xi_{n,i}\bigg)&= (1-2\delta)\Upsilon_{\alpha'}\gamma^\top\Sigma^{-1}\gamma,
\end{align*}
where $\Xi_{n,i}$ is defined in equation \eqref{eqn_ximultiMLRM}, $\Upsilon_{\alpha'}$ defined in equation \eqref{eqn_definingUpsilon}, $\Sigma^{-1}$ from equation \eqref{eqn_asympVarCovMatrixMLRM_general} and $\gamma\in\R^{d+1}$ arbitrary but fixed.
\end{lemma}

\begin{proof}
Consider an arbitrary but fixed vector $\gamma\in\mathbb{R}^{d+1}$. We initiate this proof by performing the following calculations:
\begin{align*}
\Var\bigg(\sum_{i=1}^n\gamma^\top \Xi_{n,i}\bigg)&=\gamma^\top\Var\bigg(\sum_{i=1}^n \xi_{n,i}\bigg)\gamma\\
&=\gamma^\top\frac{n(1-2\delta)^2}{mK^2\sigma_0^4}\bigg(\frac{1-2\delta}{m}X^\top X\bigg)^{-1}X^\top\Var\big(\tilde{Y}_n\big)X\bigg(\frac{1-2\delta}{m}X^\top X\bigg)^{-1}\gamma,
\end{align*}
where
\begin{align*}
\tilde{Y}_n:=\begin{pmatrix}
\sum_{i=1}^n\frac{(\Delta_i\tilde{X})^2(\textbf{y}_1)}{n\Delta_n^{\alpha'}}e^{\nor{\kappa\bigcdot\textbf{y}_1}_1}\\ \vdots \\ \sum_{i=1}^n\frac{(\Delta_i\tilde{X})^2(\textbf{y}_m)}{n\Delta_n^{\alpha'}}e^{\nor{\kappa\bigcdot\textbf{y}_m}_1}
\end{pmatrix}=\begin{pmatrix}
V_{n,\Delta_n}(\textbf{y}_1)\\ \vdots \\ V_{n,\Delta_n}(\textbf{y}_m)
\end{pmatrix}\in\R^{m}.
\end{align*}
Using Proposition \ref{prop_RRVMulti}, we can determine the components of the variance-covariance matrix $V_{n,m}:=\Var(\tilde{Y}_n)$ of $\tilde{Y}_{n,i}$, yielding:
\begin{align*}
(V_{n,m})_{j_1,j_2}=\begin{cases}
\frac{\Upsilon_{\alpha'}}{n}K^2\sigma_0^4\big(1+\Delta_n^{1/2}\vee \Delta_n^{1-\alpha'}\big) &, \text{ if } 1\leq j_1=j_2\leq m \\
\Oo\big(\Delta_n^{2-\alpha'}(\nor{\textbf{y}_{j_1}-\textbf{y}_{j_2}}_0^{-(d+1)}+\delta^{-(d+1)})\big) &,\text{ if } 1\leq j_1,j_2\leq m \text{ for }j_1\neq j_2
\end{cases},
\end{align*}
for $1\leq j_1,j_2\leq m$. Hence, we have
\begin{align*}
\Var\bigg(\sum_{i=1}^n\gamma^\top \Xi_{n,i}\bigg)&=\gamma^\top\frac{(1-2\delta)^2}{m}\bigg(\frac{1-2\delta}{m}X^\top X\bigg)^{-1}X^\top\bigg(\frac{n}{K^2\sigma_0^4}V_{n,m}\bigg)X\bigg(\frac{1-2\delta}{m}X^\top X\bigg)^{-1}\gamma,
\end{align*}
where we define:
\begin{align*}
\frac{n}{K^2\sigma_0^4}V_{n,m}=:V_{n,m,1}+V_{n,m,2},
\end{align*}
with
\begin{align*}
V_{n,m,1}:=\Upsilon_{\alpha'}\big(1+\Delta_n^{1/2}\vee \Delta_n^{1-\alpha'}\big) E_m,
\end{align*}
where $E_m$ denotes the $m\times m$ dimensional identity matrix and 
\begin{align*}
V_{n,m,2}:=\begin{cases}
0 &, \text{ if } 1\leq j_1=j_2\leq m \\
\Oo\big(\Delta_n^{1-\alpha'}(\nor{\textbf{y}_{j_1}-\textbf{y}_{j_2})}_0^{-(d+1)}+\delta^{-(d+1)})\big) &,\text{ if } 1\leq j_1,j_2\leq m \text{ for }j_1\neq j_2
\end{cases}.
\end{align*}
We conclude that 
\begin{align*}
\Var\bigg(\sum_{i=1}^n\gamma^\top \Xi_{n,i}\bigg)&=\gamma^\top\bigg(\frac{(1-2\delta)^2}{m}\bigg(\frac{1-2\delta}{m}X^\top X\bigg)^{-1}X^\top V_{n,m,1}X\bigg(\frac{1-2\delta}{m}X^\top X\bigg)^{-1}\\ 
&~~~~~+\frac{(1-2\delta)^2}{m}\bigg(\frac{1-2\delta}{m}X^\top X\bigg)^{-1}X^\top V_{n,m,2}X\bigg(\frac{1-2\delta}{m}X^\top X\bigg)^{-1}\bigg)\gamma\\
&=\gamma^\top\bigg((1-2\delta)\Upsilon_{\alpha'}\big(1+\Delta_n^{1/2}\vee \Delta_n^{1-\alpha'}\big)\bigg(\frac{1-2\delta}{m}X^\top X\bigg)^{-1}\\
&~~~~~~~~~~\times\bigg(\frac{1-2\delta}{m}X^\top X\bigg)\bigg(\frac{1-2\delta}{m}X^\top X\bigg)^{-1}\\
&~~~~~+(1-2\delta)\bigg(\frac{1-2\delta}{m}X^\top X\bigg)^{-1}\bigg(\frac{1-2\delta}{m}X^\top V_{n,m,2} X\bigg)\bigg(\frac{1-2\delta}{m}X^\top X\bigg)^{-1}\bigg)\gamma\\
&=\gamma^\top\bigg((1-2\delta)\Upsilon_{\alpha'}\big(1+\Delta_n^{1/2}\vee \Delta_n^{1-\alpha'}\big)\bigg(\frac{1-2\delta}{m}X^\top X\bigg)^{-1}\\
&~~~~~+(1-2\delta)\bigg(\frac{1-2\delta}{m}X^\top X\bigg)^{-1}\bigg(\frac{1-2\delta}{m}X^\top V_{n,m,2} X\bigg)\bigg(\frac{1-2\delta}{m}X^\top X\bigg)^{-1}\bigg)\gamma.
\end{align*}
Let $m=m_n$ be in accordance with Assumption \ref{assumption_observations_multi}. As the convergence of $(1-2\delta)/m_n\cdot X^\top X$ is established for $n\to\infty$, the focus shifts on demonstrating the convergence of $m_n^{-1} (X^\top V_{n,m_n,2}X)$ towards the zero matrix $\mathbf{0}$. 
Consider matrices $A\in\mathbb{R}^{a\times b}$, $B\in\mathbb{R}^{b \times b}$ and $C\in\mathbb{R}^{b\times a}$, where $(B)_{i_1,i_2}\geq 0$ and $(A)_{l,i},(C)_{i,l}\in[0,1]$ for all $1\leq i,i_1,i_2\leq b,1\leq l\leq a$. In such a scenario we obtain:
\begin{align*}
(ABC)_{i,l}\leq (\textbf{1}_{a, b}B\textbf{1}_{b, a})_{i,l},
\end{align*}
for each $1\leq i,l\leq a$, where $\textbf{1}_{a,b}=\{1\}^{a\times b}$ denotes the matrix with each entry being one. Thus, we find:
\begin{align*}
\bigg(\frac{1}{m}X^\top V_{n,m,2}X\bigg)_{i,l}&\leq\bigg(\frac{1}{m}\textbf{1}_{(d+1),m} V_{n,m,2}\textbf{1}_{m,(d+1)} \bigg)_{i,l},
\end{align*}
for each $1\leq i,l\leq (d+1)$. It holds for $1\leq i\leq(d+1)$ and $1\leq l\leq m$ that
\begin{align*}
\bigg(\frac{1}{m}\textbf{1}_{(d+1),m} V_{n,m,2} \bigg)_{i,l}&=\Oo\bigg(\frac{\Delta_m^{1-\alpha'}}{m}\Big(\sum_{\substack{j_1=1\\ j_1\neq l}}^m \nor{\textbf{y}_{j_1}-\textbf{y}_l}_0^{-(d+1)}     +(m-1)\delta^{-(d+1)}\Big)\bigg),
\end{align*}
and therefore we have for $1\leq i,l\leq (d+1)$ that 
\begin{align*}
\bigg(\frac{1}{m}\textbf{1}_{(d+1),m} V_{n,m,2}\textbf{1}_{m,(d+1)} \bigg)_{i,l}&=\Oo\bigg(\frac{\Delta_m^{1-\alpha'}}{m}\Big(\sum_{j_2=1}^{m}\sum_{\substack{j_1=1\\ j_1\neq j_2}}^m \nor{\textbf{y}_{j_1}-\textbf{y}_{j_2}}_0^{-(d+1)}     +m(m-1)\delta^{-(d+1)}\Big)\bigg)=\\
&=\Oo\big(\Delta_n^{1-\alpha'}m^{d+2}\big).
\end{align*}
Utilizing Assumption \ref{assumption_observations_multi}, we can establish that
\begin{align*}
\bigg(\frac{1}{m_n}X^\top V_{n,m_n,2}X\bigg)_{i,l}=\Oo\big(\Delta_n^{1-\alpha'}m_n^{d+2}\big)\overset{n\tooi}{\longrightarrow}0,
\end{align*}
for all $1\leq i,l\leq (d+1)$. This, in turn, implies:
\begin{align*}
\frac{1}{m_n}X^\top V_{n,m_n,2}X\overset{n\tooi}{\longrightarrow}\textbf{0}.
\end{align*}
The conclusion follows accordingly.
\end{proof}

\ \\
The preceding lemma demonstrated that the estimator $\hat{\Psi}$ for the parameter $\Psi$ from display \eqref{eqn_psiEstiamtorAndPara} possesses an asymptotic variance of $(1-2\delta)\Upsilon_{\alpha'}\Sigma^{-1}$, as assumed. We will now present a lemma, which helps proving the Conditions (I) and (II) from Corollary \ref{prop_CLTUtev_multivariate}.

\begin{lemma}\label{lemma_FirstCondMultiMLRM}
On the Assumptions \ref{assumption_observations_multi}, \ref{assumption_regMulti} and \ref{assumption_fullRank}, we have 
\begin{align*}
\Var\bigg(\sum_{i=a}^b\gamma^\top \Xi_{n,i}\bigg)\leq C \sum_{i=a}^b\Var(\gamma^\top\Xi_{n,i}),
\end{align*}
for all $1\leq a\leq b\leq n$, $\Xi_{n,i}$ defined in equation \eqref{eqn_ximultiMLRM}, an universal constant $C>0$ and $\gamma\in\R^{d+1}$ arbitrary but fixed.
\end{lemma}

\begin{proof}
For an arbitrary but fixed vector $\gamma\in\mathbb{R}^{d+1}$, we can establish, in a manner analogous to Lemma \ref{lemma_assympVarMultiMLRM}, that
\begin{align*}
\Var\bigg(\sum_{i=a}^b\gamma^\top \Xi_{n,i}\bigg)&=\gamma^\top \Var\bigg(\sum_{i=a}^b \xi_{n,i}\bigg)\gamma\\
&=\gamma^\top\frac{(b-a+1)^2(1-2\delta)^2}{nmK^2\sigma_0^4}\bigg(\frac{1-2\delta}{m}X^\top X\bigg)^{-1}X^\top\Var\big(\tilde{Y}_{a,b}\big)X\bigg(\frac{1-2\delta}{m}X^\top X\bigg)^{-1}\gamma,
\end{align*}
where
\begin{align*}
\tilde{Y}_{a,b}:=\begin{pmatrix}
\sum_{i=a}^b\frac{(\Delta_i\tilde{X})^2(\textbf{y}_1)}{(b-a+1)\Delta_n^{\alpha'}}e^{\nor{\kappa\bigcdot\textbf{y}_1}_1}\\ \vdots \\ \sum_{i=a}^b\frac{(\Delta_i\tilde{X})^2(\textbf{y}_m)}{(b-a+1)\Delta_n^{\alpha'}}e^{\nor{\kappa\bigcdot\textbf{y}_m}_1}
\end{pmatrix}\in\R^{m}.
\end{align*}
For the variance $\Var(\tilde{Y}_{a,b}):=V_{a,b,n,m}:=V_{a,b,n,m,1}+V_{a,b,n,m,2}$ we find:
\begin{align*}
(V_{a,b,n,m,1})_{j_1,j_2}&:=\begin{cases}
\frac{\Upsilon_{\alpha'}}{(b-a+1)}K^2\sigma_0^4\big(1+\Delta_n^{1/2}\vee \Delta_n^{1-\alpha'}\big) &, \text{ if } 1\leq j_1=j_2\leq m \\
0 &,\text{ if } 1\leq j_1,j_2\leq m \text{ for }j_1\neq j_2
\end{cases} \\
(V_{a,b,n,m,2})_{j_1,j_2}&:=\begin{cases}
0 &, \text{ if } 1\leq j_1=j_2\leq m \\
\Oo\big(\frac{1}{b-a+1}\Delta_n^{1-\alpha'}(\nor{\textbf{y}_{j_1}-\textbf{y}_{j_2}}_0^{-(d+1)}+\delta^{-(d+1)})\big) &,\text{ if } 1\leq j_1,j_2\leq m \text{ for }j_1\neq j_2
\end{cases},
\end{align*}
and thus, we have 
\begin{align*}
&\Var\bigg(\sum_{i=a}^b\gamma^\top \Xi_{n,i}\bigg)\\
&=\Oo\bigg(\frac{(b-a+1)^2(1-2\delta)}{nK^2\sigma_0^4}\cdot\frac{K^2\sigma_0^4\Upsilon_{\alpha'}}{b-a+1}\gamma^\top\bigg(\frac{1-2\delta}{m}X^\top X\bigg)^{-1}\bigg(\frac{1-2\delta}{m}X^\top X\bigg)\bigg(\frac{1-2\delta}{m}X^\top X\bigg)^{-1}\gamma\bigg)\\
&~~~~~+\frac{(b-a+1)(1-2\delta)^2}{nK^2\sigma_0^4}\gamma^\top\bigg(\frac{1-2\delta}{m}X^\top X\bigg)^{-1}\bigg(\frac{b-a+1}{m}X^\top V_{a,b,n,m,2} X\bigg)\bigg(\frac{1-2\delta}{m}X^\top X\bigg)^{-1}\gamma.
\end{align*}
Similar to the proof of Lemma \ref{lemma_assympVarMultiMLRM}, it can be deduced that
\begin{align*}
\bigg(\frac{b-a+1}{m}X^\top V_{a,b,n,m,2} X\bigg)_{i,l}=\Oo(\Delta_n^{1-\alpha'}m^{d+2}),
\end{align*}
where we used that $((1-2\delta)/m_n\cdot X^\top X)^{-1}\rightarrow \Sigma^{-1}$, as $n\tooi$, and $\gamma^\top\Sigma^{-1}\gamma=\Oo(\nor{\gamma}_\infty) $.
Therefore, it holds:
\begin{align*}
\Var\bigg(\sum_{i=a}^b\gamma^\top \Xi_{n,i}\bigg)&=\Oo\big(\nor{\gamma}_\infty\Delta_n(b-a+1)+\nor{\gamma}_\infty\Delta_n(b-a+1)\Delta_n^{1-\alpha'}m^{d+2}\big)=\Oo\big(\nor{\gamma}_\infty\Delta_n(b-a+1)\big).
\end{align*}
Applying a similar approach to $\Var(\gamma^\top\Xi_{n,i})$ yields:
\begin{align*}
\Var(\gamma^\top\Xi_{n,i})=\gamma^\top\frac{(1-2\delta)^2}{nmK^2\sigma_0^4}\bigg(\frac{1-2\delta}{m}X^\top X\bigg)^{-1}X^\top\Var\big(\tilde{Y}_{i}\big)X\bigg(\frac{1-2\delta}{m}X^\top X\bigg)^{-1}\gamma,
\end{align*}
where
\begin{align*}
\tilde{Y}_{i}:=\begin{pmatrix}
\frac{(\Delta_i\tilde{X})^2(\textbf{y}_1)}{\Delta_n^{\alpha'}}e^{\nor{\kappa\bigcdot\textbf{y}_1}_1}\\ \vdots \\ \frac{(\Delta_i\tilde{X})^2(\textbf{y}_m)}{\Delta_n^{\alpha'}}e^{\nor{\kappa\bigcdot\textbf{y}_m}_1}
\end{pmatrix}\in\R^{m}.
\end{align*}
Defining $\Var(\tilde{Y}_i):=V_{i,n,m}:=V_{i,n,m,1}+V_{i,n,m,2}$, where
\begin{align*}
(V_{i,n,m,1})_{j_1,j_2}&:=\begin{cases}
\Upsilon_{\alpha'}K^2\sigma_0^4\big(1+\Delta_n^{1/2}\vee \Delta_n^{1-\alpha'}\big) &, \text{ if } 1\leq j_1=j_2\leq m \\
0 &,\text{ if } 1\leq j_1,j_2\leq m \text{ for }j_1\neq j_2
\end{cases}, \\
(V_{i,n,m,2})_{j_1,j_2}&:=\begin{cases}
0 &, \text{ if } 1\leq j_1=j_2\leq m \\
\Oo\big(\Delta_n^{1-\alpha'}(\nor{\textbf{y}_{j_1}-\textbf{y}_{j_2})}_0^{-(d+1)}+\delta^{-(d+1)}\big) &,\text{ if } 1\leq j_1,j_2\leq m \text{ for }j_1\neq j_2
\end{cases},
\end{align*}
yields:
\begin{align*}
\Var(\gamma^\top\Xi_{n,i})=\Oo(\nor{\gamma}_\infty\Delta_n+\nor{\gamma}_\infty\Delta_n\Delta_n^{1-\alpha'}m^{d+2})=\Oo(\nor{\gamma}_\infty\Delta_n).
\end{align*}
Consequently, we obtain $\sum_{i=a}^b\Var(\gamma^\top\Xi_{n,i})=\Oo\big(\nor{\gamma}_\infty\Delta_n(b-a+1)\big)$, which concludes the proof.
\end{proof}

\ \\
The subsequent lemma establishes the proof for the third condition of Corollary \ref{prop_CLTUtev_multivariate}.

\begin{lemma}\label{lemma_3CondMLRM}
On the Assumptions \ref{assumption_observations_multi}, \ref{assumption_regMulti} and \ref{assumption_fullRank}, it holds that
\begin{align*}
\sum_{i=1}^n\E[(\gamma^\top\Xi_{n,i})^4]=\Oo\big(\nor{\gamma}_\infty^4\Delta_nm^2\big),
\end{align*}
where $\Xi_{n,i}$ defined in equation \eqref{eqn_ximultiMLRM} and $\gamma\in\R^{d+1}$ is arbitrary but fixed.
\end{lemma}

\begin{proof}
We initiate the proof by examining:
\begin{align*}
\E[(\gamma^\top\Xi_{n,i})^4]&=\Oo\Big(\E\big[(\gamma^\top\xi_{n,i})^4\big]\Big).
\end{align*}
Thus, we proceed with analysing $\E[(\gamma^\top\xi_{n,i})^4]$. By utilizing the Cauchy-Schwarz inequality, we obtain:
\begin{align*}
\E[(\gamma^\top\xi_{n,i})^4]&=\E\bigg[\sum_{l_1,\ldots,l_4=1}^{d+1}\gamma_{l_1}(\xi_{n,i})_{l_1}\cdots\gamma_{l_4}(\xi_{n,i})_{l_4}\bigg]\\
&\leq\sum_{l_1,\ldots,l_4=1}^{d+1}\gamma_{l_1}\cdots\gamma_{l_4}\E\big[(\xi_{n,i})_{l_1}^4\big]^{1/4}\cdots \E\big[(\xi_{n,i})_{l_4}^4\big]^{1/4}\\
&\leq \nor{\gamma}_\infty^4 (d+1)^4 \max_{l=1,\ldots,d+1}\E\big[(\xi_{n,i})_{l}^4\big].
\end{align*}
We exploit the fact that $X\leq \textbf{1}_{m,d+1}$, where $\textbf{1}_{a,b}\in\R^{a\times b}$ represents the matrix of ones, which leads to:
\begin{align}
\xi_{n,i}
&\leq \frac{(1-2\delta)e^{\nor{\kappa}_1}}{\sqrt{nm}\Delta_{n}^{\alpha'}K\sigma_0^2}\bigg(\frac{1-2\delta}{m}X^\top X\bigg)^{-1}\textbf{1}_{(d+1),m}\begin{pmatrix}
(\Delta_i\tilde{X})^2(\textbf{y}_1)\\ \vdots \\ (\Delta_i\tilde{X})^2(\textbf{y}_m)
\end{pmatrix}\label{eqn_inequalityforXTransposed}\\
&=\frac{(1-2\delta)e^{\nor{\kappa}_1}}{\sqrt{nm}\Delta_{n}^{\alpha'}K\sigma_0^2}\bigg(\frac{1-2\delta}{m}X^\top X\bigg)^{-1}\begin{pmatrix}
\sum_{j=1}^m(\Delta_i\tilde{X})^2(\textbf{y}_j)\\ \vdots \\ \sum_{j=1}^m(\Delta_i\tilde{X})^2(\textbf{y}_j))
\end{pmatrix}.\notag
\end{align}
Thus, we find that
\begin{align*}
\E[(\gamma^\top\xi_{n,i})^4]&\leq \nor{\gamma}_\infty^4\frac{(1-2\delta)^4e^{4\nor{\kappa}_1}(d+1)^4}{n^2m^2\Delta_{n}^{4\alpha'}K^4\sigma_0^8} \max_{l=1,\ldots,d+1}\E\bigg[\bigg(\sum_{j=1}^m(\Delta_i\tilde{X})^2(\textbf{y}_j)\bigg)^4\bigg(\Big(\frac{1-2\delta}{m}X^\top X\Big)^{-1}\textbf{1}_{d+1,1}\bigg)_{l}^4\bigg]\\
&= \nor{\gamma}_\infty^4\frac{(1-2\delta)^4e^{4\nor{\kappa}_1}(d+1)^4}{n^2m^2\Delta_{n}^{4\alpha'}K^4\sigma_0^8} \E\bigg[\Big(\sum_{j=1}^m(\Delta_i\tilde{X})^2(\textbf{y}_j)\Big)^4\bigg] \max_{l=1,\ldots,d+1}\bigg(\Big(\frac{1-2\delta}{m}X^\top X\Big)^{-1}\textbf{1}_{d+1,1}\bigg)_{l}^4.
\end{align*}
Given that the matrix $\big((1-2\delta)m_n^{-1}(X^\top X)\big)^{-1}$ is converging to $\Sigma^{-1}$, as $n\tooi$, we can constrain:
\begin{align*}
\max_{l=1,\ldots,d+1}\bigg(\Big(\frac{1-2\delta}{m}X^\top X\Big)^{-1}\textbf{1}_{d+1,1}\bigg)_{l}^4\leq \bigg((d+1)\bigg|\bigg|\Big(\frac{1-2\delta}{m_n}X^\top X\Big)^{-1}\bigg|\bigg|_\infty\bigg)^4<\infty,
\end{align*}
for all $n\in\N$ and especially for $n\tooi$. As a result, employing the Cauchy-Schwarz inequality yields:
\begin{align*}
\E[(\gamma^\top\xi_{n,i})^4]&\leq C\nor{\gamma}_\infty^4\frac{(1-2\delta)^4e^{4\nor{\kappa}_1}(d+1)^4}{n^2m^2\Delta_{n}^{4\alpha'}K^4\sigma_0^8} \E\bigg[\Big(\sum_{j=1}^m(\Delta_i\tilde{X})^2(\textbf{y}_j)\Big)^4\bigg]\\
&\leq C\nor{\gamma}_\infty^4\frac{m^2(1-2\delta)^4e^{4\nor{\kappa}_1}(d+1)^4}{n^2\Delta_{n}^{4\alpha'}K^4\sigma_0^8}\max_{j=1,\ldots,m}\E\big[(\Delta_i\tilde{X})^8(\textbf{y}_{j})\big].
\end{align*}
Similarly to the demonstration of Condition (III) in Proposition \ref{prop_cltVolaEstMulti}, we have $\E[(\Delta_i\tilde{X})^8(\textbf{y})]=\Oo(\Delta_n^{4\alpha})$. This leads us to:
\begin{align*}
\sum_{i=1}^n\E[(\gamma^\top\xi_{n,i})^4]&=\Oo\big(\nor{\gamma}_\infty^4\Delta_nm^2\big),
\end{align*}
which completes the proof.
\end{proof}

\ \\
The following corollary establishes that the temporal dependencies within the triangular array, as outlined in Condition (IV) of Corollary \ref{prop_CLTUtev_multivariate}, can be bounded.
\begin{cor}\label{corollary_4CondMLRM}
On the Assumptions \ref{assumption_observations_multi}, \ref{assumption_regMulti} and \ref{assumption_fullRank}, it holds for $1\leq r<r+u\leq v\leq n$ and
\begin{align*}
\tilde{Q}_1^r=\sum\limits_{i=1}^r \gamma^\top\xi_{n,i},~~~~~~~~\tilde{Q}_{r+u}^v=\sum\limits_{i=r+u}^v \gamma^\top\xi_{n,i},
\end{align*}
where $\xi_{n,i}$ is defined in equation \eqref{eqn_ximultiMLRM}, that there is a constant $C$, with $0<C<\infty$, such that for all $t\in\R$ it holds:
\begin{align*}
\abs{\Cov\bigg(e^{\im t(\tilde{Q}_1^r-\E[\tilde{Q}_1^r])},~e^{\im t(\tilde{Q}_{r+u}^v-\E[\tilde{Q}_{r+u}^v])}\bigg)}\leq \frac{Ct^2}{u^{3/4}}\sqrt{\Var(\tilde{Q}_1^r)\Var(\tilde{Q}_{r+u}^v)}.
\end{align*}
\end{cor}

\begin{proof}
We follow a similar approach as in display \eqref{eqn_inequalityforXTransposed}, resulting in:
\begin{align*}
\gamma^\top\xi_{n,i}&\leq \frac{\sqrt{n}(1-2\delta)}{\sqrt{m}K\sigma_0^2}\gamma^\top\bigg(\frac{1-2\delta}{m}X^\top X\bigg)^{-1}\textbf{1}_{(d+1),m}\begin{pmatrix}
\frac{(\Delta_i\tilde{X})^2(\textbf{y}_1)}{n\Delta_n^{\alpha'}}e^{\nor{\kappa\bigcdot\textbf{y}_1}_1}\\ \vdots \\ \frac{(\Delta_i\tilde{X})^2(\textbf{y}_m)}{n\Delta_n^{\alpha'}}e^{\nor{\kappa\bigcdot\textbf{y}_m}_1}
\end{pmatrix}\\
&= \frac{1-2\delta}{\sqrt{nm}\Delta_n^{\alpha'}K\sigma_0^2}\sum_{j=1}^m(\Delta_i\tilde{X})^2(\textbf{y}_j)e^{\nor{\kappa\bigcdot\textbf{y}_1}_1}\gamma^\top\bigg(\frac{1-2\delta}{m}X^\top X\bigg)^{-1}\textbf{1}_{(d+1),1}\\
&\leq C\nor{\gamma}_\infty\sigma^2\frac{\eta^{d/2}}{\sqrt{nm}\Delta_n^{\alpha'}K}\sum_{j=1}^m(\Delta_i\tilde{X})^2(\textbf{y}_j)e^{\nor{\kappa\bigcdot\textbf{y}_1}_1}.
\end{align*}
With reference to Corollary \ref{corollary_4thConditionMultiDimAndMultiSpace}, it is evident that the statement holds for:
\begin{align*}
\frac{\eta^{d/2}}{\sqrt{nm}\Delta_n^{\alpha'}K}\sum_{j=1}^m(\Delta_i\tilde{X})^2(\textbf{y}_j)e^{\nor{\kappa\bigcdot \textbf{y}_j}_1},
\end{align*}
which completes the proof.
\end{proof}

\begin{proof}
To prove Proposition \ref{clt_PsiMulti}, we leverage Corollary \ref{prop_CLTUtev_multivariate}. The asymptotic variance is provided by Lemma \ref{lemma_assympVarMultiMLRM}. Condition (I) is fulfilled as demonstrated in Lemma \ref{lemma_FirstCondMultiMLRM}. In order to establish Condition (II), it suffices to consider the $\Var(\Xi_{n,i})$, as $\Xi$ is centred. Revisiting Lemma \ref{lemma_FirstCondMultiMLRM} confirms Condition (II). The fulfillment of Conditions (III) and (IV) is validated by Lemma \ref{lemma_3CondMLRM} and Corollary \ref{corollary_4CondMLRM}, respectively, which concludes the proof.
\end{proof}

We close this section by providing the prove for Corollary \ref{corollary_hatBetaCLT_61}.
\begin{proof}
Utilizing the multivariate delta method on the central limit theorem presented in Proposition \ref{clt_PsiMulti} and employing the function $h^{-1}(\textbf{x})=(e^{x_1}/K,-x_2,\ldots,-x_{d+1})$, as defined in equation \eqref{eqn_functionHMultiForDeltaMethod}, yields:
\begin{align*}
\sqrt{nm_n}(\hat{\upsilon}-\upsilon)=\sqrt{nm_n}\big(h^{-1}(\hat{\Psi})-h^{-1}(\Psi)\big)\overset{d}{\longrightarrow}\mathcal{N}\big(\textbf{0},\Upsilon_{\alpha'}(1-2\delta)J_{h^{-1}}(\Psi)\Sigma^{-1}J_{h^{-1}}(\Psi)^\top\big),
\end{align*}
where $J_{h^{-1}}$ denotes the Jacobian matrix of $h^{-1}$, given by:
\begin{align*}
J_{h^{-1}}(\textbf{x})=\begin{pmatrix}
e^{x_1}/K & 0&0 & \hdots & 0 \\
0 & -1 & 0 & \hdots & 0 \\
0 & 0 & -1 & \hdots & 0 \\
\vdots & \vdots & \vdots & \ddots & \vdots \\
0 & 0 & 0 &\hdots & -1
\end{pmatrix}.
\end{align*}
Defining the following matrix:
\begin{align}
J_{\sigma_0^2}:=J_{h^{-1}}(\Psi)=\begin{pmatrix}
\sigma_0^2 & 0&0 & \hdots & 0 \\
0 & -1 & 0 & \hdots & 0 \\
0 & 0 & -1 & \hdots & 0 \\
\vdots & \vdots & \vdots & \ddots & \vdots \\
0 & 0 & 0 &\hdots & -1
\end{pmatrix},\label{eqn_JSigma0SqDef}
\end{align}
completes the proof.
\end{proof}
\  \\ \\
We close this section by providing the covariance structure between quadratic increments and consecutive temporal increments. 
Therefore, we introduce the following definitions:
\begin{align*}
V_{p_1,\Delta_{n}}(\textbf{y}):=\frac{1}{p_1\Delta_{n}^{\alpha'}}\sum_{i=1}^{p_1}(\Delta_{n,i}\tilde{X})^2(\textbf{y})e^{||\kappa\bigcdot \textbf{y}||_1}~~~~~\text{and}~~~~~V_{p_2,\Delta_{2n}}(\textbf{y}):=\frac{1}{p_2\Delta_{2n}^{\alpha'}}\sum_{i=1}^{p_2}(\Delta_{2n,i}\tilde{X})^2(\textbf{y})e^{||\kappa\bigcdot \textbf{y}||_1},
\end{align*}
with $1\leq p_1\leq n$ and $1\leq p_2\leq 2n$. Furthermore, utilizing \eqref{eqn_1212} yields that
\begin{align*}
V_{n,\Delta_{n}}(\textbf{y})&=\frac{e^{||\kappa\bigcdot \textbf{y}||}}{n\Delta_n^{\alpha'}}\text{RV}_{n}(\textbf{y})=2^{1-\alpha'}V_{2n,\Delta_{2n}}(\textbf{y})+\frac{4n\Delta_{2n}^{\alpha'}}{n\Delta_{n}^{\alpha'}}W_{2n,\Delta_{2n}},
\end{align*}
where we define for $1\leq p\leq 2n$:
\begin{align}
W_{p,\Delta_{2n}}(\textbf{y}):=\frac{1}{p\Delta_{2n}^{\alpha'}}\sum_{i=1}^{p}\mathbbm{1}_{2\N}(i)(\Delta_{2n,i}\tilde{X})(\textbf{y})(\Delta_{2n,i-1}\tilde{X})(\textbf{y})e^{||\kappa\bigcdot \textbf{y}||}.\label{eqn_W_definition_forAlpha}
\end{align}
Note that we have 
\begin{align}
V_{p,\Delta_n}(\textbf{y})=2^{1-\alpha'}V_{2p,\Delta_{2n}}(\textbf{y})+2^{2-\alpha'}W_{2p,\Delta_{2n}}(\textbf{y}),\label{eqn_Vp_and_V2p_link}
\end{align}
for a $1\leq p\leq n$.
Hence, the estimator $\hat{\alpha}'$ from equation \eqref{eqn_defalphaEst} can be decomposed as follows:
\begin{align*}
\hat{\alpha}_{2n,m}'&=\alpha'+\frac{1}{\log(2)m\sigma_0^2 K}\sum_{j=1}^m\Big((2^{1-\alpha'}-1)\overline{V_{2n,\Delta_{2n}}(\textbf{y}_j)}+2^{2-\alpha'}\overline{(W_{2n,\Delta_{2n}}(\textbf{y}_j)}\Big)+\Oo(\Delta_n)+\Oo_\Pp(\Delta_n)
\end{align*}
The corresponding triangular array is given by $\Xi_{2n,i}:=\xi_{2n,i}-\E[\xi_{2n,i}]$, where
\begin{align}
\xi_{2n,i}&:=\sum_{j=1}^m\frac{e^{\nor{\kappa\bigcdot\textbf{y}_j}_1}}{\log(2)\sqrt{2nm}\Delta_{2n}^{\alpha'}\sigma_0^2 K}\bigg((2^{1-\alpha'}-1)(\Delta_{2n,i}\tilde{X})^2(\textbf{y}_j)+2^{2-\alpha'}\mathbbm{1}_{2\N}(i)(\Delta_{2n,i}\tilde{X})(\textbf{y}_j)(\Delta_{2n,i-1}\tilde{X})(\textbf{y}_j)\bigg)\label{eqn_xiAlpha}\\
&=\xi_{2n,i}^1+\xi_{2n,i}^2\notag,
\end{align}
with 
\begin{align*}
\xi_{2n,i}^1&:= \frac{2^{1-\alpha'}-1}{\log(2)\sqrt{2nm}\Delta_{2n}^{\alpha'}\sigma_0^2 K} \sum_{j=1}^m(\Delta_{2n,i}\tilde{X})^2(\textbf{y}_j)e^{\nor{\kappa\bigcdot\textbf{y}_j}_1}, \\
\xi_{2n,i}^2&:=\mathbbm{1}_{2\N}(i)\frac{2^{2-\alpha'}}{\log(2)\sqrt{2nm}\Delta_{2n}^{\alpha'}\sigma_0^2 K}\sum_{j=1}^m(\Delta_{2n,i}\tilde{X})(\textbf{y}_j)(\Delta_{2n,i-1}\tilde{X})(\textbf{y}_j)e^{\nor{\kappa\bigcdot\textbf{y}_j}_1}.
\end{align*}
We now provide a necessary proposition for deriving the asymptotic variance of the triangular array from equation \eqref{eqn_xiAlpha}. Combining these result with Proposition \ref{prop_RRVMulti} and using analogous techniques as used for proofing Corollary \ref{corollary_4thConditionMultiDimAndMultiSpace} and Proposition \ref{prop_cltVolaEstMulti}, we can conclude the CLT from Proposition \ref{prop_cltAlpha}. 
\begin{theorem}\label{prop_RRV_2Grids}
On the Assumptions \ref{assumption_observations_multi} and \ref{assumption_regMulti}, we have for the covariance structure of the two temporal resolutions $\Delta_n$ and $\Delta_{2n}$ that
\begin{align*}
\Cov\big(V_{p,\Delta_{2n}}(\textbf{y}_1),W_{p,\Delta_{2n}}(\textbf{y}_2)\big)&= \frac{\Lambda_{\alpha'}}{2p}\bigg(\frac{\Gamma(1-\alpha')\sigma^4}{2^d(\pi\eta)^{d/2}\alpha'\Gamma(d/2)}\bigg)^2\bigg(1+\mathcal{O}\bigg(\Delta_{2n}^{1/2}\vee \frac{\Delta_{2n}^{1-\alpha'}}{\delta^{d+1}}\vee\frac{1}{p}\bigg)\bigg)\\
&~~~~~ +\Oo\bigg(\frac{\Delta_{2n}^{1-\alpha'}}{p}\Big(\mathbbm{1}_{\{\textbf{y}_1\neq \textbf{y}_2\}}\nor{\textbf{y}_1-\textbf{y}_2}_0^{-(d+1)}+\delta^{-(d+1)}\Big)\bigg),
\end{align*}
where $\textbf{y}_1,\textbf{y}_2\in[\delta,1-\delta]^d$, $\Lambda_{\alpha'}$ is a numerical constant depending on $\alpha'\in(0,1)$, defined in equation \eqref{eqn_definingLambda} and $2\leq p\leq 2n$. 
\end{theorem}
\begin{proof}
Analogously to Proposition \ref{prop_RRVMulti}, we first obtain that 
\begin{align*}
\Cov\big(V_{p,\Delta_{2n}}(\textbf{y}_1),W_{p,\Delta_{2n}}(\textbf{y}_2)\big)&=\frac{2e^{\nor{\kappa\bigcdot(\textbf{y}_1+\textbf{y}_2)}_1}}{p\Delta_{2n}^{2\alpha'}}\sum_{\textbf{k}_1,\textbf{k}_2\in\N^d}e_{\textbf{k}_1}(\textbf{y}_1)e_{\textbf{k}_1}(\textbf{y}_2)e_{\textbf{k}_2}(\textbf{y}_1)e_{\textbf{k}_2}(\textbf{y}_2)D_{\textbf{k}_1,\textbf{k}_2},
\end{align*}
where we redefine
\begin{align*}
D_{\textbf{k}_1,\textbf{k}_2}&:=\frac{1}{p}\sum_{i,j=1}^p\mathbbm{1}_{2\N}(j)\Cov\Big(\big(\tilde{B}_{i,\textbf{k}_1}+C_{i,\textbf{k}_1}\big)\big(\tilde{B}_{i,\textbf{k}_2}+C_{i,\textbf{k}_2}\big),~\big(\tilde{B}_{j,\textbf{k}_1}+C_{j,\textbf{k}_1}\big)\big(\tilde{B}_{j-1,\textbf{k}_2}+C_{j-1,\textbf{k}_2}\big)\Big)\\
&=\frac{1}{p} \sum\limits_{i,j=1}^p
\mathbbm{1}_{2\N}(j)\bigg(\E\Big[\big(\tilde{B}_{i,\textbf{k}_1}+C_{i,\textbf{k}_1}\big)\big(\tilde{B}_{j,\textbf{k}_1}+C_{j,\textbf{k}_1}\big)\Big]
\E\Big[\big(\tilde{B}_{i,\textbf{k}_2}+C_{i,\textbf{k}_2}\big)\big(\tilde{B}_{j-1,\textbf{k}_2}+C_{j-1,\textbf{k}_2}\big)\Big]\\
&~~~~~~~+\E\Big[\big(\tilde{B}_{i,\textbf{k}_1}+C_{i,\textbf{k}_1}\big)
\big(\tilde{B}_{j-1,\textbf{k}_2}+C_{j-1,\textbf{k}_2}\big)\Big]\E\Big[\big(\tilde{B}_{i,\textbf{k}_2}+C_{i,\textbf{k}_2}\big)\big(\tilde{B}_{j,\textbf{k}_1}+C_{j,\textbf{k}_1}\big)\Big]\bigg).
\end{align*}
Assume $\textbf{k}_1\neq \textbf{k}_2$, then we have 
\begin{align*}
D_{\textbf{k}_1,\textbf{k}_2}&=\frac{1}{p} \sum\limits_{i,j=1}^p
\mathbbm{1}_{2\N}(j)\bigg(\E\Big[\big(\tilde{B}_{i,\textbf{k}_1}+C_{i,\textbf{k}_1}\big)\big(\tilde{B}_{j,\textbf{k}_1}+C_{j,\textbf{k}_1}\big)\Big]
\E\Big[\big(\tilde{B}_{i,\textbf{k}_2}+C_{i,\textbf{k}_2}\big)\big(\tilde{B}_{j-1,\textbf{k}_2}+C_{j-1,\textbf{k}_2}\big)\Big]\\
&=\frac{1}{p} \sum\limits_{i,j=1}^p\mathbbm{1}_{2\N}(j)
\Big(\tilde{\Sigma}_{i,j}^{B,\textbf{k}_1}+\Sigma_{i,j}^{BC,\textbf{k}_1}+\Sigma_{j,i}^{BC,\textbf{k}_1}+\Sigma_{i,j}^{C,\textbf{k}_1}\Big)\Big(\tilde{\Sigma}_{i,j-1}^{B,\textbf{k}_2}+\Sigma_{i,j-1}^{BC,\textbf{k}_2}+\Sigma_{j-1,i}^{BC,\textbf{k}_2}+\Sigma_{i,j-1}^{C,\textbf{k}_2}\Big).
\end{align*}
For the covariance terms we have by Proposition \ref{prop_RRVMulti}, that 
\begin{align*}
\frac{1}{p} \sum\limits_{i,j=1}^p\mathbbm{1}_{2\N}(j)\tilde{\Sigma}_{i,j}^{B,\textbf{k}_1}\tilde{\Sigma}_{i,j-1}^{B,\textbf{k}_2}&=\sigma^4\frac{ \big(1-e^{-\lambda_{\textbf{k}_1}\Delta_{2n}}\big)^2\big(1-e^{-\lambda_{\textbf{k}_2}\Delta_{2n}}\big)^2 }{4\lambda_{\textbf{k}_1}^{1+\alpha}\lambda_{\textbf{k}_2}^{1+\alpha}}\frac{1}{p} \sum\limits_{i,j=1}^p\mathbbm{1}_{2\N}(j) e^{-\lambda_{\textbf{k}_1}\Delta_{2n}\abs{i-j}} e^{-\lambda_{\textbf{k}_2}\Delta_{2n}\abs{i-j+1}}.
\end{align*}
For the geometric sum in the latter display, we obtain:
\begin{align*}
\sum_{i,j=1}^pq_1^{\abs{i-j}}q_2^{\abs{i-j+1}}\mathbbm{1}_{2\N}(j)
&=q_2\sum_{i=2}^p(q_1q_2)^{i}\sum_{j=2}^{i}(q_1q_2)^{-j}\mathbbm{1}_{2\N}(j)+q_2^{-1}\sum_{j=2}^p(q_1q_2)^{j}\mathbbm{1}_{2\N}(j)\sum_{i=1}^{j-1}(q_1q_2)^{-i},
\end{align*}
where $q_1,q_2\neq 0$. Furthermore, for a $q\neq 1$ it holds by analogous computations as for the partial sum of the geometric series, that
\begin{align*}
\sum_{i=0}^nq^i\mathbbm{1}_{2\N}(i)=\begin{cases}
\frac{1-q^{n+2}}{1-q^2} &,\text{ if }n\text{ is even}\\
\frac{1-q^{n+1}}{1-q^2} &,\text{ if }n\text{ is odd}\\
\end{cases},
\end{align*}
where we consider zero as even. Hence, we get:
\begin{align*}
\sum_{i,j=1}^pq_1^{\abs{i-j}}q_2^{\abs{i-j+1}}\mathbbm{1}_{2\N}(j)&=\frac{q_2}{(q_1q_2)^2(1-(q_1q_2)^{-2})}\bigg(\sum_{i=2}^p (q_1q_2)^i\big(1-(q_1q_2)^{-i}\big) \mathbbm{1}_{2\N}(i)\\
&~~~~~+\sum_{i=2}^p (q_1q_2)^i\big(1-(q_1q_2)^{-(i-1)}\big) \mathbbm{1}_{(2\N)^\complement}(i)\bigg)\\
&~~~~~+\frac{q_2^{-1}}{q_1q_2(1-(q_1q_2)^{-1})}\sum_{j=2}^p (q_1q_2)^{j}\big(1-(q_1q_2)^{-(j-1)}\big)\mathbbm{1}_{2\N}(j)
\end{align*}
Now using, that $\abs{q_1},\abs{q_2}<1$ and that it holds for the floor function by the Fourier representation that
\begin{align*}
\frac{1}{p}\lfloor cp\rfloor &=\begin{cases}
c &, \text{ if } cp\in \Z \\
c-\frac{1}{2p}+\frac{1}{p\pi}\sum_{k=1}^\infty \frac{\sin(2\pi k cp)}{k} & ,\text{ if } cp\notin \Z ,
\end{cases},
\end{align*}
for $c\neq 0$ and $p\in\N$, we observe the following: 
\begin{align}
\sum_{i,j=1}^pq_1^{\abs{i-j}}q_2^{\abs{i-j+1}}\mathbbm{1}_{2\N}(j)&=\bigg(\frac{q_2}{1-(q_1q_2)^2}\Big( \frac{p}{2}+\frac{p}{2}q_1q_2\Big)+\frac{q_2^{-1}}{1-q_1q_2}\cdot\frac{p}{2}q_1q_2\bigg)\bigg(1+\Oo\Big(\frac{p^{-1}}{1-q_1q_2}\Big)\bigg)\notag\\
&=\frac{q_1+q_2}{2(1-q_1q_2)}\bigg(1+\Oo\Big(\frac{p^{-1}}{1-q_1q_2}\Big)\bigg).\label{eqn_GS_BB_ij1}
\end{align}
Therefore, we have 
\begin{align}
\frac{1}{p} \sum\limits_{i,j=1}^p\mathbbm{1}_{2\N}(j)\tilde{\Sigma}_{i,j}^{B,\textbf{k}_1}\tilde{\Sigma}_{i,j-1}^{B,\textbf{k}_2}&=\sigma^4\frac{ \big(1-e^{-\lambda_{\textbf{k}_1}\Delta_{2n}}\big)^2\big(1-e^{-\lambda_{\textbf{k}_2}\Delta_{2n}}\big)^2 }{4\lambda_{\textbf{k}_1}^{1+\alpha}\lambda_{\textbf{k}_2}^{1+\alpha}}\notag\\
&~~~~~\times\frac{e^{-\lambda_{\textbf{k}_1}\Delta_{2n}}+e^{-\lambda_{\textbf{k}_2}\Delta_{2n}}}{2(1-e^{-(\lambda_{\textbf{k}_1}+\lambda_{\textbf{k}_1})\Delta_{2n}})}\bigg(1+\Oo\Big(1\wedge \frac{p^{-1}}{1-e^{-(\lambda_{\textbf{k}_1}+\lambda_{\textbf{k}_2})}}\Big)\bigg).\label{eqn_CovW_STRC_1}
\end{align}
Furthermore, we have 
\begin{align*}
\frac{1}{p} \sum\limits_{i,j=1}^p\mathbbm{1}_{2\N}(j)\Sigma_{i,j}^{C,\textbf{k}_1}\Sigma_{i,j-1}^{C,\textbf{k}_2}&= \frac{\sigma^4}{p} \sum\limits_{i,j=1}^p\frac{(1-e^{-2\lambda_{\textbf{k}_1}\Delta_{2n}})(1-e^{-2\lambda_{\textbf{k}_2}\Delta_{2n}})}{4\lambda_{\textbf{k}_1}^{1+\alpha}\lambda_{\textbf{k}_2}^{1+\alpha}}\mathbbm{1}_{\{j=i\}}\mathbbm{1}_{\{j-1=i\}}\mathbbm{1}_{2\N}(j)=0,
\end{align*}
as well as
\begin{align*}
\frac{1}{p} \sum\limits_{i,j=1}^p\mathbbm{1}_{2\N}(j) \Sigma_{i,j}^{BC,\textbf{k}_1}\Sigma_{i,j-1}^{BC,\textbf{k}_2}&=\sigma^4\frac{(1-e^{-\lambda_{\textbf{k}_1}\Delta_{2n}})(1-e^{-\lambda_{\textbf{k}_2}\Delta_{2n}})}{4\lambda_{\textbf{k}_1}^{1+\alpha}\lambda_{\textbf{k}_2}^{1+\alpha}}\big(e^{\lambda_{\textbf{k}_1}\Delta_{2n}}-e^{-\lambda_{\textbf{k}_1}\Delta_{2n}}\big)\big(e^{\lambda_{\textbf{k}_2}\Delta_{2n}}-e^{-\lambda_{\textbf{k}_2}\Delta_{2n}}\big)
\\
&~~~~~\times\frac{1}{p} \sum\limits_{i,j=1}^p\mathbbm{1}_{\{i>j\}}\mathbbm{1}_{2\N}(j) e^{-\lambda_{\textbf{k}_1}\Delta_{2n}(i-j)}e^{-\lambda_{\textbf{k}_2}\Delta_{2n}(i-j+1)}.
\end{align*}
For the sum structure in the latter display we obtain:
\begin{align*}
\frac{1}{p} \sum\limits_{i,j=1}^p\mathbbm{1}_{\{i>j\}}\mathbbm{1}_{2\N}(j) e^{-\lambda_{\textbf{k}_1}\Delta_{2n}(i-j)}e^{-\lambda_{\textbf{k}_2}\Delta_{2n}(i-j+1)}&=\frac{e^{-\lambda_{\textbf{k}_2}\Delta_{2n}}}{p} \sum\limits_{i,j=1}^p\mathbbm{1}_{\{i>j\}}\mathbbm{1}_{2\N}(j) e^{-(\lambda_{\textbf{k}_1}+\lambda_{\textbf{k}_2})\Delta_{2n}(i-j)}.
\end{align*}
Assume $\abs{q}<1$, then we have
\begin{align}
\frac{1}{p} \sum\limits_{i,j=1}^p \mathbbm{1}_{\{i>j\}}\mathbbm{1}_{2\N}(j)q^{i-j}
&=\frac{q}{2(1-q)}\bigg(1+\Oo\Big(\frac{p^{-1}}{1-q}\Big)\bigg),\label{eqn_GS_ForBCBC}
\end{align}
where we used analogous steps leading to display \eqref{eqn_GS_BB_ij1}. Hence, we get:
\begin{align}
\frac{1}{p} \sum\limits_{i,j=1}^p\mathbbm{1}_{2\N}(j) \Sigma_{i,j}^{BC,\textbf{k}_1}\Sigma_{i,j-1}^{BC,\textbf{k}_2}
&=\sigma^4\frac{(1-e^{-\lambda_{\textbf{k}_1}\Delta_{2n}})(1-e^{-\lambda_{\textbf{k}_2}\Delta_{2n}})}{4\lambda_{\textbf{k}_1}^{1+\alpha}\lambda_{\textbf{k}_2}^{1+\alpha}}\big(1-e^{-2\lambda_{\textbf{k}_1}\Delta_{2n}}\big)\big(1-e^{-2\lambda_{\textbf{k}_2}\Delta_{2n}}\big)\notag\\
&~~~~~\times \frac{e^{-\lambda_{\textbf{k}_2}\Delta_{2n}}}{2(1-e^{-(\lambda_{\textbf{k}_1}+\lambda_{\textbf{k}_2})\Delta_{2n}})}\bigg(1+\Oo\Big(1 \wedge\frac{p^{-1}}{1-e^{-(\lambda_{\textbf{k}_1}+\lambda_{\textbf{k}_2})\Delta_{2n}}}\Big)\bigg).
\label{eqn_CovW_STRC_2}
\end{align}
Moreover, by analogous steps, we have 
\begin{align}
\frac{1}{p} \sum\limits_{i,j=1}^p\mathbbm{1}_{2\N}(j) \Sigma_{j,i}^{BC,\textbf{k}_1}\Sigma_{j-1,i}^{BC,\textbf{k}_2}
&=\sigma^4\frac{(1-e^{-\lambda_{\textbf{k}_1}\Delta_{2n}})(1-e^{-\lambda_{\textbf{k}_2}\Delta_{2n}})}{4\lambda_{\textbf{k}_1}^{1+\alpha}\lambda_{\textbf{k}_2}^{1+\alpha}}\big(1-e^{-2\lambda_{\textbf{k}_1}\Delta_{2n}}\big)\big(1-e^{-2\lambda_{\textbf{k}_2}\Delta_{2n}}\big)\notag\\
&~~~~~\times 
\frac{e^{-\lambda_{\textbf{k}_1}\Delta_{2n}}}{2(1-e^{-(\lambda_{\textbf{k}_1}+\lambda_{\textbf{k}_2})\Delta_{2n}})}\bigg(1+\Oo\Big(1\wedge\frac{p^{-1}}{1-e^{-(\lambda_{\textbf{k}_1}+\lambda_{\textbf{k}_2})\Delta_{2n}}}\Big)\bigg).\label{eqn_CovW_STRC_3}
\end{align}
where we used that
\begin{align}
 \frac{1}{p} \sum\limits_{i,j=1}^p \mathbbm{1}_{\{i<j-1\}}\mathbbm{1}_{2\N}(j)q^{j-i}
 &=\frac{q^2}{2(1-q)}\bigg(1+\Oo\Big(\frac{p^{-1}}{1-q}\Big)\bigg).\label{eqn_GS_ForBSBSji}
\end{align}
For the cross-terms we obtain that
\begin{align*}
&\frac{1}{p} \sum\limits_{i,j=1}^p\mathbbm{1}_{2\N}(j) \tilde{\Sigma}_{i,j}^{B,\textbf{k}_1}\big(\Sigma_{i,j-1}^{BC,\textbf{k}_2}+\Sigma_{j-1,i}^{BC,\textbf{k}_2}\big)
=\sigma^4\frac{(1-e^{-\lambda_{\textbf{k}_1}\Delta_{2n}})^2(e^{-\lambda_{\textbf{k}_2}\Delta_{2n}}-1)}{4\lambda_{\textbf{k}_1}^{1+\alpha}\lambda_{\textbf{k}_2}^{1+\alpha}}\big(e^{\lambda_{\textbf{k}_2}\Delta_{2n}}-e^{-\lambda_{\textbf{k}_2}\Delta_{2n}}\big)\\
&~~~~~\times\bigg( \frac{e^{-\lambda_{\textbf{k}_2}\Delta_{2n}}}{p} \sum\limits_{i,j=1}^p\mathbbm{1}_{2\N}(j) \mathbbm{1}_{\{i>j-1\}}e^{-(\lambda_{\textbf{k}_1}+\lambda_{\textbf{k}_2})\Delta_{2n}(i-j)}+\frac{e^{\lambda_{\textbf{k}_2}\Delta_{2n}}}{p} \sum\limits_{i,j=1}^p\mathbbm{1}_{2\N}(j) \mathbbm{1}_{\{i<j-1\}}e^{-(\lambda_{\textbf{k}_1}+\lambda_{\textbf{k}_2})\Delta_{2n}(j-i)}\big)\bigg) .
\end{align*}
Analogously to equation \eqref{eqn_GS_ForBCBC}, we have 
\begin{align*}
\frac{1}{p} \sum\limits_{i,j=1}^p \mathbbm{1}_{\{i>j-1\}}\mathbbm{1}_{2\N}(j)q^{i-j}
&=\frac{1}{2(1-q)}\bigg(1+\Oo\Big(\frac{p^{-1}}{1-q}\Big)\bigg),
\end{align*}
which yields in combination with equation \eqref{eqn_GS_ForBSBSji} that
\begin{align}
&\frac{1}{p} \sum\limits_{i,j=1}^p\mathbbm{1}_{2\N}(j) \tilde{\Sigma}_{i,j}^{B,\textbf{k}_1}\big(\Sigma_{i,j-1}^{BC,\textbf{k}_2}+\Sigma_{j-1,i}^{BC,\textbf{k}_2}\big)
=\sigma^4\frac{(1-e^{-\lambda_{\textbf{k}_1}\Delta_{2n}})^2(e^{-\lambda_{\textbf{k}_2}\Delta_{2n}}-1)}{4\lambda_{\textbf{k}_1}^{1+\alpha}\lambda_{\textbf{k}_2}^{1+\alpha}}\big(1-e^{-2\lambda_{\textbf{k}_2}\Delta_{2n}}\big)\notag\\
&~~~~~\times 
\frac{1+e^{-2\lambda_{\textbf{k}_1}\Delta_{2n}}}{2(1-e^{-(\lambda_{\textbf{k}_1}+\lambda_{\textbf{k}_2})\Delta_{2n}})}\bigg(1+\Oo\Big(1\wedge\frac{p^{-1}}{1-e^{-(\lambda_{\textbf{k}_1}+\lambda_{\textbf{k}_2})\Delta_{2n}}}\Big)\bigg)\label{eqn_CovW_STRC_4}
\end{align}
Moreover, it holds that
\begin{align}
&\frac{1}{p} \sum\limits_{i,j=1}^p\mathbbm{1}_{2\N}(j) \tilde{\Sigma}_{i,j-1}^{B,\textbf{k}_2}\big(\Sigma_{i,j}^{BC,\textbf{k}_1}+\Sigma_{j,i}^{BC,\textbf{k}_1}\big)=\sigma^4\frac{(1-e^{-\lambda_{\textbf{k}_2}\Delta_{2n}})^2(e^{-\lambda_{\textbf{k}_1}\Delta_{2n}}-1)}{4\lambda_{\textbf{k}_1}^{1+\alpha}\lambda_{\textbf{k}_2}^{1+\alpha}}\big(e^{\lambda_{\textbf{k}_1}\Delta_{2n}}-e^{-\lambda_{\textbf{k}_1}\Delta_{2n}}\big)\notag\\
&~~~~~\times\bigg(\frac{e^{-\lambda_{\textbf{k}_2}\Delta_{2n}}}{p} \sum\limits_{i,j=1}^p\mathbbm{1}_{2\N}(j)\mathbbm{1}_{\{i>j\}}e^{-(\lambda_{\textbf{k}_1}+\lambda_{\textbf{k}_2})\Delta_{2n}(i-j)}+\frac{e^{\lambda_{\textbf{k}_2}\Delta_{2n}}}{p} \sum\limits_{i,j=1}^p\mathbbm{1}_{2\N}(j)\mathbbm{1}_{\{j>i\}}e^{-(\lambda_{\textbf{k}_1}+\lambda_{\textbf{k}_2})\Delta_{2n}(j-i)}\bigg)\notag\\
&=\sigma^4\frac{(1-e^{-\lambda_{\textbf{k}_2}\Delta_{2n}})^2(e^{-\lambda_{\textbf{k}_1}\Delta_{2n}}-1)}{4\lambda_{\textbf{k}_1}^{1+\alpha}\lambda_{\textbf{k}_2}^{1+\alpha}}\big(1-e^{-2\lambda_{\textbf{k}_1}\Delta_{2n}}\big)\notag\\
&~~~~~\times \big(1+e^{-2\lambda_{\textbf{k}_2}\Delta_{2n}}\big)
\frac{1}{2(1-e^{-(\lambda_{\textbf{k}_1}+\lambda_{\textbf{k}_2})\Delta_{2n}})}\bigg(1+\Oo\Big(1\wedge\frac{p^{-1}}{1-e^{-(\lambda_{\textbf{k}_1}+\lambda_{\textbf{k}_2})\Delta_{2n}}}\Big)\bigg),\label{eqn_CovW_STRC_5}
\end{align}
where we used equation \eqref{eqn_GS_ForBCBC} and 
\begin{align*}
 \frac{1}{p} \sum\limits_{i,j=1}^p \mathbbm{1}_{\{i<j\}}\mathbbm{1}_{2\N}(j)q^{j-i}
 &=\frac{q}{2(1-q)}\bigg(1+\Oo\Big(\frac{p^{-1}}{1-q}\Big)\bigg).
\end{align*}
We also observe that
\begin{align}
\frac{1}{p} \sum\limits_{i,j=1}^p\mathbbm{1}_{2\N}(j) \tilde{\Sigma}_{i,j}^{B,\textbf{k}_1}\Sigma_{i,j-1}^{C,\textbf{k}_2}&=\sigma^4\frac{(1-e^{-\lambda_{\textbf{k}_1}\Delta_{2n}})^2(1-e^{-2\lambda_{\textbf{k}_2}\Delta_{2n}})}{4\lambda_{\textbf{k}_1}^{1+\alpha}\lambda_{\textbf{k}_2}^{1+\alpha}}
\frac{1}{p} \sum\limits_{i,j=1}^p\mathbbm{1}_{2\N}(j)e^{-\lambda_{\textbf{k}_1}\Delta_{2n}\abs{i-j}}
\mathbbm{1}_{\{j-1=i\}}\notag\\
&=\sigma^4e^{-\lambda_{\textbf{k}_1}\Delta_{2n}}\frac{(1-e^{-\lambda_{\textbf{k}_1}\Delta_{2n}})^2(1-e^{-2\lambda_{\textbf{k}_2}\Delta_{2n}})}{8\lambda_{\textbf{k}_1}^{1+\alpha}\lambda_{\textbf{k}_2}^{1+\alpha}}\Big(1+\Oo\big(p^{-1}\big)\Big),\label{eqn_CovW_STRC_6}
\end{align}
as well as
\begin{align}
\frac{1}{p} \sum\limits_{i,j=1}^p\mathbbm{1}_{2\N}(j) \tilde{\Sigma}_{i,j-1}^{B,\textbf{k}_2}\Sigma_{i,j}^{C,\textbf{k}_1}&=\sigma^4 \frac{(1-e^{-2\lambda_{\textbf{k}_1}\Delta_{2n}})(1-e^{-\lambda_{\textbf{k}_2}\Delta_{2n}})^2}{4\lambda_{\textbf{k}_1}^{1+\alpha}\lambda_{\textbf{k}_2}^{1+\alpha}}
\frac{1}{p} \sum\limits_{i,j=1}^p\mathbbm{1}_{2\N}(j)\mathbbm{1}_{\{j=i\}}e^{-\lambda_{\textbf{k}_2}\Delta_{2n}\abs{i-j+1}}\notag\\
&=\sigma^4 e^{-\lambda_{\textbf{k}_2}\Delta_{2n}}\frac{(1-e^{-2\lambda_{\textbf{k}_1}\Delta_{2n}})(1-e^{-\lambda_{\textbf{k}_2}\Delta_{2n}})^2}{8\lambda_{\textbf{k}_1}^{1+\alpha}\lambda_{\textbf{k}_2}^{1+\alpha}}\Big(1+\Oo\big(p^{-1}\big)\Big).\label{eqn_CovW_STRC_7}
\end{align}
In comparison to Proposition \ref{prop_RRVMulti}, the following structures do not vanish and we get
\begin{align}
\frac{1}{p} \sum\limits_{i,j=1}^p\mathbbm{1}_{2\N}(j)\Sigma_{j,i}^{BC,\textbf{k}_1}\Sigma_{i,j-1}^{C,\textbf{k}_2}
&=\sigma^4 \frac{(1-e^{-2\lambda_{\textbf{k}_2}\Delta_{2n}})(e^{-\lambda_{\textbf{k}_1}\Delta_{2n}}-1)}{8\lambda_{\textbf{k}_1}^{1+\alpha}\lambda_{\textbf{k}_2}^{1+\alpha}}
\Big(1-e^{-2\lambda_{\textbf{k}_1}\Delta_{2n}}\Big)\Big(1+\Oo\big(p^{-1}\big)\Big),\label{eqn_CovW_STRC_8}
\end{align}
as well as
\begin{align}
\frac{1}{p} \sum\limits_{i,j=1}^p\mathbbm{1}_{2\N}(j)\Sigma_{i,j}^{C,\textbf{k}_1}\Sigma_{i,j-1}^{BC,\textbf{k}_2}
&=\sigma^4 \frac{(1-e^{-2\lambda_{\textbf{k}_1}\Delta_{2n}})(e^{-\lambda_{\textbf{k}_2}\Delta_{2n}}-1)}{8\lambda_{\textbf{k}_1}^{1+\alpha}\lambda_{\textbf{k}_2}^{1+\alpha}}
\Big(1-e^{-2\lambda_{\textbf{k}_2}\Delta_{2n}}\Big)\Big(1+\Oo\big(p^{-1}\big)\Big),
\label{eqn_CovW_STRC_9}
\end{align}
whereas the following terms still vanish:
\begin{align*}
\frac{1}{p} \sum\limits_{i,j=1}^p\mathbbm{1}_{2\N}(j)\Sigma_{i,j}^{BC,\textbf{k}_1}\Sigma_{j-1,i}^{BC,\textbf{k}_2}&=0,
&\frac{1}{p} \sum\limits_{i,j=1}^p\mathbbm{1}_{2\N}(j)\Sigma_{i,j}^{BC,\textbf{k}_1}\Sigma_{i,j-1}^{C,\textbf{k}_2}=0,\\
\frac{1}{p} \sum\limits_{i,j=1}^p\mathbbm{1}_{2\N}(j)\Sigma_{j,i}^{BC,\textbf{k}_1}\Sigma_{i,j-1}^{BC,\textbf{k}_2}&=0,
&\frac{1}{p} \sum\limits_{i,j=1}^p\mathbbm{1}_{2\N}(j)\Sigma_{i,j}^{C,\textbf{k}_1}\Sigma_{j-1,i}^{BC,\textbf{k}_2}=0.
\end{align*}
Combining the calculations form the displays \eqref{eqn_CovW_STRC_1},\eqref{eqn_CovW_STRC_2},\eqref{eqn_CovW_STRC_3},\eqref{eqn_CovW_STRC_4},\eqref{eqn_CovW_STRC_5},\eqref{eqn_CovW_STRC_6},\eqref{eqn_CovW_STRC_7},\eqref{eqn_CovW_STRC_8} and \eqref{eqn_CovW_STRC_9}, yields for $\textbf{k}_1\neq \textbf{k}_2$ that
\begin{align*}
D_{\textbf{k}_1,\textbf{k}_2}
&=\sigma^4\bigg(\frac{ (1-e^{-\lambda_{\textbf{k}_1}\Delta_{2n}})^2(1-e^{-\lambda_{\textbf{k}_2}\Delta_{2n}})^2 }{8\lambda_{\textbf{k}_1}^{1+\alpha}\lambda_{\textbf{k}_2}^{1+\alpha}}\Big(\frac{e^{-\lambda_{\textbf{k}_1}\Delta_{2n}}+e^{-\lambda_{\textbf{k}_2}\Delta_{2n}}}{1-e^{-(\lambda_{\textbf{k}_1}+\lambda_{\textbf{k}_1})\Delta_{2n}}}-\frac{(1-e^{-2\lambda_{\textbf{k}_2}\Delta_{2n}})(1+e^{-2\lambda_{\textbf{k}_1}\Delta_{2n}})}{(1-e^{-(\lambda_{\textbf{k}_1}+\lambda_{\textbf{k}_2})\Delta_{2n}})(1-e^{-\lambda_{\textbf{k}_2}\Delta_{2n}})}\\
&~~~~~-
\frac{(1-e^{-2\lambda_{\textbf{k}_1}\Delta_{2n}})(1+e^{-2\lambda_{\textbf{k}_2}\Delta_{2n}})}{(1-e^{-\lambda_{\textbf{k}_1}\Delta_{2n}})(1-e^{-(\lambda_{\textbf{k}_1}+\lambda_{\textbf{k}_2})\Delta_{2n}})}+ \frac{(e^{-\lambda_{\textbf{k}_1}\Delta_{2n}}+e^{-\lambda_{\textbf{k}_2}\Delta_{2n}})(1-e^{-2\lambda_{\textbf{k}_1}\Delta_{2n}})(1-e^{-2\lambda_{\textbf{k}_2}\Delta_{2n}})}{(1-e^{-\lambda_{\textbf{k}_1}\Delta_{2n}})(1-e^{-\lambda_{\textbf{k}_2}\Delta_{2n}})(1-e^{-(\lambda_{\textbf{k}_1}+\lambda_{\textbf{k}_2})\Delta_{2n}})}\Big)\\
&~~~~~+\frac{(1-e^{-\lambda_{\textbf{k}_1}\Delta_{2n}})(1-e^{-\lambda_{\textbf{k}_2}\Delta_{2n}})}{8\lambda_{\textbf{k}_1}^{1+\alpha}\lambda_{\textbf{k}_2}^{1+\alpha}}\Big(\frac{e^{-\lambda_{\textbf{k}_1}\Delta_{2n}}(1-e^{-\lambda_{\textbf{k}_1}\Delta_{2n}})(1-e^{-2\lambda_{\textbf{k}_2}\Delta_{2n}})}{(1-e^{-\lambda_{\textbf{k}_2}\Delta_{2n}})}\\
&~~~~~+ \frac{e^{-\lambda_{\textbf{k}_2}\Delta_{2n}}(1-e^{-\lambda_{\textbf{k}_2}\Delta_{2n}})(1-e^{-2\lambda_{\textbf{k}_1}\Delta_{2n}})}{(1-e^{-\lambda_{\textbf{k}_1}\Delta_{2n}})}-\frac{(1-e^{-2\lambda_{\textbf{k}_1}\Delta_{2n}})(1-e^{-2\lambda_{\textbf{k}_2}\Delta_{2n}})}{(1-e^{-\lambda_{\textbf{k}_2}\Delta_{2n}})}\\
&~~~~~-\frac{(1-e^{-2\lambda_{\textbf{k}_1}\Delta_{2n}})(1-e^{-2\lambda_{\textbf{k}_2}\Delta_{2n}})}{(1-e^{-\lambda_{\textbf{k}_1}\Delta_{2n}})}
\Big)\bigg)\bigg(1+\Oo\Big(1\wedge\frac{p^{-1}}{1-e^{-(\lambda_{\textbf{k}_1}+\lambda_{\textbf{k}_2})\Delta_{2n}}}\Big)\bigg)\\
&=\sigma^4\bigg(\frac{ (1-e^{-\lambda_{\textbf{k}_1}\Delta_{2n}})^2(1-e^{-\lambda_{\textbf{k}_2}\Delta_{2n}})^2 }{8\lambda_{\textbf{k}_1}^{1+\alpha}\lambda_{\textbf{k}_2}^{1+\alpha}}\Big(-2+\frac{e^{-\lambda_{\textbf{k}_1}\Delta_{2n}}+e^{-\lambda_{\textbf{k}_2}\Delta_{2n}}}{1-e^{-(\lambda_{\textbf{k}_1}+\lambda_{\textbf{k}_2})\Delta_{2n}}}\Big)\\
&~~~~~-2\frac{(1-e^{-\lambda_{\textbf{k}_1}\Delta_{2n}})(1-e^{-\lambda_{\textbf{k}_2}\Delta_{2n}})}{8\lambda_{\textbf{k}_1}^{1+\alpha}\lambda_{\textbf{k}_2}^{1+\alpha}}(1-e^{-(\lambda_{\textbf{k}_1}+\lambda_{\textbf{k}_2})\Delta_{2n}})\bigg)\bigg(1+\Oo\Big(1\wedge\frac{p^{-1}}{1-e^{-(\lambda_{\textbf{k}_1}+\lambda_{\textbf{k}_2})\Delta_{2n}}}\Big)\bigg).
\end{align*}
Recalling the calculations of the covariance yields:
\begin{align*}
\Cov\big(V_{p,\Delta_{2n}}(\textbf{y}_1),W_{p,\Delta_{2n}}(\textbf{y}_2)\big)&=\frac{2e^{\nor{\kappa\bigcdot(\textbf{y}_1+\textbf{y}_2)}_1}\sigma^4}{p\Delta_{2n}^{2\alpha'}}\sum_{\substack{\textbf{k}_1,\textbf{k}_2\in\N^d \\ \textbf{k}_1\neq \textbf{k}_2}}e_{\textbf{k}_1}(\textbf{y}_1)e_{\textbf{k}_1}(\textbf{y}_2)e_{\textbf{k}_2}(\textbf{y}_1)e_{\textbf{k}_2}(\textbf{y}_2)\bar{D}_{\textbf{k}_1,\textbf{k}_2}\\
&~~~~~+\Oo\bigg(\frac{1}{p^2\Delta_{2n}^{2\alpha'}}\sum_{\substack{\textbf{k}_1,\textbf{k}_2\in\N^d \\ \textbf{k}_1\neq \textbf{k}_2}} \frac{\bar{D}_{\textbf{k}_1,\textbf{k}_2}}{1-e^{-(\lambda_{\textbf{k}_1}+\lambda_{\textbf{k}_2})\Delta_{2n}}}\bigg)\\
&~~~~~+\frac{2e^{\nor{\kappa\bigcdot(\textbf{y}_1+\textbf{y}_2)}_1}}{p\Delta_{2n}^{2\alpha'}}\sum_{\textbf{k}\in\N^d}e_{\textbf{k}}^2(\textbf{y}_1)e^2_{\textbf{k}}(\textbf{y}_2)D_{\textbf{k},\textbf{k}},
\end{align*}
where
\begin{align*}
\bar{D}_{\textbf{k}_1,\textbf{k}_2}&=\frac{ (1-e^{-\lambda_{\textbf{k}_1}\Delta_{2n}})^2(1-e^{-\lambda_{\textbf{k}_2}\Delta_{2n}})^2 }{8\lambda_{\textbf{k}_1}^{1+\alpha}\lambda_{\textbf{k}_2}^{1+\alpha}}\Big(-2+\frac{e^{-\lambda_{\textbf{k}_1}\Delta_{2n}}+e^{-\lambda_{\textbf{k}_2}\Delta_{2n}}}{1-e^{-(\lambda_{\textbf{k}_1}+\lambda_{\textbf{k}_2})\Delta_{2n}}}\Big)\\
&~~~~~-2\frac{(1-e^{-\lambda_{\textbf{k}_1}\Delta_{2n}})(1-e^{-\lambda_{\textbf{k}_2}\Delta_{2n}})}{8\lambda_{\textbf{k}_1}^{1+\alpha}\lambda_{\textbf{k}_2}^{1+\alpha}}(1-e^{-(\lambda_{\textbf{k}_1}+\lambda_{\textbf{k}_2})\Delta_{2n}}).
\end{align*}
First, we obtain for sufficiently large $p$ that
\begin{align*}
\Oo\bigg(\frac{1}{p^2\Delta_{2n}^{2\alpha'}}\sum_{\substack{\textbf{k}_1,\textbf{k}_2\in\N^d \\ \textbf{k}_1\neq \textbf{k}_2}} \frac{\bar{D}_{\textbf{k}_1,\textbf{k}_2}}{1-e^{-(\lambda_{\textbf{k}_1}+\lambda_{\textbf{k}_2})\Delta_{2n}}}\bigg)&=\Oo\bigg(\frac{1}{p^2}\Big(\sum_{\textbf{k}\in\N^d}\frac{1-e^{-\lambda_{\textbf{k}}\Delta_{2n}}}{2(\lambda_\textbf{k}\Delta_{2n})^{1+\alpha}}\Big)^2\bigg)=\Oo(p^{-2})
\end{align*}
where we used \ref{lemma_calcfAlphaDelta} and analogous steps as in Proposition \ref{prop_RRVMulti}. Hence, we obtain:
\begin{align*}
\Oo\bigg(\frac{1}{p^2\Delta_{2n}^{2\alpha'}}\sum_{\substack{\textbf{k}_1,\textbf{k}_2\in\N^d \\ \textbf{k}_1\neq \textbf{k}_2}} \frac{\bar{D}_{\textbf{k}_1,\textbf{k}_2}}{1-e^{-(\lambda_{\textbf{k}_1}+\lambda_{\textbf{k}_2})\Delta_{2n}}}\bigg)&=\Oo\bigg(\frac{1}{p}\Big(1\wedge \frac{1}{p}\Big)\bigg).
\end{align*}
Furthermore, we have for $\textbf{k}_1=\textbf{k}_2=\textbf{k}$ that
\begin{align*}
D_{\textbf{k},\textbf{k}}&=\frac{2}{p} \sum\limits_{i,j=1}^p
\mathbbm{1}_{2\N}(j)\E\Big[\big(\tilde{B}_{i,\textbf{k}}+C_{i,\textbf{k}}\big)\big(\tilde{B}_{j,\textbf{k}}+C_{j,\textbf{k}}\big)\Big]
\E\Big[\big(\tilde{B}_{i,\textbf{k}}+C_{i,\textbf{k}}\big)\big(\tilde{B}_{j-1,\textbf{k}}+C_{j-1,\textbf{k}}\big)\Big]\\
&=-\frac{(1-e^{-\lambda_{\textbf{k}}\Delta_{2n}})^2}{4\lambda_{\textbf{k}}^{2(1+\alpha)}}\bigg((1-e^{-\lambda_{\textbf{k}}\Delta_{2n}})^2-\frac{e^{-\lambda_{\textbf{k}}\Delta_{2n}}(1-e^{-\lambda_{\textbf{k}}\Delta_{2n}})^2}{1-e^{-2\lambda_{\textbf{k}}\Delta_{2n}}}+1-e^{-2\lambda_{\textbf{k}}\Delta_{2n}}\bigg)\\
&~~~~~\times\bigg(1+\Oo\Big(1\wedge\frac{p^{-1}}{1-e^{-2\lambda_{\textbf{k}}\Delta_{2n}}}\Big)\bigg).
\end{align*}
Defining the following term: 
\begin{align*}
\bar{D}_{\textbf{k},\textbf{k}}:=\frac{ (1-e^{-\lambda_{\textbf{k}}\Delta_{2n}})^4 }{8\lambda_{\textbf{k}}^{2(1+\alpha)}}\Big(-2+2\frac{e^{-\lambda_{\textbf{k}}\Delta_{2n}}}{1-e^{-2\lambda_{\textbf{k}}\Delta_{2n}}}\Big)-2\frac{(1-e^{-\lambda_{\textbf{k}}\Delta_{2n}})^2}{8\lambda_{\textbf{k}}^{2(1+\alpha)}}(1-e^{-2\lambda_{\textbf{k}}\Delta_{2n}}),
\end{align*}
yields:
\begin{align*}
\frac{1}{\Delta_{2n}^{2\alpha'}p}\sum_{\textbf{k}\in\N^d}\bar{D}_{\textbf{k},\textbf{k}}=\Oo\bigg(\frac{\Delta_{2n}^{d/2}}{p}\Delta_{2n}^{d/2}\sum_{\textbf{k}\in\N^d}\Big(\frac{(1-e^{-\lambda_{\textbf{k}}\Delta_{2n}})}{2(\lambda_{\textbf{k}}\Delta_{2n})^{1+\alpha}}\Big)^2\bigg)=\Oo\big(p^{-1}\Delta_{2n}^{2(1-\alpha')}\big),
\end{align*}
where we used analogous steps as in display \eqref{eqn_negelct_kkTerminCov}. We decompose the leading term $\bar{D}_{\textbf{k}_1,\textbf{k}_2}$ as follows:
\begin{align*}
\bar{D}_{\textbf{k}_1,\textbf{k}_2}&=\bar{D}_{\textbf{k}_1,\textbf{k}_2}^1+\bar{D}_{\textbf{k}_1,\textbf{k}_2}^2+\bar{D}_{\textbf{k}_1,\textbf{k}_2}^3+\bar{D}_{\textbf{k}_1,\textbf{k}_2}^4,
\end{align*}
where
\begin{align*}
\bar{D}_{\textbf{k}_1,\textbf{k}_2}^1&= -\frac{ (1-e^{-\lambda_{\textbf{k}_1}\Delta_{2n}})^2(1-e^{-\lambda_{\textbf{k}_2}\Delta_{2n}})^2 }{4\lambda_{\textbf{k}_1}^{1+\alpha}\lambda_{\textbf{k}_2}^{1+\alpha}}=-\frac{\Delta_{2n}^{2(1+\alpha)}}{4}f_{2,\alpha}(\lambda_{\textbf{k}_1}\Delta_{2n})f_{2,\alpha}(\lambda_{\textbf{k}_2}\Delta_{2n}),\\
\bar{D}_{\textbf{k}_1,\textbf{k}_2}^2&= \frac{ (1-e^{-\lambda_{\textbf{k}_1}\Delta_{2n}})^2(1-e^{-\lambda_{\textbf{k}_2}\Delta_{2n}})^2 }{8\lambda_{\textbf{k}_1}^{1+\alpha}\lambda_{\textbf{k}_2}^{1+\alpha}}\cdot \frac{e^{-\lambda_{\textbf{k}_1}\Delta_{2n}}+e^{-\lambda_{\textbf{k}_2}\Delta_{2n}}}{1-e^{-(\lambda_{\textbf{k}_1}+\lambda_{\textbf{k}_2})\Delta_{2n}}}\\
&=\frac{\Delta_{2n}^{2(1+\alpha)}}{2}\sum_{r=0}^\infty\big( g_{1,\alpha,r+1}(\lambda_{\textbf{k}_1}\Delta_{2n})g_{1,\alpha,r}(\lambda_{\textbf{k}_2}\Delta_{2n})+ g_{1,\alpha,r+1}(\lambda_{\textbf{k}_2}\Delta_{2n})g_{1,\alpha,r}(\lambda_{\textbf{k}_1}\Delta_{2n})\big),\\
\bar{D}_{\textbf{k}_1,\textbf{k}_2}^3&=-\frac{(1-e^{-\lambda_{\textbf{k}_1}\Delta_{2n}})(1-e^{-\lambda_{\textbf{k}_2}\Delta_{2n}})}{4\lambda_{\textbf{k}_1}^{1+\alpha}\lambda_{\textbf{k}_2}^{1+\alpha}}=-\frac{\Delta_{2n}^{2(1+\alpha)}}{4}f_{1,\alpha}(\lambda_{\textbf{k}_1}\Delta_{2n})f_{1,\alpha}(\lambda_{\textbf{k}_2}\Delta_{2n})  ,\\
\bar{D}_{\textbf{k}_1,\textbf{k}_2}^4&= \frac{(1-e^{-\lambda_{\textbf{k}_1}\Delta_{2n}})(1-e^{-\lambda_{\textbf{k}_2}\Delta_{2n}})}{4\lambda_{\textbf{k}_1}^{1+\alpha}\lambda_{\textbf{k}_2}^{1+\alpha}}e^{-(\lambda_{\textbf{k}_1}+\lambda_{\textbf{k}_2})\Delta_{2n}}=\Delta_{2n}^{2(1+\alpha)}g_{2,\alpha,1}(\lambda_{\textbf{k}_1}\Delta_{2n})g_{2,\alpha,1}(\lambda_{\textbf{k}_2}\Delta_{2n}).
\end{align*}
Here, we use the following functions defined by: 
\begin{align*}
f_{1,\alpha}(x):=f_{\alpha}(x)=\frac{1-e^{-x}}{x^{1+\alpha}}&,~~~~~f_{2,\alpha}(x):=\frac{(1-e^{-x})^2}{x^{1+\alpha}}\\
g_{1,\alpha,\tau}(x):=g_{\alpha,\tau}(x)=\frac{(1-e^{-x})^2}{2x^{1+\alpha}}e^{-\tau x}&,~~~~~g_{2,\alpha,\tau}(x):=\frac{1-e^{-x}}{2x^{1+\alpha}}e^{-\tau x}.
\end{align*}
By Lemma \ref{lemma_checkConditionsAprroxlemma}, we know that $f_{1,\alpha}\in\mathcal{Q}_{\beta_1}$ and $g_{1,\alpha,\tau}\in\mathcal{Q}_{\beta_2}$, where $\beta_1=\big(2\alpha,2(1+\alpha),2(2+2\alpha)\big)$ and $\beta_2=\big(2\alpha,2(1+\alpha),2(1+2\alpha)\big)$. By analogous computations as used in Lemma \ref{lemma_checkConditionsAprroxlemma}, we obtain that $f_{2,\alpha}\in\mathcal{Q}_{\beta_1}$ and $g_{2,\alpha,\tau}\in\mathcal{Q}_{\beta_1}$. Assume $\textbf{y}_1\neq \textbf{y}_2$. We can repeat the calculations leading to equation \eqref{eqn_covEksRemainder} and have 
\begin{align*}
\Cov\big(V_{p,\Delta_{2n}}(\textbf{y}_1),W_{p,\Delta_{2n}}(\textbf{y}_2)\big)&=\Oo\bigg(\frac{\Delta_{2n}^{1-\alpha'}}{p}\big(\nor{\textbf{y}_1-\textbf{y}_2}_0^{-(d+1)}+\delta^{-(d+1)}\big)\bigg)+\mathcal{O}\bigg(\frac{1}{p}\Big(\Delta_{2n}^{2(1-\alpha')}+\frac{1}{p}\wedge1\Big)\bigg)\\
&=\Oo\bigg(\frac{\Delta_{2n}^{1-\alpha'}}{p}\big(\nor{\textbf{y}_1-\textbf{y}_2}_0^{-(d+1)}+\delta^{-(d+1)}\big)\bigg).
\end{align*}
Therefore, it remains to analyse the case where $\textbf{y}_1= \textbf{y}_2$. Again, utilizing that the function $f_{1,\alpha}$, $f_{2,\alpha}$ and $g_{2,\alpha,\tau}$ are in the same class $\mathcal{Q}_{\beta_1}$ as the function $f_{\alpha}$ as used in Proposition \ref{prop_RRVMulti}, as well as $g_{1,\alpha,\tau}=g_{\alpha,\tau}$, we can conclude analogous to equation \eqref{eqn_rrvPropEqnToLastTerm} that
\begin{align*}
\Cov\big(V_{p,\Delta_{2n}}(\textbf{y}_1),W_{p,\Delta_{2n}}(\textbf{y}_2)\big)&=\frac{2\sigma^4}{p\Delta_{2n}^{2\alpha'}}\sum_{\substack{\textbf{k}_1,\textbf{k}_2\in\N^d \\ \textbf{k}_1\neq \textbf{k}_2}}\bar{D}_{\textbf{k}_1,\textbf{k}_2}+\mathcal{O}\bigg(\frac{1}{p}\Big(\Delta_{2n}^{1/2}\vee \frac{\Delta_{2n}^{1-\alpha'}}{\delta^{d+1}}+\frac{1}{p}\wedge1\Big)\bigg).
\end{align*}
First, we obtain that
\begin{align*}
\frac{1}{\Delta_{2n}^{2\alpha'}}\sum_{\substack{\textbf{k}_1,\textbf{k}_2\in\N^d \\ \textbf{k}_1\neq \textbf{k}_2}}\bar{D}_{\textbf{k}_1,\textbf{k}_2}^1&=-\frac{1}{4}\bigg(\Delta_{2n}^{d/2}\sum_{\textbf{k}\in\N^d}f_{2,\alpha}(\lambda_{\textbf{k}}\Delta_{2n})\bigg)^2:=I_1,\\
\frac{1}{\Delta_{2n}^{2\alpha'}}\sum_{\substack{\textbf{k}_1,\textbf{k}_2\in\N^d \\ \textbf{k}_1\neq \textbf{k}_2}}\bar{D}_{\textbf{k}_1,\textbf{k}_2}^2&=\frac{2}{2}\sum_{r=0}^\infty \bigg(\Delta_{2n}^{d/2}\sum_{\textbf{k}\in\N^d}g_{1,\alpha,r+1}(\lambda_{\textbf{k}}\Delta_{2n}) \bigg)\bigg(\Delta_{2n}^{d/2}\sum_{\textbf{k}\in\N^d} g_{1,\alpha,r}(\lambda_{\textbf{k}}\Delta_{2n}) \bigg):=I_2,\\
\frac{1}{\Delta_{2n}^{2\alpha'}}\sum_{\substack{\textbf{k}_1,\textbf{k}_2\in\N^d \\ \textbf{k}_1\neq \textbf{k}_2}}\bar{D}_{\textbf{k}_1,\textbf{k}_2}^3&=-\frac{1}{4}\bigg(\Delta_{2n}^{d/2}\sum_{\textbf{k}\in\N^d}f_{1,\alpha}(\lambda_{\textbf{k}}\Delta_{2n})\bigg)^2:=I_3,\\
\frac{1}{\Delta_{2n}^{2\alpha'}}\sum_{\substack{\textbf{k}_1,\textbf{k}_2\in\N^d \\ \textbf{k}_1\neq \textbf{k}_2}}\bar{D}_{\textbf{k}_1,\textbf{k}_2}^3&=\bigg(\Delta_{2n}^{d/2}\sum_{\textbf{k}\in\N^d}g_{2,\alpha,1}(\lambda_{\textbf{k}}\Delta_{2n})\bigg)^2:=I_4.
\end{align*}
Using Corollary \ref{corollary_toLemmaRiemannApproxMulti} and Lemma \ref{lemma_calcfAlphaDelta} yields:
\begin{align*}
I_1&=-\frac{1}{4}\bigg(\frac{1}{2^d(\pi\eta)^{d/2}\Gamma(d/2)}\bigg)^2\bigg(\frac{(2-2^{\alpha })\pi }{\Gamma (1+\alpha )\sin(\pi  \alpha)}\bigg)^2=-\frac{1}{4}\bigg(\frac{\Gamma(1-\alpha')}{2^d(\pi\eta)^{d/2}\alpha'\Gamma(d/2)}\bigg)^2(2^{\alpha'}-2)^2,\\
I_2&=\frac{1}{4}\bigg(\frac{\Gamma(1-\alpha')}{2^d(\pi\eta)^{d/2}\alpha'\Gamma(d/2)}\bigg)^2\sum_{r=0}^\infty \left(-(r+1) ^{\alpha'}+2 (r +2)^{\alpha'}-(r+3)^{\alpha'}\right)\left(-r ^{\alpha'}+2 (r +1)^{\alpha'}-(r +2)^{\alpha'}\right),\\
I_3&=-\frac{1}{4}\bigg(\frac{\Gamma(1-\alpha')}{2^d(\pi\eta)^{d/2}\alpha'\Gamma(d/2)}\bigg)^2,\\
I_4&=\frac{1}{4}\bigg(\frac{\Gamma(1-\alpha')}{2^d(\pi\eta)^{d/2}\alpha'\Gamma(d/2)}\bigg)^2(2^{\alpha'}-1)^2.
\end{align*}
Hence, we obtain that
\begin{align*}
&\Cov\big(V_{p,\Delta_{2n}}(\textbf{y}_1),W_{p,\Delta_{2n}}(\textbf{y}_2)\big)=\frac{1}{2p}\bigg(\frac{\Gamma(1-\alpha')\sigma^4}{2^d(\pi\eta)^{d/2}\alpha'\Gamma(d/2)}\bigg)^2\bigg((2^{\alpha'}-1)^2-(2^{\alpha'}-2)^2-1\\
&~~~~~+\sum_{r=0}^\infty \left(-(r+1) ^{\alpha'}+2 (r +2)^{\alpha'}-(r+3)^{\alpha'}\right)\left(-r ^{\alpha'}+2 (r +1)^{\alpha'}-(r +2)^{\alpha'}\right)\bigg)\\
&~~~~~+\mathcal{O}\bigg(\frac{1}{p}\Big(\Delta_{2n}^{1/2}\vee \frac{\Delta_{2n}^{1-\alpha'}}{\delta^{d+1}}+\frac{1}{p}\wedge1\Big)\bigg).
\end{align*}
Defining the following constant:
\begin{align}
\Lambda_{\alpha'}:= 2(2^{\alpha'}-2)+\sum_{r=0}^\infty \left(-(r+1) ^{\alpha'}+2 (r +2)^{\alpha'}-(r+3)^{\alpha'}\right)\left(-r ^{\alpha'}+2 (r +1)^{\alpha'}-(r +2)^{\alpha'}\right),\label{eqn_definingLambda}
\end{align}
completes the proof.
\end{proof}

\section*{Acknowledgement}
I wish to express my appreciation to my Ph.D. advisor, Markus Bibinger, for the careful reading of this manuscript and his useful suggestions.

\addcontentsline{toc}{section}{References}
\bibliographystyle{plain}
\bibliography{library}

\begin{thebibliography}{10}

\bibitem{randolf2022}
Randolf Altmeyer, Till Bretschneider, Josef Janák, and Markus Reiß.
\newblock Parameter estimation in an spde model for cell repolarization.
\newblock {\em SIAM/ASA Journal on Uncertainty Quantification},
  10(1):179--199\EatDot, 2022.

\bibitem{randolf2021}
Randolf Altmeyer and Markus Reiß.
\newblock {Nonparametric estimation for linear SPDEs from local measurements}.
\newblock {\em The Annals of Applied Probability}, 31(1):1 -- 38\EatDot, 2021.

\bibitem{arakawa1999multiple}
Tsuneo Arakawa and Masanobu Kaneko.
\newblock Multiple zeta values, poly-bernoulli numbers, and related zeta
  functions.
\newblock {\em Nagoya Mathematical Journal}, 153:189--209, 1999.

\bibitem{bibinger2023efficient}
Markus Bibinger and Patrick Bossert.
\newblock Efficient parameter estimation for parabolic spdes based on a
  log-linear model for realized volatilities.
\newblock {\em Japanese Journal of Statistics and Data Science}, pages 1--23,
  2023.

\bibitem{trabs}
Markus Bibinger and Mathias Trabs.
\newblock Volatility estimation for stochastic pdes using high-frequency
  observations.
\newblock {\em Stochastic Processes and their Applications}, 130(5):3005 --
  3052\EatDot, 2020.

\bibitem{chong2020high}
Carsten Chong.
\newblock High-frequency analysis of parabolic stochastic pdes.
\newblock {\em The Annals of Statistics}, 48(2):1143--1167, 2020.

\bibitem{cialenco}
Igor Cialenco and Yicong Huang.
\newblock A note on parameter estimation for discretely sampled spdes.
\newblock {\em Stochastics and Dynamics}, 20(03):2050016\EatDot, 2020.

\bibitem{daPrato}
Giuseppe Da~Prato and Jerzy Zabczyk.
\newblock {\em Stochastic equations in infinite dimensions}.
\newblock Cambridge university press, 2014.

\bibitem{davie2001convergence}
A~Davie and J~Gaines.
\newblock Convergence of numerical schemes for the solution of parabolic
  stochastic partial differential equations.
\newblock {\em Mathematics of Computation}, 70(233):121--134, 2001.

\bibitem{fioravanti2023interpolating}
Guido Fioravanti, Sara Martino, Michela Cameletti, and Andrea Toreti.
\newblock Interpolating climate variables by using inla and the spde approach.
\newblock {\em International Journal of Climatology}, 2023.

\bibitem{fuglstad}
Geir-Arne Fuglstad and Stefano Castruccio.
\newblock {Compression of climate simulations with a nonstationary global
  SpatioTemporal SPDE model}.
\newblock {\em The Annals of Applied Statistics}, 14(2):542 -- 559\EatDot,
  2020.

\bibitem{gun2018multiple}
Sanoli Gun and Biswajyoti Saha.
\newblock Multiple lerch zeta functions and an idea of ramanujan.
\newblock {\em Michigan Mathematical Journal}, 67(2):267--287, 2018.

\bibitem{hambly}
Ben Hambly and Andreas S{\o}jmark.
\newblock An spde model for systemic risk with endogenous contagion.
\newblock {\em Finance and Stochastics}, 23(3):535--594\EatDot, 2019.

\bibitem{hildebrandt2020}
Florian Hildebrandt.
\newblock On generating fully discrete samples of the stochastic heat equation
  on an interval.
\newblock {\em Statistics \& Probability Letters}, 162:108750\EatDot, 2020.

\bibitem{hildebrandt}
Florian Hildebrandt and Mathias Trabs.
\newblock {Parameter estimation for SPDEs based on discrete observations in
  time and space}.
\newblock {\em Electronic Journal of Statistics}, 15(1):2716 -- 2776\EatDot,
  2021.

\bibitem{isserlis1918formula}
Leon Isserlis.
\newblock On a formula for the product-moment coefficient of any order of a
  normal frequency distribution in any number of variables.
\newblock {\em Biometrika}, 12(1/2):134--139, 1918.

\bibitem{jacod2011discretization}
Jean Jacod and Philip Protter.
\newblock {\em Discretization of processes}, volume~67.
\newblock Springer Science \& Business Media, 2011.

\bibitem{uchida}
Yusuke Kaino and Masayuki Uchida.
\newblock Parametric estimation for a parabolic linear spde model based on
  sampled data.
\newblock arXiv:1909.13557, 2019.

\bibitem{lindgren2022spde}
Finn Lindgren, David Bolin, and H{\aa}vard Rue.
\newblock The spde approach for gaussian and non-gaussian fields: 10 years and
  still running.
\newblock {\em Spatial Statistics}, 50:100599, 2022.

\bibitem{lototsky2017stochastic}
Sergey~V Lototsky, Boris~L Rozovsky, et~al.
\newblock {\em Stochastic partial differential equations}.
\newblock Springer, 2017.

\bibitem{mena2019efficient}
Hermann Mena and Lena Pfurtscheller.
\newblock An efficient spde approach for el ni{\~n}o.
\newblock {\em Applied Mathematics and Computation}, 352:146--156, 2019.

\bibitem{utev}
Magda Peligrad, Sergey Utev, et~al.
\newblock Central limit theorem for linear processes.
\newblock {\em The Annals of Probability}, 25(1):443--456, 1997.

\bibitem{pereira2020matrix}
Mike Pereira, Nicolas Desassis, C{\'e}dric Magneron, and Nathan Palmer.
\newblock A matrix-free approach to geostatistical filtering.
\newblock {\em arXiv preprint arXiv:2004.02799}, 2020.

\bibitem{tonaki2023parameter}
Yozo Tonaki, Yusuke Kaino, and Masayuki Uchida.
\newblock Parameter estimation for linear parabolic spdes in two space
  dimensions based on high frequency data.
\newblock {\em Scandinavian Journal of Statistics}, 2023.

\end{thebibliography}

\end{document}